\documentclass[12pt]{amsart}
\usepackage{amsfonts}
\usepackage{mathrsfs}
\usepackage{amssymb}
\usepackage{amsmath}
\usepackage{tikz-cd}
\usepackage{centernot}
\usepackage{enumerate}
\usepackage{relsize}
\usepackage{scalerel}
\usepackage{bm}
\usepackage{amsrefs}
\usepackage{stmaryrd}
\usepackage{graphicx}
\usepackage{fullpage}
\usepackage{mathtools}
\usepackage[T2A,T1]{fontenc}
\usepackage[utf8]{inputenc}
\usepackage[russian,english]{babel}

\newtheorem{theorem}{Theorem}[subsection]
\newtheorem{lemma}[theorem]{Lemma}
\newtheorem{proposition}[theorem]{Proposition}
\newtheorem{corollary}[theorem]{Corollary}

\theoremstyle{definition}
\newtheorem{definition}[theorem]{Definition}

\usepackage{tikz}
\usetikzlibrary{calc,decorations.markings}
\pgfdeclarelayer{edgelayer}
\pgfdeclarelayer{nodelayer}
\pgfsetlayers{edgelayer,nodelayer,main}
\tikzstyle{none}=[inner sep=1pt]
\tikzstyle{None}=[inner sep=1pt, fill=white]
\tikzstyle{dashedcircle}=[circle, draw=gray, dashed, inner sep=6pt]
\tikzstyle{box}=[draw=black, fill=white, inner sep=.5ex, rounded corners=.1ex]
\tikzstyle{roundedbox}=[draw=black, fill=white, inner sep=.5ex, rounded corners=1ex]
\tikzstyle{cross}=[preaction={draw=white, -, line width=3pt}]
\tikzstyle{arrow}=[postaction=decorate]
\newcommand{\markat}{0.5}
\newcommand{\markwithsym}{>}
\newcommand{\markwith}{{\arrow[black]{\markwithsym}}}
\pgfkeys{/tikz/.cd, markat/.store in=\markat, markwith/.store in=\markwithsym} 
\tikzset{decoration={markings, mark=at position \markat with \markwith}}

\usepackage[T3,T1]{fontenc}
\DeclareSymbolFont{tipa}{T3}{cmr}{m}{n}
\DeclareMathAccent{\bend}{\mathalpha}{tipa}{16}

\newcommand{\Q}{\mathcal Q}

\newcommand{\W}{\mathcal W}
\newcommand{\X}{\mathcal X}
\newcommand{\Y}{\mathcal Y}
\newcommand{\Z}{\mathcal Z}

\newcommand{\CC}{\mathbb C}

\newcommand{\RR}{\mathbb R}

\newcommand{\<}{\langle}
\renewcommand{\>}{\rangle}
\renewcommand{\:}{\colon}
\newcommand{\inv}{^{-1}}
\newcommand{\subsetof}{\subseteq}
\newcommand{\suchthat}{\,|\,}
\newcommand{\tensor}{\mathbin{\otimes}}
\newcommand{\stensor}{\mathbin{\bar{\otimes}}}
\newcommand{\iso}{\cong}
\newcommand{\AND}{\wedge}
\newcommand{\OR}{\vee}
\newcommand{\NOT}{\neg}
\newcommand{\IMPLIES}{\rightarrow}
\newcommand{\IFF}{\leftrightarrow}

\newcommand{\feq}{\!=\!}
\newcommand{\fin}{\!\in\!}
\newcommand{\ftimes}{\!\times\!}
\newcommand{\Tr}{\mathrm{Tr}}
\newcommand{\At}{\mathrm{At}}
\newcommand{\Rel}{\mathrm{Rel}}
\newcommand{\Pred}{\mathrm{Pred}}
\renewcommand{\[}{\llbracket\,}
\renewcommand{\]}{\,\rrbracket}
\newcommand{\counit}{\epsilon}
\newcommand{\inc}{\mathrm{inc}}
\newcommand{\incdag}{\mathrm{inc}^\dagger}
\newcommand{\incnag}{\mathrm{inc}^{\phantom \dagger}}
\newcommand{\sascon}{\mathbin{\&}}
\newcommand{\ghost}{\vphantom{\cdot}}

\newcommand{\EV}{\quad\Longleftrightarrow\quad}
\newcommand{\RRm}{{[0,\infty]}}
\newcommand{\idhom}{\mathnormal{1}}
\newcommand{\Rep}{\mathrm{Rep}}
\newcommand{\atensor}{\mathbin{\odot}}
\newcommand{\Mor}{\mathrm{Mor}}
\newcommand{\Mult}{\mathrm{Mult}}
\newcommand{\cat}{\mathbf}
\newcommand{\evaluatesto}{\curvearrowright}

\allowdisplaybreaks

\begin{document}

\title{Discrete quantum structures}
\author{Andre Kornell}
\address{Department of Computer Science, Tulane University, New Orleans, Louisiana 70118}
\email{akornell@tulane.edu}
\thanks{This work was partially supported by the AFOSR under MURI grant FA9550-16-1-0082.}

\begin{abstract}
A majority of established quantum generalizations of discrete structures are shown to be instances of a single quantum generalization. In particular, the quantum graphs of Duan, Severini and Winter, the quantum metric spaces of Kuperberg and Weaver, the quantum isomorphisms of Atserias, Man\v{c}inska, Roberson, \v{S}\'{a}mal, Severini and Varvitsiotis, and the quantum groups of Woronowicz that are all discrete in the sense that the underlying von Neumann algebra is hereditarily atomic are shown to be subclasses of a single class of discrete quantum structures. Such a discrete quantum structure is defined to be a discrete quantum space equipped with relations and functions of various arities. Weaver's quantum predicate logic, a generalization of the quantum propositional logic of Birkhoff and von Neumann, provides canonical quantum generalizations for a large class of properties. The equality relation on discrete quantum spaces that is introduced here plays a central role in this approach to mathematical quantization.
\end{abstract}

\maketitle

\section{Introduction}

\subsection{Equality}\label{introduction.A}

This paper establishes a connection between quantum logic and discrete noncommutative mathematics. The study of quantum logic was initiated by Birkhoff and von Neumann, who drew an analogy between the lattice of projection operators in a von Neumann algebra and the lattice of measurable subsets of a measure space, modulo null sets \cite{BirkhoffVonNeumann}*{secs.~5, 6}, providing our interpretation of the Boolean connectives $\neg$, $\wedge$ and $\vee$. The lattice of projection operators was then investigated as a propositional logic, providing our interpretation of the Boolean connective $\rightarrow$ \cite{Sasaki}\cite{Kunsemuller}\cite{Finch}\cite{Hardegree}. Weaver extended this quantum propositional logic to a quantum predicate logic, providing our interpretation of the quantifiers $\forall$ and $\exists$ \cite{Weaver}*{sec.~2.6}. Motivated by the same physical and logical considerations, we continue this line of research by suggesting an interpretation of the equality relation.

Noncommutative mathematics in the sense of noncommutative geometry may be said to originate with the observation of Gelfand and Na\u{\i}mark that commutative unital C*-algebras are in duality with compact Hausdorff spaces \cite{GelfandNaimark}*{Lem.~1}. The notion of a locally compact quantum space, i.e., a pseudospace, as an object that is formally dual to a C*-algebra was first clearly enunciated by Woronowicz \cite{Woronowicz}*{sec.~1}. The notion of a discrete quantum space as an object that is formally dual to a $c_0$-direct sum of full matrix algebras then appeared implicitly in the work of Podle\'s and Woronowicz on Potryagin duality for compact quantum groups \cite{PodlesWoronowicz}*{sec.~2}. The discrete quantum structures considered in this paper are all essentially discrete quantum spaces equipped with additional structure.

Discrete quantum structures generalize the structures of many-sorted first-order logic \cite{Schmidt}, which consist of sets equipped with functions and relations, e.g., groups, graphs and vector spaces. Discrete quantum structures have not been previously considered in full generality, but their definition is already implicit in the established generalizations of noncommutative mathematics. For technical simplicity, we generalize sets to von Neumann algebras that are $\ell^\infty$-direct sums of full matrix algebras, rather than to C*-algebras that are $c_0$-direct sums of full matrix algebras. These are the hereditarily atomic von Neumann algebras \cite{Kornell}*{Prop.~5.4}.
We generalize the Cartesian product to the spatial tensor product, we generalize relations to projections, and we generalize functions to unital normal $*$-homomorphisms in the opposite direction.

The only complication to this straightforward narrative is that sometimes the order of multiplication in a von Neumann algebra $M$ is unexpectedly reversed. For example, the multiplication of a discrete quantum group is a unital normal $*$-homomorphism $M \to M \stensor M$, but the adjacency relation of a discrete quantum graph is a projection in $M \stensor M^{op}.$ Of course, if $M$ is commutative, then $M^{op} = M$, so this complication is a phenomenon that is peculiar to the quantum setting.

For each von Neumann algebra $M$, we interpret the equality relation to be the largest projection $\delta_M \in M \stensor M^{op}$ that is orthogonal to $p \tensor (1-p)$ for every projection $p \in M$. If $M = \ell^\infty(A)$ for some set $A$, then $\delta_M$ is the projection that corresponds to the diagonal of the Cartesian square $A \times A$. However, if $M = L^\infty(\RR)$, then $\delta_M = 0$. The intuitive explanation for this phenomenon is that the diagonal of the Cartesian square $\RR \times \RR$ has Lebesgue measure zero. We show that the equality relation $\delta_M$ is suitably nondegenerate for precisely the class of hereditarily atomic von Neumann algebras:

\begin{theorem}[also \ref{appendix.D.2}]
Let $M$ be a von Neumann algebra, and let $\delta_M \in M \stensor M^{op}$ be the largest projection that is orthogonal to $p \tensor (1-p)$ for every projection $p \in M$. Then, the following are equivalent:
\begin{enumerate}
\item $\delta_M$ is not orthogonal to $p \tensor p$ for any nonzero projection $p \in M$;
\item $M$ is hereditarily atomic.
\end{enumerate}
\end{theorem}

This theorem provides an additional justification for our focus on this class of von Neumann algebras. It also provides an additional characterization of this class \cite{Kornell}*{Prop.~5.4}. In contrast, it is routine to verify that $0$ is the only projection in $M_2(\CC) \stensor M_2(\CC)$ that is orthogonal to $p \tensor (1-p)$ for every projection $p \in M_2(\CC)$.

The projection $\delta_M$ may be equivalently defined as the infimum of all projections of the form $p \tensor p + (1-p) \tensor (1-p)$, where $p$ is a projection in $M$. To interpret this definition physically, we regard von Neumann algebras as abstract physical systems and their projections as Boolean observables on those physical systems. Hereditarily atomic von Neumann algebras correspond to those physical systems that are discrete in the sense that each observable admits a complete set of pairwise orthogonal eigenstates \cite{Kornell}*{Prop.~5.4}. The spatial tensor product of two hereditarily atomic von Neumann algebras corresponds to the composition of two spatially separated discrete physical systems, because the spatial tensor product coincides with the categorical tensor product in this case \cite{Guichardet}*{Prop.~8.6}.

Let $M$ be a hereditarily atomic von Neumann algebra. From the physical perspective, $\delta_M$ is a Boolean observable on the composite physical system $M \stensor M^{op}$ that guarantees equal outcomes for pairs of equal Boolean observables on $M$ and $M^{op}$. Indeed, $M$ and $M^{op}$ formally have the same projections. Furthermore, $\delta_M$ is the largest Boolean observable with this property. Thus, $\delta_M$ may be characterized as the Boolean observable that is true in exactly those states on $M \stensor M^{op}$ that guarantee equal outcomes for pairs of equal Boolean observables. Such states are used in perfect quantum strategies for synchronous games \cite{AtseriasMancinskaRobersonSamalSeveriniVarvitsiotis}*{sec.~5.2}\cite{MancinskaRobersonVaritsiotis}\cite{PaulsenSeveriniStahlkeTodorovWinter}\cite{CameronMontanaroNewmanSeveriniWinter}.

Viewed as abstract physical systems, $M$ and $M^{op}$ have exactly the same states and observables. Probed separately, their physics is indistinguishable. However, the composite systems $M\stensor M^{op}$ and $M \stensor M$ exhibit different physics. Of course, $M\stensor M^{op}$ and $M \stensor M$ are isomorphic as von Neumann algebras because $M$ is hereditarily atomic, but there is generally no isomorphism between them that fixes the projections of both tensor factors, i.e., the Boolean observables of both subsystems, as the example $M = M_2(\CC)$ demonstrates. Fancifully, we might regard $M$ and $M^{op}$ as otherwise isomorphic physical systems that are oriented oppositely in time. 

To motivate this intuition, we consider the example of an electron-positron pair that is produced by a neutral pion decay \cite{Drell}. We may model the spin of the electron by $M_2(\CC)$ and the spin of the positron by $M_2(\CC)^{op}$. The conservation of angular momentum then implies that their magnetic moments are equal along any axis of measurement. Thus, the composite system $M_2(\CC) \stensor M_2(\CC)^{op}$ is in the unique state such that a measurement of $\delta_{M_2(\CC)}$ is guaranteed to yield $1$. Furthermore, this observation demonstrates that the spin of the positron must be modeled by the opposite operator algebra, because the composite system must have a state with zero total angular momentum. Of course, a positron is sometimes regarded as an electron traveling backward in time \cite{Stueckelberg}\cite{Feynman}.

\subsection{Quantum sets}

The development in this paper proceeds in terms of quantum sets, their functions and their binary relations. The reader may choose to view each quantum set $\X$ as an object that is formally dual to a hereditarily atomic von Neumann algebra, in the same way that a pseudospace is formally dual to a C*-algebra \cite{Woronowicz}. In this account, the category of quantum sets and functions is defined to be the opposite of the category of hereditarily atomic von Neumann algebras and unital normal $*$-homomorphisms, and the category of quantum sets and binary relations is defined to be the opposite of the category of hereditarily atomic von Neumann algebras and Weaver's quantum relations \cite{Weaver2}. The former category is then included into the latter category via the equivalence in \cite{Kornell2}.

Formally, we instead define a quantum set $\X$ to be an object whose data consists of a set $\At(\X)$ of nonzero finite-dimensional Hilbert spaces called the atoms of $\X$ \cite{Kornell}. The corresponding hereditarily atomic von Neumann algebra $\ell^\infty(\X)$ is then defined to be the $\ell^\infty$-direct sum of the factors $L(X)$ for $X \in \At(\X)$. Each quantum set $\X$ has a dual $\X^*$ that is obtained by dualizing all the atoms of $\X$, and $\ell^\infty(\X^*)$ is naturally isomorphic to $\ell^\infty(\X)^{op}$. Each pair of quantum sets, $\X$ and $\Y$, has a Cartesian product $\X \times \Y$ that is obtained by forming all possible tensor products of an atom of $\X$ with an atom of $\Y$, and $\ell^\infty(\X \times \Y)$ is naturally isomorphic to $\ell^\infty(\X) \stensor \ell^\infty(\Y)$.

A binary relation $R$ from a quantum set $\X$ to a quantum set $\Y$ is just a choice of subspaces $R(X,Y) \leq L(X,Y)$ for all $X \in \At(\X)$ and $Y \in \At(\Y)$. Binary relations from $\X$ to $\Y$ are in one-to-one correspondence with projections in the hereditarily atomic von Neumann algebra $\ell^\infty(\X) \stensor \ell^\infty(\Y)^{op}$ \cite{Weaver2}*{Prop.~2.23}. The one-to-one correspondence between binary relations from $\X$ to $\Y$ and  quantum relations from $\ell^\infty(\Y)$ to $\ell^\infty(\X)$ in Weaver's sense is verified in Appendix \ref{appendix.E}. The category $\cat{qRel}$ of quantum sets and binary relations is dagger compact \cite{Kornell}*{Thm.~3.6}, i.e., strongly compact, and this enables our extensive use of the graphical calculus \cite{AbramskyCoecke}\cite{Penrose}.

The semantics that we define in this paper assigns an interpretation to each nonduplicating term and to each nonduplicating formula in a language of many-sorted first-order logic that draws its nonlogical symbols from the category $\cat{qRel}$. The qualifier ``nonduplicating'' refers to a syntactic constraint that reflects the absence of a diagonal function for quantum sets and for quantum spaces more generally \cite{Woronowicz} and the impossibility of broadcasting quantum states \cite{BarnumCavesFuchsJozsaSchumacher}. The sorts of our language are quantum sets. Its relation symbols are binary relations into the monoidal unit $\mathbf 1$, and its function symbols are binary relations that are functions \cite{Kornell}*{Def.~4.1}. The equality symbol for sort $\X$ is a binary relation from $\X \times \X^*$ to $\mathbf 1$. It is both the counit of the dagger compact structure of $\cat{qRel}$ and the binary relation into $\mathbf 1$ that corresponds canonically to the projection $\delta_{\ell^\infty(\X)}$ that was defined in subsection \ref{introduction.A}.

The semantics interprets each nonduplicating formula $\phi(x_1, \ldots, x_n)$, whose distinct free variables $x_1, \ldots x_n$ are of sorts $\X_1, \ldots, \X_n$, respectively, as a binary relation $\[\phi(x_1, \ldots, x_n)\]$ from $\X_1 \times \cdots \times \X_n$ to $\mathbf 1$. In the graphical calculus, this is depicted as follows:
$$
      \begin{tikzpicture}[scale=1]
	\begin{pgfonlayer}{nodelayer}
		\node [style=box] (0) at (0,0) {$\[\phi(x_1, \ldots, x_n)\]$};
		\node [style=none] (B) at (1.3,-0.25) {};
		\node  (B0) at (1.3,-1) {$\scriptstyle\X_n$};
		\node [style=none] (A) at (-1.3,-0.25) {};
		\node  (A0) at (-1.3,-1) {$\scriptstyle\X_1$};
		\node [style=none] (10) at (0,-0.6) {$\cdots$};
		\node (C0) at (-0.9,-1) {$\scriptstyle \phantom {\X_i}$};
		\node [style=none] (C) at (-0.9,-0.25) {};
		\node (D0) at (0.9,-1) {$\scriptstyle \phantom {\X_i}$};
		\node [style=none] (D) at (0.9,-0.25) {};
	\end{pgfonlayer}
	\begin{pgfonlayer}{edgelayer}
	    \draw[arrow,markat=0.5] (B0) to (B);
	    \draw[arrow,markat=0.5] (A0) to (A);
	    \draw[arrow,markat=0.5] (C0) to (C);
	    \draw[arrow,markat=0.5] (D0) to (D);
	\end{pgfonlayer}
      \end{tikzpicture}
$$
Our core computational device relates the equality relation to the graphical calculus:

\begin{theorem}[also \ref{computation.D.2}]
Let $\phi(x_1, x_2, x_3, \ldots, x_n)$ be a nonduplicating formula, whose distinct free variables $x_1, x_2, x_3, \ldots, x_n$ are of sorts $\X_1, \X_2, \X_3, \ldots, \X_n$, respectively. Assume that $\X_2 = \X_1^*$ and let $\psi(x_3, \ldots, x_n)$ be the nonduplicating formula $\NOT (\forall x_1)\, (\forall x_2)\, (E_{\X_1}(x_1, x_2) \rightarrow \NOT \phi(x_1, x_2, x_3, \ldots, x_n))$. Then,
\begin{align*}
\begin{aligned}
      \begin{tikzpicture}[scale=1]
	\begin{pgfonlayer}{nodelayer}
		\node [style=box] (0) at (0,0) {$\[\psi(x_3, \ldots, x_n)\]$};
		\node [style=none] (B) at (1.3,-0.25) {};
		\node  (B0) at (1.3,-1) {$\scriptstyle\X_n$};
		\node [style=none] (A) at (-1.3,-0.25) {};
		\node  (A0) at (-1.3,-1) {$\scriptstyle\X_3$};
		\node [style=none] (10) at (0,-0.6) {$\cdots$};
		\node (C0) at (-0.9,-1) {$\scriptstyle \phantom {\X_i}$};
		\node [style=none] (C) at (-0.9,-0.25) {};
		\node (D0) at (0.9,-1) {$\scriptstyle \phantom {\X_i}$};
		\node [style=none] (D) at (0.9,-0.25) {};
	\end{pgfonlayer}
	\begin{pgfonlayer}{edgelayer}
	    \draw[arrow,markat=0.5] (B0) to (B);
	    \draw[arrow,markat=0.5] (A0) to (A);
	    \draw[arrow,markat=0.5] (C0) to (C);
	    \draw[arrow,markat=0.5] (D0) to (D);
	\end{pgfonlayer}
      \end{tikzpicture}
\end{aligned}
\quad = \quad
\begin{aligned}
      \begin{tikzpicture}[scale=1]
	\begin{pgfonlayer}{nodelayer}
		\node [style=box] (0) at (-0.4,0) {$\[\phi(x_1, x_2, x_3, \ldots, x_n)\]$};
		\node [style=none] (B) at (1.3,-0.25) {};
		\node  (B0) at (1.3,-1) {$\scriptstyle\X_n$};
		\node [style=none] (A) at (-1.3,-0.25) {};
		\node  (A0) at (-1.3,-1) {$\scriptstyle\X_3$};
		\node [style=none] (10) at (0,-0.6) {$\cdots$};
		\node (C0) at (-0.9,-1) {$\scriptstyle \phantom {\X_i}$};
		\node [style=none] (C) at (-0.9,-0.25) {};
		\node (D0) at (0.9,-1) {$\scriptstyle \phantom {\X_i}$};
		\node [style=none] (D) at (0.9,-0.25) {};
		\node [style=none] (Y) at (-1.7,-0.25) {};
		\node [style=none] (X) at (-2.1,-0.25) {};
	\end{pgfonlayer}
	\begin{pgfonlayer}{edgelayer}
	    \draw[arrow,markat=0.5] (B0) to (B);
	    \draw[arrow,markat=0.5] (A0) to (A);
	    \draw[arrow,markat=0.5] (C0) to (C);
	    \draw[arrow,markat=0.5] (D0) to (D);
	    \draw[arrow,  bend left = 90, looseness=2.5, markat=0.6] (Y) to (X);
	\end{pgfonlayer}
      \end{tikzpicture}
\end{aligned}\; .
\end{align*}
\end{theorem}
\noindent In the graphical calculus, the object $\X_1^*$ may be depicted as an upward-directed wire labeled $\X_1^*$ or as a downward-directed wire labeled $\X_1$. Thus, we may connect the wire depicting $\X_1$ with the wire depicting $\X_2 = \X_1^*$, as shown.

The quantum sets in this paper are essentially just discrete quantum spaces, which first arose in the study of compact quantum groups  \cite{PodlesWoronowicz}. Finite quantum sets, which are formally dual to finite-dimensional C*-algebras, were identified even earlier \cite{Woronowicz}. Finite quantum sets appear naturally in the study of quantum symmetry \cite{Wang} and quantum information \cite{MustoReutterVerdon}. In \cite{MustoReutterVerdon}, finite quantum sets appear as special symmetric dagger Frobenius algebras in the category of finite-dimensional Hilbert spaces \cite{Vicary}. The class of finite quantum sets is sufficient and convenient for many applications, but the class of all quantum sets has two significant advantages: discrete quantum groups are generally infinite, and the category of all quantum sets and functions is closed monoidal \cite{Kornell}*{Thm.~9.1}. 

Two closely related quantum generalizations of sets have been proposed. Giles generalized sets essentially to atomic von Neumann algebras, calling them q-spaces \cite{Giles}. Rump has also recently proposed a striking geometric definition of quantum sets \cite{Rump}. For both of these definitions and for our definition, a quantum set is an object that may be partitioned into irreducible components, and up to isomorphism, the predicates on each component are the closed subspaces of some Hermitian space \cite{Holland}, forming a complete atomic orthomodular lattice. For the moment, our choice of definition appears to be the most compatible with the established body of noncommutative generalizations \cite{PodlesWoronowicz}\cite{Kornell}. Other quantum analogs of set theory exist, but they are less closely related \cite{Schlesinger}\cite{Takeuti}.

\subsection{Examples}

The many quantum generalizations that underlie noncommutative mathematics are motivated by diverse considerations, but they are nevertheless mutually compatible, and this is true even within discrete noncommutative mathematics. For example, discrete quantum groups arose in the program of extending Potryagin duality to include noncommutative groups, more than two decades after its inception \cite{Kac}\cite{Kac2}\cite{Takesaki}\cite{Vallin}\cite{PodlesWoronowicz}, and the notion of quantum isomorphism between simple graphs arose quite independently in the study of quantum nonlocality originating from the Kochen-Specker Theorem \cite{KochenSpecker}\cite{HeywoodRedhead}\cite{GalliardWolf}\cite{CameronMontanaroNewmanSeveriniWinter}\cite{MancinskaRoberson}\cite{AtseriasMancinskaRobersonSamalSeveriniVarvitsiotis}, but these two quantum generalizations are now understood to be related \cite{Bichon}\cite{Banica}\cite{LupiniMancinskaRoberson}\cite{MustoReutterVerdon2}\cite{BrannanChirvasituEiflerHarrisPaulsenSuWasilewski}. 

This compatibility between quantum generalizations of disparate origins demands an explanation, and the simplest possible explanation is that the quantum generalizations that underlie noncommutative mathematics are all instances of a single quantum generalization. Such an explanation requires a general notion of quantum structure and furthermore a method for extending each class of ordinary structures to a class of quantum structures. For comparison, in \cite{Woronowicz}, Woronowicz not only defines the general notion of a locally compact quantum space, i.e. a pseudospace, but also extends the classes of finite, finite-dimensional and compact locally compact spaces to classes of finite, finite-dimensional and compact locally compact quantum spaces, respectively. The problem is to extend this approach to encompass all possible quantum structures and all possible properties.

This paper proposes a solution to this problem in the special case of discrete structures. We may define a discrete structure to consist of sets, relations and functions, and similarly, we may define a discrete quantum structure to consist of quantum sets, relations and functions \cite{Kornell}. The role of quantum predicate logic with equality is then to extend each class of ordinary structures to a class of quantum structures. The nonduplicating formulas are the classes that we are extending, and the semantics is the method by which we are extending these classes. Formally, a formula of many-sorted first-order logic is defined to be \emph{nonduplicating} if no variable appears more than once in any atomic subformula.
This restriction is weaker than it might appear; it is routine to show that every formula of many-sorted first-order logic with equality is equivalent over discrete structures to a nonduplicating formula.

Many established classes of discrete quantum structures are unified in this way. In this paper, we treat quantum graphs, quantum metric spaces, quantum posets, quantum graph homomorphisms, quantum graph isomorphisms, quantum permutations and quantum groups, all discrete in the sense that the underlying von Neumann algebra is hereditarily atomic.

The quantum graphs and quantum metric spaces that are considered here originate in the study of quantum error correction. Quantum graphs were introduced in \cite{DuanSeveriniWinter} as the confusability graphs of quantum channels, and they were generalized to arbitrary von Neumann algebras in \cite{Weaver2}. These quantum graphs are not closely related to metric graphs equipped Shr\"{o}dinger operators on each edge, which are also called quantum graphs \cite{GnutzmannSmilansky}\cite{Pauling}. Quantum metric spaces in the sense of von Neumann algebras \cite{KuperbergWeaver} generalize quantum graphs in a way that quantifies error. That research led to the notion of a quantum relation \cite{Weaver2}, which in turn led to the research in the present paper. There are other quantum generalizations of metric structure \cite{Weaver5}\cite{Rieffel}\cite{Latremoliere}, which we do not consider here.

There are multiple natural notions of a quantum poset, even in the finite-dimensional case. A quantum partial order on $M_n(\CC)$ may be defined to be an antisymmetric subalgebra of $M_n(\CC)$ \cite{Weaver2}, a hereditarily antisymmetric subalgebra of $M_n(\CC)$ \cite{Weaver4} or a nilpotent subalgebra of $M_n(\CC)$ \cite{Weaver4}. We axiomatize both the first and the last of these notions with almost identical sets of nonduplicating formulas. The difference between them illustrates two natural generalizations of conjunction to the quantum setting \cite{BirkhoffVonNeumann}\cite{Sasaki}. Discrete quantum posets of the first kind form a well-behaved category \cite{KornellLindenhoviusMislove2} that may be used to model recursion in the quantum setting \cite{KornellLindenhoviusMislove}.

Quantum graph homomorphisms and quantum graph isomorphisms originate in the study of quantum nonlocality. These terms refer to quantum analogs of relationships between simple graphs, rather than to relationships between quantum graphs. The notion of quantum graph homomorphism was defined in \cite{MancinskaRoberson}. In \cite{MustoReutterVerdon}, quantum graph homomorphisms were identified with the morphisms of a 2-category, and in \cite{Kornell}, quantum graph homomorphisms were interpreted as quantum families of graph homomorphisms within the framework of noncommutative mathematics. The notion of quantum graph isomorphism was defined in \cite{AtseriasMancinskaRobersonSamalSeveriniVarvitsiotis}, and quantum graph isomorphisms were identified with certain morphisms in \cite{MustoReutterVerdon}. Both notions also have analogs in the quantum commuting framework \cite{PaulsenTodorov}\cite{PaulsenSeveriniStahlkeTodorovWinter}\cite{OritzPaulsen}\cite{AtseriasMancinskaRobersonSamalSeveriniVarvitsiotis}, which coincides with the usual tensor product framework when all the measurement operators are taken from a hereditarily atomic von Neumann algebra \cite{Guichardet}*{Prop.~8.6}\cite{Kornell}*{Prop.~5.4}.

Quantum permutations are just quantum automorphisms of edgeless graphs, but they predate quantum isomorphisms considerably \cite{Wang}. A quantum permutation of a finite set is also called a magic unitary \cite{BanicaBichonCollins}\cite{MustoReutterVerdon}. A subclass of quantum permutations, the quantum Latin squares, have been used for the construction of unitary error bases \cite{MustoVicary}. In the language of quantum information theory, Wang's quantum permutation group \cite{Wang} may be regarded as a universal quantum permutation of the given finite set, and in the language of noncommutative mathematics, it may be regarded as the compact quantum group of all permutations of the given finite set.

As in the case of our other examples, the quantum groups that we consider are the discrete members of a larger class; discrete quantum groups are discrete locally compact quantum groups \cite{KustermansVaes}\cite{KustermansVaes2}. Infinite discrete quantum groups first occurred as the Pontryagin duals of compact matrix quantum group \cite{Woronowicz2}\cite{PodlesWoronowicz}. The class of all discrete quantum groups was then implicitly defined along with the class of all compact quantum groups \cite{Woronowicz3}. The first explicit definition of discrete quantum groups appears to have been given by Effros and Ruan \cite{EffrosRuan}, but it falls slightly outside the standard approach of starting with an operator algebra of complex-valued functions on a putative quantum space. We work with the definition of Van Daele \cite{VanDaele}. The proof given here establishing discrete quantum groups as an example is essentially due to Vaes \cite{Vaes}; any flaws are the fault of the author.

Not all established classes of discrete quantum structures are claimed to be naturally definable in quantum predicate logic. Most notably, unital normal complete positive maps, which formalize quantum channels \cite{Kraus}, are not treated in this paper. Additionally, some structures such as metric spaces and posets have more than one proposed quantum generalization, not all of which have been axiomatized.

\subsection{Many-sorted logic}\label{introduction.D}
Some quantum groups are commutative, and others are not. Some quantum graphs are complete, and others are not. The core definition of this paper unifies these two notions and many others; it specifies whether or not a quantum structure possesses a classical property for a large class of quantum structures that includes both discrete groups and discrete graphs and for a large class of classical properties that includes both commutativity and completeness. Specifically, it applies to discrete quantum structures, which consist of quantum sets, binary relations and functions, and to the classical properties that are formalized by the \emph{nonduplicating} sentences of many-sorted first-order logic. The notion of a nonduplicating formula is defined in Definition \ref{definition.G.3}. Thus, the purpose of logical formulas in this paper is to speak precisely about classical properties as would enable us to state and prove theorems at this level of generality.

More formally, we quantize the semantics of many-sorted first-order logic \cite{Schmidt}. We presently review this semantics in a form that is convenient to this goal. For simplicity, we work with a single fixed many-sorted structure that consists of all sets and the relations and functions between them. Our language includes infinitely many \textit{variables} for each set. The class of all \emph{terms} is defined recursively: a variable of sort $A$ is a term of sort $A$, and for each function $f\: A_1\times \cdots \times A_m \to B$, if $t_1, \ldots, t_m$ are terms of sorts $A_1, \ldots, A_m$, respectively, then the expression $f(t_1, \ldots, t_m)$ is a term of sort $B$. An \emph{atomic formula} is then defined to be an expression of the form $R(t_1, \cdots, t_n)$, where $R \subseteq A_1 \times \cdots \times A_n$ and $t_1, \ldots, t_n$ are terms of sorts $A_1, \ldots, A_n$, respectively. The natural numbers $m$ and $n$ may be equal to $0$. Finally, the class of all \emph{formulas} is defined recursively: an atomic formula is a formula, and if $\phi$ and $\psi$ are formulas and $v$ is a variable, then the expressions $\neg \phi$, $\phi \wedge \psi$, $\phi \vee \psi$, $\phi \rightarrow \psi$, $(\forall v)\, \phi$ and $(\exists v)\, \phi$ are formulas. A \emph{sentence} is defined to be simply a formula with no free variables.

The sentences that we have defined are mathematical objects like groups and topological spaces. Tarski's analysis of semantics \cite{Tarski}\cite{TarskiVaught} leads to the definition of truth as a property of sentences, in the same sense that commutativity and compactness are properties of groups and topological spaces, respectively. To formulate this definition, we first formalize the specification of subsets by properties. For each sequence of variables $v_1, \ldots, v_n$ of sorts $A_1, \ldots, A_n$, respectively, and each formula $\phi$ whose free variables are among $v_1, \ldots, v_n$, we define the subset
$$
\[(v_1\ldots, v_n)  \suchthat \phi \] \subsetof A_1 \times \cdots \times A_n,
$$
the set of all tuples $(a_1, \ldots, a_n)\in A_1 \times \cdots \times A_n$ that \emph{satisfy} $\phi$. This definition proceeds by recursion on the class of all formulas, and in the end we obtain a partial class function whose first argument is a tuple of variables and whose second argument is a formula.

A sentence $\phi$ is then defined to be \emph{true} if $\[() \suchthat \phi \] = \{()\}$. In effect, we have that $$ \[(v_1\ldots, v_n)  \suchthat \phi \] = \{(a_1,\ldots, a_n) \in A_1 \times \cdots \times A_n  \suchthat \phi(a_1, \dots, a_n)\text{ is true}\}.$$
This equation emphasizes that the formula $\phi$ is a formula of the object language, a mathematical object about which we may state theorems, rather than a formula of the metalanguage, the language in which we state our theorems. The notation $\phi(a_1, \ldots, a_n)$ expresses the substitution of the elements $a_1, \ldots, a_n$ for the variables $v_1, \ldots, v_n$, where each element $a_i \in A_i$ is regarded as a function from the singleton set $\{\ast\}:=\{()\}$ to the set $A_i$.

\subsection{Conventions.}
Each variable is of a unique sort, which is a quantum set. However, as an aid to memory and intuition, we write $(\exists x \fin \X)\, \phi$ in place of $(\exists x)\, \phi$ and $\[x \in \X \suchthat \phi\]$ in place of $\[(x) \suchthat \phi\]$, where $\X$ is the sort of $x$. If $\phi$ is a formula, then we write $\phi(x_1, \ldots, x_n)$ to indicate that the free variables of $\phi$ are among $x_1, \ldots, x_n$ and that variables $x_1, \ldots, x_n$ are pairwise distinct. We do likewise for terms.


Let $H$ and $K$ be Hilbert spaces. We write $L(H,K)$ for the set of all bounded operators from $H$ to $K$, we write $L(H)$ for the set $L(H,H)$ of all bounded operators on $H$, and we write $H^*$ for the set $L(H, \CC)$ of all bounded functionals on $H$. Let $a$ be a linear operator from $H$ to $K$. We write $a^\dagger \in L(K,H)$ for the Hermitian adjoint of $a$, we write $a^* \in L(K^*, H^*)$ for the Banach space transpose of $a$, and we write $a_* \in L(H^*, K^*)$ for the ``conjugate'' $(a^\dagger)^*$ of $a$. Note that if $A$ is an operator algebra on $H$, then $A^*$ is canonically isomorphic to the opposite of $A$, that is, to the algebra $A$ with the order of multiplication reversed. We retain the stock term ``$*$-homomorphism'' to mean a homomorphism that respects the Hermitian adjoint operation $a \mapsto a^\dagger$.

We write $\CC a$ for the linear span of a bounded operator $a$. If $a$ and $b$ are bounded operators, then we write $a \cdot b$ for the product of $a$ and $b$, in order to separate the factors visually and to make the operator product readily distinguishable from function application. If $V$ and $W$ are subspaces of bounded operators, then we write $V \cdot W$ for the linear span of $\{v\cdot w \suchthat v \in V, w \in W\}$.

Let $A$ and $B$ be sets. We regard each binary relation from $A$ to $B$ foremost as a morphism from $A$ to $B$ in the category of sets and binary relations. Similarly, we regard a relation of arity $(A_1, \ldots, A_n)$ foremost as morphism from $A_1 \times \cdots \times A_n$ to $\{\ast\}$, the monoidal unit of the canonical monoidal structure on the category of sets and binary relations. Thus, a binary relation on $A$ is essentially the same thing as a relation of arity $(A,A)$, but we regard the former as morphism from $A$ to $A$, and we regard the latter as a morphism from $A \times A$ to $\{\ast\}$. In the same vein, we may regard any element $a \in A$ as a morphism from $\{\ast\}$ to $A$.

We use the adjective ``ordinary'' to emphasize that we are using a noun in its standard mathematical sense. Thus, an ordinary set is just a set.

\section{definition}\label{definition}

We now expound the interpretation of nonduplicating first-order formulas over quantum sets. We recall quantum sets in subsection \ref{definition.A}, and we define their relations in subsection \ref{definition.B}. We define the interpretation of primitive formulas in subsection \ref{definition.C}, and we extend this interpretation to all nonduplicating relational formulas in subsection \ref{definition.D}. Then, we define quantifiers over the diagonal in subsection \ref{definition.E}, and we use these quantifiers to define function graphs in subsection \ref{definition.F}. Finally, we define the interpretation of arbitrary nonduplicating formulas in subsection \ref{definition.G}.

\subsection{Quantum sets.}\label{definition.A} A quantum set is essentially just a set of nonzero finite-dimensional Hilbert spaces, intuitively, a union of indecomposable quantum sets. This section is a brief summary of some relevant definitions from \cite{Kornell}.

\begin{definition}\label{definition.A.1}
A \emph{quantum set} $\X$ is uniquely determined by a set $\At(\X)$ of nonzero finite-dimensional Hilbert spaces, called the \emph{atoms} of $\X$.
\end{definition}

Each quantum set $\X$ is associated to the von Neumann algebra $\ell^\infty(\X) = \bigoplus_{X \in \At(\X)}L(X)$, which intuitively consists of all bounded complex-valued functions on $\X$. This algebra is typically not commutative, and thus the elements of $\X$ are figures of speech, rather like the points of a quantum space. Formally, $\X$ is equal to $\At(\X)$, but intuitively, they are distinct objects, and this notational distinction affects the meaning of our expressions. For example, $\ell^\infty(\X)$ is generally not isomorphic to $\ell^\infty(\At(\X))$. Indeed, the former von Neumann algebra is generally not commutative, but the latter von Neumann algebra is always commutative, because $\At(\X)$ is just an ordinary set, which happens to consist of Hilbert spaces. This explains the circuitous language in Definition \ref{definition.A.1}.

In quantum mathematics, we should recover the classical theory whenever the relevant operator algebras are all commutative. This is the definitional feature of any quantum generalization in the sense of noncommutative geometry. In our case, we observe that $\ell^\infty(\X)$ is commutative if and only if each atom of $\X$ is one-dimensional. Intuitively, such atoms correspond to those elements of $\X$ which exist individually, apart from the other elements. This gloss clarifies how ordinary sets should be incorporated into the picture.

\begin{definition}\label{definition.A.2}
To each ordinary set $A$, we associate a quantum set $`A$ whose atoms are one-dimensional Hilbert spaces, with one such atom for each element of $A$. More generally, we say that a quantum set $\X$ is \emph{classical} if and only if each of its atoms is one-dimensional.
\end{definition}

\noindent We may gloss the first sentence of Definition \ref{definition.A.2} by the equation $\At(`A) = \{\CC_a\suchthat a \in A\}$, where $\CC_a$ denotes a one-dimensional Hilbert space that is somehow labeled by the element $a$. The exact formalization of this labeling is inconsequential; it is only important that distinct elements $a_1$ and $a_2$ correspond to distinct Hilbert spaces $\CC_{a_1}$ and $\CC_{a_2}$, so that we have a canonical bijection $A \to \At(`A)$.

A property of a quantum set $\X$ generalizes a property of an ordinary set $A$ if we obtain the latter from the former by replacing $\X$ by $`A$, and it is likewise for operations. For example, the Cartesian product of quantum sets generalizes the Cartesian product of ordinary sets:

\begin{definition}\label{definition.A.3}
The \emph{Cartesian product} $\X \times \Y$ of quantum sets $\X$ and $\Y$ is defined by $\At(\X \times \Y) = \{ X \otimes Y \suchthat X \in \At(X), Y \in \At(\Y)\}$.
\end{definition}

\noindent This is not an exact quantum generalization because $`A \times ` B$ may be formally distinct from $`(A \times B)$. However, $`A \times ` B$ is isomorphic to  $`(A \times B)$ in the obvious sense. The Cartesian product of quantum sets corresponds to the spatial tensor of von Neumann algebras in the sense that $\ell^\infty(\X \times \Y) \iso \ell^\infty(\X) \stensor \ell^\infty(\Y)$ for all quantum sets $\X$ and $\Y$.

Each quantum set $\X$ has a \emph{dual} $\X^*$ that is defined by the equation $\At(\X^*) = \{X^* \suchthat X \in \At(\X)\}$ \cite{Kornell}*{Def.~3.4}. Dualization in this sense corresponds to reversing the order of multiplication in a von Neumann algebra: $\ell^\infty(\X)^{op}$ is canonically isomorphic to $\ell^\infty(\X^*)$ via the map $a \mapsto a^*$, where $a^*$ is the transpose rather than the adjoint of a bounded operator $a$. Formally, we define $a^*(X) = a(X)^*$ for each atom $X \in \At(\X)$. The quantum sets $\X$ and $\X^*$ are distinct but closely related. The example of pair production in subsection \ref{introduction.A} suggests the intuition that the elements of $\X^*$ are the antielements of $\X$. Like the elements of $\X$, the antielements of $\X$ are just figures of speech; they are not mathematical objects.

\subsection{Relations on quantum sets}\label{definition.B}
Intuitively, we may view each quantum set as the phase space of an abstract physical system that is discrete in the sense that each observable admits an orthonormal basis of eigenvectors \cite{Kornell}*{Prop.~5.4}. Thus, the predicates, i.e., unary relations on a quantum set $\X$ should be in bijection with the projections in $\ell^\infty(\X)$. Such a projection is formally a family of projections $p_X \in L(X)$, for $X \in \At(\X)$, so we may define a predicate $P$ on $\X$ to be simply a family of subspaces $P(X) \leq X$, for $X \in \At(\X)$, as it is done in \cite{Kornell}*{App.~B}. However, for technical and intuitive reasons, we prefer to work with subspaces of the dual Hilbert spaces.

\begin{definition}\label{definition.B.1}
Let $\X$ be a quantum set. A \emph{predicate} $P$ on $\X$ is a function assigning a subspace $P(X) \leq L(X, \CC)$ to each atom $X$ of $\X$. 
\end{definition}

\noindent The canonical one-to-one correspondence between predicates $P$ on $\X$ and projections $p$ in $\ell^\infty(\X)$ is defined by $P(X) = L(X, \CC) \cdot p(X)$, for $X \in \At(\X)$.

For each atom $X$, the subspaces of $L(X, \CC)$ form a modular orthomodular lattice, and thus, the predicates on $\X$ themselves form a modular orthomodular lattice $\mathrm{Pred}(\X)$, with its operations defined atomwise. This is essentially the orthomodular lattice of projections in $\ell^\infty(\X)$. We use the standard notations $\AND$ and $\OR$ for the meets and the joins, respectively, as well as $\bot_\X$ and $\top_\X$ for the smallest and largest predicates on a quantum set $\X$, respectively, but we notate the orthocomplementation by $\neg$.

Each predicate $P$ on $\X$ has a \emph{conjugate} $P_*$, a predicate on $\X^*$ that is defined by 
$P_*(X^*) = \{ \xi_* \suchthat \xi \in P(X)\}$ for each atom $X \in \At(\X)$. The functional $\xi_*$ is formally defined by $\xi_* = (\xi^*)^\dagger$, and since $\xi \in X^*$, we may also characterize $\xi_*$ by the equation $\xi_*(\eta)= \langle \xi | \eta \rangle$, for $\eta \in X^*$. Intuitively, $P_*$ holds of those elements of $\X^*$ such that $P$ holds of their counterparts in $\X$. We thus obtain an isomorphism of orthomodular lattices $\Pred(\X) \to \Pred(\X^*)$.

We now define the \emph{Cartesian product} of two predicates, generalizing the Cartesian product of two subsets to the quantum setting.

\begin{definition}\label{definition.B.2}
If $P$ and $Q$ are predicates on quantum sets $\X$ and $\Y$, respectively, then the predicate $P \times Q$ on $\X \times \Y$ is defined by $(P\times Q)(X \otimes Y) = P(X) \otimes Q(Y)$, for $X \in \At(\X)$ and $Y \in \At(\Y)$.
\end{definition}

\noindent Both $P(X)$ and $Q(Y)$ are vector spaces of functionals, and $P(X) \tensor Q(Y)$ denotes another vector space of functionals, so we have suppressed the canonical isomorphism $\CC \otimes \CC \iso \CC$. Thus, $(P \times Q)(X \tensor Y)$ is essentially just the span of bilinear functionals $(x,y) \mapsto \xi(x)\eta(y)$, for $\xi \in P(X)$ and $\eta \in Q(Y)$. The construction $(P,Q) \mapsto P \times Q$ corresponds to the tensor product of two projections, i.e., to the conjunction of two Boolean observables on the composite of two abstract physical systems.

Finally, we define relations on quantum sets, generalizing the relations of ordinary many-sorted logic:

\begin{definition}\label{definition.B.3}
Let $\X_1, \ldots, \X_n$ be quantum sets. A \emph{relation} of arity $(\X_1, \ldots, \X_n)$ is a predicate on the Cartesian product $\X_1 \times \cdots \times \X_n$, with $n \geq 0$.
\end{definition}

\noindent Thus, a relation of arity $(\X_1, \ldots, \X_n)$ essentially just assigns a vector space of multilinear functionals $X_1 \times \cdots \times X_n \to \CC$ to each choice of atoms $X_1 \in \At(\X_1)$, $X_2 \in \At(\X_2)$, etc. We write $\Rel(\X_1, \ldots, \X_n)$ for the set of all relations of arity $(\X_1, \ldots, \X_n)$.
Permuting the quantum sets $\X_1, \ldots, \X_n$ according to some permutation $\pi$ of the index set $\{1, \ldots, n\}$, we expect and obtain a bijection between $\Rel(\X_1, \ldots, \X_n)$ and $\Rel(\X_{\pi(1)}, \ldots, \X_{\pi(n)})$:

\begin{definition}\label{definition.B.4}
Let $\X_1, \ldots, \X_n$ be quantum sets, and let $\pi$ be a permutation of the index set $\{1, \ldots, n\}$. For each relation $R$ of arity $(\X_{\pi(1)}, \ldots, \X_{\pi(n)})$, define the relation $\pi_\#(R)$ of arity $(\X_{1}, \ldots, \X_{n})$ by $$\pi_\#(R)(X_{1} \tensor \cdots \tensor X_{n}) = R(X_{\pi(1)} \tensor \cdots \tensor X_{\pi(n)}) \cdot u_\pi,$$ for all $X_1 \in \At(\X_1)$, $X_2, \in \At(\X_2)$, etc., where $u_\pi\: X_{1} \tensor \cdots \tensor X_{n}\to X_{\pi(1)} \tensor \cdots \tensor X_{\pi(n)}$ is the unitary operator that permutes the tensor factors according to $\pi$.
\end{definition}

The construction $R \mapsto \pi_\#(R)$ is clearly a bijection $\Rel(\X_{\pi(1)}, \ldots, \X_{\pi(n)}) \to \Rel(\X_{1}, \ldots, \X_{n})$, with inverse $S \mapsto (\pi\inv)_\#(S)$. Furthermore, it is an isomorphism of orthomodular lattices. Its effect on the projections corresponding to these relations is given by the canonical unital normal $*$-homomorphism $ \ell^\infty(\X_{\pi(1)}) \stensor \cdots \stensor \ell^\infty(\X_{\pi(n)}) \to \ell^\infty(\X_1) \stensor \cdots \stensor \ell^\infty(\X_n) .$

\subsection{Interpreting primitive formulas.}\label{definition.C} 
We work with the language of many-sorted first-order logic whose sorts are the quantum sets of subsection \ref{definition.A} and whose relation symbols are the relations of subsection \ref{definition.B}. Each sort, that is, each quantum set is assigned an infinite stock of variables, intuitively ranging over that quantum set. Within formulas, we write $x \in \X$ to annotate that $x$ has sort $\X$, replacing the more traditional notation $x: \X$. We will incorporate function symbols in subsection \ref{definition.F}.

\begin{definition}\label{definition.C.1}
A \emph{primitive atomic formula} is an expression of the form $R(x_1, \ldots, x_n)$, where the $R$ is a relation of some arity $(\X_1, \ldots, \X_n)$ and $x_1 \ldots, x_n$ are distinct variables of sorts $\X_1, \ldots, \X_n$, respectively. The class of \emph{primitive formulas} is defined recursively: each primitive atomic formula is a primitive formula, and if $\phi$ and $\psi$ are primitive formulas and $x$ is a variable of some sort $\X$, then the expressions $\neg \phi$, $\phi\AND \psi$ and $(\forall x \fin \X)\, \phi$ are primitive formulas. A \emph{primitive sentence} is a primitive formula with no free variables. We will sometimes write $\phi(x_1, \ldots,x_n)$ in place of $\phi$ to indicate that the free variables of $\phi$ are among $x_1, \ldots, x_n$ and that the variables $x_1, \ldots, x_n$ are pairwise distinct.
\end{definition}

For each sequence of distinct variables $x_1, \ldots, x_n$ of sorts $\X_1, \ldots, \X_n$, respectively, and each primitive formula $\phi(x_1, \ldots, x_n)$, we now define a relation 
$$\[(x_1, \ldots, x_n) \in \X_1 \times \cdots \times \X_n \suchthat \phi(x_1, \ldots, x_n)  \] $$
of arity $(\X_1, \ldots, \X_n)$ to be our \emph{interpretation} of $\phi$ in the \emph{context} $x_1 \in \X_1, \ldots, x_n \in \X_n$. We will occasionally simply write $\[\phi(x_1, \ldots, x_n)\]$ when the context is obvious.

The notation 
$
\[ x_1: X_1,\ldots, x_n : X_n \vdash \phi(x_1, \ldots, x_n)\]
$
is more or less standard to categorical logic, but it is not as intuitive in this setting. The chosen notation is intended to suggest the standard notation for defining subsets, e.g., $\{ (x,y)\in \RR \times \RR\suchthat x^2 + y^2 =1 \}$. We use brackets rather than braces because the familiar bijection between subsets and predicates does not survive the quantum generalization \cite{Kornell}*{sec.~10}. We are defining predicates.

\begin{definition}\label{core}\label{definition.C.2}
Let $\X_1, \ldots, \X_n$ be quantum sets, and let $x_1, \ldots, x_n$ be distinct variables of sorts $\X_1, \ldots, \X_n$, respectively. For each permutation $\pi$ of $\{1, \ldots, n\}$, and each relation $R$ of arity $(\X_{\pi(1)}, \ldots, \X_{\pi(m)})$, for some $m \leq n$, we define
\begin{align*}\[(x_1,\ldots, x_n) \in \X_{1} \times \cdots \times  \X_{n} \suchthat R(x_{\pi(1)}, \ldots, x_{\pi(m)})\]
=
\pi_\#(R \times \top_{\X_{\pi(m+1)}} \times \cdots \times \top_{\X_{\pi(n)}}).
\end{align*}
\noindent It is straightforward to verify that this relation depends only on the values of $\pi$ on $ \{1, \ldots, m\}$.
Furthermore, for arbitrary primitive formulas $\phi(x_1, \ldots, x_n)$ and $\psi(x_1, \ldots, x_n)$, we define
\vspace{1ex}
\begin{enumerate}[(1)]  \setlength{\itemsep}{2ex} 
\item $\[ (x_1, \ldots, x_n)\in \X_1 \times \cdots \times \X_n \suchthat \NOT \phi (x_1, \ldots, x_n)\]$ \\
$ = \neg \[(x_1, \ldots, x_n)\in \X_1 \times \cdots \times  \X_n \suchthat \phi (x_1, \ldots, x_n)\]$;
\item
$\[(x_1, \ldots, x_n)\in \X_1 \times \cdots \times \X_n \suchthat  \phi (x_1, \ldots, x_n) \AND \psi (x_1, \ldots, x_n)  \]$ \\ $= \[(x_1, \ldots, x_n)\in \X_1 \times \cdots \times \X_n \suchthat \phi (x_1, \ldots, x_n)\] \AND \[(x_1, \ldots, x_n)\in \X_1 \times \cdots \times \X_n \suchthat \psi (x_1, \ldots, x_n)\]$;
\item
$\[(x_2, \ldots, x_{n})\in \X_2 \times \cdots \times \X_{n} \suchthat (\forall x_1 \fin \X_1) \, \phi (x_1, \ldots, x_n)\] $
\\ $= \sup\{R \in \Rel(\X_2, \ldots, \X_{n}) \suchthat \top_{\X_1} \times R \leq \[(x_1, \ldots, x_n)\in \X_1 \times \cdots \times  \X_n \suchthat \phi (x_1, \ldots, x_n)\]  \}$.
\end{enumerate}
\end{definition}

The quantum sets $\X_1, \ldots, \X_n$ in Definition \ref{definition.C.2} are arbitrary, as are the variables $x_1, \ldots, x_n$, so we have defined the interpretation of all primitive formulas by recursion over that class.

Our interpretation of universal quantification may be justified by observing that it is a straightforward generalization of the classical interpretation. A further argument was given by Weaver when he introduced this definition \cite{Weaver}. In essence, Weaver drew an analogy between the elements of a set and the pure normal states of a type I factor, and we could do the same here. The same analogy was later drawn by Rump \cite{Rump}. However, we do not view pure states as a direct analogue of elements, instead holding fast to the orthodox understanding of elements and points in noncommutative mathematics that is expressed in subsection \ref{definition.A}.

The interpretation of a primitive formula in the empty context is a relation of arity $()$, i.e., a predicate on the empty Cartesian product of quantum sets. By convention, this empty Cartesian product is the quantum set $\mathbf 1$ whose only atom is the field $\CC$ of complex numbers, considered as a one-dimensional Hilbert space. It has exactly two predicates, the predicate $\top = \top_{\mathbf 1}$, defined by $\top(\CC) = L(\CC, \CC)$, and $\bot = \bot_{\mathbf 1}$, defined by $\bot(\CC) = 0$. It is natural to say that a formula $\phi()$, which has no free variables, is \emph{true} if $\[\phi()\] = \top$.

\begin{proposition}[also \ref{appendix.B.2}]\label{definition.C.3}
Let $\X_1, \ldots, \X_p$ be quantum sets, and let $x_1, \ldots x_p$ be distinct variables of sorts $\X_1, \ldots, \X_p$, respectively. For each permutation $\sigma$ of $\{1, \ldots, p\}$, and each primitive formula $\phi(x_1, \ldots, x_n)$, with $n \leq p$, we have that
\begin{align*}
\[ (x_{\sigma(1)}, & \ldots, x_{\sigma(p)})  \in \X_{\sigma(1)} \times \cdots \times \X_{\sigma(p)}
\suchthat
\phi (x_1, \ldots, x_n) \]
\\ & =
(\sigma \inv)_\# (
\[ (x_1, \ldots, x_n) \in \X_1 \times \cdots \times \X_n
\suchthat
\phi(x_1, \ldots, x_n) \]
\times \top_{\X_{n+1}} \times \cdots \times \top_{\X_{p}}).
\end{align*}
\end{proposition}

This is the expected but necessary observation that permuting the context corresponds exactly to permuting the arity of the resulting relation and that additionally any unused variable $x$ of some sort $\X$ corresponds to a factor of $\top_\X$. This behavior is built in to the definition of our interpretation of primitive atomic formulas, but an inductive argument is necessary to show that it persists for primitive formulas of higher syntactic complexity. The proof is relegated to Appendix \ref{appendix.B}.

\subsection{Defined logical symbols.}\label{definition.D} As in classical logic, the disjunction connective $\OR$ and the existential quantifier $\exists$ may be expressed in terms of their duals.

\begin{definition}\label{definition.D.1}
For primitive formulas $\phi(x_1, \ldots, x_n)$ and $\psi(x_1, \ldots, x_n)$, we write $$\psi(x_1, \ldots, x_n) \OR \psi(x_1, \ldots, x_n)$$ as an abbreviation for $\NOT( \NOT \psi(x_1, \ldots, x_n) \AND  \NOT \psi(x_1, \ldots, x_n))$, and we write $$(\exists x_1 \fin \X_1)\, \phi(x_1, \ldots, x_n)$$ as an abbreviation for $\NOT (\forall x_1 \fin \X_1)\,\NOT \phi(x_1, \ldots, x_n)$.
\end{definition}

\begin{proposition}\label{definition.D.2}
Let $\X_1, \ldots, \X_n$ be quantum sets, and let $x_1, \ldots, x_n$ be distinct variables of sorts $\X_1, \ldots, \X_n$, respectively. For primitive formulas $\phi(x_1, \ldots, x_n)$ and $\psi(x_1, \ldots, x_n)$,
\begin{enumerate}[(1)]
\item $\[(x_1, \ldots, x_n)\in \X_1 \times \cdots \times \X_n \suchthat  \phi (x_1, \ldots, x_n) \OR \psi (x_1, \ldots, x_n)  \]$ \\ $= \[(x_1, \ldots, x_n)\in \X_1 \times \cdots \times \X_n \suchthat \phi (x_1, \ldots, x_n)\] \OR \[(x_1, \ldots, x_n)\in \X_1 \times \cdots \times \X_n \suchthat \psi (x_1, \ldots, x_n)\]$;
\item $\[(x_2, \ldots, x_{n})\in \X_2 \times \cdots \times \X_{n} \suchthat (\exists x_1 \fin \X_1)\, \phi (x_1, \ldots, x_n)\] $
\\ $= \mathrm{inf}\{R \in \Rel(\X_2, \ldots, \X_{n}) \suchthat \top_{\X_1} \times R \geq \[(x_1, \ldots, x_n)\in \X_1 \times \cdots \times  \X_n \suchthat \phi (x_1, \ldots, x_n)\]  \}$.
\end{enumerate}
\end{proposition}

\begin{proof}
Straightforward.
\end{proof}

In classical logic, an implication $\phi(x_1, \ldots, x_n) \IMPLIES \psi(x_1, \ldots, x_n)$ may be viewed as abbreviating the formula $\NOT \phi(x_1, \ldots, x_n) \OR \psi(x_1, \ldots, x_n)$, but it is now widely understood both that this expression is entirely unsatisfactory to the quantum setting and that no such expression is entirely satisfactory. Hardegree observed \cite{Hardegree} that there are exactly three polynomials $P \IMPLIES Q$ in propositional variables $P$ and $Q$, for the operations $\NOT$, $\AND$, and $\OR$, that satisfy the following requirements in every orthomodular lattice:
\begin{enumerate}
\item $P \AND (P \IMPLIES Q) \leq Q$,
\item $(P \IMPLIES Q) \AND \NOT Q \leq \NOT P$,
\item $P \IMPLIES Q = \top$ if and only if $P \leq Q$.
\end{enumerate}
They are as follows:
\begin{enumerate}
\item $\NOT P \OR (P \AND Q)$,
\item $(\NOT P \AND \NOT Q) \OR Q$,
\item $(P \AND Q) \OR (\NOT P \AND Q) \OR (\NOT P \AND \NOT Q)$.
\end{enumerate}
None of these expressions is entirely satisfactory because none of them satisfies the expected transitivity law $(P \IMPLIES Q) \AND (Q \IMPLIES R) \leq P \IMPLIES R$. In this article, we interpret the implication $P \IMPLIES Q$ to be the Sasaki arrow $\NOT P \OR (P \AND Q)$.

This choice may be motivated by physical considerations. If $P$ and $Q$ are propositions about a physical system \cite{BirkhoffVonNeumann} and $\NOT P \OR (P \AND Q)$ is true of the initial state with probability one, then a positive outcome in a measurement of the truth value of $P$ guarantees a positive outcome in a successive measurement of the truth value of $Q$. Moreover, $\NOT P \OR (P \AND Q)$ is the weakest such proposition in the sense that any other proposition with this property implies it \cite{Sasaki}\cite{Finch}. This choice of implication may also be pragmatically justified by its role in the proof of Proposition \ref{computation.D.1}.

\begin{definition}\label{definition.D.3}
For primitive formulas $\phi(x_1, \ldots, x_n)$ and $\psi(x_1, \ldots, x_n)$, we write $$\phi(x_1, \ldots, x_n) \to \psi(x_1, \ldots, x_n)$$ as an abbreviation for
$ \NOT \phi(x_1, \ldots, x_n) \OR (\phi(x_1, \ldots, x_n) \AND \psi(x_1, \ldots, x_n)).$ We also write $$\phi(x_1, \ldots, x_n) \leftrightarrow \psi(x_1, \ldots, x_n)$$ as an abbreviation for $(\phi(x_1, \ldots, x_n) \to \psi(x_1, \ldots, x_n)) \AND (\psi(x_1, \ldots, x_n) \to \phi(x_1, \ldots, x_n))$.
\end{definition}

Contradiction $\bot$ and equality $=$ are commonly regarded as logical symbols on the basis that their interpretation does not really depend on the structure being considered. This distinction between logical and nonlogical relations is not meaningful within our approach of interpreting primitive formulas in a single many-sorted structure, the class of all quantum sets equipped with all their relations. Each relation symbol is a relation that denotes itself, and we do not consider other structures in which that symbol may denote some other relation. However, with this semantics in hand, it is entirely straightforward to define the notion of a discrete quantum model that accommodates both a logical equality symbol and various nonlogical relations symbols \cite{Hodges}.

Contradiction $\bot$ is a relation of arity $()$; it was defined in subsection \ref{definition.C}. The equality relation $E_\X$ on a quantum set $\X$ is a relation of arity $(\X,\X^*)$, which we now define:

\begin{definition}\label{definition.D.4}
Let $\X$ be a quantum set. The equality relation on $\X$ is the relation $E_\X$ of arity $(\X, \X^*)$ defined by $E_\X(X \tensor X^*) = \CC \counit_X$ for all atoms $X \in \At(\X)$ and $E_\X(X_1 \tensor X_2^*) = 0$ for distinct atoms $X_1, X_2 \in \At(\X)$, where $\counit_X$ is the evaluation operator $X \otimes X^* \to \CC$.
\end{definition}

\noindent 

The equality projection $\delta_{\ell^\infty(\X)}$ that was defined in subsection \ref{introduction.A} may be regarded as an element of $\ell^\infty(\X \times \X^*)$; see subsection \ref{definition.A}. Viewed in this way, it is the projection that corresponds to the predicate $E_\X$ in the sense that $E_\X(X_1 \tensor X_2^*) = L(X_1 \tensor X_2^*, \CC)\cdot \delta_{\ell^\infty(\X)}$ for all $X_1, X_2 \in \At(\X)$. This is proved in Appendix \ref{appendix.H}. From the intuitive perspective that $\X^*$ consists of the antielements of $\X$, the relation $E_\X$ identifies elements of $\X$ with antielements of $\X$, and there is generally no way to identify the elements of one copy of $\X$ with the elements of another. As in subsection \ref{definition.A}, the elements and antielements of a quantum set are just figures of speech.

\subsection{Quantifying over the diagonal.}\label{definition.E} 

The characteristic feature of the equality relation in the quantum setting is its mixed arity. For this reason, axioms often quantify over both the underlying quantum set of a discrete quantum structure and over its dual, and the relations that constitute that structure often have mixed arity. For example, the reflexivity of a relation $R$ equipping a quantum set $\X$ is naturally expressed by the sentence $(\forall x_1 \fin \X)\,(\forall x_2 \fin \X^*)\,(E_\X(x_1, x_2) \IMPLIES  R(x_1,x_2))$. The variables $x_1$ and $x_2$ must have sorts $\X$ and $\X^*$, respectively, because $E_\X$ has arity $(\X, \X^*)$. It follows that $R$ should also have arity $(\X,\X^*)$. Thus, a reflexive relation on $\X$ should have arity $(\X, \X^*)$.

The formulation of reflexivity given in the above paragraph suggests a device for expressing the quantification of a variable ranging simultaneously over a quantum set $\X$ and over its dual $\X^*$. For greater convenience, we might modify our conventions to allow a single bound variable to appear once as an $\X$-sorted argument and once as an $\X^*$-sorted argument in any atomic formula. However, to avoid the risk of confusion and the cost of time borne by introducing this notation, we make do with a minor addition to our syntax that canonizes this device as a defined quantifier. We do so in part because this quantifier occurs frequently in the axiomatizations of already established quantum generalizations of discrete structures; see section \ref{examples}.

\begin{definition}\label{definition.E.1}
Let $\X$ be a quantum set. Let $\phi(x_1, x_2, x_3, \ldots, x_n)$ be a primitive formula with $x_1$ and $x_2$ of sorts $\X$ and $\X^*$, respectively. We write $$(\forall (x_1 \feq x_2) \fin \X \ftimes \X^*)\,\phi(x_1,  \ldots, x_n)$$ as an abbreviation for $(\forall x_2 \fin \X^*)\,(\forall x_1 \fin \X)\,(E_\X(x_1, x_2) \IMPLIES \phi(x_1,\ldots, x_n))$. We also write $$(\exists(x_1 \feq x_2) \fin \X \ftimes \X^*)\,\phi(x_1,  \ldots, x_n)$$ as an abbreviation for $\NOT (\forall (x_1 \feq x_2) \fin \X \ftimes \X^*)\,\NOT \phi(x_1, \ldots, x_n)$.
\end{definition}

For clarity, we will often decorate a variable that ranges over the dual of a given quantum set with an asterisk as a part of that symbol. For two variables that are paired by the quantifier that we have just defined, it is convenient for the variables to differ by exactly the asterisk, e.g., $(\forall (x \feq x_{*}) \fin \X \ftimes \X^*)\,R(x, x_{*})$, for $R$ a relation of arity $(\X, \X^*)$. The variables $x$ and $x_{*}$ are entirely distinct.

The example of pair production in subsection \ref{introduction.A} suggests a vivid intuition for the universal diagonal quantifier. We reframe this example in terms of quantum sets: The phase space of the electron's spin is a quantum set $\X$ that consists of a single two-dimensional atom $X$. The phase space of the positron's spin is then $\X^*$, and the phase space of the composite system is $\X \times \X^*$. Each relation $R$ of arity $(\X, \X^*)$ corresponds to a projection $r$ in $\ell^\infty(\X \times \X^*)$, i.e., to a Boolean observable on the composite system $\ell^\infty(\X) \stensor \ell^\infty(\X)^{op}$. We will soon show that $\[(\forall (x \feq x_*) \fin \X \ftimes \X^*)\, R(x, x_*)\] = \top$ if and only if $\delta_{\ell^\infty(\X)} \leq r$; see Proposition \ref{computation.C.2}. Therefore, $\[(\forall (x \feq x_*) \fin \X \ftimes \X^*)\, R(x, x_*)\] = \top$ if and only if a measurement of $r$ is guaranteed to yield $1$ whenever the composite system is prepared via a neutral pion decay.

The suggested intuition is that we may analogously produce element-antielement pairs from any nonempty quantum set $\X$ and that the universal diagonal quantifier refers to all element-antielements pairs produced in this way. To the extent that a variable $x$ of sort $\X$ may be regarded as naming an element of $\X$, the variable $x_*$ of sort $\X^*$ may be regarded as naming the corresponding antielement. However, the variables $x$ and $x_*$ are formally unrelated, and the elements and antielements of $\X$ are just figures of speech.

\subsection{Function graphs.}\label{functions}\label{definition.F}  Functions may be treated logically as relations. Classically, we may identify each function $f$ from a set $X$ to a set $Y$ with its graph relation $\[(x,y)\in X \times Y \suchthat f(x) = y\]$. We follow the same approach in the quantum setting.

Let us suppose that $F$ is a function from a quantum set $\X$ to a quantum set $\Y$ in some appropriate sense. As the variable $x$ ranges over $\X$, the term $F(x)$ ranges in $\Y$, so the graph relation of $F$ is a relation defined by the formula $E_\Y(F(x), y)$. The equality relation $E_\Y$ has arity $(\Y, \Y^*)$, so the variable $y$ must range over $\Y^*$, not $\Y$. Thus, the graph relation of a function $F$ from a quantum set $\X$ to a quantum set $\Y$ should be a relation of arity $(\X, \Y^*)$.  Therefore, we define a function graph from $\X$ to $\Y$ to be a relation $G$ of arity $(\X, \Y^*)$ that is univalent in $\X$ and total in $\X$, expressing both properties by primitive formulas:

\begin{definition}\label{definition.F.1}
Let $\X$ and $\Y$ be quantum sets. A relation $G$ of arity $(\X,\Y^*)$ is said to be a \emph{function graph} if it satisfies
\begin{enumerate}
\item $\[(\forall x \fin \X)\,(\exists y_* \fin \Y^*)\,G(x,y_*)\] = \top$, and 
\item $\[(\forall y_1 \fin \Y)\,(\forall y_{2*} \fin \Y^*)\,((\exists (x \feq x_*) \fin \X\ftimes \X^*)\,(G_*(x_*,y_{1}) \AND G(x, y_{2*})) \IMPLIES E_\Y(y_1, y_{2*}))\] \\ = \top$.
\end{enumerate}
\end{definition}

This definition is recognizable from ordinary logic. The placement of asterisks in the second formula is essentially dictated by the arity of $E_\Y$. We reason that the variables $y_1$ and $y_{2*}$ must be of sorts $\Y$ and $\Y^*$, respectively, so the first atomic subformula must use the conjugate relation $G_*$ in place of $G$. This implies that the variable $x_*$ in the first atomic subformula must be of sort $\X^*$. Likewise, the variable $x$ in the second atomic subformula must be of sort $\X$.

For each function graph $G$ of arity $(\X, \Y^*)$, we extend the language by adding a function symbol $\breve G$. Anticipating Theorem \ref{computation.E.2}, we formally define $\breve G$ to be the binary relation $(G \times I_\Y) \circ (I_\X \times E_\Y^*)$ from $\X$ to $\Y$ \cite{Kornell}*{sec.~3}. Reasoning graphically, as described in subsection \ref{computation.A}, it is easy to see that $R = E_\Y \circ (\breve R \times I_{\Y^*})$ for each relation $R$ of arity $(\X, \Y^*)$. Thus, the mapping $G \mapsto \breve G$ is injective, and we introduce no ambiguity by defining our function symbols in this way. 

We remark that a relation may have more than one arity. Finite Cartesian products are formally defined associating to the left, so a function graph $G$ of arity $(\X_1 \times \cdots \times \X_m, \Y^*)$ is also a relation of arity $(\X_1, \ldots, \X_m, \Y^*)$. Hence, we will use the function symbol $\breve G$ both with one argument and with $n$ arguments, Definition \ref{definition.F.1} notwithstanding. This is made precise in subsection \ref{definition.G}.

\subsection{Interpreting nonduplicating formulas}\label{definition.G}

We now define the class of nonduplicating formulas, and we extend the semantics of subsection \ref{definition.C} to this class.

\begin{definition}\label{definition.G.1}
The class of \emph{nonduplicating terms} is defined recursively: a variable of sort $\X$ is a nonduplicating term of sort $\X$, and for each function graph $G$ of arity $(\X_1, \ldots, \X_m, \Y^*)$, if $s_1, \ldots, s_m$ are nonduplicating terms of sorts $\X_1, \ldots, \X_m$, respectively, and no two of these terms have a variable in common, then the expression $\breve G(s_1, \ldots, s_m)$ is a nonduplicating term of sort $\Y$. Furthermore, for each relation $R$ of arity $(\Y_1, \ldots, \Y_n)$, if $t_1, \ldots, t_n$ are nonduplicating terms of sorts $\Y_1, \ldots, \Y_n$, respectively, and no two of these terms have a variable in common, then the expression $R(t_1, \ldots, t_n)$ is a \emph{nonduplicating atomic formula}.
\end{definition}

\begin{definition}\label{definition.G.2}
Let $\Y_1, \ldots, \Y_n$ be quantum sets, let $R$ be a relation of arity $(\Y_1, \ldots \Y_n)$. If $R(t_1, \ldots, t_n)$ is a nonduplicating atomic formula that is not primitive, then it abbreviates the nonduplicating formula
$$ (\exists (y_n \feq y_{n*}) \fin \Y_n \ftimes \Y_n^*)\cdots(\exists (y_1 \feq y_{1*}) \fin \Y_1  \ftimes \Y_1^*)\,
(R(y_1, \ldots, y_n) \AND t_1 \evaluatesto y_{1*} \AND \cdots \AND t_n \evaluatesto y_{n*}),$$
where the variables $y_1, \ldots, y_n$ and $y_{1*}, \ldots, y_{n*}$ are all new in the sense that they do not occur in the formula $R(t_1, \ldots, t_n)$. In this context, a formula of the form $t \evaluatesto y_*$, for $t$ of sort $\Y$ and $y_*$ of sort $\Y^*$, abbreviates $E_\Y(t, y_*)$ if $t$ is a variable and $G(s_1, \ldots, s_m, y_*)$ if $t$ is of the form $\breve G(s_1, \ldots, s_m)$, for some terms $s_1, \dots, s_m$.
\end{definition}

Thus, every nonduplicating atomic formula that is not primitive abbreviates a primitive formula. For example, the formula $P(\breve G(x))$, with $G$ a function graph of arity $(\X, \Y^*)$, abbreviates the formula $(\exists (y\feq y_*) \fin \Y \ftimes \Y^*)\,(P(y) \AND G(x, y_*))$, which in turn abbreviates the formula $(\exists y_* \fin \Y^*)\,(\exists y \fin \Y)\,(E_\Y(y, y_*) \IMPLIES ( P(y) \AND G(x,y_*)))$, which finally abbreviates the primitive formula $(\exists y_* \fin \Y^*)\,(\forall y \fin \Y)\,(\NOT E_\Y(y, y_*) \OR (E_\Y(y, y_*) \AND ( P(y) \AND G(x,y_*))))$. This observation extends easily to the class of all nonduplicating formulas:

\begin{definition}\label{definition.G.3}
The class of \emph{nonduplicating formulas} is defined recursively: each nonduplicating atomic formula is a nonduplicating formula, and if $\phi$ and $\psi$ are nonduplicating formulas and $x$ is variable of some sort $\X$, then the expressions $\neg \phi$, $\phi \AND \psi$, $\phi \OR \psi$, $\phi \rightarrow \psi$, $(\forall x \in \X)\,\phi$ and $(\exists x \in \X)\,\phi$ are nonduplicating formulas.
\end{definition}

The abbreviations that we have defined in section \ref{definition} together define a translation, i.e., a class function from nonduplicating formulas to primitive formulas, which fixes all of the primitive formulas. This extends our interpretation of primitive formulas to all nonduplicating formulas. Formally, for each nonduplicating formula $\phi(x_1, \ldots, x_n)$,
we define
$$\[(x_1, \ldots, x_n)\in\X_1 \times \cdots \times \X_n \suchthat \phi(x_1, \ldots, x_n)\] =  \[(x_1, \ldots, x_n)\in\X_1 \times \cdots \times \X_n \suchthat \tilde\phi(x_1, \ldots, x_n)\],$$
where $\tilde\phi(x_1, \ldots, x_n)$ is the translation of $\phi(x_1, \ldots, x_n)$. Note that the translation has exactly the same free variables as the original formula.

For the sake of the exposition, quantification over the diagonal remains an informal abbreviation, i.e., $(\forall (x \feq x_*) \fin \X \ftimes \X^*)\, \phi(x,x_*,y_1,\ldots,y_n)$ is the formula $(\forall x_* \fin \X^*)\, (\forall x \fin \X)\, (E_\X(x, x_*) \rightarrow \phi(x, x_*, y_1, \ldots, y_n))$ for each nonduplicating formula $\phi(x, x_*, y_1, \ldots, y_n)$, and it is likewise for $(\exists (x \feq x_*) \fin \X \ftimes \X^*)\, \phi(x,x_*,y_1,\ldots,y_n)$. Similarly, equivalence remains an informal abbreviation, i.e., $\phi(x_1, \ldots, x_n) \leftrightarrow \psi(x_1, \ldots, x_n)$ is the formula $(\phi(x_1, \ldots, x_n) \rightarrow \psi(x_1, \ldots, x_n)) \AND (\psi(x_1, \ldots, x_n) \rightarrow \phi(x_1, \ldots, x_n))$ for all nonduplicating formulas $\phi(x_1, \ldots, x_n)$ and $\psi(x_1, \ldots, x_n)$.

\section{computation}\label{computation}

Computation with the relations that we have defined is most easily performed with the aid of wire diagrams. A relation $R$ of some arity $(\X_1, \ldots, \X_n)$ is also a binary relation from $\X_1 \times \cdots \times \X_n$ to $\mathbf 1$ in the sense of \cite{Kornell}. Hence $R(X_1 \tensor \cdots \tensor X_n) = R(X_1 \tensor \cdots \tensor X_n, \CC)$ for all atoms $X_1 \in \At(\X_1)$, $X_2 \in \At(\X_2)$, etc.
The category of quantum sets and binary relations is compact closed and therefore supports a graphical calculus in which binary relations are depicted as boxes and quantum sets are depicted as wires \cite{AbramskyCoecke}. 

\subsection{Wire diagrams}\label{computation.A}
A binary relation $B$ from a product $\X_1 \times \cdots \times \X_n$ to a product $\Y_1 \times \cdots \times \Y_m$ is depicted as a box with $n$ wires entering the box from the bottom, each associated to one of the quantum sets $\X_1, \ldots, \X_n$, and with $m$ wires leaving the box from the top, each associated to one of the quantum sets $\Y_1, \ldots, \Y_m$. A relation $R$ of arity $(\X_1, \ldots, \X_n)$ is therefore depicted as a box with wires coming just from below. See Figure \ref{figure}, below.
\begin{figure}[h]
$$
      \begin{tikzpicture}[scale=1]
	\begin{pgfonlayer}{nodelayer}
		\node [style=box] (0) at (0,0) {\,\;\;$B$\;\;\,};
		\node [style=none] (B) at (0.4,0) {};
		\node [style=none] (B0) at (0.4,-0.7) {$\scriptstyle\X_n$};
		\node [style=none] (A) at (-0.4,0) {};
		\node [style=none] (A0) at (-0.4,-0.7) {$\scriptstyle\X_1$};
		\node [style=none] (10) at (0,-0.4) {$\cdots$};
		\node [style=none] (A1) at (-0.4, 0.7) {$\scriptstyle\Y_1$};
		\node [style=none] (B1) at (0.4, 0.7) {$\scriptstyle\Y_m$};
		\node [style=none] (11) at (0, 0.4) {$\cdots$};
	\end{pgfonlayer}
	\begin{pgfonlayer}{edgelayer}
	    \draw[arrow,markat=0.4] (B0) to (B);
	    \draw[arrow,markat=0.4] (A0) to (A);
	    \draw[arrow,markat=0.8] (A) to (A1);
	    \draw[arrow,markat=0.8] (B) to (B1);
	\end{pgfonlayer}
      \end{tikzpicture}
\qquad \qquad
      \begin{tikzpicture}[scale=1]
	\begin{pgfonlayer}{nodelayer}
		\node [style=box] (0) at (0,0) {\,\;\;$R$\;\;\,};
		\node [style=none] (B) at (0.4,0) {};
		\node [style=none] (B0) at (0.4,-0.7) {$\scriptstyle\X_n$};
		\node [style=none] (A) at (-0.4,0) {};
		\node [style=none] (A0) at (-0.4,-0.7) {$\scriptstyle\X_1$};
		\node [style=none] (10) at (0,-0.4) {$\cdots$};
	\end{pgfonlayer}
	\begin{pgfonlayer}{edgelayer}
	    \draw[arrow,markat=0.4] (B0) to (B);
	    \draw[arrow,markat=0.4] (A0) to (A);
	\end{pgfonlayer}
      \end{tikzpicture}
\qquad \qquad
      \begin{tikzpicture}[scale=1]
	\begin{pgfonlayer}{nodelayer}
		\node [style=none] (0) at (0,0) {$E_\X$};
		\node [style=none] (B) at (0.4,0) {};
		\node [style=none] (B1) at (0.4,0.1) {};
		\node [style=none] (B0) at (0.4,-0.7) {$\scriptstyle\X$};
		\node [style=none] (A) at (-0.4,0) {};
		\node [style=none] (A1) at (-0.4,0.1) {};
		\node [style=none] (A0) at (-0.4,-0.7) {$\scriptstyle\X$};
	\end{pgfonlayer}
	\begin{pgfonlayer}{edgelayer}
	    \draw (B0) to (B1);
	    \draw (A0) to (A1);
	    \draw [arrow, bend left=90, looseness=2.25] (A) to (B);
	\end{pgfonlayer}
      \end{tikzpicture}
\qquad \qquad
      \begin{tikzpicture}[scale=1]
	\begin{pgfonlayer}{nodelayer}
		\node [style=none] (A) at (0.5,0) {$I_\X$};
		\node (A1) at (0,0.7) {};
		\node [style=none] (A0) at (0,-0.7) {$\scriptstyle\X$};
	\end{pgfonlayer}
	\begin{pgfonlayer}{edgelayer}
	    \draw[arrow] (A0) to (A1);
	\end{pgfonlayer}
      \end{tikzpicture}
\qquad \qquad
      \begin{tikzpicture}[scale=1]
	\begin{pgfonlayer}{nodelayer}
		\node [style=none] (A) at (0,0) {};
		\node [style=none] (1) at (0,-0.1) {$\bullet$};
		\node [style=none] (A0) at (0,-0.7) {$\scriptstyle\X$};
		\node (0) at (0, 0.3) {$\top_\X$};
	\end{pgfonlayer}
	\begin{pgfonlayer}{edgelayer}
	    \draw[arrow] (A0) to (A);
	\end{pgfonlayer}
      \end{tikzpicture}
$$
\caption{Some depicted binary relations.}\label{figure}
\end{figure}
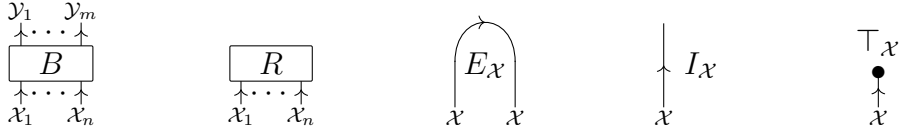

We orient each wire, with downward-oriented wires corresponding to dual quantum sets. In other words, a downward-oriented wire labeled $\X$ corresponds to the quantum set $\X^*$. The advantage of this notation is that the equality relation, which is also the counit of the dagger compact structure on the category of quantum sets and binary relations, can be depicted simply as an arc. For each quantum set $\X$, the identity binary relation $I_\X$ on $\X$ is depicted simply as a wire, and the maximum predicate $\top_\X$ is depicted by a ``loose end'', which we will sometimes ``pull away'', that is, completely omit. See Figure \ref{figure}, above.

In this diagrammatic calculus, the monoidal product of two morphisms, i.e., of two binary relations, is depicted by placing the corresponding diagrams side by side. Thus, for all quantum sets $\X$ and $\Y$, we have the equation 
$$
\begin{aligned}
      \begin{tikzpicture}[scale=1]
	\begin{pgfonlayer}{nodelayer}
		\node [style=none] (A) at (0,-0.3) {};
		\node [style=none] (1) at (0,-0.4) {$\bullet$};
		\node [style=none] (A0) at (0,-1) {$\scriptstyle\X \times \Y$};
	\end{pgfonlayer}
	\begin{pgfonlayer}{edgelayer}
	    \draw[arrow] (A0) to (A);
	\end{pgfonlayer}
      \end{tikzpicture}
\end{aligned}
      \quad
      =
      \quad
\begin{aligned}
      \begin{tikzpicture}[scale=1]
	\begin{pgfonlayer}{nodelayer}
		\node [style=none] (A) at (0,-0.3) {};
		\node [style=none] (A1) at (0,-0.4) {$\bullet$};
		\node [style=none] (A0) at (0,-1) {$\scriptstyle\X$};
		\node [style=none] (B) at (0.5,-0.3) {};
		\node [style=none] (B1) at (0.5,-0.4) {$\bullet$};
		\node [style=none] (B0) at (0.5,-1) {$\scriptstyle\Y$};
	\end{pgfonlayer}
	\begin{pgfonlayer}{edgelayer}
	    \draw[arrow] (A0) to (A);
  	    \draw[arrow] (B0) to (B);
	\end{pgfonlayer}
      \end{tikzpicture}
\end{aligned}
$$
because $\top_{\X \times \Y} = \top_\X \times \top_\Y$. Similarly, the composition of binary relations is depicted by placing one diagram above the other and tying together the correspond wires. For example, if $B$ is a binary relation from $\X$ to $\Y$, then the binary relation $\bend B := E_\Y \circ (B \times I_{\Y^*})$ form $\X \times \Y^*$ to $\mathbf 1$ is depicted in the following diagram:
$$
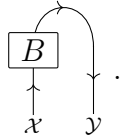
\begin{aligned}
      \begin{tikzpicture}[scale=1]
	\begin{pgfonlayer}{nodelayer}
		\node [style=box] (0) at (-0.4 ,0) {\,$B$\,};
		\node [style=none] (B) at (0.4,0) {};
		\node [style=none] (B1) at (0.4,0.1) {};
		\node [style=none] (B0) at (0.4,-1) {$\scriptstyle\Y$};
		\node [style=none] (A) at (-0.4,0) {};
		\node [style=none] (A0) at (-0.4,-1) {$\scriptstyle\X$};
	\end{pgfonlayer}
	\begin{pgfonlayer}{edgelayer}
	    \draw [arrow] (B1) to (B0);
	    \draw [arrow] (A0) to (A);
	    \draw [arrow, bend left=90, looseness=2.25] (A) to (B);
	\end{pgfonlayer}
      \end{tikzpicture}
      \end{aligned}\;.
$$

We will often use variables to label the wires of a diagram in order to distinguish various occurrences of the same quantum set, particularly when depicting the interpretation of a formula. For example, the relation $\[(x,x_*,y_1, y_2, y_3) \in \X \times \X^* \times \Y \times \Y \times \Y \suchthat E_\X(x, x_*)\]$ is depicted in the following diagram:
$$
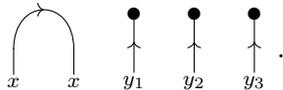
\begin{aligned}
      \begin{tikzpicture}[scale=1]
	\begin{pgfonlayer}{nodelayer}
		\node [style=none] (0) at (0,0) {};
		\node [style=none] (B) at (0.4,0) {};
		\node [style=none] (B1) at (0.4,0.1) {};
		\node [style=none] (B0) at (0.4,-0.4) {$\scriptstyle x$};
		\node [style=none] (A) at (-0.4,0) {};
		\node [style=none] (A1) at (-0.4,0.1) {};
		\node [style=none] (A0) at (-0.4,-0.4) {$\scriptstyle x$};
		\node [style=none] (C0) at (1.2,-0.4) {$\scriptstyle y_1$};
		\node [style=none] (D0) at (2.0,-0.4) {$\scriptstyle y_2$};
		\node [style=none] (E0) at (2.8,-0.4) {$\scriptstyle y_3$};
		\node [style=none] (C1) at (1.2,0.1) {};
		\node [style=none] (D1) at (2.0,0.1) {};
		\node [style=none] (E1) at (2.8,0.1) {};
		\node [style=none] (C) at (1.2,0.5) {$\bullet$};
		\node [style=none] (D) at (2.0,0.5) {$\bullet$};
		\node [style=none] (E) at (2.8,0.5) {$\bullet$};
		\node [style=none] (C2) at (1.2,0.6) {};
		\node [style=none] (D2) at (2.0,0.6) {};
		\node [style=none] (E2) at (2.8,0.6) {};
	\end{pgfonlayer}
	\begin{pgfonlayer}{edgelayer}
	    \draw (B0) to (B1);
	    \draw (A0) to (A1);
	    \draw [arrow, bend left=90, looseness=2.25] (A) to (B);
	    \draw [arrow] (C0) to (C2);
	    \draw [arrow] (D0) to (D2);
	    \draw [arrow] (E0) to (E2);
	\end{pgfonlayer}
      \end{tikzpicture}
      \end{aligned}\;.
$$

The defining properties of a dagger compact category are such that wires may be deformed in the intuitive way. Boxes may be moved around or even turned upside down, which corresponds to dualization in the sense of the dagger compact structure. Thus, for any binary relation $B$ from a quantum set $\X$ to a quantum set $\Y$, we have the following:

$$
\begin{aligned}
\begin{tikzpicture}[scale=1]
	\begin{pgfonlayer}{nodelayer}
    \node [style=box] (0) at (0,0) {\,$B$\,};
    \node [style=none] (top) at (0,1) {$\scriptstyle\Y$};
    \node [style=none] (bot) at (0,-1) {$\scriptstyle\X$};
	\end{pgfonlayer}
	\begin{pgfonlayer}{edgelayer}
	\draw [arrow] (bot) to (0);
	\draw [arrow] (0) to (top);
	\end{pgfonlayer}
   \end{tikzpicture}
\end{aligned}
\qquad = \qquad
\begin{aligned}
\begin{tikzpicture}[scale=1]
	\begin{pgfonlayer}{nodelayer}
    \node [style=box] (0) at (0,0) {\rotatebox[origin=c]{180}{\,$B$\,}};
    \node [style=none] (top) at (1,1) {};
    \node [style=none] (bot) at (-1,-1) {};
    \node (left) at (-1,0) {$\scriptstyle \X$};
    \node (right) at (1,0) {$\scriptstyle\Y$};
	\end{pgfonlayer}
	\begin{pgfonlayer}{edgelayer}
	\draw [arrow] (bot) to (left);
	\draw [arrow] (right) to (top);
	\draw [arrow, bend left=90, looseness=2.25] (left) to (0);
	\draw [arrow, bend right=90, looseness=2.25] (0) to (right);
	\end{pgfonlayer}
   \end{tikzpicture}
\end{aligned}
\qquad = \qquad
\begin{aligned}
\begin{tikzpicture}[scale=1]
	\begin{pgfonlayer}{nodelayer}
    \node [style=box] (0) at (0,0) {\,$B^*$};
    \node [style=none] (top) at (1,1) {};
    \node [style=none] (bot) at (-1,-1) {};
    \node (left) at (-1,0) {$\scriptstyle \X$};
    \node (right) at (1,0) {$\scriptstyle\Y$};
	\end{pgfonlayer}
	\begin{pgfonlayer}{edgelayer}
	\draw [arrow] (bot) to (left);
	\draw [arrow] (right) to (top);
	\draw [arrow, bend left=90, looseness=2.25] (left) to (0);
	\draw [arrow, bend right=90, looseness=2.25] (0) to (right);
	\end{pgfonlayer}
   \end{tikzpicture}
\end{aligned}\;.
$$

The expression $\Rel(\X;\Y)$ denotes the set of all binary relations from a quantum set $\X$ to a quantum set $\Y$ in \cite{Kornell} and in the present paper. For quantum sets $\X_1, \ldots, \X_n$ and $\Y_1, \ldots, \Y_m$, monoidal closure yields a canonical bijection between $\Rel(\X_1 \times \cdots \times \X_n; \Y_1 \times \cdots \times \Y_m)$ and $\Rel(\X_1 \times \cdots \times \X_n \times \Y_1^* \times \cdots \times \Y_m^*; \mathbf 1)$, so whether a wire leaves the diagram upward or downward has no fundamental significance beyond sorting the factors between the domain and the codomain.  Thus, binary relations are all essentially predicates. Similarly, our wire diagrams are essentially diagrams in a space with no top or bottom; each box denotes some predicate, and each emanates wires according to the arity of that predicate.

One significant advantage of diagrammatic computation is the ease with which we can permute the variables of a context. Formally, we appeal to the following proposition, which is proved in Appendix \ref{appendix.J}:

\begin{proposition}[also \ref{appendix.J.1}]\label{computation.A.1}
Let $\X_1, \ldots, \X_n$ be quantum sets, and let $\pi$ be a permutation of $\{1, \ldots, n\}$. Let $U_\pi$ be the canonical isomorphism \cite{MacLane}*{Thm.~XI.1.1} from $\X_1\times \cdots \times \X_n$ to $\X_{\pi(1)} \times \cdots \times \X_{\pi(n)}$ in the symmetric monoidal category of quantum sets and binary relations \cite{Kornell}*{sec.~3}. Then, $\pi_\#(R) = R \circ U_\pi$ for all relations $R$ of arity $(\X_{\pi(1)}, \ldots, \X_{\pi(n)})$.
\end{proposition}

\noindent Together, Propositions \ref{definition.C.3} and \ref{computation.A.1} allow us to quickly compute simplified diagrams depicting interpreted formulas. For example, writing $$R := \[(x_1, x_{2*}, y) \in \X \times \X^* \times \Y \suchthat E_\X(x_1, x_{2*})\],$$ we may compute that
$$
\[(x_{2*}, y, x_1) \in \X^* \times \Y \times \X \suchthat E_\X(x_1, x_{2*})\]
\; = \quad
\begin{aligned}
      \begin{tikzpicture}[scale=1]
	\begin{pgfonlayer}{nodelayer}
		\node [style=box] (0) at (0,-.1) {$\ghost\,\;\; R\;\;\,\ghost$};
		\node [style=none] (A) at (-0.4,-0.2) {};
		\node (A0) at (-0.4,-1.3) {$\scriptstyle x_{2}$};
		\node [style=none] (B) at (0, -0.2) {};
	    \node (B0) at (0,-1.3) {$\scriptstyle y$};
		\node [style=none] (C) at (0.4,-0.2) {};
		\node (C0) at (0.4,-1.3) {$\scriptstyle x_1$};
	\end{pgfonlayer}
	\begin{pgfonlayer}{edgelayer}
	    \draw [arrow, in=90, out=270, markat=0.8] (B.center) to (A0);
	    \draw [arrow, in=270, out=90, markat=0.3] (B0) to (C.center);
	    \draw [arrow, in=270, out=90, markat=0.3] (C0) to (A.center);
	\end{pgfonlayer}
      \end{tikzpicture}
\end{aligned}
\quad = \quad
\begin{aligned}
      \begin{tikzpicture}[scale=1]
	\begin{pgfonlayer}{nodelayer}
		\node [style=none] (A) at (-0.4,-0.2) {};
		\node (A0) at (-0.4,-1.3) {$\scriptstyle x_{2}$};
		\node [style=none] (B) at (0, -0.2) {};
	    \node (B0) at (0,-1.3) {$\scriptstyle y$};
		\node [style=none] (C) at (0.4,-0.2) {$\bullet$};
		\node (C0) at (0.4,-1.3) {$\scriptstyle x_1$};
	\end{pgfonlayer}
	\begin{pgfonlayer}{edgelayer}
	    \draw [arrow, in=90, out=270, markat=0.8] (B.center) to (A0);
	    \draw [arrow, in=270, out=90, markat=0.3] (B0) to (C.center);
	    \draw [arrow, in=270, out=90, markat=0.3] (C0) to (A.center);
	    \draw [arrow, bend left=90, looseness=2.25] (A.center)to (B.center);
	\end{pgfonlayer}
      \end{tikzpicture}
\end{aligned}
\quad = \quad
\begin{aligned}
      \begin{tikzpicture}[scale=1]
	\begin{pgfonlayer}{nodelayer}
		\node [style=none] (A) at (-0.4,-0.2) {};
		\node (A0) at (-0.4,-1) {$\scriptstyle x_{2}$};
		\node [style=none] (B) at (0, -0.2) {$\bullet$};
	    \node (B0) at (0,-1) {$\scriptstyle y$};
		\node [style=none] (C) at (0.4,-0.2) {};
		\node (C0) at (0.4,-1) {$\scriptstyle x_1$};
	\end{pgfonlayer}
	\begin{pgfonlayer}{edgelayer}
	    \draw [arrow, markat=0.75] (A.center) to (A0);
	    \draw [arrow] (B0) to (B.center);
	    \draw [arrow] (C0) to (C.center);
	    \draw [arrow, bend right=90, looseness=2.25] (C.center)to (A.center);
	\end{pgfonlayer}
      \end{tikzpicture}
\end{aligned}\;.
$$

\noindent The weave of wires below the box depicting the relation $R$ depicts the canonical isomorphism from $\X^* \times \Y \times \X$ to $\X \times \X^* \times \Y$ that is derived from the symmetric monoidal structure of the category of quantum sets and binary relations. 

\subsection{Standard quantifiers.}\label{computation.C} We establish two basic propositions about the standard quantifiers $\forall$ and $\exists$.

\begin{lemma}\label{computation.C.1}
Let $\phi(x_1, \ldots, x_n)$ be a nonduplicating formula, with $x_1, \dots, x_n$ of sorts $\X_1, \ldots, \X_n$, respectively. For all $m \in \{0, \ldots, n\}$, we have
\begin{align*}
\[ & (x_{m+1} \ldots, x_n)  \in \X_{m+1} \times \cdots \times \X_n
\suchthat
(\forall x_m \fin \X_m)\cdots (\forall x_1 \fin \X_1)\,\phi(x_1, \ldots, x_n)\]
\\ & =
\sup\{
R \in \Rel(\X_{m+1}, \ldots, \X_n)
\suchthat
\top_{\X_1} \times \cdots \times \top_{\X_m} \times R 
\\ & \hspace{40ex} \leq
\[ (x_1, \ldots, x_n) \in \X_1 \times \cdots \times \X_n \suchthat\phi(x_1, \ldots, x_n)\]
\}.
\end{align*}
The special case $m=n$ shows that $\[(\forall x_n \fin \X_n)\cdots (\forall x_1 \fin \X_1)\,\phi(x_1, \ldots,x_n)\] = \top$ if and only if $\[(x_1, \ldots, x_n) \in \X_1 \times \cdots \times \X_n \suchthat \phi(x_1, \ldots,x_n)\]$ is the maximum relation of arity $(\X_1, \ldots, \X_n)$.
\end{lemma}

\begin{proof}
For each $m \in \{0, \ldots, n\}$, write $R_m$ for the left side of the equality. Hence, we are to show that for all $m \in \{0, \ldots, n\}$, we have that $R_m = \sup\{
R \in \Rel(\X_{m+1}, \ldots, \X_n)
\suchthat
\top_{\X_1} \times \cdots \times \top_{\X_m} \times R \leq R_0\}$. We proceed by induction on $m$. The base case is just the obvious equality $R_0 = \sup\{R \in \Rel(\X_1, \ldots, \X_n)\suchthat R \leq R_0\}$. For the induction step, we assume that the desired equality holds for some natural $m-1 \in \{0, \ldots, n-1\}$, and thus, we have the following:
$$R_{m-1} = \sup\{
R \in \Rel(\X_m, \ldots , \X_n)
\suchthat
\top_{\X_1} \times \cdots \times \top_{\X_{m-1}} \times R \leq R_0
\},
$$
$$
R_m = \sup\{R\in \Rel(\X_{m+1}, \ldots, \X_n)
\suchthat
\top_{\X_m} \times R \leq R_{m- 1}
\}.
$$
In particular, we have that $\top_{\X_1} \times \cdots \times \top_{\X_{m-1}} \times \top_{\X_m} \times R_m \leq \top_{\X_1} \times \cdots \times \top_{\X_{m-1}} \times R_{m-1} \leq R_0$. Now, suppose that $R$ is any other relation that satisfies $\top_{\X_1} \times \cdots \times \top_{\X_{m-1}} \times \top_{\X_m} \times R \leq R_0$. It follows that $\top_{\X_m} \times R \leq R_{m-1}$ by the first equation and then that $R \leq R_m$ by the second equation. Therefore, $R_m$ is indeed the supremum of the relations $R$ satisfying $\top_{\X_1} \times \cdots \times \top_{\X_m} \times R \leq R_0$.
\end{proof}

\begin{proposition}\label{computation.C.2}
Let $\phi(x_1, \ldots, x_n)$ and $\psi(x_1, \ldots, x_n)$ be nonduplicating formulas, with $x_1, \ldots, x_n$ of sorts $\X_1, \ldots, \X_n$, respectively. Then $$\[(\forall x_n \fin \X_n)\cdots (\forall x_1 \fin \X_1)\,(\phi(x_1, \ldots, x_n) \IMPLIES \psi (x_1, \ldots, x_n))\] = \top$$ if and only if $$\[(x_1, \ldots, x_n) \in \X_1 \times \cdots \times \X_n \suchthat \phi(x_1, \ldots, x_n)\] \leq \[(x_1, \ldots, x_n) \in \X_1 \times \cdots \times \X_n \suchthat \psi(x_1, \ldots, x_n)\].$$
\end{proposition}

\begin{proof}
By Lemma \ref{computation.C.1}, the equation is true if and only if $ \[\phi(x_1, \ldots, x_n)\] \IMPLIES \[\psi(x_1, \ldots, x_n)\] $ is the maximum relation of arity $(\X_1, \ldots, \X_n)$. This condition is equivalent to the claimed inequality by a fundamental property of the Sasaki arrow \cite{Hardegree}.
\end{proof}

\begin{proposition}\label{computation.C.3}
Let $\phi(x_1, \ldots, x_n)$ be a nonduplicating formula, with $x_1, \ldots,x_n$ of sorts $\X_1, \ldots,\X_n$. Then,
\begin{align*}
\[(x_2, \ldots, x_n) & \in \X_2 \times \cdots  \times \X_n \suchthat (\exists x_1 \fin \X_1)\,\phi(x_1, \ldots, x_n)\] \\ & = \[(x_1, \ldots, x_n) \in \X_1 \times \cdots \times \X_n \suchthat \phi(x_1, \ldots, x_n)\] \circ (\top_{\X_1}^\dagger \times I_{\X_2} \times \cdots \times I_{\X_n})
\\ & = \quad
\begin{aligned}
\begin{tikzpicture}[scale=1]
	\begin{pgfonlayer}{nodelayer}
        \node [style=box] (0) at (0,0) {$\[\phi(x_1, \ldots, x_n)\]$};
        \node [style=none] (A) at (-1.2,-0.2) {};
        \node [style=none] (A1) at (-1.2,-0.8) {};
        \node [style=none] (1) at (-1.2,-0.8) {$\bullet$};
        \node [style=none] (B) at (-0.9,-0.2) {};
        \node [style=none] (C) at (-.8,-0.2) {};
        \node [style=none] (E) at (1.2,-0.2) {};
        \node [style=none] (C1) at (-.8,-1) {};
        \node [style=none] (E1) at (1.2,-1) {};
        \node [style=none] (D) at (0.2,-0.65){$\cdots\cdots$};
	\end{pgfonlayer}
	\begin{pgfonlayer}{edgelayer}
	    \draw [arrow] (C1) to (C);
	    \draw [arrow] (E1) to (E);
	    \draw [arrow] (A1) to (A);
	\end{pgfonlayer}
   \end{tikzpicture}
\end{aligned}\;.
\end{align*}
\end{proposition}

\begin{proof} Write $\[\phi(x_1, \ldots, x_n)\] =\[(x_1, \ldots,x_n) \in \X_1 \times \cdots \times \X_n \suchthat\phi(x_1, \ldots, x_n)\]$. We refer to \cite{KornellLindenhoviusMislove2}*{App.~C} for the relationship between the adjoint $\vphantom{a}^\dagger$, and the orthogonality relation $\perp$.
\begin{align*}
\[&(x_2, \ldots, x_n) \in \X_2 \times \cdots \times \X_n \suchthat (\exists x_1 \fin \X_1)\,\phi(x_1, \ldots, x_n)\]
\\ & =
\[(x_2, \ldots, x_n) \in \X_2 \times \cdots \times \X_n \suchthat \NOT (\forall x_1 \fin \X_1)\,\NOT \phi(x_1, \ldots, x_n)\]
\\ & =
\NOT \sup\{
R \in \Rel(\X_2, \ldots, \X_n)
\suchthat
\top_{\X_1} \times R \leq \[\NOT\phi(x_1, \ldots, x_n)\] \}
\\ & =
\NOT \sup\{
R \in \Rel(\X_2, \ldots, \X_n)
\suchthat
\top_{\X_1} \times R \perp  \[\phi(x_1, \ldots, x_n)\] \}
\\ & =
\NOT \sup\{
R \in \Rel(\X_2, \ldots, \X_n)
\suchthat
\[\phi(x_1, \ldots, x_n)\] \circ (\top_{\X_1} \times R)^\dagger = \bot\}
\\ & =
\NOT \sup\{
R \in \Rel(\X_2, \ldots, \X_n)
\suchthat
\[\phi(x_1, \ldots, x_n)\] \circ (\top_{\X_1}^\dagger \times R^\dagger)=\bot\}
\\ & =
\NOT \sup\{
R \in \Rel(\X_2, \ldots, \X_n)
\suchthat
\[\phi(x_1, \ldots, x_n)\] \circ (\top_{\X_1}^\dagger \times I_{\X_2} \times \cdots \times I_{\X_n})\circ R^\dagger=\bot\}
\\ & =
\NOT \sup\{
R \in \Rel(\X_2, \ldots, \X_n)
\suchthat
R \perp \[\phi(x_1, \ldots, x_n)\] \circ (\top_{\X_1}^\dagger \times I_{\X_2} \times \cdots \times I_{\X_n})\}
\\ & =
\NOT \NOT  (\[\phi(x_1, \ldots, x_n)\] \circ (\top_{\X_1}^\dagger \times I_{\X_2} \times \cdots \times I_{\X_n}))
\\ & =
\[\phi(x_1, \ldots, x_n)\] \circ (\top_{\X_1}^\dagger \times I_{\X_2} \times \cdots \times I_{\X_n}). \qedhere
\end{align*} 
\end{proof}

Applying diagrammatic reasoning, we find that existential quantifiers commute, as a corollary of Proposition \ref{computation.C.3}. Therefore, so do universal quantifiers. If $\X_1 = `A$ for some ordinary set $A$, then existential quantification over $\X_1$ is equivalent to a disjunction over $A$, essentially because the maximum binary relation from a singleton $\{\ast\}$ to $A$ is the disjunction of the elements of $A$, each considered as a binary relation from $\{\ast\}$ to $A$. See Lemma \ref{appendix.C.1}.

\subsection{Diagonal quantifiers.}\label{computation.D} We now characterize our two defined quantifiers over the diagonal.

\begin{proposition}\label{computation.D.1}
Let $\phi(x, x_*, y_1, \ldots, y_n)$ be a nonduplicating formula, with $x$ of sort $\X$, with $x_*$ of sort $\X^*$, and with $y_1, \ldots, y_n$ of sorts $\Y_1, \ldots \Y_n$, respectively. Write $\[\phi(x, x_*, y_1, \ldots, y_n)\]$ as an abbreviation for
$\[ (x,x_*,y_1, \ldots, y_n) \in \X \times \X^* \times \Y_1 \times \cdots \times \Y_n
\suchthat
\phi(x,x_*, y_1, \ldots, y_n)
\]$. Then, 
\begin{align*}
\[(y_1, \ldots& , y_n) \in \Y_1 \times \cdots \times \Y_n \suchthat (\forall  (x \feq x_*) \fin \X \ftimes \X^*)\, \phi(x, x_*, y_1, \ldots, y_n)\] 
\\ & =
\sup\{
R \in \Rel(\Y_1, \ldots, \Y_n)
\suchthat
E_\X \times R \leq
\[
\phi(x,x_*, y_1, \ldots, y_n)
\] \}.
\end{align*}
\end{proposition}

\begin{proof}
\begin{align*}
\[(y_1&, \ldots , y_n) \in \Y_1 \times \cdots \times \Y_n \suchthat (\forall  (x \feq x_*) \fin \X \ftimes \X^*)\, \phi(x, x_*, y_1, \ldots, y_n)\] 
\\ & =
\[(y_1, \ldots , y_n) \in \Y_1 \times \cdots \times \Y_n \suchthat (\forall x_* \fin \X^*)\,(\forall x \fin \X)\,(E_\X(x,x_*) \IMPLIES  \phi(x, x_*, y_1, \ldots, y_n))\] 
\\ & =
\sup\{R \in \Rel(\Y_1, \ldots, \Y_n)
\suchthat
\top_{\X} \times \top_{\X^*} \times R \leq \[E_\X(x,x_*) \IMPLIES \phi(x, x_*, y_1, \ldots, y_n)\]
\}
\\ & =
\sup\{R \in \Rel(\Y_1, \ldots, \Y_n)
\suchthat
\top_{\X} \times \top_{\X^*} \times R \leq \[E_\X(x,x_*)\] \IMPLIES \[\phi(x, x_*, y_1, \ldots, y_n)\]
\}
\\ & =
\sup\{R \in \Rel(\Y_1, \ldots, \Y_n)
\suchthat
(\top_{\X} \times \top_{\X^*} \times R) \sascon \[E_\X(x,x_*)\]\leq \[\phi(x, x_*, y_1, \ldots, y_n)\]
\}
\\ & =
\sup\{R \in \Rel(\Y_1, \ldots, \Y_n)
\suchthat
(\top_{\X} \times \top_{\X^*} \times R) \sascon (E_\X \times \top_{\Y_1} \times \cdots \times \top_{\Y_n}) 
\\ & \hspace{60ex}
\leq \[\phi(x, x_*, y_1, \ldots, y_n)\]
\}
\\ & =
\sup\{R \in \Rel(\Y_1, \ldots, \Y_n)
\suchthat
E_\X \times R \leq \[\phi(x, x_*, y_1, \ldots, y_n)\]
\}.
\end{align*}
In this context, $\sascon$ denotes the Sasaki projection connective, defined by $P \sascon Q = (P \OR \NOT Q ) \AND Q$ \cite{Sasaki}*{Def.~5.1}. For every relation $Q$, the mapping $P \mapsto P \sascon Q$ is left adjoint to the mapping $P \mapsto Q \rightarrow P$ \cite{Finch}.
\end{proof}

The following theorem serves as a bridge between the semantics defined in section \ref{definition}, and the interpretation of wire diagrams in the dagger compact category of quantum sets and binary relations \cite{Kornell}*{sec.~3}.

\begin{theorem}\label{computation.D.2}
From the assumptions of Proposition \ref{computation.D.1}, we have the following equality:
\begin{align*}
\[(y_1, \ldots, y_n) \in \Y_1 \times \cdots \times \Y_n & \suchthat (\exists(x\feq x_*) \fin \X \ftimes \X^*)\,\phi(x,x_*, y_1, \ldots, y_n) \] 
\\ &= \;
\[\phi(x,x_*, y_1, \ldots, y_n)\] \circ (E_\X^\dagger \times I_{\Y_1} \times \cdots \times I_{\Y_n})
\\ & \quad
\\ &= \quad
\begin{aligned}
\begin{tikzpicture}[scale=1]
	\begin{pgfonlayer}{nodelayer}
        \node [style=box] (0) at (0,0) {$\[\phi(x,x_*, y_1, \ldots, y_n)\]$};
        \node [style=none] (A) at (-1.7,-0.2) {};
        \node [style=none] (B) at (-0.8,-0.2) {};
        \node [style=none] (C) at (-0.3,-0.2) {};
        \node [style=none] (E) at (1.7,-0.2) {};
        \node [style=none] (C1) at (-0.3,-1) {};
        \node [style=none] (E1) at (1.7,-1) {};
        \node [style=none] (D) at (0.75,-0.65){$\cdots\cdots$};
	\end{pgfonlayer}
	\begin{pgfonlayer}{edgelayer}
	    \draw [arrow, bend left=90, looseness=2.25] (B) to (A);
	    \draw [arrow] (C1) to (C);
	    \draw [arrow] (E1) to (E);
	\end{pgfonlayer}
   \end{tikzpicture}
\end{aligned} \; .
\end{align*}
\end{theorem}

\begin{proof}
Appealing to Proposition \ref{computation.D.1} for the second equality, we reason that
\begin{align*}
\[&(y_1, \ldots, y_n) \in \Y_1 \times \cdots \times \Y_n  \suchthat (\exists(x\feq  x_*) \fin \X \ftimes \X^*)\,\phi(x,x_*, y_1, \ldots, y_n) \]
\\ & =
\[(y_1, \ldots, y_n) \in \Y_1 \times \cdots \times \Y_n  \suchthat \NOT (\forall(x\feq  x_*) \fin \X \ftimes \X^*)\, \NOT \phi(x,x_*, y_1, \ldots, y_n) \]
\\ & =
\NOT \sup\{
R \in \Rel(\Y_1, \ldots, \Y_n)
\suchthat
E_\X \times R \leq  \[\NOT \phi(x,x_*, y_1, \ldots, y_n)\] \}
\\ & =
\NOT \sup\{
R \in \Rel(\Y_1, \ldots, \Y_n)
\suchthat
E_\X \times R \perp  \[\phi(x,x_*, y_1, \ldots, y_n)\] \}
\\ & =
\NOT \sup\{
R \in \Rel(\Y_1, \ldots, \Y_n)
\suchthat
 \[\phi(x,x_*, y_1, \ldots, y_n)\] \circ (E_\X \times R)^\dagger =\bot\}
\\ & =
\NOT \sup\{
R \in \Rel(\Y_1, \ldots, \Y_n)
\suchthat
 \[\phi(x,x_*, y_1, \ldots, y_n)\] \circ (E_\X^\dagger \times R^\dagger)=\bot\}
\\ & =
\NOT \sup\{
R \in \Rel(\Y_1, \ldots, \Y_n)
\suchthat
 \[\phi(x,x_*, y_1, \ldots, y_n)\] \circ (E_\X^\dagger \times I_{\Y_1} \times \cdots \times I_{\Y_n})\circ R^\dagger=\bot\}
\\ & =
\NOT \sup\{
R \in \Rel(\Y_1, \ldots, \Y_n)
\suchthat
R \perp \[\phi(x,x_*, y_1, \ldots, y_n)\] \circ (E_\X^\dagger \times I_{\Y_1} \times \cdots \times I_{\Y_n})\}
\\ & =
\NOT \NOT ( \[\phi(x,x_*, y_1, \ldots, y_n)\] \circ (E_\X^\dagger \times I_{\Y_1} \times \cdots \times I_{\Y_n}))
\\ & =
\[\phi(x,x_*, y_1, \ldots, y_n)\] \circ (E_\X^\dagger \times I_{\Y_1} \times \cdots \times I_{\Y_n}). \qedhere
\end{align*}
\end{proof}

If $\X = `A$ for some ordinary set $A$, then existential quantification over the diagonal of $\X \times \X^*$ is equivalent to a disjunction over $A$. See Lemma \ref{appendix.C.2}.

\subsection{Functions.}\label{computation.E} Let $\X$ and $\Y$ be quantum sets. We now show that functions from $\X$ to $\Y$ in the sense of \cite{Kornell}*{Def.~4.1} are in canonical bijective correspondence with function graphs of arity $(\X, \Y^*)$ in the sense of Definition \ref{definition.F.1}. This correspondence is given by 
$$
F  \;\mapsto\; E_\Y \circ (F \times I_{\Y^*}),
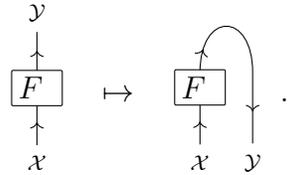
$$
$$
\begin{aligned}
      \begin{tikzpicture}[scale=1]
	\begin{pgfonlayer}{nodelayer}
		\node [style=box] (F) at (0,0) {$F^{\phantom{*}}$};
		\node [style=none] (Xanch) at (0,-0.2) {};
		\node [style=none] (Yanch) at (0,0.2) {};
		\node  (X) at (0,-1) {$\scriptstyle \X$};
		\node  (Y) at (0,1) {$\scriptstyle \Y$};
	\end{pgfonlayer}
	\begin{pgfonlayer}{edgelayer}
	    \draw [arrow,markat=0.55] (X) to (Xanch.center);
	    \draw [arrow,markat=0.55] (Yanch.center) to (Y);
	\end{pgfonlayer}
      \end{tikzpicture}
\end{aligned}
\quad
 \mapsto
\quad
\begin{aligned}
      \begin{tikzpicture}[scale=1]
	\begin{pgfonlayer}{nodelayer}
		\node  (phantomY) at (0,1) {$\scriptstyle \phantom{\Y}$};
		\node [style=box] (F) at (0,0) {$F^{\phantom{*}}$};
		\node [style=none] (Xanch) at (0,-0.2) {};
		\node [style=none] (Yanch) at (0.7,0.2) {};
		\node [style=none] (Fanch) at (0,0.2) {};
		\node  (X) at (0,-1) {$\scriptstyle \X$};
		\node  (Y) at (0.7,-1) {$\scriptstyle \Y$};
	\end{pgfonlayer}
	\begin{pgfonlayer}{edgelayer}
	    \draw [arrow,markat=0.55] (X) to (Xanch.center);
	    \draw [arrow,markat=0.55] (Yanch.center) to (Y);
	    \draw [arrow, bend left=90, looseness=3, markat=0.2] (Fanch.center) to (Yanch.center);
	\end{pgfonlayer}
      \end{tikzpicture}
\end{aligned}\;.
$$
Because we regard the distinction between domain wires and codomain wires to be simply an aid to computation, we view $F$ and $\bend F:=E_\Y \circ (F \times I_{\Y^*})$ to be essentially identical. Thus, intuitively, we show that the functions defined in \cite{Kornell} are the same as the functions that we have defined here.

\begin{lemma}\label{computation.E.1}
Let $F$ be a partial function from $\X$ to $\Y$ in the sense of \cite{Kornell}, i.e., a binary relation from $\X$ to $\Y$ satisfying the inequality $F \circ F^\dagger \leq I_\Y$. Then, $\bend F$ is a relation of arity $(\X, \Y^*)$ that satisfies condition (2) of Definition \ref{definition.F.1}. Furthermore, this construction is bijective.
\end{lemma}

\begin{proof}
We argue that the inequality $F \circ F^\dagger \leq I_\Y$ is equivalent to the inequality 
\begin{equation*}\label{computation.E.1.eq}\tag{$*$}\[(y_1,y_{2*}) \in \Y \times \Y^* \suchthat (\exists(x\feq x_*)\fin \X \ftimes \X^*)\, (\bend F_*(x_*,y_1) \AND \bend F(x,y_{2*}))\] \leq E_\Y.
\end{equation*} The inequality $F \circ F^\dagger \leq I_\Y$ may be depicted as on the left, and it is equivalent to the inequalities depicted to its right by the graphical calculus:
$$
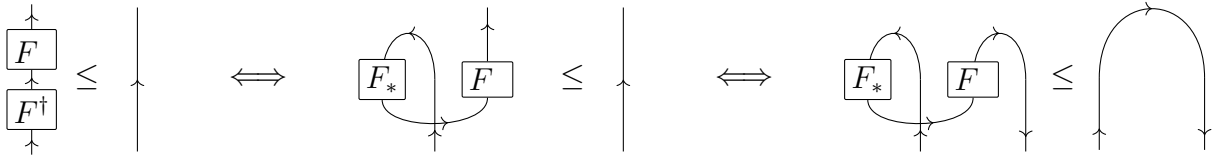
\begin{aligned}
      \begin{tikzpicture}[scale=1]
	\begin{pgfonlayer}{nodelayer}
	    \node [style=none] (air) at (0,1.1) {\phantom{air}};
	    \node [style=box] (F) at (0, 0.4)  {$F^{\phantom{\dagger}}$};
	    \node [style=box] (Fdag) at (0, -0.4)  {$F^{\dagger}$};
        \node [style=none] (top) at (0,1) {};
        \node [style=none] (bot) at (0,-1) {};        
	\end{pgfonlayer}
	\begin{pgfonlayer}{edgelayer}
        \draw [arrow,markat=0.6] (bot.center) to (Fdag.south);
        \draw [arrow,markat=0.6] (Fdag.north) to (F.south);
        \draw [arrow,markat=0.6] (F.north) to (top.center);
	\end{pgfonlayer}
      \end{tikzpicture}
\end{aligned}
\; \leq \;
\begin{aligned}
      \begin{tikzpicture}[scale=1]
	\begin{pgfonlayer}{nodelayer}
	    \node [style=none] (air) at (0,1.1) {\phantom{air}};
        \node [style=none] (top) at (0,1) {};
        \node [style=none] (bot) at (0,-1) {};        
	\end{pgfonlayer}
	\begin{pgfonlayer}{edgelayer}
        \draw [arrow] (bot) to (top);
	\end{pgfonlayer}
      \end{tikzpicture}
\end{aligned}
\qquad \Longleftrightarrow \qquad
\begin{aligned}
      \begin{tikzpicture}[scale=1]
	\begin{pgfonlayer}{nodelayer}
	    \node [style=none] (air) at (0,1.1) {\phantom{air}};
		\node [style=box] (F) at (0.7,0) {$F^{\phantom{*}}$};
		\node [style=box] (F*) at (-0.7,0 ) {$F_*$};
		\node [style=none] (ce) at (0,0) {};
		\node [style=none] (A) at (0,-1) {};
		\node [style=none] (B) at (1.4,-1) {};
		\node [style=none] (C) at (0.7,1) {};
	\end{pgfonlayer}
	\begin{pgfonlayer}{edgelayer}
        \draw [arrow, bend right=90, looseness=3] (ce.center) to (F*.center);
        \draw [arrow, bend right=90, looseness=0.8, markat=0.6] (F*.south) to (F.south);
        \draw [arrow, markat=0.3] (A.north) to (ce.center);
        \draw [arrow, markat=0.7] (F.north) to (C.south);
	\end{pgfonlayer}
      \end{tikzpicture}
\end{aligned}
\; \leq \;
\begin{aligned}
      \begin{tikzpicture}[scale=1]
	\begin{pgfonlayer}{nodelayer}
	    \node [style=none] (air) at (0,1.1) {\phantom{air}};
        \node [style=none] (top) at (0,1) {};
        \node [style=none] (bot) at (0,-1) {};        
	\end{pgfonlayer}
	\begin{pgfonlayer}{edgelayer}
        \draw [arrow] (bot) to (top);
	\end{pgfonlayer}
      \end{tikzpicture}
\end{aligned}
\qquad \Longleftrightarrow \qquad 
\begin{aligned}
      \begin{tikzpicture}[scale=1]
	\begin{pgfonlayer}{nodelayer}
	    \node [style=none] (air) at (0,1.1) {\phantom{air}};
		\node [style=box] (F) at (0.7,0) {$F^{\phantom{*}}$};
		\node [style=box] (F*) at (-0.7,0 ) {$F_*$};
		\node [style=none] (ce) at (0,0) {};
		\node [style=none] (ri) at (1.4,0) {};
		\node [style=none] (A) at (0,-1) {};
		\node [style=none] (B) at (1.4,-1) {};		
	\end{pgfonlayer}
	\begin{pgfonlayer}{edgelayer}
        \draw [arrow, bend right=90, looseness=3] (ce.center) to (F*.center);
        \draw [arrow, bend right=90, looseness=0.8, markat=0.6] (F*.south) to (F.south);
        \draw [arrow, bend left=90, looseness=3] (F.center) to (ri.center);
        \draw [arrow, markat=0.3] (A.north) to (ce.center);
        \draw [arrow, markat=0.8] (ri.center) to (B.north);
	\end{pgfonlayer}
      \end{tikzpicture}
\end{aligned}
\; \leq \;
\begin{aligned}
      \begin{tikzpicture}[scale=1]
	\begin{pgfonlayer}{nodelayer}
		\node [style=none] (F) at (0.7,0) {};
		\node [style=none] (F*) at (-0.7,0 ) {};
		\node [style=none] (A) at (-0.7,-1) {};
		\node [style=none] (B) at (0.7,-1) {};		
	\end{pgfonlayer}
	\begin{pgfonlayer}{edgelayer}
        \draw [arrow, bend left=90, looseness=2.25] (F*.center) to (F.center);
        \draw [arrow, markat=0.3] (A.north) to (F*.center);
        \draw [arrow, markat=0.8] (F.center) to (B.north);
	\end{pgfonlayer}
      \end{tikzpicture}
\end{aligned}
$$
The third graphical inequality depicts inequality (\ref{computation.E.1.eq}) because \begin{align*}\[\bend F_*(x_*,y_1) \AND \bend F(x,y_{2*})\] & = (\[\bend F_*(x_*,y_1)\] \times \top_\X \times \top_{\Y^*}) \AND (\[\top_{\X^*} \times \top_\Y \times \[\bend F(x,y_{2*})\]) \\ &  =\[\bend F_*(x_*,y_1)\] \times \[\bend F(x,y_{2*})\],
\end{align*} and this is a relation of arity $(\X^*, \Y, \X, \Y^*)$ that may be depicted as follows:
$$
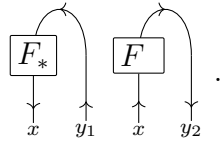
\begin{aligned}
      \begin{tikzpicture}[scale=1]
	\begin{pgfonlayer}{nodelayer}
		\node [style=box] (F) at (0.7,0) {$F^{\phantom{*}}$};
		\node [style=box] (F*) at (-0.7,0 ) {$F_*$};
		\node [style=none] (ce) at (0,0) {};
		\node [style=none] (ri) at (1.4,0) {};
		\node [style=none] (A) at (0,-1) {$\scriptstyle y_1$};
		\node [style=none] (A0) at (-0.7,-1) {$\scriptstyle x$};
		\node [style=none] (B) at (1.4,-1) {$\scriptstyle y_2$};
		\node [style=none] (B0) at (0.7,-1) {$\scriptstyle x$};	
	\end{pgfonlayer}
	\begin{pgfonlayer}{edgelayer}
        \draw [arrow, bend right=90, looseness=3] (ce.center) to (F*.center);
        \draw [arrow, bend left=90, looseness=3] (F.center) to (ri.center);
        \draw [arrow, markat=0.3] (A.north) to (ce.center);
        \draw [arrow, markat=0.8] (F*.center) to (A0.north);
        \draw [arrow, markat=0.8] (ri.center) to (B.north);
        \draw [arrow, markat=0.45] (B0.north) to (F);
	\end{pgfonlayer}
      \end{tikzpicture}
\end{aligned}\;.
$$
Therefore, $F \circ F^\dagger \leq I_\Y$ is equivalent to inequality (\ref{computation.E.1.eq}). Inequality (\ref{computation.E.1.eq}) is in turn equivalent to condition (2) of Definition \ref{definition.F.1} by Proposition \ref{computation.C.2}. By the graphical calculus, the construction $F \mapsto \bend F$ is a bijection from binary relations $\X \to \Y$ to relations of arity $(\X, \Y^*)$, so the lemma is proved.
\end{proof}

\begin{theorem}\label{computation.E.2}
Let $F$ be a function from $\X$ to $\Y$ in the sense of \cite{Kornell}, i.e., a binary relation from $\X$ to $\Y$ satisfying the inequalities $F^\dagger \circ F \geq I_\X$ and $F \circ F^\dagger \leq I_\Y$. Then, $\bend F$ is a function graph. Furthermore, this construction is bijective. Applying {Thm.~7.4} of \cite{Kornell}, we obtain a canonical bijection between function graphs of arity $(\X,\Y^*)$ and unital normal $*$-homomorphisms from $\ell^\infty(\Y)$ to $\ell^\infty(\X)$.
\end{theorem}

\begin{proof}
With Lemma \ref{computation.E.1} in hand, it remains only to show that $F$ satisfies $F^\dagger \circ F \geq I_\X$ if and only if $ \[(\forall x \fin \X)\,(\exists y_* \fin \Y)\, \bend F(x,y_*)\] = \top$. As we observed in subsection \ref{computation.C}, this equality holds if and only if $\[x \in \X \suchthat (\exists y_* \fin \Y)\, \bend F(x,y_*)\]$ is the maximum predicate $\top_\X$. Reasoning diagrammatically, we have that
$$
\[x \in \X \suchthat (\exists y_* \fin \Y)\, \bend F(x,y_*)\]
\; = \;
\begin{aligned}
      \begin{tikzpicture}[scale=1]
	\begin{pgfonlayer}{nodelayer}
		\node [style=box] (F) at (0.7,0) {$F$};
		\node [style=none] (ri) at (1.4,0) {};
		\node [style=none] (B) at (1.4,-0.7) {$\bullet$};
		\node [style=none] (B0) at (0.7,-1) {};	
	\end{pgfonlayer}
	\begin{pgfonlayer}{edgelayer}
        \draw [arrow, bend left=90, looseness=3] (F.center) to (ri.center);
        \draw [arrow, markat=0.5] (ri.center) to (B.center);
        \draw [arrow, markat=0.3] (B0.north) to (F);
	\end{pgfonlayer}
      \end{tikzpicture}
\end{aligned}
\; = \;
\begin{aligned}
      \begin{tikzpicture}[scale=1]
	\begin{pgfonlayer}{nodelayer}
		\node [style=box] (F) at (0.7,0) {$F$};
		\node [style=none] (B) at (0.7,0.7) {$\bullet$};
		\node [style=none] (B0) at (0.7,-1) {};	
	\end{pgfonlayer}
	\begin{pgfonlayer}{edgelayer}
        \draw [arrow, markat=0.5] (F.north) to (B.center);
        \draw [arrow, markat=0.3] (B0.north) to (F);
	\end{pgfonlayer}
      \end{tikzpicture}
\end{aligned}\;.
$$
We conclude that $ \[(\forall x \fin \X)\,(\exists y_* \fin \Y)\, \bend F(x,y_*)\] = \top$ if and only if $\top_\Y \circ F = \top_\X$.

Thus, the construction $F \mapsto \bend F$ is a bijection between partial functions $F$ satisfying $\top_\Y \circ F = \top_\X$ and function graphs. By \cite{Kornell}*{Lem.~B.4}, $\top_\Y \circ F = \top_\X$ if and only if the normal $*$-homomorphism $F^\star$ is unital, and by \cite{Kornell}*{Lem.~6.4}, the latter condition holds if and only if $F^\dagger \circ F \geq I_\X$. Therefore, the construction $F \mapsto \bend F$ restricts to a bijection from functions to function graphs.
\end{proof}

Having established a one-to-one correspondence between functions and function graphs, it becomes natural to use functions for function symbols. Indeed, this is what we have been doing. For each function $F$ from a quantum set $\X$ to a quantum set $\Y$ and each function graph $G$ of arity $(\X, \Y^*)$, the equation $G = \bend F$ is easily seen to be equivalent to the equation $F = \breve G$ via the graphical calculus.
We did not directly define our function symbols to be functions in subsection \ref{definition.F} to delay drawing from \cite{Kornell}, in order to demonstrate that this notion may be motivated from elementary physical and logical considerations.

\subsection{Terms}\label{computation.F} One effect of Definition \ref{definition.G.2} is that nonduplicating terms may be interpreted as compositions of functions in the expected way.

\begin{definition}\label{computation.F.1}
Let $\X_1, \ldots, \X_n$ be quantum sets, and let $x_1, \ldots, x_n$ be distinct variables of sorts $\X_1, \ldots, \X_n$, respectively. Let $\Y$ be a quantum set, and let $t(x_1, \ldots, x_n)$ be a term of sort $\Y$. We define
$
\[(x_1, \ldots, x_n) \in \X_1 \times \cdots \times \X_n \suchthat t(x_1, \ldots,x_n)\]
$
to be
$$(\[(x_1, \ldots, x_n,y_*) \in \X_1 \times \cdots \times \X_n \times \Y^* \suchthat t(x_1, \ldots,x_n) \evaluatesto y_*\] \times I_\Y)\circ (I_{\X_1} \times \cdots \times I_{\X_n} \times E_{\Y}^*),
$$
where the formula $t(x_1, \ldots,x_n) \evaluatesto y_*$ is defined in Definition \ref{definition.G.2}.
Graphically,
\begin{align*}
\begin{aligned}
\begin{tikzpicture}
\begin{pgfonlayer}{nodelayer}
    \node[style=box] (term) at (0,0) {$\[t(x_1, \ldots,x_n)\]$};
    \node[style=none] (yanch) at (0, 0) {};
    \node[style=none] (xnanch) at (1.2, 0) {};
    \node[style=none] (x1anch) at (-1.2, 0) {};
    \node (xn) at (1.2, -1.4) {$\scriptstyle x_n$};
    \node (x1) at (-1.2, -1.4) {$\scriptstyle x_1$};
    \node[style=none] (y) at (0, 1.4) {$\scriptstyle y$};
    \node (dots) at (0,-0.8) {$\cdots$};
\end{pgfonlayer}
\begin{pgfonlayer}{edgelayer}
    \draw [arrow, markat=0.4] (x1) to (x1anch.center);
    \draw [arrow, markat=0.4] (xn) to (xnanch.center);
    \draw [arrow, markat=0.6] (yanch.center) to (y);
\end{pgfonlayer}
\end{tikzpicture}
\end{aligned}
\quad := \quad
\begin{aligned}
\begin{tikzpicture}
\begin{pgfonlayer}{nodelayer}
    \node[style=box] (equality) at (0,0) {$\[t(x_1, \ldots,x_n) \evaluatesto y_*\]$};
    \node[style=none] (ybend) at (2.6, 0) {};
    \node[style=none] (yanch) at (1.5, 0) {};
    \node[style=none] (xnanch) at (0.8, 0) {};
    \node[style=none] (x1anch) at (-1.6, 0) {};
    \node (xn) at (0.8, -1.4) {$\scriptstyle x_n$};
    \node (x1) at (-1.6, -1.4) {$\scriptstyle x_1$};
    \node[style=none] (y) at (2.6, 1.4) {$\scriptstyle y$};
    \node (dots) at (-0.4,-0.8) {$\cdots$};
\end{pgfonlayer}
\begin{pgfonlayer}{edgelayer}
    \draw [arrow, markat=0.4] (x1) to (x1anch.center);
    \draw [arrow, markat=0.4] (xn) to (xnanch.center);
    \draw [arrow] (ybend.center) to (y);
    \draw [arrow, bend right = 90, looseness=3] (yanch.center) to (ybend.center);
\end{pgfonlayer}
\end{tikzpicture}
\end{aligned}
\; .
\end{align*}
\end{definition}

Let $F$ be a function $\X_1 \times \cdots \times \X_n \to \Y$. By definition, the formula $F(x_1,\ldots,x_n) \evaluatesto y_*$ abbreviates the nonduplicating formula $\bend F(x_1, \ldots, x_n, y_*)$, and therefore,
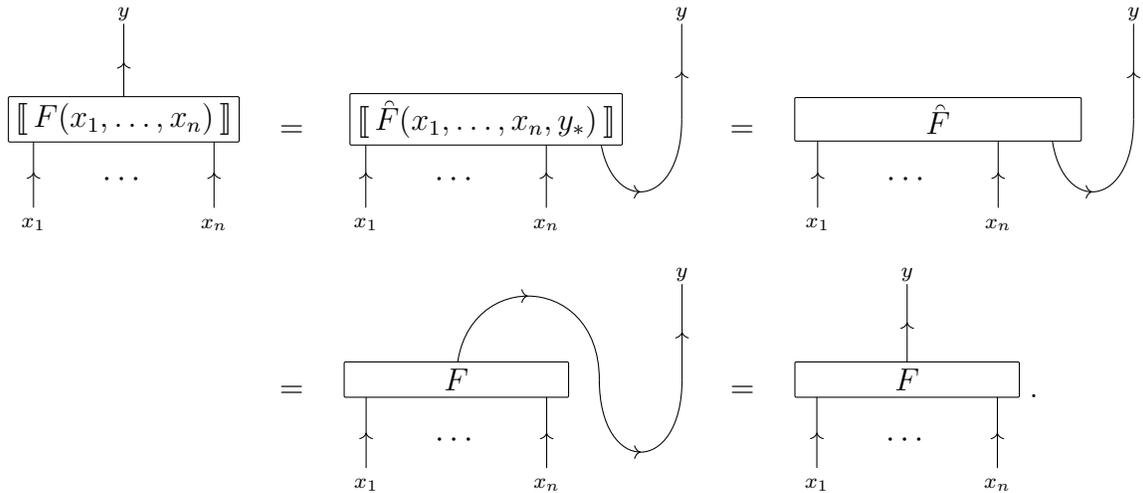
\begin{align*}
\begin{aligned}
\begin{tikzpicture}
\begin{pgfonlayer}{nodelayer}
    \node[style=box] (term) at (0,0) {$\[F(x_1, \ldots,x_n)\]$};
    \node[style=none] (yanch) at (0, 0) {};
    \node[style=none] (xnanch) at (1.2, 0) {};
    \node[style=none] (x1anch) at (-1.2, 0) {};
    \node (xn) at (1.2, -1.4) {$\scriptstyle x_n$};
    \node (x1) at (-1.2, -1.4) {$\scriptstyle x_1$};
    \node[style=none] (y) at (0, 1.4) {$\scriptstyle y$};
    \node (dots) at (0,-0.8) {$\cdots$};
\end{pgfonlayer}
\begin{pgfonlayer}{edgelayer}
    \draw [arrow, markat=0.4] (x1) to (x1anch.center);
    \draw [arrow, markat=0.4] (xn) to (xnanch.center);
    \draw [arrow, markat=0.6] (yanch.center) to (y);
\end{pgfonlayer}
\end{tikzpicture}
\end{aligned}
 & \quad = \quad
\begin{aligned}
\begin{tikzpicture}
\begin{pgfonlayer}{nodelayer}
    \node[style=box] (equality) at (0,0) {$\[\bend F(x_1, \ldots, x_n, y_*)\]$};
    \node[style=none] (ybend) at (2.6, 0) {};
    \node[style=none] (yanch) at (1.5, 0) {};
    \node[style=none] (xnanch) at (0.8, 0) {};
    \node[style=none] (x1anch) at (-1.6, 0) {};
    \node (xn) at (0.8, -1.4) {$\scriptstyle x_n$};
    \node (x1) at (-1.6, -1.4) {$\scriptstyle x_1$};
    \node[style=none] (y) at (2.6, 1.4) {$\scriptstyle y$};
    \node (dots) at (-0.4,-0.8) {$\cdots$};
\end{pgfonlayer}
\begin{pgfonlayer}{edgelayer}
    \draw [arrow, markat=0.4] (x1) to (x1anch.center);
    \draw [arrow, markat=0.4] (xn) to (xnanch.center);
    \draw [arrow] (ybend.center) to (y);
    \draw [arrow, bend right = 90, looseness=3] (yanch.center) to (ybend.center);
\end{pgfonlayer}
\end{tikzpicture}
\end{aligned}
\quad = \quad
\begin{aligned}
\begin{tikzpicture}
\begin{pgfonlayer}{nodelayer}
    \node[style=box] (equality) at (0,0) {$\ghost \qquad \qquad \bend F \qquad \qquad \ghost$};
    \node[style=none] (ybend) at (2.6, 0) {};
    \node[style=none] (yanch) at (1.5, 0) {};
    \node[style=none] (xnanch) at (0.8, 0) {};
    \node[style=none] (x1anch) at (-1.6, 0) {};
    \node (xn) at (0.8, -1.4) {$\scriptstyle x_n$};
    \node (x1) at (-1.6, -1.4) {$\scriptstyle x_1$};
    \node[style=none] (y) at (2.6, 1.4) {$\scriptstyle y$};
    \node (dots) at (-0.4,-0.8) {$\cdots$};
\end{pgfonlayer}
\begin{pgfonlayer}{edgelayer}
    \draw [arrow, markat=0.4] (x1) to (x1anch.center);
    \draw [arrow, markat=0.4] (xn) to (xnanch.center);
    \draw [arrow] (ybend.center) to (y);
    \draw [arrow, bend right = 90, looseness=3] (yanch.center) to (ybend.center);
\end{pgfonlayer}
\end{tikzpicture}
\end{aligned}
\\ & \quad  = \quad
\begin{aligned}
\begin{tikzpicture}
\begin{pgfonlayer}{nodelayer}
    \node[style=box] (F) at (-0.4,0) {$\ghost \qquad\quad F\qquad\quad \ghost $};
    \node[style=none] (y*anch) at (2.6, 0) {};
    \node[style=none] (yanch) at (1.5, 0) {};
    \node[style=none] (xnanch) at (0.8, 0) {};
    \node[style=none] (x1anch) at (-1.6, 0) {};
    \node (xn) at (0.8, -1.4) {$\scriptstyle x_n$};
    \node (x1) at (-1.6, -1.4) {$\scriptstyle x_1$};
    \node[style=none] (y) at (2.6, 1.4) {$\scriptstyle y$};
    \node (dots) at (-0.4,-0.8) {$\cdots$};
\end{pgfonlayer}
\begin{pgfonlayer}{edgelayer}
    \draw [arrow, markat=0.4] (x1) to (x1anch.center);
    \draw [arrow, markat=0.4] (xn) to (xnanch.center);
    \draw [arrow] (y*anch.center) to (y);
    \draw [arrow, bend right = 90, looseness=3] (yanch.center) to (y*anch.center);
    \draw [arrow, bend left = 90, looseness =2] (F.center) to (yanch.center);
\end{pgfonlayer}
\end{tikzpicture}
\end{aligned}
\quad =\quad
\begin{aligned}
\begin{tikzpicture}
\begin{pgfonlayer}{nodelayer}
    \node[style=box] (term) at (0,0) {$\ghost \qquad \quad F \quad \qquad \ghost$};
    \node[style=none] (yanch) at (0, 0) {};
    \node[style=none] (xnanch) at (1.2, 0) {};
    \node[style=none] (x1anch) at (-1.2, 0) {};
    \node (xn) at (1.2, -1.4) {$\scriptstyle x_n$};
    \node (x1) at (-1.2, -1.4) {$\scriptstyle x_1$};
    \node[style=none] (y) at (0, 1.4) {$\scriptstyle y$};
    \node (dots) at (0,-0.8) {$\cdots$};
\end{pgfonlayer}
\begin{pgfonlayer}{edgelayer}
    \draw [arrow, markat=0.4] (x1) to (x1anch.center);
    \draw [arrow, markat=0.4] (xn) to (xnanch.center);
    \draw [arrow, markat=0.6] (yanch.center) to (y);
\end{pgfonlayer}
\end{tikzpicture}
\end{aligned} \;.
\end{align*}
Thus, $\[(x_1,\ldots, x_n) \in \X_1 \times \cdots \times \X_n \suchthat F(x_1, \ldots, x_n)\] = F$. 

Let $i \in \{1, \ldots, n\}$. The variable $x_i$ is a term of sort $\X_i$. The formula $x_i \evaluatesto y_*$ abbreviates the nonduplicating formula $E_{\X_i}(x_i,y_*)$, and therefore,
\begin{align*}
\begin{aligned}
\begin{tikzpicture}
\begin{pgfonlayer}{nodelayer}
    \node[style=box] (term) at (0,0) {$\ghost \qquad \;\[x_i\]\;\qquad \ghost$};
    \node[style=none] (yanch) at (0, 0) {};
    \node[style=none] (xnanch) at (1.2, 0) {};
    \node[style=none] (x1anch) at (-1.2, 0) {};
    \node[style=none] (xianch) at (0,0) {};
    \node (xi) at (0, -1.4) {$\scriptstyle x_i$};
    \node (xn) at (1.2, -1.4) {$\scriptstyle x_n$};
    \node (x1) at (-1.2, -1.4) {$\scriptstyle x_1$};
    \node[style=none] (y) at (0, 1.4) {$\scriptstyle y$};
    \node (dotsleft) at (-0.6,-0.8) {$\cdots$};
    \node (dotsright) at (0.6,-0.8) {$\cdots$};
\end{pgfonlayer}
\begin{pgfonlayer}{edgelayer}
    \draw [arrow, markat=0.4] (x1) to (x1anch.center);
    \draw [arrow, markat=0.4] (xi) to (xianch.center);
    \draw [arrow, markat=0.4] (xn) to (xnanch.center);
    \draw [arrow, markat=0.6] (yanch.center) to (y);
\end{pgfonlayer}
\end{tikzpicture}
\end{aligned}
 & \quad = \quad
\begin{aligned}
\begin{tikzpicture}
\begin{pgfonlayer}{nodelayer}
    \node[style=box] (equality) at (0,0) {$\ghost \quad\;\[E_{\X_i}(x_i, y_*)\]\;\quad \ghost$};
    \node[style=none] (ybend) at (2.6, 0) {};
    \node[style=none] (yanch) at (1.5, 0) {};
    \node[style=none] (xnanch) at (0.8, 0) {};
    \node[style=none] (x1anch) at (-1.6, 0) {};
    \node[style=none] (xianch) at (-0.4,0) {};
    \node (xi) at (-0.4, -1.4) {$\scriptstyle x_i$};
    \node (xn) at (0.8, -1.4) {$\scriptstyle x_n$};
    \node (x1) at (-1.6, -1.4) {$\scriptstyle x_1$};
    \node[style=none] (y) at (2.6, 1.4) {$\scriptstyle y$};
    \node (dotsleft) at (-1,-0.8) {$\cdots$};
    \node (dotsright) at (0.2,-0.8) {$\cdots$};
\end{pgfonlayer}
\begin{pgfonlayer}{edgelayer}
    \draw [arrow, markat=0.4] (x1) to (x1anch.center);
    \draw [arrow, markat=0.4] (xi) to (xianch.center);
    \draw [arrow, markat=0.4] (xn) to (xnanch.center);
    \draw [arrow] (ybend.center) to (y);
    \draw [arrow, bend right = 90, looseness=3] (yanch.center) to (ybend.center);
\end{pgfonlayer}
\end{tikzpicture}
\end{aligned}
\\ & \quad = \quad
\begin{aligned}
\begin{tikzpicture}
\begin{pgfonlayer}{nodelayer}
    \node[style=none] (ybend) at (2.6, 0) {};
    \node[style=none] (yanch) at (1.5, 0) {};
    \node[style=none] (xnanch) at (0.8, 0) {$\bullet$};
    \node[style=none] (x1anch) at (-1.6, 0) {$\bullet$};
    \node[style=none] (xianch) at (-0.4,0) {};
    \node (xi) at (-0.4, -1.4) {$\scriptstyle x_i$};
    \node (xn) at (0.8, -1.4) {$\scriptstyle x_n$};
    \node (x1) at (-1.6, -1.4) {$\scriptstyle x_1$};
    \node[style=none] (y) at (2.6, 1.4) {$\scriptstyle y$};
    \node (dotsleft) at (-1,-0.8) {$\cdots$};
    \node (dotsright) at (0.2,-0.8) {$\cdots$};
    \node (dotsleftmid) at (-1,0) {$\cdots$};
    \node (dotsrightmid) at (0.2,-0) {$\cdots$};
\end{pgfonlayer}
\begin{pgfonlayer}{edgelayer}
    \draw [arrow, markat=0.4] (x1) to (x1anch.center);
    \draw [arrow, markat=0.4] (xi) to (xianch.center);
    \draw [arrow, markat=0.4] (xn) to (xnanch.center);
    \draw [arrow] (ybend.center) to (y);
    \draw [arrow, bend right = 90, looseness=3] (yanch.center) to (ybend.center);
    \draw [arrow, bend left = 90, looseness =2] (xianch.center) to (yanch.center);
\end{pgfonlayer}
\end{tikzpicture}
\end{aligned}
 \quad  = \quad
\begin{aligned}
\begin{tikzpicture}
\begin{pgfonlayer}{nodelayer}
    \node[style=none] (xnanch) at (0.8, 0) {$\bullet$};
    \node[style=none] (x1anch) at (-1.6, 0) {$\bullet$};
    \node[style=none] (xianch) at (-0.4,0) {};
    \node (xi) at (-0.4, -1.4) {$\scriptstyle x_i$};
    \node (xn) at (0.8, -1.4) {$\scriptstyle x_n$};
    \node (x1) at (-1.6, -1.4) {$\scriptstyle x_1$};
    \node[style=none] (y) at (-0.4, 1.4) {$\scriptstyle y$};
    \node (dotsleft) at (-1,-0.8) {$\cdots$};
    \node (dotsright) at (0.2,-0.8) {$\cdots$};
    \node (dotsleftmid) at (-1,0) {$\cdots$};
    \node (dotsrightmid) at (0.2,-0) {$\cdots$};
\end{pgfonlayer}
\begin{pgfonlayer}{edgelayer}
    \draw [arrow, markat=0.4] (x1) to (x1anch.center);
    \draw [arrow, markat=0.4] (xi) to (xianch.center);
    \draw [arrow, markat=0.4] (xn) to (xnanch.center);
    \draw [arrow] (xianch.center) to (y);
\end{pgfonlayer}
\end{tikzpicture}
\end{aligned}\;.
\end{align*}
Thus, $\[(x_1, \ldots, x_n) \in \X_1 \times \cdots \times \X_n \suchthat x_i\]= 
\top_{\X_1} \times \cdots \times \top_{\X_{i-1}} \times I_{\X_i} \times \top_{\X_{i+1}} \times \cdots \times \top_{\X_n}$. This is the projection function $P_i\: \X_1 \times \cdots \times \X_n \to \X_i$, which is dual to the canonical inclusion unital normal $*$-homomorphism $P_i^\star\: \ell^\infty(\X_i) \hookrightarrow \ell^\infty(\X_1) \stensor \cdots \stensor \ell^\infty(\X_n)$ \cite{Kornell}*{sec.~10}\cite{KornellLindenhoviusMislove2}*{App.~B}.

\begin{lemma}\label{computation.F.2}
Let $\Y_1, \ldots, \Y_m$ be quantum sets, and let $R$ be a relation of arity $(\Y_1, \ldots, \Y_m)$. For each index $i \in \{1, \ldots, m\}$, let $t_i$ be a term of sort $\Y_i$, whose distinct variables $x_{i,1}, \ldots, x_{i,n_i}$ are of sorts $\X_{1,i}, \ldots \X_{i,n_i}$, respectively. For each index $i \in \{1, \ldots, m\}$, let $\[t_i\]$ be an abbreviation for $\[(x_{i,1}, \ldots, x_{i,n_i}) \in \X_{i,1}\times \cdots \times \X_{i,n_i} \suchthat t_i\]$. Similarly, let $\[R(t_1, \ldots, t_m)\]$ be an abbreviation for $$\[
(x_{1,1}, \ldots, x_{1,n_1}, \ldots, x_{m,1}, \ldots, x_{m,n_m}) \in \X_{1,1}\times \cdots \times \X_{1,n_1} \times \cdots \times \X_{m,1} \times \cdots \times \X_{m,n_m}
\suchthat R(t_1, \ldots, t_m)\].$$ If $R(t_1, \ldots, t_m)$ is nonduplicating, then
$$\[R(t_1, \ldots, t_m)\] = R \circ (\[t_1\] \times \cdots \times \[t_m\]).$$
\end{lemma}

\begin{proof}
Assume that $R(t_1, \ldots, t_m)$ is nonduplicating. We calculate that
\begin{align*}
& \[R(t_1, \ldots, t_m)\]
\\ & =
\[ (\exists (y_m \feq y_{m*}) \fin \Y \ftimes \Y^*)\cdots (\exists (y_1 \feq y_{1*}) \fin \Y \ftimes \Y^*)\,
\\ & \ghost \hspace{44ex}(R(y_1, \ldots, y_m) \AND t_1\evaluatesto y_{1*} \AND \cdots \AND t_m \evaluatesto y_{m*}) \]
\\ & = \quad
\begin{aligned}
\begin{tikzpicture}
\begin{pgfonlayer}{nodelayer}
    \node[style=box] (t1) at (0,0) {$\ghost \[t_1\evaluatesto y_{1*}\] \ghost$};
    \node[style=none] (x11) at (-0.8,-1.3) {};
    \node[style=none] (x11!) at (-0.8, 0) {};
    \node[style=none] (x1n) at (0.3,-1.3) {};
    \node[style=none] (x1n!) at (0.3, 0) {};
    \node (dots1) at (-0.2,-0.85) {$\cdots$};
    \node[style=none] (y1star!) at (0.8, -0.2) {};
    \node (dots) at (1.8,0) {$\cdots$};
    \node[style=box] (tm) at (3.6,0) {$\ghost \[t_m\evaluatesto y_{m*}\] \ghost$};
    \node[style=none] (xm1) at (2.8,-1.3) {};
    \node[style=none] (xm1!) at (2.8, 0) {};
    \node[style=none] (xmn) at (3.9,-1.3) {};
    \node[style=none] (xmn!) at (3.9, 0) {};
    \node (dotsm) at (3.4,-0.85) {$\cdots$};
    \node[style=none] (ymstar!) at (4.4, -0.2) {};
    \node[style=box] (R) at (7,0) {$\[R(y_1, \ldots,y_m)\]$};
    \node[style=none] (y1!) at (6,-0.2) {};
    \node[style=none] (ym!) at (8,-0.2) {};
    \node (dotsR) at (7,-0.5) {$\cdots$};
\end{pgfonlayer}
\begin{pgfonlayer}{edgelayer}
    \draw [arrow, markat=0.4] (x11) to (x11!);
    \draw [arrow, markat=0.4] (x1n) to (x1n!);
    \draw [arrow, markat=0.4] (xm1) to (xm1!);
    \draw [arrow, markat=0.4] (xmn) to (xmn!);
    \draw [arrow, bend right=90, looseness=0.35, markat = 0.3] (y1star!.center) to (y1!.center);
    \draw [arrow, bend right=90, looseness=0.55, markat = 0.4] (ymstar!.center) to (ym!.center);
\end{pgfonlayer}
\end{tikzpicture}
\end{aligned}
\\ & = \quad
\begin{aligned}
\begin{tikzpicture}
\begin{pgfonlayer}{nodelayer}
    \node[style=box] (R) at (0,0) {$\[R(y_1, \ldots,y_m)\]$};
    \node[style=none] (t1!) at (-1,-0.2) {};
    \node[style=none] (tm!) at (1,-0.2) {};
    \node (dotsR) at (0,-1) {$\cdots$};
    \node[style=box] (t1) at (-1, -1) {$\[t_1\]$};
    \node[style=none] (x11!) at (-1.3,-1) {};
    \node[style=none] (x1n!) at (-0.7,-1) {};
    \node[style=none] (x11) at (-1.3,-2) {};
    \node[style=none] (x1n) at (-0.7,-2) {};
    \node (dots1) at (-1, -1.6) {$\scriptstyle \cdots$};
    \node[style=box] (tm) at (1, -1) {$\[t_m\]$};
    \node[style=none] (xmn!) at (1.3,-1) {};
    \node[style=none] (xm1!) at (0.7,-1) {};
    \node[style=none] (xmn) at (1.3,-2) {};
    \node[style=none] (xm1) at (0.7,-2) {};
    \node (dots2) at (1, -1.6) {$\scriptstyle \cdots$};
\end{pgfonlayer}
\begin{pgfonlayer}{edgelayer}
    \draw [arrow] (t1) to (t1!);
    \draw [arrow] (tm) to (tm!);
    \draw [arrow,markat=0.4] (x11) to (x11!);
    \draw [arrow,markat=0.4] (x1n) to (x1n!);
    \draw [arrow,markat=0.4] (xm1) to (xm1!);
    \draw [arrow,markat=0.4] (xmn) to (xmn!);
\end{pgfonlayer}
\end{tikzpicture}
\end{aligned}
\quad = \;
R \circ (\[t_1\] \times \cdots \times \[t_m\])\;.\qedhere
\end{align*}
\end{proof}

\begin{lemma}\label{computation.F.3}
Let $\Y_1, \ldots, \Y_m$ and $\Z$ be quantum sets, and let $F$ be a function from $\Y_1 \times \cdots \times \Y_m$ to $\Z$. For each index $i \in \{1, \ldots, m\}$, let $t_i$ be a term of sort $\Y_i$, whose distinct variables $x_{i,1}, \ldots, x_{i,n_i}$ are of sorts $\X_{1,i}, \ldots \X_{i,n_i}$, respectively. For each index $i \in \{1, \ldots, m\}$, let $\[t_i\]$ be an abbreviation for $\[(x_{i,1}, \ldots, x_{i,n_i}) \in \X_{i,1}\times \cdots \times \X_{i,n_i} \suchthat t_i\]$. Similarly, let $\[F(t_1, \ldots, t_m)\]$ be an abbreviation for $$\[
(x_{1,1}, \ldots, x_{1,n_1}, \ldots, x_{m,1}, \ldots, x_{m,n_m}) \in \X_{1,1}\times \cdots \times \X_{1,n_1} \times \cdots \times \X_{m,1} \times \cdots \times \X_{m,n_m}
\suchthat F(t_1, \ldots, t_m)\].$$ If $F(t_1, \ldots, t_m)$ is nonduplicating, then
$$\[F(t_1, \ldots, t_m)\] = F \circ (\[t_1\] \times \cdots \times \[t_m\]).$$
\end{lemma}

\begin{proof}
Assume that $F(t_1, \ldots, t_m)$ is nonduplicating. Let $\X=\X_{1,1}\times \cdots \times \X_{1,n_1} \times \cdots \times \X_{m,1} \times \cdots \times \X_{m,n_m}$, and let $\Y = \Y_1 \times \cdots \times\Y_m$. We apply Lemma \ref{computation.F.2} to calculate that
\begin{align*}
\[F(t_1, \ldots, t_m)\]
& =
(\[F(t_1, \ldots, t_m)\evaluatesto z_*\] \times I_\Z) \circ (I_\X \times E_\Z^*)
\\ &=
(\[\bend F(t_1, \ldots, t_m, z_*)\] \times I_\Z) \circ (I_\X \times E_\Z^*)
\\ & =
((\bend F \circ  (\[t_1\] \times \cdots \times \[t_m\] \times \[z_*\])) \times I_\Z)\circ (I_\X \times E_\Z^*)
\\ & =
((\bend F \circ  (\[t_1\] \times \cdots \times \[t_m\] \times I_{\Z^*})) \times I_\Z)\circ (I_\X \times E_\Z^*)
\\ & =
(\bend F \times I_\Z) \circ (I_\Y \times E_\Z^*) \circ (\[t_1\] \times \cdots \times \[t_m\])
\\ & =
F \circ (\[t_1\] \times \cdots \times \[t_m\]). \qedhere
\end{align*}
\end{proof}

We may conjugate a term $t$ by conjugating each function symbol and variable that appears in that term. Formally, if $t$ is of the form $F(t_1, \ldots, t_m)$, then we define $t_*$ to be $F_*(t_{1*}, \ldots, t_{m*})$, and if $t$ is a variable $x$, then we define $t_*$ to be $x_*$, the conjugate variable. If the term $t$ has variables among $x_1, \ldots, x_n$, of sorts quantum sets $\X_1, \ldots \X_n$, respectively, then we may apply Lemma \ref{computation.F.3} to show that $$\[(x_{1*}, \ldots, x_{n*}) \in \X_1^* \times \cdots \times \X_n^* \suchthat t_*\] = \[(x_{1}, \ldots, x_{n}) \in \X_1 \times \cdots \times \X_n \suchthat t\]_*.$$

\begin{proposition}\label{computation.F.4}
Let $\X_1, \ldots, \X_n$ and $\Y$ be quantum sets. Let $x_1, \ldots, x_n$ be distinct variables of sorts $\X_1, \ldots, \X_n$, respectively, and let $x_{1*}, \ldots, x_{n*}$ be distinct variables of sorts $\X_{1*}, \ldots, \X_{n*}$, respectively. Let $s$ and $t$ be terms of sort $\Y$ whose free variables are among $x_1, \ldots, x_n$. Then, $\[(x_1, \ldots, x_n) \in \X_1 \times \cdots \times \X_n \suchthat s \] = \[(x_1, \ldots, x_n) \in \X_1 \times \cdots \times \X_n \suchthat t \]$ if and only if
\begin{equation}\label{computation.F.4.eq}\tag{$\dagger$}\[(\forall (x_n \feq x_{n*})\fin \X_n \ftimes \X_n^*)\cdots (\forall (x_1 \feq x_{1*}) \fin \X_1 \ftimes \X_1^*)\,E_\Y(s, t_*)\] = \top.\end{equation}
\end{proposition}

\begin{proof}
By the duality of diagonal quantifiers (Definition \ref{definition.E.1}), equation (\ref{computation.F.4.eq}) holds if and only if $
\[(\exists (x_n \feq x_{n*})\fin \X_n \ftimes \X_n^*)\cdots (\exists (x_1 \feq x_{1*}) \fin \X_1 \ftimes \X_1^*)\, \neg E_\Y(s, t_*)\] = \bot$ or equivalently
$$
\begin{aligned}
      \begin{tikzpicture}[scale=1]
	\begin{pgfonlayer}{nodelayer}
		\node [style=box] (notR) at (0,0) {$\[\neg E_\Y(s, t_*)\]$};
		\node [style=none] (Fdots) at (-.6,-.45) {$\cdots$};
		\node [style=none] (F1anchor) at (-1,-.2) {};
    	\node [style=none] (F1start) at (-1,-.6){};
		\node [style=none] (Fnanchor) at (-.2,-.2) {};
		\node [style=none] (Fnstart) at (-.2,-.6){};
		\node [style=none] (Gdots) at (.6,-.45) {$\cdots$};
		\node [style=none] (G1anchor) at (1,-.2) {};
    	\node [style=none] (G1start) at (1,-.6){};
		\node [style=none] (Gnanchor) at (.2,-.2) {};
		\node [style=none] (Gnstart) at (.2,-.6){};
	\end{pgfonlayer}
	\begin{pgfonlayer}{edgelayer}
	    \draw [arrow] (F1start.center) to (F1anchor);
	    \draw [arrow] (Fnstart.center) to (Fnanchor);
	    \draw [arrow,markat=0.8] (G1anchor) to (G1start.center);
	    \draw [arrow,markat=0.8] (Gnanchor) to (Gnstart.center);
	    \draw [arrow, bend left=90, looseness=1.5] (Gnstart.center) to (F1start.center);
	    \draw [arrow, bend left=90, looseness=1.5] (G1start.center) to (Fnstart.center);
	\end{pgfonlayer}
      \end{tikzpicture}
\end{aligned}
\quad = \quad
\bot
\; .
$$
We recognize the diagram on the left as depicting $\[\neg E_\Y(s, t_*)\] \circ E_\X^\dagger$, for $\X = \X_1 \times \cdots \times \X_n$. Thus, equation (\ref{computation.F.4.eq}) holds if and only if $(\neg \[E_\Y(s, t_*)\] ) \circ E_\X^\dagger = \bot$ or equivalently $E_\X \perp \neg \[E_\Y(s, t_*)\]$ or equivalently $E_\X \leq \[E_\Y(s, t_*)\]$. By Lemma \ref{computation.F.2}, we have that $\[E_\Y(s, t_*)\] = E_\Y \circ (\[s\] \times \[t\]_*)$, and therefore, equation (\ref{computation.F.4.eq}) is equivalent to the inequality
$$
\begin{aligned}
      \begin{tikzpicture}[scale=1]
	\begin{pgfonlayer}{nodelayer}
		\node [style=none] (F1anchor) at (-1,-.2) {};
    	\node [style=none] (F1start) at (-1,-1.2){$\scriptstyle \X$};
		\node [style=none] (Gnanchor) at (.2,-.2) {};
		\node [style=none] (Gnstart) at (.2,-1.2){$\scriptstyle \X$};
	\end{pgfonlayer}
	\begin{pgfonlayer}{edgelayer}
	    \draw [arrow] (F1start) to (F1anchor.center);
	    \draw [arrow,markat=0.65] (Gnanchor.center) to (Gnstart);
	    \draw [arrow, bend left=90, looseness=1.85] (F1anchor.center) to (Gnanchor.center);
	\end{pgfonlayer}
      \end{tikzpicture}
\end{aligned}
\quad \leq \quad
\begin{aligned}
      \begin{tikzpicture}[scale=1]
	\begin{pgfonlayer}{nodelayer}
		\node [style=box] (F) at (-.6,0) {$\[s\]_{\phantom{*}}$};
		\node [style=box] (G) at (.6,0) {$\[t\]_*$};
		\node [style=none] (F1anchor) at (-.6,-.2) {};
    	\node [style=none] (F1start) at (-.6,-.8){$\scriptstyle \X$};
		\node [style=none] (Gnanchor) at (.6,-.2) {};
		\node [style=none] (Gnstart) at (.6,-.8){$\scriptstyle \X$};
		\node [style=none] (Ftop) at (-.6,.2){};
		\node [style=none] (Gtop) at (.6,.2){};
	\end{pgfonlayer}
	\begin{pgfonlayer}{edgelayer}
	    \draw [arrow] (F1start) to (F1anchor);
	    \draw [arrow,markat=0.75] (Gnanchor) to (Gnstart);
	    \draw [arrow, bend left=90, looseness=1.85] (Ftop) to (Gtop);
	\end{pgfonlayer}
      \end{tikzpicture}
\end{aligned}\;.
$$
Straightening the wires, we conclude that equation (\ref{computation.F.4.eq}) is equivalent to $I_\X\leq \[t\]^\dagger \circ \[s\]$.

It is a basic fact about functions between quantum sets that $I_\X \leq \[t\]^\dagger \circ \[s\]$ if and only if $\[s\] = \[t\]$. Indeed, $I_\X \leq \[t\]^\dagger \circ \[s\]$ implies that $\[t\] \leq \[t\] \circ \[t\]^\dagger \circ \[s\] \leq \[s\]$, and similarly, it implies that $\[s\]^\dagger \leq \[t\]^\dagger \circ \[s\] \circ \[s\]^\dagger \leq \[t\]^\dagger$, so $\[s\] \leq \[t\]$. Altogether, $I_\X \leq \[t\]^\dagger \circ \[s\]$ implies that $\[s\] =\[t\]$. Conversely, if $\[s\] = \[t\]$, then $I_\X \leq \[t\]^\dagger \circ \[s\]$, by the definition of a function between quantum sets. Thus, we conclude that equation (\ref{computation.F.4.eq}) is equivalent to $\[s\] = \[t\]$.
\end{proof}

\section{examples}\label{examples}
The definition of equality as a relation of mixed arity can be motivated on conceptual grounds, as in subsection \ref{introduction.A}, but the most compelling justification for this definition is that it provides a link in a mathematical connection between the quantum logic of Birkhoff and von Neumann and several established classes of discrete quantum structures. Definition \ref{core} does not in itself provide an unambiguous quantization method because formulas that are equivalent in the classical setting need not also be equivalent in the quantum setting. Indeed, the remainder of section \ref{definition} may be viewed as providing some heuristic guidelines for this approach to quantization. Nevertheless, organic examples of axiomatic quantization are further evidence of the mechanisms underlying the coherence the noncommutative dictionary.

Before considering the examples, we briefly recall a useful property of binary relations between quantum sets \cite{KornellLindenhoviusMislove2}*{App.~C}. For each quantum set $\X$, the set of binary relations on $\X$ has a trace that is defined by
$$
\Tr_\X(R)
\; := \quad
\begin{aligned}
      \begin{tikzpicture}[scale=1]
	\begin{pgfonlayer}{nodelayer}
		\node [style=box] (R) at (0,0) {$R$};
		\node [style=none] (top) at (0,0.2) {};
		\node [style=none] (bot) at (0,-0.2) {};
		\node [style=none] (topright) at (0.7,0.2) {};
		\node [style=none] (botright) at (0.7,-0.2) {};
	\end{pgfonlayer}
	\begin{pgfonlayer}{edgelayer}
        \draw [arrow, bend left=90, looseness=2.25] (top.center) to (topright.center);
        \draw [arrow, bend left=90, looseness=2.25] (botright.center) to (bot.center);
        \draw [arrow] (topright.center) to (botright.center);
	\end{pgfonlayer}
      \end{tikzpicture}
\end{aligned}\;.
$$
Two binary relations $R$ and $S$ from a quantum set $\X$ to a quantum set $\Y$ are orthogonal to each other \cite{Kornell}*{Def.~3.8(5)} if and only if $\Tr_\X(S^\dagger \circ R) = \bot$ \cite{KornellLindenhoviusMislove2}*{Prop.~C.2}. We may express this equivalence graphically:
$$ R \perp S
\EV
\begin{aligned}
      \begin{tikzpicture}[scale=1]
	\begin{pgfonlayer}{nodelayer}
		\node [style=box] (R) at (0,0) {$R$};
		\node [style=box] (Sdag) at (0,1) {$S^\dagger$};
		\node [style=none] (top) at (0,1.2) {};
		\node [style=none] (bot) at (0,-0.2) {};
		\node [style=none] (topright) at (0.7,1.2) {};
		\node [style=none] (botright) at (0.7,-0.2) {};
	\end{pgfonlayer}
	\begin{pgfonlayer}{edgelayer}
        \draw [arrow, bend left=90, looseness=2.25] (top.center) to (topright.center);
        \draw [arrow, bend left=90, looseness=2.25] (botright.center) to (bot.center);
        \draw [arrow] (topright.center) to (botright.center);
        \draw [arrow] (R.center) to (Sdag.center);
	\end{pgfonlayer}
      \end{tikzpicture}
\end{aligned}
\; = \bot
\EV
\begin{aligned}
      \begin{tikzpicture}[scale=1]
	\begin{pgfonlayer}{nodelayer}
		\node [style=box] (R) at (0,0) {$R$};
		\node [style=none] (top0) at (0,0.2) {};
		\node [style=none] (bot0) at (0,-0.2) {};
		\node [style=box] (conS) at (1.0,0) {$S_*$};
		\node [style=none] (top2) at (1.0,0.2) {};
		\node [style=none] (bot2) at (1.0,-0.2) {};
	\end{pgfonlayer}
	\begin{pgfonlayer}{edgelayer}
        \draw [arrow, bend left=90, looseness=2] (top0.center) to  (top2.center);
        \draw [arrow, bend left=90, looseness=2] (bot2.center) to  (bot0.center);
	\end{pgfonlayer}
      \end{tikzpicture}
\end{aligned}
\; = \; \bot\;.
$$

\subsection{Quantum graphs}\label{examples.A} Quantum graphs are a quantum generalization of simple graphs. Quantum graphs were first defined in the context of zero-error communication to be operator systems on a finite-dimensional Hilbert space \cite{DuanSeveriniWinter}*{sec.~II}. More generally, a quantum graph structure on an arbitrary von Neumann algebra $M \subsetof L(H)$ is an ultraweakly closed operator system $V$ such that $m'v \in V$ and $vm' \in V$ for all $m' \in M'$ and all $v \in V$ \cite{Weaver2}*{Def.~2.6(d)}\cite{Weaver3}. By definition, an operator system contains the scalar operators, and this feature reflects the convention that each vertex of a simple graph is adjacent to itself, which is natural to the context of zero-error communication. Hence, in subsection \ref{examples.A}, a simple graph is a set $A$ that is equipped with a binary relation $\sim_A$ that is both reflexive and symmetric.

\begin{proposition}\label{examples.A.1} Let $\X$ be a quantum set, and let $R$ be a binary relation on $\X$. Then, $I_\X \leq R$ if and only if
$ \[(\forall (x \feq x_*) \fin \X \ftimes \X^*)\, \bend R(x, x_*) \] = \top.$
\end{proposition}

\begin{proof}
\begin{align*}
 & \[(\forall (x \feq x_*) \fin \X \ftimes \X^*)\, \bend R(x, x_*) \] = \top
\quad \Longleftrightarrow \quad
\[(\exists (x \feq x_*) \fin \X \ftimes \X^*)\, \NOT \bend R(x, x_*) \] = \bot
\\ & \quad \Longleftrightarrow \quad
\begin{aligned}
      \begin{tikzpicture}[scale=1]
	\begin{pgfonlayer}{nodelayer}
		\node [style=box] (notR) at (0,0) {$\NOT R$};
		\node [style=none] (top) at (0,0.2) {};
		\node [style=none] (bot) at (0,-0.2) {};
		\node [style=none] (topright) at (0.7,0.2) {};
		\node [style=none] (botright) at (0.7,-0.2) {};
	\end{pgfonlayer}
	\begin{pgfonlayer}{edgelayer}
        \draw [arrow, bend left=90, looseness=2.25] (top.center) to (topright.center);
        \draw [arrow, bend left=90, looseness=2.25] (botright.center) to (bot.center);
        \draw [arrow] (topright.center) to (botright.center);
	\end{pgfonlayer}
      \end{tikzpicture}
\end{aligned}
= \bot
\quad \Longleftrightarrow \quad
\NOT R \perp I_\X
\quad \Longleftrightarrow \quad
I_\X \leq R \; .\qedhere
\end{align*}
\end{proof}

\begin{lemma}\label{examples.A.2}
Let $\X$ and $\Y$ be quantum sets. Let $R$ be a binary relation from $\X$ to $\Y$, and let $S = R^\dagger$. Then, $\[(y, x_*) \in \Y \times \X^* \suchthat \bend R_*(x_*, y)\] =  \bend S$.
\end{lemma}

\begin{proof}
\begin{align*}
&
\[(y, x_*) \in \Y \times \X^* \suchthat \bend R_*(x_*, y)\]
\; = \;
\begin{aligned}
      \begin{tikzpicture}[scale=1]
	\begin{pgfonlayer}{nodelayer}
		\node [style=box] (conR) at (0,0) {$R_*$};
		\node [style=none] (top) at (0,0.2) {};
		\node [style=none] (bot) at (0,-0.2) {};
		\node [style=none] (topright) at (0.7,0.2) {};
		\node [style=none] (botright) at (0.7,-0.2) {};
		\node (x1) at (0,-1.2) {$\scriptstyle y$};
		\node (x2) at (0.7,-1.2) {$\scriptstyle x$};
	\end{pgfonlayer}
	\begin{pgfonlayer}{edgelayer}
        \draw [arrow, bend right=90, looseness=2.25] (topright.center) to (top.center)  ;
        \draw [arrow] (botright.center) to (topright.center);
        \draw [arrow, out =270, in = 90, markat = 0.8] (bot.center) to (x2);
        \draw [arrow, out =90, in = 270, markat = 0.3]  (x1) to (botright.center);
	\end{pgfonlayer}
      \end{tikzpicture}
\end{aligned}
\; = \;
\begin{aligned}
      \begin{tikzpicture}[scale=1]
	\begin{pgfonlayer}{nodelayer}
		\node [style=box] (conR) at (0.7,0) {$R_*$};
		\node [style=none] (top) at (0,0.2) {};
		\node [style=none] (bot) at (0,-0.2) {};
		\node [style=none] (topright) at (0.7,0.2) {};
		\node [style=none] (botright) at (0.7,-0.2) {};
		\node (x1) at (0,-1.2) {$\scriptstyle y$};
		\node (x2) at (0.7,-1.2) {$\scriptstyle x$};
	\end{pgfonlayer}
	\begin{pgfonlayer}{edgelayer}
        \draw [arrow, bend left=90, looseness=2.25] (top.center) to (topright.center);
        \draw [arrow] (bot.center) to (top.center);
        \draw [arrow, out =270, in = 90, markat = 0.8] (botright.center) to (x2);
        \draw [arrow, out =90, in = 270, markat = 0.3]  (x1) to (bot.center);
	\end{pgfonlayer}
      \end{tikzpicture}
\end{aligned}
\; = \;
\begin{aligned}
      \begin{tikzpicture}[scale=1]
	\begin{pgfonlayer}{nodelayer}
		\node [style=box] (conR) at (0,0) {$R^\dagger$};
		\node [style=none] (top) at (0,0.2) {};
		\node [style=none] (bot) at (0,-0.2) {};
		\node [style=none] (topright) at (0.7,0.2) {};
		\node [style=none] (botright) at (0.7,-0.2) {};
		\node (x1) at (0,-1.2) {$\scriptstyle y$};
		\node (x2) at (0.7,-1.2) {$\scriptstyle x$};
	\end{pgfonlayer}
	\begin{pgfonlayer}{edgelayer}
        \draw [arrow, bend left=90, looseness=2.25] (top.center) to (topright.center);
        \draw [arrow] (topright.center) to (botright.center);
        \draw [arrow, out =270, in = 90, markat = 0.8] (botright.center) to (x2);
        \draw [arrow, out =90, in = 270, markat = 0.3]  (x1) to (bot.center);
	\end{pgfonlayer}
      \end{tikzpicture}
\end{aligned}
\; = \;
\bend{S} \; . \qedhere
\end{align*}
\end{proof}

\begin{lemma}\label{examples.A.3}
Let $\X$ and $\Y$ be quantum sets. Let $R$ and $S$ be binary relations from $\X$ to $\Y$. Then, $R \leq S$ if and only if 
\begin{equation*}\tag{$\ddag$}\label{examples.A.3.eq}\[(\forall (x \feq x_*) \fin \X \ftimes \X^*)\,(\forall (y \feq y_*) \fin \Y \ftimes \Y^*)\,( \bend R (x, y_*) \IMPLIES \bend S_*(x_*, y))\] = \top.
\end{equation*}
\end{lemma}

\begin{proof}
Taking advantage of the fact that the formulas $\bend R (x, y_*)$ and $\bend S_*(x_*, y)$ have no variables in common, we find that equation (\ref{examples.A.3.eq}) is equivalent to each of the following conditions:
\begin{align*}
& \[(\exists (x \feq x_{*}) \fin \X \ftimes \X^*)\,(\exists (y \feq y_*) \fin \Y \ftimes \Y^*)\, \NOT (\bend R (x, y_*) \IMPLIES \bend S_*(x_*, y))\] = \bot
\\ & \quad \Longleftrightarrow \quad
\begin{aligned}
      \begin{tikzpicture}[scale=1]
	\begin{pgfonlayer}{nodelayer}
		\node [style=box] (R) at (0,0) {$R$};
		\node [style=none] (top0) at (0,0.2) {};
		\node [style=none] (bot0) at (0,-0.2) {};
		\node [style=none] (top1) at (0.7,0.2) {};
		\node [style=none] (bot1) at (0.7,-0.2) {};
		\node [style=box] (notconS) at (1.4,0) {$\NOT S_*$};
		\node [style=none] (top2) at (1.4,0.2) {};
		\node [style=none] (bot2) at (1.4,-0.2) {};
		\node [style=none] (top3) at (2.1,0.2) {};
		\node [style=none] (bot3) at (2.1,-0.2) {};
	\end{pgfonlayer}
	\begin{pgfonlayer}{edgelayer}
        \draw [arrow, bend left=90, looseness=2.25] (top0.center) to  (top1.center);
        \draw [arrow] (top1.center) to (bot1.center);
        \draw [arrow, bend right=90, looseness=2.25] (top3.center) to (top2.center);
        \draw [arrow] (bot3.center) to (top3.center);
        \draw [arrow, bend right= 90, looseness=1.5]  (bot1.center) to (bot3.center);
        \draw [arrow, bend left =90, looseness=1.5] (bot2.center) to (bot0.center);
	\end{pgfonlayer}
      \end{tikzpicture}
\end{aligned}
\; = \; \bot
\quad \Longleftrightarrow \quad
\begin{aligned}
      \begin{tikzpicture}[scale=1]
	\begin{pgfonlayer}{nodelayer}
		\node [style=box] (R) at (0,0) {$R$};
		\node [style=none] (top0) at (0,0.2) {};
		\node [style=none] (bot0) at (0,-0.2) {};
		\node [style=none] (top1) at (0.7,0.2) {};
		\node [style=none] (bot1) at (0.7,-0.2) {};
		\node [style=box] (notconS) at (1.4,0) {$\NOT S_*$};
		\node [style=none] (top2) at (1.4,0.2) {};
		\node [style=none] (bot2) at (1.4,-0.2) {};
		\node [style=none] (top3) at (2.1,0.2) {};
		\node [style=none] (bot3) at (2.1,-0.2) {};
	\end{pgfonlayer}
	\begin{pgfonlayer}{edgelayer}
        \draw [arrow, bend left=90, looseness=1.5] (top0.center) to  (top2.center);
        \draw [arrow, bend left=90, looseness=1.5] (bot2.center) to  (bot0.center);
	\end{pgfonlayer}
      \end{tikzpicture}
\end{aligned}
\; = \; \bot
\quad \Longleftrightarrow \quad
R \perp \NOT S
\\ & \quad \Longleftrightarrow \quad
R \leq S\qedhere
\end{align*}
\end{proof}

\begin{proposition}\label{examples.A.4}
Let $\X$ be a quantum set, and let $R$ be a binary relation on $\X$. Then, the following are equivalent:
\begin{enumerate}
\item $R \leq R^\dagger$,
\item $\[(\forall x_1 \in \X)\, (\forall x_{2*} \in \X^*)\, (\bend R(x_1, x_{2*}) \IMPLIES \bend R_*(x_{2*}, x_1))\] = \top$,
\item $\[(\forall (x_1 \feq x_{1*}) \fin \X \ftimes \X^*)\,(\forall (x_2 \feq x_{2*}) \fin \X \ftimes \X^*)\, (\bend R (x_1, x_{2*}) \IMPLIES \bend R(x_2, x_{1*}))\] = \top. $
\end{enumerate}
\end{proposition}

\begin{proof}
Let $S = R^\dagger$. Applying Lemma \ref{examples.A.2} twice, we find that $\[(x_1, x_{2*}) \in \X \times \X^* \suchthat \bend R_*(x_{2*}, x_1)\] =\bend S$ and $\[(x_2, x_{1^*}) \in \X \times \X^* \suchthat \bend S_*(x_{1*}, x_2)\] = \bend R$. Combining the former equality with Proposition \ref{computation.C.2}, we conclude that condition (2) is equivalent to condition (1). Combining the latter equality with Lemma \ref{examples.A.3}, we conclude that condition (3) is equivalent to condition (1). Indeed, we have that $\[(x_2, x_{1^*}) \in \X \times \X^* \suchthat \bend S_*(x_{1*}, x_2)\] = \[(x_2, x_{1^*}) \in \X \times \X^* \suchthat \bend R(x_2, x_{1*})\]$.
\end{proof}

For each atom $X$ of a quantum set $\X$, let $\inc^{\phantom \dagger}_X \in L(X, \bigoplus \At(\X))$ be the inclusion isometry of $X$ into the $\ell^2$-direct sum of the atoms of $\X$.

\begin{theorem}\label{examples.A.5}
Let $\X$ be a quantum set. Then, there is a one-to-one correspondence between quantum graph structures $V$ on $\ell^\infty(\X)$ in the sense of \cite{Weaver2}*{Def.~2.6(d)} and binary relations $R$ on $\X$ such that
\begin{enumerate}
\item $\[(\forall (x \feq x_*) \fin \X \ftimes \X^*)\, \bend R(x, x_*) \] = \top$;
\item $\[(\forall (x_1 \feq x_{1*}) \fin \X \ftimes \X^*)\,(\forall (x_2 \feq x_{2*}) \fin \X \ftimes \X^*)\, (\bend R (x_1, x_{2*}) \IMPLIES \bend R(x_2, x_{1*}))\] = \top$.
\end{enumerate}
The correspondence is given by $R(X_1, X_2) =  \inc_{X_2}^\dagger \cdot V \cdot \inc_{X_1}^{\phantom \dagger}$, for $X_1, X_2 \in \At(\X)$.
\end{theorem}

\begin{proof}
This is an immediate consequence of Propositions \ref{examples.A.1} and \ref{examples.A.4}, and the equivalence between binary relations and quantum relations described in Appendix \ref{appendix.E}.
\end{proof}

\begin{corollary}\label{examples.A.6}
Furthermore, assume that $\At(\X) = \{H\}$ for some nonzero finite-dimensional Hilbert space $H$. Then, there is a one-to-one correspondence between operator systems $V \subsetof L(H)$ and binary relations $R$ on $\X$ satisfying conditions (1) and (2) in the statement of Theorem \ref{examples.A.5}. The correspondence is given by $R(H,H) = V$.
\end{corollary}

\begin{proof}
In the special case $\At(\X) = \{H\}$, we have that $\ell^\infty(\X) = L(H)$ and that $\ell^\infty(\X)' = \CC 1_H$, so every operator system on $H$ is also a quantum graph structure on $L(H)$. The isometry $\incnag_H$ is simply the identity on $H$.
\end{proof}

\subsection{Quantum preordered sets.} A quantum preorder on a von Neumann algebra $M \subseteq L(H)$ is an ultraweakly closed algebra $N \subseteq L(H)$ that contains $M'$ \cite{Weaver2}*{Def.~2.6(b)}. Quantum preorders have not attracted much research interest except as a stepping stone to quantum partial orders \cite{Weaver2}*{Def.~2.6(c)}; this is the role that they play here too.

\begin{lemma}\label{examples.B.1}
Let $\X$, $\Y$ and $\Z$ be quantum sets. Let $R$ be a binary relation from $\X$ to $\Y$, let $S$ be a binary relation from $\Y$ to $\Z$, and let $T$ be the binary relation from $\X$ to $\Z$ defined by $T = S \circ R$. Then,
$ 
\[(x,z_*) \in \X \times \Z^* \suchthat (\exists (y\feq y_*) \in \Y \ftimes \Y^*)\, (\bend R(x,y_*) \AND \bend S(y,z_*))\] = \bend T.
$
\end{lemma}

\begin{proof}
$$
\[(x, y_*, y, z_*) \in \X \times \Y^* \times \Y \times \Z^* \suchthat \bend R(x, y_*) \AND \bend S(y,z_*)\]
\;=\;
\begin{aligned}
      \begin{tikzpicture}[scale=1]
	\begin{pgfonlayer}{nodelayer}
	\node [style=box] (box1) at (-0.7,0) {$R$};
		\node [style=none] (flex1) at (0,0) {};
		\node [style=none] (sink1) at (0,-1) {$\scriptstyle y$};
		\node [style=none] (source1) at (-0.7,-1) {$\scriptstyle x$};
		\node [style=box] (box2) at (0.7,0) {$S$};
		\node [style=none] (flex2) at (1.4,0) {};
		\node [style=none] (sink2) at (1.4,-1) {$\scriptstyle z$};
		\node [style=none] (source2) at (0.7,-1) {$\scriptstyle y$};
	\end{pgfonlayer}
	\begin{pgfonlayer}{edgelayer}
        \draw [arrow, bend left=90, looseness=3] (box1.center) to (flex1.center);
        \draw [arrow, markat=0.8] (flex1.center) to (sink1.north);
        \draw [arrow, markat=0.45] (source1.north) to (box1);
        \draw [arrow, bend left=90, looseness=3] (box2.center) to (flex2.center);
        \draw [arrow, markat=0.8] (flex2.center) to (sink2.north);
        \draw [arrow, markat=0.45] (source2.north) to (box2);
	\end{pgfonlayer}
      \end{tikzpicture}
\end{aligned}\;.
$$

\begin{align*}
\[(x,z_*) \in \X \times \Z^* \suchthat (\exists (y\feq y_*) \fin \Y \ftimes \Y^*)\,&( \bend R(x,y_*) \AND \bend S(y,z_*))\]
\\ &
\;=\;
\begin{aligned}
      \begin{tikzpicture}[scale=1]
	\begin{pgfonlayer}{nodelayer}
	\node [style=box] (box1) at (-0.7,0) {$R$};
		\node [style=none] (flex1) at (0,0) {};
		\node [style=none] (source1) at (-0.7,-1) {$\scriptstyle x$};
		\node [style=box] (box2) at (0.7,0) {$S$};
		\node [style=none] (flex2) at (1.4,0) {};
		\node [style=none] (sink2) at (1.4,-1) {$\scriptstyle z$};
	\end{pgfonlayer}
	\begin{pgfonlayer}{edgelayer}
        \draw [arrow, bend left=90, looseness=3] (box1.center) to (flex1.center);
        \draw [arrow, markat=0.45] (source1.north) to (box1);
        \draw [arrow, bend left=90, looseness=3] (box2.center) to (flex2.center);
        \draw [arrow, markat=0.8] (flex2.center) to (sink2.north);
        \draw [arrow, bend right = 90, looseness=3] (flex1.center) to (box2.center);
	\end{pgfonlayer}
      \end{tikzpicture}
\end{aligned}
\; = \;
\begin{aligned}
      \begin{tikzpicture}[scale=1]
	\begin{pgfonlayer}{nodelayer}
	    \node [style=none] (air) at (0,1.1) {\phantom{air}};
	    \node [style=box] (S) at (0, 0.4)  {$S$};
	    \node [style=box] (R) at (0, -0.4)  {$R$};
        \node [style=none] (source) at (0,-1.1) {$\scriptstyle x$};    
        \node [style=none] (sink) at (1, -1.1) {$\scriptstyle z$};
        \node [style=none] (flex) at (1,0.4) {};
	\end{pgfonlayer}
	\begin{pgfonlayer}{edgelayer}
        \draw [arrow,markat=0.6] (source.north) to (R.south);
        \draw [arrow,markat=0.6] (R.north) to (S.south);
        \draw [arrow, bend left = 90, looseness=2.5] (S.center) to (flex.center);
        \draw [arrow] (flex.center) to (sink.north);
	\end{pgfonlayer}
      \end{tikzpicture}
\end{aligned}
\; = \;
\bend{T} \; . \qedhere
\end{align*}
\end{proof}

\begin{lemma}\label{examples.B.2}
Let $\X$, $\Y$ and $\Z$ be quantum sets. Let $R$ be a binary relation from $\X$ to $\Y$, let $S$ be a binary relation from $\Y$ to $\Z$, and let $T$ be a binary relation from $\X$ to $\Z$. Then, the following are equivalent:
\begin{enumerate}
\item $S \circ R \leq T$;
\item $\[(\forall x \in \X)\, (\forall z_* \in \Z^*)\,((\exists (y\feq y_*) \fin \Y \ftimes \Y^*)\, (\bend R(x,y_*) \AND \bend S(y,z_*)) \IMPLIES \bend T(x,z_*))\] = \top$;
\item $\[(\forall (x\feq x_*) \fin \X \ftimes \X^*)\, (\forall (y \feq y_*) \fin \Y \ftimes \Y^*)\, \forall (z \feq z_*) \fin \Z \ftimes \Z^*.$ \\ \phantom{A} \hfill $((\bend R(x, y_*) \AND \bend S(y, z_*)) \IMPLIES \bend T_*(x_*, z))\] = \top$.
\end{enumerate}
\end{lemma}

\begin{proof}
Let $T_0 = S \circ R$. By Lemma \ref{examples.B.1}, $\[(x,z_*) \in \X \times \Z^* \suchthat (\exists (y\feq y_*) \fin \Y \ftimes \Y^*)\, (\bend R(x,y_*) \AND \bend S(y,z_*))\] = \bend T_0 = \[(x,z_*) \in \X \times \Z^* \suchthat \bend T_0(x,z_*)\]$, so by Proposition \ref{computation.C.2}, condition (2) is equivalent to $T_0 \leq T$, i.e., to condition (1).

No two of the formulas $\bend R(x, y_*)$, $\bend S(y, z_*)$, and $\bend T_*(x_*, z)$ have variables in common, which implies that the formula $\NOT ((\bend R(x, y_*) \AND \bend S(y, z_*)) \IMPLIES \bend T_*(x_*, z))$ has the same interpretation as the formula $\bend R(x, y_*) \AND  \bend S(y, z_*) \AND \NOT \bend T_*(x_*, z)$.
Similarly, the formula $\NOT ((\exists (y \feq y_*) \fin \Y \ftimes \Y^*)\,(\bend R(x, y_*) \AND \bend S(y, z_*)) \IMPLIES \bend T_*(x_*, z))$ has the same interpretation as the formula $(\exists (y \feq y_*) \fin \Y \ftimes \Y^*)\,(\bend R(x, y_*) \AND \bend S(y, z_*)) \AND \NOT \bend T_*(x_*, z)$. It follows by Theorem \ref{computation.D.2} and the duality between the quantifier expressions $(\forall (y \feq y_*) \fin \Y \ftimes \Y^*)$ and $(\exists (y \feq y_*) \fin \Y \ftimes \Y^*)$ that the formula $(\forall (y \feq y_*) \fin \Y \ftimes \Y^*)\, ((\bend R(x, y_*) \AND \bend S(y, z_*)) \IMPLIES \bend T_*(x_*, z))$ has the same interpretation as the formula $(\exists (y \feq y_*) \fin \Y \ftimes \Y^*)\,(\bend R(x, y_*) \AND \bend S(y, z_*)) \IMPLIES \bend T_*(x_*, z)$. We conclude that condition (3) is equivalent to condition (2) by Proposition \ref{examples.A.4}. 
\end{proof}

\begin{theorem}\label{examples.B.3}
Let $\X$ be a quantum set. Then, there is a one-to-one correspondence between quantum preorders $V$ on $\ell^\infty(\X)$ in the sense of \cite{Weaver2}*{Def.~2.6(b)} and binary relations $R$ on $\X$ such that
\begin{enumerate}
\item $\[(\forall (x \feq x_*) \fin \X \ftimes \X^*)\, \bend R(x, x_*) \] = \top$;
\item $\[(\forall (x_1\feq x_{1*}) \fin \X \ftimes \X^*)\, (\forall (x_2 \feq x_{2*}) \fin \X \ftimes \X^*)\, (\forall (x_3 \feq x_{3*}) \fin \X \ftimes \X^*)$ \\ \phantom{A} \hfill $((\bend R(x_1, x_{2*}) \AND \bend R(x_2, x_{3*})) \IMPLIES \bend R_*(x_{1*}, x_3))\] = \top$.
\end{enumerate}
The correspondence is given by $R(X_1, X_2) =  \inc_{X_2}^\dagger \cdot V \cdot \inc_{X_1}^{\phantom \dagger}$, for $X_1, X_2 \in \At(\X)$.
\end{theorem}

\begin{proof}
This is an immediate consequence of Propositions \ref{examples.A.1} and \ref{examples.B.2} and the equivalence between binary relations and quantum relations described in Appendix \ref{appendix.E}.
\end{proof}

\begin{corollary}\label{examples.B.4}
Let $\X$ be a quantum set. Then, there is a one-to-one correspondence between von Neumann algebras $M \subsetof \ell^\infty(\X)$ and binary relations $R$ on $\X$ such that
\begin{enumerate}
\item $\[(\forall (x \feq x_*) \fin \X \ftimes \X^*)\, \bend R(x, x_*) \] = \top$;
\item $\[(\forall (x_1 \feq x_{1*}) \fin \X \ftimes \X^*)\,(\forall (x_2 \feq x_{2*}) \fin \X \ftimes \X^*)\, (\bend R (x_1, x_{2*}) \IMPLIES \bend R(x_2, x_{1*}))\] = \top$;
\item $\[(\forall (x_1\feq x_{1*}) \fin \X \ftimes \X^*)\, (\forall (x_2 \feq x_{2*}) \fin \X \ftimes \X^*)\, (\forall (x_3 \feq x_{3*}) \fin \X \ftimes \X^*)$ \\ \phantom{A} \hfill $((\bend R(x_1, x_{2*}) \AND \bend R(x_2, x_{3*})) \IMPLIES \bend R_*(x_{1*}, x_3))\] = \top$.
\end{enumerate}
The correspondence is given by $R(X_1, X_2) =  \inc_{X_2}^\dagger \cdot M' \cdot \inc_{X_1}^{\phantom \dagger}$, for $X_1, X_2 \in \At(\X)$.
\end{corollary}

\begin{proof}
As an immediate consequence of Theorems \ref{examples.A.5} and \ref{examples.B.3}, we have a one-to-one correspondence between quantum equivalence relations $V$ on $\ell^\infty(\X)$ in the sense of \cite{Weaver2}*{Def.~2.6(a)} and binary relations $R$ on $\X$ satisfying conditions (1), (2) and (3). A quantum equivalence relation on $\ell^\infty(\X)$ is just a von Neumann algebra that contains $\ell^\infty(\X)'$, and such von Neumann algebras are in one-to-one correspondence with the von Neumann algebras contained in $\ell^\infty(\X)$ via the commutant operation.
\end{proof}

The significance of Corollary \ref{examples.B.4} is that according to the noncommutative dictionary, the unital ultraweakly closed $*$-subalgebras of $\ell^\infty(\X)$ correspond to the quotients of $\X$.

\subsection{Quantum posets.}\label{examples.C} The example of quantum posets is quite similar to the example of quantum preordered sets, but it is nevertheless significant, both because quantum posets are of inherent interest \cite{Weaver4}\cite{KornellLindenhoviusMislove} and because this example touches on a basic feature of quantum logic. This basic feature is that the notion of inconsistency between two predicates has multiple natural generalizations to the quantum setting. More generally, the notion of the conjunction of two predicates has multiple natural generalizations to the quantum setting.

The phenomenon is familiar: a pair of predicates $P$ and $Q$ may be inconsistent in the sense that $P \AND Q = \bot$, or they may be inconsistent in the stronger sense that $P \perp Q$. Physically, a pair of predicates $P$ and $Q$ that are inconsistent in the former sense but not in the latter sense correspond to Boolean observables with the property that no state provides the certainty of both $P$ and $Q$, but there are states that provide the certainty of $P$ and the possibility of $Q$, and vice versa. This distinction in turn involves the distinction between the meet connective $\AND$ and the Sasaki projection connective $\sascon$. Both connectives have a natural role in quantum predicate logic; see Lemma \ref{appendix.C.1} and \cite{Finch}.

\begin{proposition}\label{examples.C.1}
Let $\X$ be a quantum set, and let $P$ and $Q$ be predicates on $\X$. Then,
\begin{enumerate}
\item $P \AND Q = \bot_\X$ if and only $\[(\forall x \in \X)\, \NOT( P(x) \AND Q(x))\] = \top$;
\item $P \perp Q$ if and only if $\[(\forall x \in \X)\, \NOT (P(x) \sascon Q(x))\] = \top$;
\item $P \perp Q$ if and only if $\[(\forall (x \feq x_*) \fin \X \ftimes \X^*)\,\NOT (P(x) \AND Q_*(x_*))\] = \top$.
\end{enumerate}
\end{proposition}

\begin{proof}
Each of the three equivalences may be established via the corresponding condition depicted graphically below:
\begin{align*}&
\begin{aligned}
      \begin{tikzpicture}[scale=1]
	\begin{pgfonlayer}{nodelayer}
		\node [style=box] (PandQ) at (0,0) {$P \AND Q$};
		\node [style=none] (end) at (0,-1) {$\bullet$};
	\end{pgfonlayer}
	\begin{pgfonlayer}{edgelayer}
        \draw [arrow] (end.center) to  (PandQ.center);
	\end{pgfonlayer}
      \end{tikzpicture}
\end{aligned}
\;=\; \bot \;,
\hspace{10ex}
\begin{aligned}
      \begin{tikzpicture}[scale=1]
	\begin{pgfonlayer}{nodelayer}
		\node [style=box] (PandQ) at (0,0) {$P \sascon Q$};
		\node [style=none] (end) at (0,-1) {$\bullet$};
	\end{pgfonlayer}
	\begin{pgfonlayer}{edgelayer}
        \draw [arrow] (end.center) to  (PandQ.center);
	\end{pgfonlayer}
      \end{tikzpicture}
\end{aligned}
\;=\; \bot \;,
\hspace{10ex}
\begin{aligned}
      \begin{tikzpicture}[scale=1]
	\begin{pgfonlayer}{nodelayer}
		\node [style=box] (P) at (0,0) {$P_{\phantom *}$};
		\node [style=box] (Q) at (1,0) {$Q_*$};
	\end{pgfonlayer}
	\begin{pgfonlayer}{edgelayer}
        \draw [arrow, bend left=90, looseness=3] (Q.center) to  (P.center);
	\end{pgfonlayer}
      \end{tikzpicture}
\end{aligned}
\;=\; \bot\;.\qedhere
\end{align*}
\end{proof}

\begin{corollary}\label{examples.C.2}
The one-to-one correspondence of Theorem \ref{examples.B.3} restricts to a one-to-one correspondence between quantum partial orders $V$ on $\ell^\infty(\X)$ in the sense of \cite{Weaver2}*{Def.~2.6(c)} and binary relations $R$ on $\X$ satisfying conditions (1) and (2) in the statement of Theorem \ref{examples.B.3} as well as
\begin{enumerate}\setcounter{enumi}{2}
\item $\[(\forall x_1 \in \X)\, (\forall x_{2*} \in \X)\, ((\bend R(x_1, x_{2*}) \AND \bend R_*(x_{2*},x_1)) \IMPLIES E_\X(x_1, x_{2*}))\] = \top$.
\end{enumerate}
\end{corollary}

\begin{proof}
Let $S = R^\dagger$. By Lemma \ref{examples.A.2}, we have that $\[(x_1,x_{2*}) \in \X \times \X^* \suchthat \bend R_*(x_{2*},x_1)\] = \bend S$. By Proposition \ref{computation.C.2}, we conclude that condition (3) is equivalent to $\bend R \AND \bend S \leq E_\X$, i.e., to $R \AND R^\dagger \leq I_\X$, i.e., to $V \cap V^\dagger \subseteq \ell^\infty(\X)'$.
\end{proof}

\begin{lemma}\label{examples.C.3}
Let $\X$ be a quantum set. There is a one-to-one correspondence between binary relations $R$ on $\X$ such that $I_\X \leq R$, $R \circ R \leq R$ and $R \sascon R^\dagger \leq I_\X$ and binary relations $S$ on $\X$ such that $S \perp I_\X$ and $S \circ S \leq S$. The correspondence is given by $R \mapsto R \AND \NOT I_\X$ and by $S \mapsto S \OR I_\X$.
\end{lemma}

\begin{proof}
We gather a couple of basic facts. First, for each binary relation $R$ on $\X$ such that $I_\X \leq R$, we have that 
$R \sascon R^\dagger \leq I_\X \; \Leftrightarrow \; R \leq R^\dagger \IMPLIES I_\X \; \Leftrightarrow \; R \leq \NOT R^\dagger \OR I_\X \; \Leftrightarrow \; R \AND \NOT I_\X \leq \NOT R^\dagger \OR I_\X \; \Leftrightarrow \; (R \AND \NOT I_\X) \perp (R \AND \NOT I_\X)^\dagger$, with the first equivalence following via the Sasaki adjunction and the third equivalence following by orthomodularity. Second, for each binary relation $S$ on $\X$ we have that $S \perp S^\dagger \; \Leftrightarrow\; \Tr_\X(S \circ S) = \bot\; \Leftrightarrow \; S \circ S \perp  I_\X$ \cite{KornellLindenhoviusMislove2}*{App.~C}.

Let $R$ be a binary relation on $\X$ such that $I_\X \leq R$, $R \circ R \leq R$ and $R \sascon R^\dagger \leq I_\X$, and let $S = R \AND \NOT I_\X$. Clearly, $S \perp I_\X$. Furthermore, $S \perp S^\dagger$, which implies that $S \circ S \leq  \NOT I_\X$. We now calculate that $S \circ S \leq (R \circ R) \AND \NOT I_\X \leq R \AND \NOT I_\X = S$. Therefore, $S$ satisfies both $S \perp I_\X$ and $S \circ S \leq S$.

Let $S$ be a binary relation on $\X$ such that $S \perp I_\X$ and $S \circ S \leq S$, and let $R = S \OR I_\X$. Clearly, $I_\X \leq R$. Furthermore, $R \circ R = (S \circ S) \OR (S \circ I_\X) \OR (I_\X \circ S) \OR (I_\X \circ I_\X) \leq S \OR I_\X = R$. Finally, we observe that $S = R \AND \NOT I_\X$ by orthomodularity, and we reason that $S \perp I_\X$ implies $S \circ S \perp I_\X$, which implies $S \perp S^\dagger$, leading to $R \sascon R^\dagger \leq I_\X$. Therefore, $R$ satisfies all three inequalities $I_\X \leq R$, $R \circ R \leq R$ and $R \sascon R^\dagger \leq I_\X$.

We have shown that the construction $R \mapsto R \AND \NOT I_\X$ takes binary relations of the first kind to binary relations of the second kind, and that the construction $S \mapsto S \OR I_\X$ takes binary relations of the second kind to binary relations of the first kind. The two constructions invert each other by orthomodularity.
\end{proof}

\begin{theorem}\label{examples.C.4}
Let $H$ be a nonzero finite-dimensional Hilbert space, and let $\X$ be the quantum set defined by $\At(\X) = \{H\}$. Then, there is a one-to-one correspondence between nilpotent algebras $A \subsetof L(H)$ in the sense of \cite{Weaver4}*{sec.~6} and binary relations $R$ on $\X$ satisfying conditions (1) and (2) in the statement of Theorem \ref{examples.B.3} as well as
\begin{enumerate}\setcounter{enumi}{2}
\item $\[(\forall x_1 \in \X)\, (\forall x_{2*} \in \X)\, ((\bend R(x_1, x_{2*}) \sascon \bend R_*(x_{2*},x_1)) \IMPLIES E_\X(x_1, x_{2*}))\] = \top$.
\end{enumerate}
This correspondence is given by $R(H, H) = A + \CC 1_H$.
\end{theorem}

\begin{proof}
Condition (1) is equivalent to the inequality $I_\X \leq R$ by Proposition \ref{examples.A.1}. Condition (2) is equivalent to the inequality $R \circ R \leq R$ by Proposition \ref{examples.B.2}. Condition (3) is equivalent to the inequality $R \sascon R^\dagger \leq I_\X$, as in the proof of Corollary \ref{examples.C.2}. Thus, by Lemma \ref{examples.C.3}, the binary relations $R$ on $\X$ satisfying conditions (1), (2) and (3) are in one-to-one correspondence with the binary relations $S$ on $\X$ satisfying $S \perp I_\X$ and $S \circ S \leq S$, which is given by $R = S \OR I_\X$, i.e., by $R(H,H) = S(H,H) + \CC 1_H$.

We also have a one-to-one correspondence between the binary relations $S$ on $\X$ and subspaces $A$ of $L(H)$, which is given by $A = S(H,H)$. The binary relations $S$ that satisfy $S \perp I_\X$ correspond to subspaces of trace-zero operators, and the binary relations $S$ that satisfy $S \circ S \leq S$ correspond to subalgebras. Thus, the binary relations $S$ on $\X$ that satisfy both inequalities correspond to subalgebras of trace-zero operators. These are exactly the nilpotent subalgebras because an operator $a \in L(H)$ is nilpotent if and only if $\Tr_H(a^n) = 0$ for each positive integer $n$.

Combining the one-to-one correspondences, we find that binary relations $R$ on $\X$ satisfying conditions (1), (2) and (3) correspond to nilpotent subalgebras $A$ of $L(H)$ via the equation $R(H,H) = A + \CC 1_H$.
\end{proof}

\subsection{Functions between quantum sets.}\label{examples.D}

We have characterized functions between quantum sets as binary relations satisfying a pair of intelligible formulas, which certainly characterize ordinary functions within the dagger compact category of ordinary sets and ordinary binary relations; see Definition \ref{definition.F.1} and Theorem \ref{computation.E.2}. Now, we give streamlined characterizations of functions, injective functions and surjective functions between quantum sets, emphasizing our recovery of the corresponding concepts from noncommutative geometry.

\begin{proposition}\label{examples.D.1}
Let $\X$ and $\Y$ be quantum sets. Then, there is a one-to-one correspondence between unital normal $*$-homomorphisms $\phi$ from $\ell^\infty(\Y)$ to $\ell^\infty(\X)$ and binary relations $F$ on $\X$ such that
\begin{enumerate}
\item $\[(\forall x \fin \X)\, (\exists y_* \fin \Y^*)\, \bend F(x,y_*)\] = \top$;
\item $\[(\forall (x \feq x_*) \fin \X \ftimes \X^*)\, (\forall (y_1 \feq y_{1*}) \fin \Y \ftimes \Y^*)\, (\forall (y_2 \feq y_{2*}) \fin \Y \ftimes \Y^*)$\\ \phantom A \hfill $((\bend F(x, y_{1*}) \AND \bend F_*(x_*, y_2)) \to E_\Y(y_1, y_{2*}))\] = \top.$
\end{enumerate}
These are exactly the functions $F$ from $\X$ to $\Y$ \cite{Kornell}*{Def.~4.1}. This correspondence is given by $\phi = F^\star$ \cite{Kornell}*{Thm.~7.4}.
\end{proposition}

\begin{proof}
This proposition combines Theorem \ref{computation.E.2} with Lemma \ref{examples.A.2} and Proposition \ref{examples.B.2}. In Proposition \ref{examples.B.2}, we set $R = F^\dagger$, $S = F$ and $T = E_\Y$, to infer that
$$\[(\forall y_1 \fin \Y)\, (\forall y_{2*} \fin \Y^*)\,(  (\exists(x \feq x_*) \fin \X \ftimes \X^*)\, (\bend R(y_1,x_*) \AND \bend F(x,y_{2*})) \IMPLIES E_\Y (y_1, y_{2*}))\] = \top$$
is equivalent to
\begin{align*}\tag{\S}\label{examples.D.1.eq}
\[(\forall (y_1 \feq y_{1*}) \fin \Y \ftimes \Y^*)\, (\forall(x \feq x_*) & \fin \X \ftimes \X^*)\, (\forall (y_2 \feq y_{2*}) \fin \Y \ftimes \Y^*)\, \\ &
((\bend R(y_1, x_*) \AND \bend F(x, y_{2*})) \IMPLIES E_{\Y*}(y_{1*} , y_2))\] = \top.
\end{align*}
By Lemma \ref{examples.A.2}, $\[(y_1,x_*) \in \Y \times \X^* \suchthat \bend F_*(x_*,y_1)\] = \bend R =
\[(y_1,x_*) \in \Y \times \X^* \suchthat \bend R(y_1, x_*)\]$, so we can replace $\bend R(y_1,x_*)$ by $\bend F_*(x_*, y_1)$ in both formulas. We thus recover the definition of a function graph (Definiton \ref{definition.F.1}). Similarly, appealing to the fact $I_\X$ is self-adjoint, we can replace $E_{\Y*}(y_{1*},y_2)$ by $E_\Y(y_2, y_{1*})$ in formula (\ref{examples.D.1.eq}). Exchanging the variables $y_1$ and $y_2$, and permuting the quantifiers, we recover condition (2).
\end{proof}

\begin{proposition}\label{examples.D.2}
The one-to-one correspondence in Proposition \ref{examples.D.1} restricts to a one-to-one correspondence between surjective unital normal $*$-homomorphisms $\phi$ from $\ell^\infty(\Y)$ to $\ell^\infty(\X)$ and binary relations $F$ from $\X$ to $\Y$ satisfying conditions (1) and (2) in the statement of that proposition, together with
\begin{enumerate}\setcounter{enumi}{2}
\item $\[(\forall x \fin \X)\, (\forall x_* \fin  \X^*)\, E_\Y(F(x),F_*(x_*)) \IMPLIES E_\X(x,x_*)\] = \top.$
\end{enumerate}
These are exactly the injective functions $F$ from $\X$ to $\Y$ \cite{Kornell}*{Def.~4.3(1)}.
\end{proposition}

\begin{proof}
Appealing to Proposition \ref{computation.C.2} and Lemma \ref{computation.F.2}, we find that condition (3) is equivalent to the following inequality:
$$
\begin{aligned}
      \begin{tikzpicture}[scale=1]
	\begin{pgfonlayer}{nodelayer}
	\node [style=box] (F) at (0,0) {$F_{\phantom *}$};
	\node [style=box] (Fcon) at (1,0) {$F_*$};
	\node [style=none] (x) at (0,-1) {$\scriptstyle x$};
	\node [style=none] (xcon) at (1, -1) {$\scriptstyle x$};
	\end{pgfonlayer}
	\begin{pgfonlayer}{edgelayer}
	\draw[arrow, markat=0.45] (x.north) to (F.center);
	\draw [arrow, bend left = 90, looseness = 2.5] (F.center) to (Fcon.center);
	\draw [arrow, markat =0.7] (Fcon.center) to (xcon.north);
	\end{pgfonlayer}
      \end{tikzpicture}
\end{aligned}
\; \leq \;
\begin{aligned}
      \begin{tikzpicture}[scale=1]
	\begin{pgfonlayer}{nodelayer}
	\node [style=none] (F) at (0,0) {};
	\node [style=none] (Fcon) at (1,0) {};
	\node [style=none] (x) at (0,-1) {$\scriptstyle x$};
	\node [style=none] (xcon) at (1, -1) {$\scriptstyle x$};
	\end{pgfonlayer}
	\begin{pgfonlayer}{edgelayer}
	\draw[arrow, markat=0.45] (x.north) to (F.center);
	\draw [arrow, bend left = 90, looseness = 2.5] (F.center) to (Fcon.center);
	\draw [arrow, markat =0.7] (Fcon.center) to (xcon.north);
	\end{pgfonlayer}
      \end{tikzpicture}
\end{aligned}\;.
$$
Straightening the wire, we conclude that condition (3) is equivalent to the inequality $F^\dagger \circ F \leq I_\X$, i.e., to $F$ being injective. A function $F$ is injective if and only if $F^\star$ is surjective \cite{Kornell}*{Prop.~8.4}.
\end{proof}

\begin{proposition}\label{examples.D.3}
The one-to-one correspondence in Proposition \ref{examples.D.1} restricts to a one-to-one correspondence between injective unital normal $*$-homomorphisms $\phi$ from $\ell^\infty(\Y)$ to $\ell^\infty(\X)$ and binary relations from $\X$ to $\Y$ satisfying conditions (1) and (2) in the statement of that proposition, together with
\begin{enumerate}\setcounter{enumi}{3}
\item $\[(\forall y_* \in \Y^*)\, (\exists x \in \X)\, E_\Y( F(x), y_*)\] = \top$.
\end{enumerate}
These are exactly the surjective functions $F$ from $\X$ to $\Y$ \cite{Kornell}*{Def.~4.3(2)}.
\end{proposition}

\begin{proof}
We glean from the proof of Theorem \ref{computation.E.2} that $\top_\Y \circ F  = \top_\X$, simply because $F$ is a function. Similarly, condition (4) is equivalent to the inequality $\top_\X \circ F^\dagger = \top_\Y$, as we infer from its depiction below:
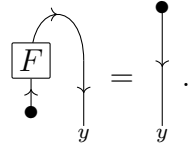
\begin{align*}
\begin{aligned}
      \begin{tikzpicture}[scale=1]
	\begin{pgfonlayer}{nodelayer}
		\node [style=box] (F) at (0.7,0) {$F$};
		\node [style=none] (ri) at (1.4,0) {};
		\node [style=none] (B) at (1.4,-1) {$\scriptstyle y$};
		\node [style=none] (B0) at (0.7,-0.7) {$\bullet$};	
	\end{pgfonlayer}
	\begin{pgfonlayer}{edgelayer}
        \draw [arrow, bend left=90, looseness=3] (F.center) to (ri.center);
        \draw [arrow, markat=0.5] (ri.center) to (B.north);
        \draw [arrow, markat=0.7] (B0.center) to (F);
	\end{pgfonlayer}
      \end{tikzpicture}
\end{aligned}
\; = \;
\begin{aligned}
      \begin{tikzpicture}[scale=1]
	\begin{pgfonlayer}{nodelayer}
		\node [style=none] (B) at (0.7,0.7) {$\bullet$};
		\node [style=none] (B0) at (0.7,-1) {$\scriptstyle y$};	
	\end{pgfonlayer}
	\begin{pgfonlayer}{edgelayer}
        \draw [arrow, markat=0.5] (B.center) to (B0.north);
	\end{pgfonlayer}
      \end{tikzpicture}
\end{aligned}\;.
\end{align*}

By the definition of a function from $\X$ to $\Y$, the function $F$ satisfies the inequality $F \circ F^\dagger \leq I_\Y$. It is surjective if and only if $F \circ F^\dagger = I_\Y$. If $F \circ F^\dagger = I_\Y$, then $ \top_\X \circ F^\dagger \geq\ \top_\Y \circ F \circ F^\dagger  =\top_\Y \circ I_\X$, and therefore $\top_\X \circ F^\dagger = \top_\Y$. Inversely, if $F \circ F^\dagger \neq I_\Y$, then the inequality $F \circ F^\dagger \leq I_\Y$ implies that $(F \circ F^\dagger)(W,Y) = 0$ for some atom $W \in \At(\Y)$ and all atoms $Y \in \At(\Y)$. Thus, $\Y$ has a nonempty subset $\W$ whose inclusion function $J$ satisfies $F \circ F^\dagger \circ J = \bot_\W^\Y$, where $\bot_\W^\Y$ denotes the minimum binary relation from $\W$ to $\Y$. If we suppose that $\top_\X \circ F^\dagger = \top_\Y$, then we may calculate that $\bot_\W = \top_\Y \circ F \circ F^\dagger \circ J = \top_\X \circ F^\dagger \circ J = \top_\Y \circ J = \top_\W$, contradicting that $\W$ is nonempty. Therefore, if $F \circ F^\dagger \neq I_\Y$, then $\top_\X \circ F^\dagger \neq \top_\Y$.

We conclude that $F$ satisfies condition (4) if and only if it is surjective. A function $F$ is surjective if and only if $F^\star$ is injective \cite{Kornell}*{Prop.~8.1}.
\end{proof}

Noting the similarity between the proof of Proposition \ref{examples.D.3} and the proof of Theorem \ref{computation.E.2}, the reader may well ask whether a binary relation $R$ from $\X$ to $\Y$ is surjective in the sense of \cite{Kornell}*{Def.~4.3(2)} if and only if $F \circ \top_\X^\dagger = \top_Y^\dagger$ or equivalently $\[(\forall y_* \in \Y^*)\, (\exists x \in \X)\, R(x,y_*)\] = \top$. The answer is no, by the same simple example that appeared at the end of section 4 of \cite{Kornell}, which we now revisit. Thus, the notion of surjectivity for binary relations has multiple natural generalizations to the quantum setting.

Let $\X$ be the quantum set whose only atom is $\CC^2$, let $a \in L(\CC^2, \CC^2)$ be an invertible matrix that is not a scalar multiple of a unitary matrix, and let $R$ be the binary relation on $\X$ defined by $R(\CC^2, \CC^2) = \CC a$. By our choice of $a$, the binary relation $R$ does not satisfy the inequality $R \circ R^\dagger \geq I_\X$. However, by our choice of $a$, the binary relation $R$ is clearly invertible, so it does satisfy the inequality $R \circ \top_\X^\dagger \geq \top_\Y^\dagger$, as a simple consequence of $\top_\X^\dagger \geq R\inv \circ \top_\Y^\dagger$. We conclude that $R$ does not satisfy $R \circ R^\dagger \geq I_\X$ and that it does satisfy $R \circ \top_\X^\dagger = \top_\Y^\dagger$.

\subsection{Quantum metric spaces}\label{examples.F}
A quantum metric on a von Neumann algebra $M \subsetof L(H)$ is a family of ultraweakly closed subspaces $(V_\alpha \subsetof L(H) \suchthat \alpha \in [0,\infty))$ such that $V_\alpha V_\beta \subsetof V_{\alpha+\beta}$ for all $\alpha, \beta \in [0,\infty)$, $V_\alpha = \bigcap_{\beta > \alpha} V_\beta$ for all $\alpha \in [0,\infty)$ and $V_0 = M'$ \cite{KuperbergWeaver}*{Def.~2.1(a) and 2.3}. Intuitively, each subspace $V_\alpha$ consists of those operators which transform the configuration of the system to a configuration that is at most distance $\alpha \in [0,\infty)$ away. In the motivating example of quantum Hamming distance, $M = M_2(\CC) \tensor \cdots \tensor M_2(\CC)$ and $V_\alpha$ is the span of operators $a_1 \tensor \cdots \tensor a_n$, with $a_i \neq 1$ for at most $\alpha$ indices $i \in \{1, \ldots, n\}$. Thus, $V_\alpha$ is spanned by operators that corrupt at most $\alpha$ qubits. Note that these quantum metrics generalize metrics for which the distance between two points may be infinite \cite{KuperbergWeaver}*{sec.~2.1, Prop.~2.5}.

\begin{lemma}\label{examples.F.1}
Let $\X_1, \ldots, \X_n$ be quantum sets, and let $A_1, \ldots, A_m$ be ordinary sets. Let $t_1, \ldots, t_m$ be terms of sorts $`A_1, \ldots, `A_m$, whose free variables are among $x_1, \ldots, x_n$ of sorts $\X_1, \ldots, \X_n$ and $x_{1*}, \ldots, x_{n*}$ of sorts $\X_1^*, \ldots, \X_n^*$. Let $r$ be an ordinary relation of arity $(A_1, \ldots, A_m)$. If $`r(t_1, \ldots, t_n)$ is nonduplicating, then the following are equivalent:
\begin{enumerate}
\item $\[(\forall (x_1 \feq x_{1*}) \fin \X_1 \ftimes \X_1^*)\cdots (\forall (x_n \feq x_{n*}) \fin \X_n \ftimes \X_n^*)\, `r(t_1, \ldots, t_m)\] = \top$;
\item for all $(a_1, \ldots, a_m) \in A_1 \times \cdots \times A_m \setminus r$,
$$\[(\exists (x_1 \feq x_{1*}) \fin \X_1 \ftimes \X_1^*)\cdots (\exists (x_n \feq x_{n*}) \fin \X_n \ftimes \X_n^*)\,( E_{`A}(t_1, `a_1) \AND \cdots \AND E_{`A}(t_m,`a_m))\] = \bot.$$
\end{enumerate}
\end{lemma}

The following proof includes formulas that may be infinite conjunctions or infinite disjunctions of other formulas. Formally, such formulas are not part of our language, but they pose no special difficulty. Indeed, Definition \ref{definition.C.1} excludes infinitary conjunction exclusively for the sake of the exposition. In light of Lemma \ref{appendix.C.1}, we may regard each infinitary conjunction $\bigwedge_{(a_1, \ldots, a_n) \in A^n} \phi(`a_{1}, \ldots, `a_{m}, y_1, \ldots, y_n)$ as abbreviating
$$ (\forall x_1 \in `A) \cdots (\forall x_n \in `A)\, \phi(x_1, \ldots,x_n, y_1, \ldots,y_n),$$
and we may do similarly for each infinitary conjunction $\bigwedge_{(a_1, \ldots, a_n) \in A^n}\phi(`a_{1*}, \ldots, `a_{m*}, y_1, \ldots, y_n)$.

\begin{proof}[Proof of Lemma \ref{examples.F.1}]
The relation $ \NOT r = A_1 \times \cdots \times A_m \setminus r$ may be regarded as a morphism in the dagger compact category of ordinary sets and binary relations. As a binary relation from $A_1 \times \cdots \times A_m$ to the singleton set $\{\ast\}$, it is defined by $\NOT r = \sup\{a_1^\dagger \times \cdots \times a_m^\dagger \suchthat (a_1, \ldots a_m) \in A_1 \times \cdots \times A_m \setminus r \}$. For brevity, we will write $\overline a = (a_1, \ldots, a_m)$.
\begin{align*}
 &\[`(\NOT r)(t_1, \ldots t_n)\]
\; = \quad
\begin{aligned}
\begin{tikzpicture}
\begin{pgfonlayer}{nodelayer}
    \node [style = box] (t1) at (-1, 0) {$\[t_1\]$};
    \node [style = box] (tm) at (1,0) {$\[t_m\]$};
    \node [style = none] (dots) at (0,0) {$\cdots$};
    \node [style = box] (notr) at (0,1) {$\ghost\qquad`(\NOT r )\qquad\ghost$};
    \node [style = none] (anchor1) at (-1,1) {};
    \node [style = none] (anchorm) at (1,1) {};
    \node [style = none] (t1dots) at (-1,-0.4) {$\scriptstyle \cdots$};
    \node [style = none] (tmdots) at (1,-0.4) {$\scriptstyle \cdots$};
    \node [style = none] (t1var1) at (-1.3,-0.5) {};
    \node [style = none] (t1var1anchor) at (-1.3,0) {};
    \node [style = none] (t1var2) at (-0.7,-0.5) {};
    \node [style = none] (t1var2anchor) at (-0.7,0) {};
    \node [style = none] (t2var1) at (1.3,-0.5) {};
    \node [style = none] (t2var1anchor) at (1.3,0) {};
    \node [style = none] (t2var2) at (0.7,-0.5) {};
    \node [style = none] (t2var2anchor) at (0.7,0) {};
\end{pgfonlayer}
\begin{pgfonlayer}{edgelayer}
    \draw [arrow] (t1.center) to (anchor1.center);
    \draw [arrow] (tm.center) to (anchorm.center);
    \draw (t1var1) to (t1var1anchor);
    \draw (t1var2) to (t1var2anchor);
    \draw (t2var1) to (t2var1anchor);
    \draw (t2var2) to (t2var2anchor);
\end{pgfonlayer}
\end{tikzpicture}
\end{aligned}
\quad = \;
\bigvee_{(a_1, \ldots, a_m) \in \NOT r}
\; \left(
\begin{aligned}
\begin{tikzpicture}
\begin{pgfonlayer}{nodelayer}
    \node [style = box] (t1) at (-1, 0) {$\[t_1\]$};
    \node [style = box] (tm) at (1,0) {$\[t_m\]$};
    \node [style = none] (dots) at (0,0) {$\cdots$};
    \node [style = box] (s1dagger) at (-1,1) {$`a_1^\dagger$};
    \node [style = box] (smdagger) at (1,1) {$`a_m^\dagger$};
    \node [style = none] (highdots) at (0,1) {$\cdots$};
    \node [style = none] (anchor1) at (-1,1) {};
    \node [style = none] (anchorm) at (1,1) {};
    \node [style = none] (t1dots) at (-1,-0.4) {$\scriptstyle \cdots$};
    \node [style = none] (tmdots) at (1,-0.4) {$\scriptstyle \cdots$};
    \node [style = none] (t1var1) at (-1.3,-0.5) {};
    \node [style = none] (t1var1anchor) at (-1.3,0) {};
    \node [style = none] (t1var2) at (-0.7,-0.5) {};
    \node [style = none] (t1var2anchor) at (-0.7,0) {};
    \node [style = none] (t2var1) at (1.3,-0.5) {};
    \node [style = none] (t2var1anchor) at (1.3,0) {};
    \node [style = none] (t2var2) at (0.7,-0.5) {};
    \node [style = none] (t2var2anchor) at (0.7,0) {};
\end{pgfonlayer}
\begin{pgfonlayer}{edgelayer}
    \draw [arrow] (t1.center) to (anchor1.center);
    \draw [arrow] (tm.center) to (anchorm.center);
    \draw (t1var1) to (t1var1anchor);
    \draw (t1var2) to (t1var2anchor);
    \draw (t2var1) to (t2var1anchor);
    \draw (t2var2) to (t2var2anchor);
\end{pgfonlayer}
\end{tikzpicture}
\end{aligned}
\right)
\\ & \; = \;
\bigvee_{(a_1, \ldots, a_m) \in \NOT r}
\; \left(
\begin{aligned}
\begin{tikzpicture}
\begin{pgfonlayer}{nodelayer}
    \node [style = box] (t1) at (-2, 0) {$\[t_1\]$};
    \node [style = box] (tm) at (1,0) {$\[t_m\]$};
    \node [style = none] (dots) at (0,0) {$\cdots$};
    \node [style = none] (t1dots) at (-2,-0.4) {$\scriptstyle \cdots$};
    \node [style = none] (tmdots) at (1,-0.4) {$\scriptstyle \cdots$};
    \node [style = none] (t1var1) at (-2.3,-0.5) {};
    \node [style = none] (t1var1anchor) at (-2.3,0) {};
    \node [style = none] (t1var2) at (-1.7,-0.5) {};
    \node [style = none] (t1var2anchor) at (-1.7,0) {};
    \node [style = none] (t2var1) at (1.3,-0.5) {};
    \node [style = none] (t2var1anchor) at (1.3,0) {};
    \node [style = none] (t2var2) at (0.7,-0.5) {};
    \node [style = none] (t2var2anchor) at (0.7,0) {};
    \node [style = box] (s1dagger) at (-0.9,0) {$`a_{1*}$};
    \node [style = box] (smdagger) at (2.1,0) {$`a_{m*}$};
\end{pgfonlayer}
\begin{pgfonlayer}{edgelayer}
    \draw (t1var1) to (t1var1anchor);
    \draw (t1var2) to (t1var2anchor);
    \draw (t2var1) to (t2var1anchor);
    \draw (t2var2) to (t2var2anchor);
    \draw [arrow, bend left = 90, looseness = 2] (t1) to (s1dagger);
    \draw [arrow, bend left = 90, looseness = 2] (tm) to (smdagger);
\end{pgfonlayer}
\end{tikzpicture}
\end{aligned}
\right)
\\ & \; = \;
\[ \textstyle \bigvee_{\overline a \in \NOT r} ( E_{`A}(t_1, `a_{1*}) \AND \cdots \AND E_{`A}(t_m, `a_{m*}))\].
\end{align*}

\begin{align*}
& \[(\forall (x_1 \feq x_{1*}) \fin \X_1 \ftimes \X_1^*)\cdots (\forall (x_n \feq x_{n*}) \fin \X_n \ftimes \X_n^*)\, `r(t_1, \ldots, t_m)\]
\\ & =
\[\NOT (\exists (x_1 \feq x_{1*}) \fin \X_1 \ftimes \X_1^*)\cdots (\exists (x_n \feq x_{n*}) \fin \X_n \ftimes \X_n^*)\, \NOT`r(t_1, \ldots, t_m)\]
\\ & =
\[\NOT (\exists (x_1 \feq x_{1*}) \fin \X_1 \ftimes \X_1^*)\cdots (\exists (x_n \feq x_{n*}) \fin \X_n \ftimes \X_n^*)\, (\NOT`r)(t_1, \ldots, t_m)\]
\\ & = 
\[\NOT (\exists (x_1 \feq x_{1*}) \fin \X_1 \ftimes \X_1^*)\cdots (\exists (x_n \feq x_{n*}) \fin \X_n \ftimes \X_n^*)\, `(\NOT r)(t_1, \ldots, t_m)\]
\\ & = 
\[\NOT (\exists (x_1 \feq x_{1*}) \fin \X_1 \ftimes \X_1^*)\cdots (\exists (x_n \feq x_{n*}) \fin \X_n \ftimes \X_n^*)\, \textstyle \bigvee_{\overline a \in \NOT r} ( E_{`A}(t_1, `a_{1*}) \AND \cdots \AND E_{`A}(t_m), `a_{m*}))\]
\\ & =
\[\NOT \textstyle \bigvee_{\overline a \in \NOT r} (\exists (x_1 \feq x_{1*}) \fin \X_1 \ftimes \X_1^*)\cdots (\exists (x_n \feq x_{n*}) \fin \X_n \ftimes \X_n^*)\,  ( E_{`A}(t_1, `a_{1*}) \AND \cdots \AND E_{`A}(t_m), `a_{m*}))\]
\\ & =
\bigwedge_{\overline a \in \NOT r}\[\NOT (\exists (x_1 \feq x_{1*}) \fin \X_1 \ftimes \X_1^*)\cdots (\exists (x_n \feq x_{n*}) \fin \X_n \ftimes \X_n^*)\,  ( E_{`A}(t_1, `a_{1*}) \AND \cdots \AND E_{`A}(t_m), `a_{m*}))\].
\end{align*}
For the second equality, we appeal to Lemma \ref{computation.F.2}; \cite{Kornell}*{Thm.~B.8} implies that for each function from a quantum set $\X$ to a quantum set $\Y$ and each binary relation $R$ from $\Y$ to $\mathbf 1$, we have that $\neg(R \circ F) = (\neg R) \circ F$. For the penultimate equality, we appeal to Theorem \ref{computation.D.2} and either to Proposition \ref{computation.C.3} or to the fact that the composition of binary relations between quantum sets respects arbitrary joins in both arguments, according to our gloss of the infinitary disjunction symbol. The conjunction of binary relations is equal to $\top$ if and only if each conjunct is equal to $\top$, so the statement of the theorem follows.
\end{proof}

\begin{proposition}\label{examples.F.2}
Let $\X$ be a quantum set, let $A$ be an ordinary set, and let $F\: \X \times \X^* \to `A$ be a function. For each $a \in A$, let $R_a$ be the binary relation on $\X$ defined by $\bend R_a = `a^\dagger \circ F$. Let $a_0 \in A$. Then, $R_{a_0} \geq I_{`A}$ if and only if
\begin{equation*}\label{examples.F.2*}\[(\forall (x_1 \feq x_{1*}) \fin \X \ftimes \X^*)\, `(=_A)(F(x_1,x_{1*}),`a_0)\] = \top. \end{equation*}
\end{proposition}

\begin{proof}
Applying Lemma \ref{examples.F.1}, we find that the following are equivalent:
\begin{align*}&
\[(\forall (x_1 \feq x_{1*}) \fin \X \ftimes \X^*)\, `(=_A)(F(x_1,x_{1*}),`a_0)\] = \top
\\ & \Longleftrightarrow \quad
\text{for all distinct $a_1, a_2 \in A$, } \quad
\begin{aligned}
\begin{tikzpicture}
\begin{pgfonlayer}{nodelayer}
    \node [style=box] (F) at (0,0) {$\phantom| \, F \,\phantom|$};
    \node [style=box] (s0) at (1.2,0) {$`a_{0}$};
    \node [style=box] (s1) at (0,1) {$`a_1^\dagger$};
    \node [style=box] (s2) at (1.2,1) {$`a_2^\dagger$};
    \node [style=none] (anchor2) at (0.25,-0.2) {};
    \node [style=none] (anchor1) at (-0.25, -0.2) {};
\end{pgfonlayer}
\begin{pgfonlayer}{edgelayer}
    \draw [arrow, bend left=90, looseness=3] (anchor2) to (anchor1);
    \draw [arrow] (F.north) to (s1);
    \draw [arrow] (s0) to (s2);
\end{pgfonlayer}
\end{tikzpicture}
\end{aligned}
\quad = \; \bot
\\ & \Longleftrightarrow \quad
\text{for all $a_1\in A$ distinct from $a_0$, } \quad
\begin{aligned}
\begin{tikzpicture}
\begin{pgfonlayer}{nodelayer}
    \node [style=box] (Rs1) at (0,0) {$\phantom|\bend R_{a_1}\phantom|$};
    \node [style=none] (anchor2) at (0.25,-0.2) {};
    \node [style=none] (anchor1) at (-0.25, -0.2) {};
\end{pgfonlayer}
\begin{pgfonlayer}{edgelayer}
    \draw [arrow, bend left=90, looseness=3] (anchor2) to (anchor1);
\end{pgfonlayer}
\end{tikzpicture}
\end{aligned} \quad = \; \bot
\\ & \Longleftrightarrow \quad
\text{for all $a_1 \in A$ distinct from $a_0$,}  \quad
\begin{aligned}
\begin{tikzpicture}
\begin{pgfonlayer}{nodelayer}
    \node [style = box] (R) at (0,0) {$R_{a_1}$};
    \node [style = none] (ri) at (0.7,0) {};
\end{pgfonlayer}
\begin{pgfonlayer}{edgelayer}
    \draw [arrow, bend left = 90, looseness = 3] (R.center) to (ri.center);
    \draw [arrow, bend left = 90, looseness = 3] (ri.center) to (R.center);
\end{pgfonlayer}
\end{tikzpicture}
\end{aligned}
\quad =\; \bot
\\ & \Longleftrightarrow \quad
\text{for all $a_1 \in A$ distinct from $a_0$,\;$R_{a_1} \perp I_{`A}$}
\quad \Longleftrightarrow \quad
I_{`A} \leq R_{a_0}.
\end{align*}

The second equivalence follows form the fact that $`a_2^\dagger \circ `a_0 = \top$ or $`a_2^\dagger \circ `a_0 = \bot$ according to whether $a_2 = a_0$ or $a_2 \neq a_0$. The last equivalence follows from the fact that $\{`a^\dagger \circ F \suchthat a \in A\}$ consists of pairwise-orthogonal relations of arity $(\X, \X^*)$ whose join is the maximum relation of arity $(\X, \X^*)$, and thus, $\{R_a \suchthat a \in A\}$ consists of binary relations on $\X$ whose join is the maximum binary relation on $\X$.
\end{proof}

\begin{proposition}\label{examples.F.3}
Let $\X$ be a quantum set, let $A$ be an ordinary set, and let $F\: \X \times \X^* \to `A$ be a function. For each $a \in A$, let $R_a$ be the binary relation on $\X$ defined by $\bend R_a = `a^\dagger \circ F$. Then, $R_a^\dagger = R_a$ for all $a \in A$ if and only if 
$$\[(\forall (x_1 \feq x_{1*}) \fin \X \ftimes \X^*)\, (\forall (x_2 \feq x_{2*}) \fin \X \ftimes \X^*)\, `(=_A)( F(x_1, x_{2*}),  F(x_2, x_{1*}))\] = \top.$$
\end{proposition}

\begin{proof}
Applying Lemma \ref{examples.F.1}, we find that the following are equivalent:
\begin{align*}&
\[(\forall (x_1 \feq x_{1*}) \fin \X \ftimes \X^*)\, (\forall (x_2 \feq x_{2*}) \fin \X \ftimes \X^*)\, `(=_A)( F(x_1, x_{2*}),  F(x_2, x_{1*}))\] = \top
\\ & \Longleftrightarrow \quad
\text{for all distinct $a_1, a_2 \in A$, } \quad
\begin{aligned}
\begin{tikzpicture}
\begin{pgfonlayer}{nodelayer}
    \node [style=box] (s1) at (0,1) {$`a_1^\dagger$};
    \node [style=box] (s2) at (1,1) {$`a_{2}^\dagger$};
    \node [style=box] (F1) at (0,0) {$\phantom|\, F\, \phantom|$};
    \node [style=box] (F2) at (1,0) {$\phantom|\, F\, \phantom |$};
    \node [style=none] (F1anch1) at (-0.25,-0.2) {};
    \node [style=none] (F1anch2) at (0.25, -0.2) {};
    \node [style=none] (F2anch1) at (0.75,-0.2) {};
    \node [style=none] (F2anch2) at (1.25, -0.2) {};
\end{pgfonlayer}
\begin{pgfonlayer}{edgelayer}
    \draw [arrow] (F1.north) to (s1.south);
    \draw [arrow] (F2.north) to (s2.south);
    \draw [arrow, bend right = 90, looseness=2.5] (F1anch2) to (F2anch1);
    \draw [arrow, bend left = 90, looseness=2] (F2anch2) to (F1anch1);
\end{pgfonlayer}
\end{tikzpicture}
\end{aligned}
\quad = \; \bot
\\ & \Longleftrightarrow \quad
\text{for all distinct $a_1, a_2 \in A$, } \quad
\begin{aligned}
\begin{tikzpicture}
\begin{pgfonlayer}{nodelayer}
    \node [style=box] (R1) at (0,0) {$\phantom|\, \bend R_{a_1} \phantom|$};
    \node [style=box] (R2) at (1.4,0) {$\phantom|\, \bend R_{a_2} \phantom |$};
    \node [style=none] (F1anch1) at (-0.25,-0.2) {};
    \node [style=none] (F1anch2) at (0.25, -0.2) {};
    \node [style=none] (F2anch1) at (1.15,-0.2) {};
    \node [style=none] (F2anch2) at (1.65, -0.2) {};
\end{pgfonlayer}
\begin{pgfonlayer}{edgelayer}
    \draw [arrow, bend right = 90, looseness=2] (F1anch2) to (F2anch1);
    \draw [arrow, bend left = 90, looseness=1.8] (F2anch2) to (F1anch1);
\end{pgfonlayer}
\end{tikzpicture}
\end{aligned} \quad = \; \bot
\\ & \Longleftrightarrow \quad
\text{for all distinct $a_1, a_2 \in A$, } \quad
\begin{aligned}
\begin{tikzpicture}
\begin{pgfonlayer}{nodelayer}
    \node [style=box] (R1) at (-0.3,0) {$  R_{a_1}$};
    \node [style=box] (R2) at (1.1,0) {$ R_{a_2}$};
    \node [style=none] (anch1) at (0.3,0) {};
    \node [style=none] (anch2) at (1.7,0) {};
\end{pgfonlayer}
\begin{pgfonlayer}{edgelayer}
    \draw [arrow, bend left = 90, looseness = 3 ] (R1.center) to (anch1.center);
    \draw [arrow, bend right =90, looseness = 2.5 ] (anch1.center) to (R2.center);
    \draw [arrow, bend left = 90, looseness = 3 ] (R2.center) to (anch2.center);
    \draw [arrow, bend left = 90, looseness =2 ] (anch2.center) to (R1.center);
\end{pgfonlayer}
\end{tikzpicture}
\end{aligned} \quad = \; \bot
\\ & \Longleftrightarrow \quad
\text{for all distinct $a_1, a_2 \in A$, } \quad
\begin{aligned}
\begin{tikzpicture}
\begin{pgfonlayer}{nodelayer}
    \node [style=box] (R1) at (0,0) {$R_{a_1}\, $};
    \node [style=box] (R2) at (1.2,0) {$R_{a_2}^* $};
\end{pgfonlayer}
\begin{pgfonlayer}{edgelayer}
    \draw [arrow, bend left = 90, looseness = 2.5] (R1.center) to (R2.center);
    \draw [arrow, bend left = 90, looseness = 2.5] (R2.center) to (R1.center);
\end{pgfonlayer}
\end{tikzpicture}
\end{aligned} \quad = \; \bot
\\ & \Longleftrightarrow \quad
\text{for all distinct $a_1, a_2 \in A$,\; $R_{a_1} \perp R_{a_2}^\dagger$}
\quad \Longleftrightarrow \quad
\text{for all $a \in A$, \; $R_a \leq R_a^\dagger$}.
\end{align*}
The second-to-last equivalence holds because $R_{a_2}^* = (R_{a_2}^\dagger)_*$. The last equivalence holds because $\{R_a\suchthat a \in A\}$ consists of pairwise-orthogonal relations on $\X$ whose join is the maximum relation on $\X$. The inequality $R_a \leq R_a^\dagger$ for all $a \in A$ is equivalent to the equality $R_a = R_a^\dagger$ for all $a \in A$ because the adjoint operation is an order isomorphism.
\end{proof}

\begin{proposition}\label{examples.F.4}
Let $\X$ be a quantum set, and let $F\: \X \times \X^* \to `\RRm$ be a function. For each $\alpha \in \RRm$, let $R_\alpha$ be the binary relation on $\X$ defined by $\bend R_\alpha = `\alpha^\dagger \circ F$. Then, $R_{\alpha_2} \circ R_{\alpha_1} \leq \bigvee_{\alpha \leq \alpha_1 + \alpha_2} R_\alpha^\dagger$ for all $\alpha_1, \alpha_2 \in \RRm$ if and only if
\begin{align*}
& \[ (\forall (x_1 \feq x_{1*}) \fin \X \ftimes \X^*)\, (\forall (x_2 \feq x_{2*}) \fin \X \ftimes \X^*)\,(\forall (x_3 \feq x_{3*}) \fin \X \ftimes \X^*) \\ & \hspace{35ex}  `(\bend\leq) (F(x_1,x_{2*}),`(+)(F(x_2,x_{3^*}), F(x_3, x_{1*})))\] = \top.
\end{align*}
\end{proposition}

\begin{proof}
Refer to \cite{KornellLindenhoviusMislove2}*{App.~C} for the basic properties of the trace on binary relations. Let $\alpha, \alpha_1, \alpha_2 \in \RRm$. We compute that
\begin{align*}
\begin{aligned}
\begin{tikzpicture}[scale = 1]
\begin{pgfonlayer}{nodelayer}
    \node [style=box] (F1) at (0,0) {$\phantom|\, F\, \phantom|$};
    \node [style=box] (F2) at (1,0) {$\phantom|\, F\, \phantom|$};
    \node [style=box] (F3) at (2,0) {$\phantom|\, F\, \phantom|$};
    \node [style=box] (s1) at (0, 1) {$`\alpha_{\phantom{1}}^\dagger$};
    \node [style=box] (s2) at (1, 1) {$`\alpha_1^\dagger$};
    \node [style=box] (s3) at (2, 1) {$`\alpha_2^\dagger$};
    \node [style=none] (anch1) at (-0.2,-0.2) {};
    \node [style=none] (anch2) at (0.2,-0.2) {};
    \node [style=none] (anch3) at (0.8,-0.2) {};
    \node [style=none] (anch4) at (1.2,-0.2) {};
    \node [style=none] (anch5) at (1.8,-0.2) {};
    \node [style=none] (anch6) at (2.2,-0.2) {};
\end{pgfonlayer}
\begin{pgfonlayer}{edgelayer}
    \draw [arrow] (F1.north) to (s1.south);
    \draw [arrow] (F2.north) to (s2.south);
    \draw [arrow] (F3.north) to (s3.south);
    \draw [arrow, bend right = 90, looseness=2.5] (anch2) to (anch3);
    \draw [arrow, bend right = 90, looseness=2.5] (anch4) to (anch5);
    \draw [arrow, bend left = 90, looseness=1.5] (anch6) to (anch1);
\end{pgfonlayer}
\end{tikzpicture}
\end{aligned}
\quad & = \quad
\begin{aligned}
\begin{tikzpicture}[scale = 1]
\begin{pgfonlayer}{nodelayer}
    \node [style=none] (cloud) at (0,1.2) {$\phantom \circ$};
    \node [style=box] (R1) at (0,0) {$\phantom|\, \bend R_{\alpha_{\phantom{1}}} \phantom|$};
    \node [style=box] (R2) at (1.5,0) {$\phantom|\, \bend R_{\alpha_1} \phantom|$};
    \node [style=box] (R3) at (3,0) {$\phantom|\, \bend R_{\alpha_2} \phantom|$};
    \node [style=none] (anch1) at (-0.3,-0.2) {};
    \node [style=none] (anch2) at (0.3,-0.2) {};
    \node [style=none] (anch3) at (1.2,-0.2) {};
    \node [style=none] (anch4) at (1.8,-0.2) {};
    \node [style=none] (anch5) at (2.7,-0.2) {};
    \node [style=none] (anch6) at (3.3,-0.2) {};
\end{pgfonlayer}
\begin{pgfonlayer}{edgelayer}
    \draw [arrow, bend right = 90, looseness=2] (anch2) to (anch3);
    \draw [arrow, bend right = 90, looseness=2] (anch4) to (anch5);
    \draw [arrow, bend left = 90, looseness=1] (anch6) to (anch1);
\end{pgfonlayer}
\end{tikzpicture}
\end{aligned}
\quad = \quad
\begin{aligned}
\begin{tikzpicture}[scale = 1]
\begin{pgfonlayer}{nodelayer}
    \node [style=none] (cloud) at (0,1.2) {$\phantom \circ$};
    \node [style=box] (R1) at (-0.3,0) {$ R_{\alpha_{\phantom{1}}}$};
    \node [style=box] (R2) at (1.2,0) {$ R_{\alpha_1}$};
    \node [style=box] (R3) at (2.7,0) {$R_{\alpha_2}$};
    \node [style=none] (anch1) at (0.3,0) {};
    \node [style=none] (anch2) at (1.8,0) {};
    \node [style=none] (anch3) at (3.3,0) {};
\end{pgfonlayer}
\begin{pgfonlayer}{edgelayer}
    \draw [arrow, bend left = 90, looseness=4] (R1.center) to (anch1.center);
    \draw [arrow, bend right = 90, looseness=2.5] (anch1.center) to (R2.center);
    \draw [arrow, bend left = 90, looseness=4] (R2.center) to (anch2.center);
    \draw [arrow, bend right = 90, looseness=2.5] (anch2.center) to (R3.center);
    \draw [arrow, bend left = 90, looseness=4] (R3.center) to (anch3.center);
    \draw [arrow, bend left = 90, looseness=1] (anch3.center) to (R1.center);
\end{pgfonlayer}
\end{tikzpicture}
\end{aligned}
\\ & = \quad 
\begin{aligned}
\begin{tikzpicture}
\begin{pgfonlayer}{nodelayer}
    \node [style=box] (R1) at (0,-1) {$R_{\alpha_{\phantom{1}}}$};
    \node [style=box] (R2) at (0,0) {$R_{\alpha_1}$};
    \node [style=box] (R3) at (0,1) {$R_{\alpha_2}$};
    \node [style=none] (anch1) at (1.5,-1) {};
    \node [style=none] (anch3) at (1.5,1) {};
\end{pgfonlayer}
\begin{pgfonlayer}{edgelayer}
    \draw [arrow] (R1.north) to (R2.south);
    \draw [arrow] (R2.north) to (R3.south);
    \draw [arrow, bend left = 90, looseness = 2] (R3.center) to (anch3.center);
    \draw [arrow] (anch3.center) to (anch1.center);
    \draw [arrow, bend left = 90, looseness =2] (anch1.center) to (R1.center);
\end{pgfonlayer}
\end{tikzpicture}
\end{aligned}
\quad = \;
\Tr_\X (R_{\alpha_2} \circ R_{\alpha_1} \circ R_{\alpha}).
\end{align*}

Let $r$ be the ordinary relation of arity $(\RRm, \RRm, \RRm)$ defined by $$r = (\bend \leq) \circ (\mathrm{id}_\RRm \times (+)) = \{(\alpha,\alpha_1,\alpha_2) \in \RRm^3 \suchthat \alpha \leq \alpha_1 + \alpha_2\}.$$ Applying Lemma \ref{examples.F.1}, we find that the following are equivalent:
\begin{align*}
& \[ (\forall (x_1 \feq x_{1*}) \fin \X \ftimes \X^*)\, (\forall (x_2 \feq x_{2*}) \fin \X \ftimes \X^*)\,(\forall (x_3 \feq x_{3*}) \fin \X \ftimes \X^*) \\ & \hspace{35ex}  `(\bend\leq) (F(x_1,x_{2*}), `(+)( F(x_2,x_{3^*}), F(x_3, x_{1*})))\] = \top
\\ & \Longleftrightarrow \quad 
\[ (\forall (x_1 \feq x_{1*}) \fin \X \ftimes \X^*)\, (\forall (x_2 \feq x_{2*}) \fin \X \ftimes \X^*)\,(\forall (x_3 \feq x_{3*}) \fin \X \ftimes \X^*) \\ & \hspace{43ex}  `r(F(x_1,x_{2*}),  F(x_2,x_{3^*}), F(x_3, x_{1*}))\] = \top
\\ & \Longleftrightarrow \quad
\text{for all $\alpha, \alpha_1, \alpha_2 \in \RRm$ such that $\alpha > \alpha_1 + \alpha_2$, }\;\Tr_\X (R_{\alpha_2} \circ R_{\alpha_1} \circ R_{\alpha}) = \bot
\\ & \Longleftrightarrow \quad
\text{for all $\alpha, \alpha_1, \alpha_2 \in \RRm$ such that $\alpha > \alpha_1 + \alpha_2$, }\;\Tr_\X ((R_{\alpha}^\dagger)^\dagger \circ (R_{\alpha_2} \circ R_{\alpha_1})) = \bot
\\ & \Longleftrightarrow \quad
\text{for all $\alpha, \alpha_1, \alpha_2 \in \RRm$ such that $\alpha > \alpha_1 + \alpha_2$, }\; R_{\alpha_2} \circ R_{\alpha_1}   \perp R_{\alpha}^\dagger
\\ & \Longleftrightarrow \quad
\text{for all $\alpha_1, \alpha_2 \in \RRm$, }\; R_{\alpha_2} \circ R_{\alpha_1} \leq \bigvee_{\alpha \leq \alpha_1 + \alpha_2} R_\alpha^\dagger.
\end{align*}
The first equivalence follows by the graphical calculus via Lemma \ref{computation.F.2}.
\end{proof}

Let $A$ be an ordinary set. The relations $`(=_A)$ and $E_{`A}$ are both, in some sense, equality relations on $`A$, but they are distinct because they have different arities. The relation $`(=_A)$ has arity $(`A,`A)$, and the relation $E_{`A}$ has arity $(`A, (`A)^*)$. Nevertheless, they are very closely related because $`A$ and $(`A)^*$ are naturally isomorphic via the ``conjugation'' function $C_A\: `A \to (`A)^*$, defined by $C_A(\CC_{a}, \CC_{a}^*) = L(\CC_{a}, \CC_{a}^*)$ for all $a \in A$, with the other components vanishing. The function $C_A$ is intuitively the identity on $`A$, and this can be arranged to hold formally \cite{Kornell}*{App.~D}.

\begin{theorem}\label{examples.F.5}
Let $\X$ be a quantum set. Let $T$ be the binary relation on $`\RRm$ defined by $T = `(\leq)$. Then, there is a one-to-one correspondence between quantum pseudometrics $(V_\beta\suchthat \beta \in [0,\infty))$ on $\ell^\infty(\X)$ in the sense of \cite{KuperbergWeaver}*{Def.~2.3} and functions $F\: \X \times \X^* \to `[0, \infty]$ such that
\begin{enumerate}
\item $\[(\forall (x_1 \feq x_{1*}) \fin \X \ftimes \X^*)\, E_{`\RRm}(F(x_1,x_{1*}),`0_*)\] = \top$;
\item $\[(\forall (x_1 \feq x_{1*}) \fin \X \ftimes \X^*)\, (\forall (x_2 \feq x_{2*}) \fin \X \ftimes \X^*) $\\ \phantom | \hfill $ E_{`\RRm}( F(x_1, x_{2*}), C_{\RRm}(F(x_{2}, x_{1*})))\] = \top$;
\item $\[(\forall (x_1 \feq x_{1*}) \fin \X \ftimes \X^*)\, (\forall (x_2 \feq x_{2*}) \fin \X \ftimes \X^*)\,(\forall (x_3 \feq x_{3*}) \fin \X \ftimes \X^*)$\\ \phantom | \hfill $\bend T( F(x_1,x_{2*}), `(+)_*( F_*(x_{1*},x_{3}), F_*(x_{3*}, x_2)))\] = \top$.
\end{enumerate}
This correspondence is given by $\sum_{\alpha \in [0,\beta]} R_\alpha(X_1, X_2) = \incdag_{X_2} \cdot V_\beta \cdot \incnag_{X_1}$, for $\beta \in [0, \infty)$ and $X_1, X_2 \in \At(\X)$, where $\bend R_\alpha = `\alpha^\dagger \circ F$ for each $\alpha \in [0,\infty]$.
\end{theorem}

\begin{proof}
Let $A = [0,\infty]$, and let $a_0 = 0$. The equation $\bend R_\alpha = `\alpha^\dagger \circ F$, for $\alpha \in A$, defines a one-to-one correspondence between functions $F\: \X \times \X^* \to `A$ and indexed families $(R_\alpha \suchthat \alpha \in A)$ of pairwise disjoint binary relations on $\X$ whose join is $\top_\X^\X$, the maximum binary relation on $\X$. In the context of this correspondence, condition (1) is equivalent to the equation in Proposition \ref{examples.F.2}, because $$E_{`A} \circ (F \times `0_*) =  E_{`A} \circ (F \times (C_A \circ `0)) = E_{`A} \circ (I_{`A} \times C_A) \circ (F \times `0) = `(=_A) \circ (F \times `0).$$ Similarly, condition (2) is equivalent to the equation in Proposition \ref{examples.F.3}, because $$E_{`A} \circ (F \times (C_A \circ F)) = E_{`A} \circ (I_{`A} \times C_A)  \circ (F \times F) = `(=_A) \circ (F \times F).$$ Thus, $F$ satisfies conditions (1) and (2) if and only if $R_0 \geq I_{`A}$ and $R_\alpha^\dagger = R_\alpha$ for all $\alpha \in [0,\infty]$. 

Assuming conditions (1) and (2), Lemma \ref{examples.A.2} implies that for all $\alpha \in [0,\infty]$ we have $\bend R_{\alpha*} \circ B_{\X,\X^*} = \bend R_\alpha$, i.e., $`\alpha^\dagger_* \circ F_* \circ B_{\X,\X^*} = `\alpha^\dagger \circ F$, where $B_{\X,\X^*}$ is the braiding of $\X$ and $\X^*$. Thus, $`\alpha^\dagger_* \circ F_* \circ B_{\X,\X^*} = `\alpha^\dagger_* \circ C_A \circ F$ for all $\alpha \in A$; we conclude that $F_* \circ B_{\X,\X^*} = C_A \circ F$. We now calculate that 
\begin{align*}
`(\bend \leq) &\circ (F \times (`(+) \circ (F \times F)))
=
`(\bend \leq) \circ (F \times (`(+) \circ B_{`\RRm,`\RRm}\circ (F \times F)))
\\ & =
`(\bend \leq) \circ (F \times (`(+) \circ (F \times F) \circ B_{\X \times \X^*, \X \times \X^*}))
\\ & =
\bend T \circ (I_{`A} \times C_A) \circ (F \times (`(+) \circ (F \times F) \circ B_{\X \times \X^*, \X \times \X^*}))
\\ & = 
\bend T \circ (F \times (C_A \circ `(+) \circ (F \times F) \circ B_{\X \times \X^*, \X \times \X^*}))
\\ & = 
\bend T \circ (F \times (`(+)_* \circ ( C_A \times C_A )  \circ (F \times F) \circ B_{\X \times \X^*, \X \times \X^*}))
\\ & = 
\bend T \circ (F \times (`(+)_* \circ (F_* \times F_*) \circ (B_{\X,\X^*} \times B_{\X,\X^*}) \circ B_{\X \times \X^*, \X \times \X^*}))
\\ & = \bend T \circ (F \times (`(+)_* \circ (F_* \times F_*))) \circ (I_{\X \times \X^*} \times ((B_{\X,\X^*} \times B_{\X,\X^*}) \circ B_{\X \times \X^*, \X \times \X^*})).
\end{align*}
We conclude that, modulo conditions (1) and (2), condition (3) is equivalent to the equation in Proposition \ref{examples.F.4}.

Therefore, we have a one-to-one correspondence between functions $F\: \X \times \X^* \to `[0,\infty]$ satisfying conditions (1), (2) and (3), and indexed families $(R_\alpha \suchthat \alpha \in [0,\infty])$ of pairwise orthogonal binary relations on $\X$, whose join is $\top_{\X}^{\X}$ and which satisfy $R_0 \geq I_{\X}$, $R_\alpha^\dagger= R_\alpha$ for all $\alpha \in [0,\infty]$ and $R_{\alpha_2} \circ R_{\alpha_1} \leq \bigvee_{\alpha \leq \alpha_1 + \alpha_2} R_\alpha$ for all $\alpha_1, \alpha_2 \in [0,\infty]$.

The equation $ \bigvee_{\alpha \leq \beta} R_\alpha = S_\beta$ defines a one-to-one correspondence between indexed families $(R_\alpha \suchthat \alpha \in [0,\infty])$ such that $\bigvee_{\alpha} R_\alpha = \top_\X^\X$ and $R_{\alpha_1} \perp R_{\alpha_2}$ for all distinct $\alpha_1, \alpha_2 \in [0,\infty]$ and indexed families $(S_\beta\suchthat \beta \in [0,\infty))$ such that $S_{\beta_0} = \bigwedge_{\beta > \beta_0} S_\beta$ for all $\beta_0 \in [0,\infty) $. The existence of this correspondence becomes readily apparent by identifying the binary relations $R_\alpha$ and the binary relations $S_\beta$ with projections in the hereditarily atomic von Neumann algebra $\ell^\infty(\X \times \X^*)$. In the context of this correspondence, $R_0 \geq I_\X$ is equivalent to $S_0 \geq I_\X$, $R^\dagger_\alpha = R_\alpha$ for all $\alpha \in [0,\infty]$ is equivalent to $S_\beta^\dagger = S_\beta$ for all $\beta \in [0,\infty)$, and $R_{\alpha_2} \circ R_{\alpha_1} \leq \bigvee_{\alpha \leq \alpha_1 + \alpha_2} R_\alpha$ for all $\alpha_1, \alpha_2 \in [0,\infty]$ is equivalent to $S_{\beta_1} \circ S_{\beta_2} \leq S_{\beta_1+\beta_2}$ for all $\beta_1, \beta_2 \in [0,\infty)$. Families $(S_\beta \suchthat \beta \in [0,\infty))$ with these four properties are in obvious one-to-once correspondence with quantum pseudometrics $(V_\beta \suchthat \beta \in [0,\infty))$ on $\ell^\infty(\X)$ via the expression $S_\beta(X_1,X_2) = \incdag_{X_2} \cdot V_\beta \cdot \incnag_{X_1}$, as in Appendix \ref{appendix.E.1}.

In summary, the equation $\bend R_\alpha = `\alpha^\dagger \circ F$, for $\alpha \in [0,\infty]$, defines a one-to-one correspondence between functions $F\: \X \times \X^* \to `[0,\infty]$ satisfying conditions (1), (2) and (3), and families  $(R_\alpha \suchthat \alpha \in [0,\infty])$ of binary relations on $\X$ satisfying
\begin{enumerate}[\indent (a)]
\item $R_{\alpha_1} \perp R_{\alpha_2}$ for distinct $\alpha_1, \alpha_2 \in [0,\infty]$,
\item $\bigvee_{\alpha} R_\alpha = \top_\X^\X$,
\item $R_0 \geq I_\X$,
\item $R_\alpha^\dagger = R_\alpha$ for $\alpha \in [0,\infty]$,
\item $R_{\alpha_2} \circ R_{\alpha_1} \leq \bigvee_{\alpha \leq \alpha_1 + \alpha_2} R_\alpha$ for $\alpha_1, \alpha_2 \in [0,\infty]$.
\end{enumerate}
Furthermore, the equation $\sum_{\alpha \leq \beta} R_\alpha (X_1,X_2) = \incdag_{X_2} \cdot V_\beta \cdot \incnag_{X_1}$, for $\beta \in[0,\infty)$, defines a one-to-one correspondence between such families $(R_\alpha \suchthat \alpha \in [0,\infty])$ and quantum pseudometrics $(V_\beta \suchthat \beta \in [0,\infty))$ on $\ell^\infty(\X)$. Thus, the theorem is proved.
\end{proof}

\begin{corollary}\label{examples.F.6}
The one-to-one correspondence of Theorem \ref{examples.F.5} restricts to a one-to-one correspondence between quantum metrics $(V_\beta\suchthat \beta \in [0,\infty))$ on $\ell^\infty(\X)$ in the sense of \cite{KuperbergWeaver}*{Def.~2.3}, and functions $F\: \X \times \X^* \to `[0, \infty]$ satisfying conditions (1), (2) and (3) in the statement of that theorem, together with
\begin{enumerate}\setcounter{enumi}{3}
\item $\[( \forall x_1 \fin \X)\, (\forall x_{2*} \fin \X^*)\, (E_{`[0,\infty]}(F(x_1, x_{2*}),0_*) \IMPLIES E_\X(x_1, x_{2*}))\] = \top$.
\end{enumerate}
\end{corollary}

\begin{proof}
We have that $\[x_1 \in \X, x_{2*} \in \X^* \suchthat E_{`[0,\infty]}(F(x_1, x_{2*}),0_*)\] = `0^\dagger \circ F = \bend R_0$ by the graphical calculus, so condition (4) is equivalent to $\bend R_0 \leq E_\X$ by Proposition \ref{computation.C.2}. We now reason that $\bend R_0 \leq E_\X$ is equivalent to $R_0 \leq I_\X$, which is equivalent to $V_0 \leq \ell^\infty(\X)'$ via the correspondence between binary relations and quantum relations given in Appendix \ref{appendix.E.1} because $\ell^\infty(\X)'$ is the identity quantum relation on $\ell^\infty(\X)$.
\end{proof}

\subsection{Quantum families of graph isomorphisms}\label{examples.G}
The graph coloring game is played by two cooperating players, traditionally named Alice and Bob, against a referee. The parameters of the game are a finite simple graph $G$ and a finite set of colors $C$, and the rules of the game are such that, classically, Alice and Bob have a winning strategy if and only if $G$ can be properly colored by $C$. Alice and Bob are forbidden from communicating with each other during the course of the game, and the existence of a proper graph coloring ensures that they are able to successfully coordinate their responses to the referee. However, if Alice and Bob possess entangled quantum systems, then they may have a winning strategy even if no proper graph coloring exists. In this case, one says that a quantum graph coloring exists. Subsection I.A of \cite{Kornell} contains a longer discussion of the graph coloring game and its connection to quantum sets.

The graph coloring game can be generalized to the graph homomorphism game, whose parameters are two finite simple graphs $G$ and $H$. Classically, Alice and Bob have a winning strategy if and only if there exists a graph homomorphism from $G$ to $H$. We show that Alice and Bob have a winning strategy utilizing finite entangled quantum systems if and only if there exists a nonempty quantum family of graph homomorphisms from $G$ to $H$, suitably expressed in the quantum predicate logic of this paper. The family is quantum in the sense that it is indexed by a quantum set. This equivalence between the existence a winning strategy utilizing finite entangled quantum systems and the existence of a nonempty quantum family of graph homomorphisms is essentially already present in \cite{Kornell}*{Prop.~1.2}, and even in \cite{CameronMontanaroNewmanSeveriniWinter}*{sec.~II}. Thus, the novelty of Proposition \ref{examples.G.2} consists primarily in the axiomatization of these quantum families of graph homomorphisms in quantum predicate logic.

Similarly, we axiomatize quantum families of graph \emph{isomorphisms}. Naturally, Alice and Bob have a winning strategy for the graph isomorphism game \cite{AtseriasMancinskaRobersonSamalSeveriniVarvitsiotis}*{sec.~1.1}, possibly utilizing finite entangled quantum systems, if and only if there exists a nonempty quantum family of graph isomorphisms. The example of the quantum families of graph isomorphisms is particularly significant due to extraordinary progress in understanding the quantum isomorphism relation \cite{MustoReutterVerdon}\cite{LupiniMancinskaRoberson}\cite{MustoReutterVerdon2}\cite{MancinskaRoberson2}.

\begin{lemma}\label{examples.G.0}
Let $\X$ be a quantum set, and let $A$ be an ordinary set. Then, there is a one-to-one correspondence between functions $F\: \X \to `A$ and families $(P_a\in \Rel(\X) \suchthat a \in A)$ such that $P_{a_1} \perp P_{a_2}$ for all $a_1 \neq a_2$ and $\bigvee_{a \in A} P_a = \top_{`A}$. This correspondence is given by $P_a = `a^\dagger \circ F$.
\end{lemma}

\begin{proof} The equation $P_a = `a^\dagger \circ F$, for $a \in A$, defines a one-to-one correspondence between binary relations $F$ from $\X$ to $`A$ and indexed families $(P_a\in \Rel(\X)\suchthat a \in A)$, essentially by the definition of a binary relation between quantum sets. It only remains to show that the inequalities $F \circ F^\dagger \leq I_{`A}$ and $F^\dagger \circ F \geq I_\X$ are together equivalent to the stated conditions on $(P_a\suchthat a \in A)$ under this correspondence. Reasoning in terms of the trace on binary relations \cite{KornellLindenhoviusMislove2}*{App.~C},
\begin{align*}
F& \circ F^\dagger \leq I_{`A} 
\EV
F \circ F^\dagger \perp \NOT I_{`A} 
\EV
\Tr(\NOT I_{`A}^\dagger \circ F \circ F^\dagger) = \bot
\\ & \EV
\Tr\left( \left(\bigvee_{a_1 \neq a_2} `a_1 \circ `a_2^\dagger\right) \circ F \circ F^\dagger \right) = \bot
\EV
\bigvee_{a_1 \neq a_2} \Tr(`a_1 \circ `a_2^\dagger \circ F \circ F^\dagger) = \bot
\\ & \EV
\text{for all distinct $a_1 ,a_2 \in A$,} \quad \Tr((`a_2^\dagger \circ F) \circ (a_1^\dagger \circ F)^\dagger) = \bot
\\ & \EV
\text{for all distinct $a_1 ,a_2 \in A$,} \quad P_{a_1} \perp P_{a_2}.
\end{align*}
Under the assumption $F \circ F^\dagger \leq I_{`A}$, the inequality $F^\dagger \circ F \geq I_\X$ is equivalent to the equality $\top_{`A} \circ F = \top_{\X}$ by Lemmas 6.4 and B.4 of \cite{Kornell}:
\begin{align*}&
F \circ F^\dagger \leq I_{`A}
\EV
\top_{`A} \circ F = \top_{\X}
\EV
\left(\bigvee_{a \in A} `a^\dagger\right) \circ F  = \top_\X
\EV
\bigvee_{a \in A} P_a = \top_\X.
\end{align*}
\end{proof}

\begin{lemma}\label{examples.G.1}
Let $A$ and $B$ be ordinary sets, and let $\X$ be a quantum set. We have a one-to-one correspondence between families $(P_{ab} \in \Rel(\X) \suchthat a \in A,\, b \in B)$ such that
\begin{enumerate}
\item $P_{ab_1} \perp P_{a b_2}$ for all $a \in A$ and distinct $b_1, b_2 \in B$ and
\item $\displaystyle \bigvee_{b \in B} P_{ab} = \top_\X$ for all $a \in A$,
\end{enumerate}
and functions $F\: \X \times `A \to `B$. It is given by the equation $P_{ab} = `b^\dagger \circ F \circ (I_\X \times `a)$, for $a \in A$ and $b \in B$.
\end{lemma}

\begin{proof}
The quantum set $\X \times `A$ is a coproduct of copies of $\X$, with one copy for each element of $A$. The inclusions of this coproduct are exactly the functions $I_\X \times `a$ for $a \in A$. Thus, we have a one-to-one correspondence between families of functions $(F_a \: \X \to `B\suchthat a \in A)$ and functions $F\: \X \times `A \to `B$, which is given by $F_a = F \circ (I_\X \times `a)$ for $a \in A$. The statement of the lemma then follows by Lemma \ref{examples.G.0}.
\end{proof}

\begin{proposition}\label{examples.G.1.5}
Let $A$ and $B$ be ordinary sets, let $\X$ be a quantum set, and let $F$ be a function $\X \times `A \to `B$. For each $a \in A$ and each $b \in B$, let $P_{ab} = `b^\dagger \circ F \circ (I_\X \times `a)$. Then, $\bigvee_{a \in A} P_{ab} = \top_\X$ for all $b \in B$ if and only if
$$
\[(\forall x \in \X)\,(\forall b_* \in `B_*)\,(\exists a \in `A)\, E_{`B}(F(x,a),b_*)\] = \top.
$$
\end{proposition}

\begin{proof}
First, we compute that for all $a \in A$ and all $b \in B$,
\begin{align*}
P_{ab}
\; &= \quad
\begin{aligned}
\begin{tikzpicture}
\begin{pgfonlayer}{nodelayer}
    \node [style=box] (F) at (0, 0) {$\phantom{F}\;F\;\phantom{F}$};
    \node [style=none] (xanch) at (-0.4,0) {};
    \node [style=none] (aanch) at (0.4,0) {};
    \node (x) at (-0.4,-1.2) {$\scriptstyle x$};
    \node [style=box] (a) at (0.4, -0.8) {$`a_{\phantom{*}}$};
    \node [style=box] (b) at (0, 0.8) {$`b^\dagger$};
\end{pgfonlayer}
\begin{pgfonlayer}{edgelayer}
    \draw [arrow, markat =0.4] (x) to (xanch);
    \draw [arrow, markat=0.4] (a) to (aanch);
    \draw [arrow, markat =0.8] (F.center) to (b);
\end{pgfonlayer}
\end{tikzpicture}
\end{aligned}
\quad = \quad
\begin{aligned}
\begin{tikzpicture}
\begin{pgfonlayer}{nodelayer}
    \node [style=box] (F) at (0, 0) {$\phantom{F}\;F\;\phantom{F}$};
    \node [style=none] (xanch) at (-0.4,0) {};
    \node [style=none] (aanch) at (0.4,0) {};
    \node [style=none] (banch) at (1.2,0) {};
    \node (x) at (-0.4,-1.2) {$\scriptstyle x$};
    \node [style=box] (a) at (0.4, -0.8) {$`a_{\phantom{*}}$};
    \node [style=box] (b) at (1.2, -0.8) {$`b_*$};
\end{pgfonlayer}
\begin{pgfonlayer}{edgelayer}
    \draw [arrow, markat =0.4] (x) to (xanch);
    \draw [arrow, markat=0.4] (a) to (aanch);
    \draw [arrow, markat =0.8] (banch.center) to (b);
    \draw [arrow, bend left = 90, looseness=2.3] (F.center) to (banch.center);
\end{pgfonlayer}
\end{tikzpicture}
\end{aligned}
\quad = \quad
\[x \in \X\suchthat  E_{`B}(F(x,`a),`b_*)\].
\end{align*}
Now we apply Lemma \ref{appendix.C.1} twice to reason that
\begin{align*}
(\forall b \in B)\; \bigvee_{a \in A} P_{ab} = \top_\X
& \EV
\bigwedge_{b \in B} \bigvee_{a \in A} P_{ab} = \top_\X
\\ & \EV
\bigwedge_{b \in B} \bigvee_{a \in A} \[x \in \X\suchthat  E_{`B}(F(x,`a),`b_*)\] = \top_\X
\\ & \EV
\bigwedge_{b \in B} \[x \in \X\suchthat (\exists a \fin A)\, E_{`B}(F(x,a),`b_*)\] = \top_\X
\\ & \EV
\[x \in \X\suchthat (\forall b_* \fin B^*)\, (\exists a \fin A)\, E_{`B}(F(x,a),b_*)\] = \top_\X
\\ & \EV
\[(\forall x \fin \X)\, (\forall b_* \fin B^*)\, (\exists a \fin A)\, E_{`B}(F(x,a),b_*)\] = \top,
\end{align*}
where the last equivalence follows from Lemma \ref{computation.C.1}.
\end{proof}

\begin{proposition}\label{examples.G.2}
Let $A$ and $B$ be sets equipped with binary relations $r$ and $s$, respectively, and let $R = `r$ and $S = `s$.
Let $\X$ be a quantum set, and let $F\: \X \times `A \to `B$ be a function. For each $a \in A$ and each $b \in B$, let $P_{ab}$ be the relation of arity $(\X)$ defined by $P_{ab} = `b^\dagger \circ F \circ (I_\X \times `a)$. Then, $P_{a_1b_1} \perp P_{a_2b_2}$ for all $(a_1,a_2) \in r$ and $(b_1, b_2) \in \NOT s$ if and only if
\begin{align*}
&\[
(\forall (x \feq x_*) \fin \X \ftimes \X^*)\, (\forall (a_1 \feq a_{1*}) \in `A \ftimes `A^*)\,(\forall (a_2 \feq a_{2*}) \fin `A \ftimes `A^*)\, \\ & \ghost \hspace{40ex}
(\bend R (a_{1},a_{2*}) \IMPLIES \bend S_*(F_*(x_*,a_{1*}), F(x,a_{2})))
\] = \top.
\end{align*}
\end{proposition}

\begin{proof}
The equation in the statement of the proposition is equivalent to
\begin{align*}
&\[
(\exists (x \feq x_*) \fin \X \ftimes \X^*)\, (\exists (a_1 \feq a_{1*}) \fin `A \ftimes `A^*)\,(\exists (a_2 \feq a_{2*}) \fin `A \ftimes `A^*)\, \\ & \ghost \hspace{40ex}
(\bend R (a_{1},a_{2*}) \AND \NOT \bend S_* (F_*(x_*,a_{1*}), F(x,a_{2}))
\] = \bot,
\end{align*}
which may be rendered graphically as
\begin{align*}
\begin{aligned}
\begin{tikzpicture}
\begin{pgfonlayer}{nodelayer}
    \node [style=box] (F) at (-0.4, -1) {$F_*$};
    \node [style=box] (Fstar) at (0.6,-1) {$F_{\phantom{*}}$};
    \node [style=box] (notS) at (0.1,0) {\phantom | \,$\NOT \bend S_{*}$\,\phantom |};
    \node [style=none] (anchF) at (-0.4,-0.2) {};
    \node [style=none] (anchFstar) at (0.6,-0.2) {};
    \node [style=box] (R) at (-1.9,0) {\phantom |\,$\phantom{\NOT}\bend R_{\phantom{*}}$\, \phantom |};
    \node [style=none] (s1) at (-2.4, -0.2){};
    \node [style=none] (s2*) at (-1.4, -0.2) {};
    \node [style=none] (x*) at (-0.55,-1.2) {};
    \node [style=none] (s1*) at (-0.25,-1.2) {};
    \node [style=none] (x) at (0.45, -1.2) {};
    \node [style=none] (s2) at (0.75, -1.2) {};
    \node [style=none] (s1bend) at (-2.4, -1.2) {};
    \node [style=none] (s2bend) at (-1.4, -1.2) {};
\end{pgfonlayer}
\begin{pgfonlayer}{edgelayer}
    \draw [arrow,markat=0.7] (anchF.north) to (F);
    \draw [arrow] (Fstar) to (anchFstar.north);
    \draw [arrow] (s1bend.center) to (s1);
    \draw [arrow] (s2*) to (s2bend.center);
    \draw [arrow, bend right = 90, looseness = 2.5] (x*.center) to (x.center);
    \draw [arrow, bend right = 90, looseness = 1.5] (s2bend.center) to (s2.center);
    \draw [arrow, bend left = 90, looseness = 1.5] (s1*.center) to (s1bend.center);
\end{pgfonlayer}
\end{tikzpicture}
\end{aligned}
\quad = \; \bot.
\end{align*}
Because $A$ and $B$ are ordinary sets, we have $\bend R = \{`a_1^\dagger \times `a_{2*}^\dagger \suchthat (a_1,a_2) \in r\}$ and $\NOT \bend S_* = \{`b_{1*}^\dagger \times `b_2 \suchthat (b_1,b_2) \in \NOT s \}$. Thus, the equation in the statement of the proposition is equivalent to
\begin{align*}
\bigvee_{\begin{smallmatrix}(a_1,a_2) \in r\\(b_1, b_2) \in \NOT s\end{smallmatrix}}
\begin{aligned}
\begin{tikzpicture}
\begin{pgfonlayer}{nodelayer}
    \node [style=box] (F) at (-0.4, -1) {$F_*$};
    \node [style=box] (Fstar) at (0.6,-1) {$F_{\phantom{*}}$};
    \node [style=none] (anchF) at (-0.4,-0.2) {};
    \node [style=none] (anchFstar) at (0.6,-0.2) {};
    \node [style=none] (s1) at (-2.4, -0.2){};
    \node [style=none] (s2*) at (-1.4, -0.2) {};
    \node [style=none] (x*) at (-0.55,-1.2) {};
    \node [style=none] (s1*) at (-0.25,-1.2) {};
    \node [style=none] (x) at (0.45, -1.2) {};
    \node [style=none] (s2) at (0.75, -1.2) {};
    \node [style=none] (s1bend) at (-2.4, -1.2) {};
    \node [style=none] (s2bend) at (-1.4, -1.2) {};
    \node [style=box] (s1dag) at (-2.4,0) {$`a_{1\phantom{*}\!}^\dagger$};
    \node [style=box] (s2*dag) at (-1.4,0) {$`a_{2*\!}^\dagger$};
    \node [style=box] (t1dag) at (-0.4,0) {$`b_{1*\!}^\dagger$};
    \node [style=box] (t2*dag) at (0.6,0) {$`b_{2\phantom{*}\!}^\dagger$};
\end{pgfonlayer}
\begin{pgfonlayer}{edgelayer}
    \draw [arrow,markat=0.7] (anchF.north) to (F);
    \draw [arrow] (Fstar) to (anchFstar.north);
    \draw [arrow] (s1bend.center) to (s1);
    \draw [arrow] (s2*) to (s2bend.center);
    \draw [arrow, bend right = 90, looseness = 2.5] (x*.center) to (x.center);
    \draw [arrow, bend right = 90, looseness = 1.5] (s2bend.center) to (s2.center);
    \draw [arrow, bend left = 90, looseness = 1.5] (s1*.center) to (s1bend.center);
\end{pgfonlayer}
\end{tikzpicture}
\end{aligned}
\quad = \; \bot.
\end{align*}
Reasoning graphically, we conclude that this equation expresses the condition that $`b_1^\dagger \circ F \circ (I_\X \times `a_1)$ is orthogonal to $`b_2^\dagger \circ F \circ (I_\X \times `a_2)$ whenever $(a_1,a_2) \in r$ and $(b_1,b_2) \in \NOT s$.
\end{proof}

\begin{proposition}\label{examples.G.2.5}
Let $A$ and $B$ be ordinary sets. Let $H$ be a nonzero finite-dimensional Hilbert space, and let $\X$ be the quantum set defined by $\At(\X) = \{H\}$. Then, there is a one-to-one correspondence between families of projections $(p_{ab} \in L(H) \suchthat a \in A,\, b \in B)$ such that $\sum_{b \in B} p_{ab} = 1_H$ for all $a \in A$ and functions $F\: \X \times `A \to `B$. This correspondence is obtained by combining the one-to-one correspondence of Lemma \ref{examples.G.1} with the canonical one-to-one correspondence between projection operators on $H$ and relations of arity $(\X)$. It is given by the equation $F(H \tensor \CC_a, \CC_b)= L(H \tensor \CC_a, \CC_b)\cdot (p_{ab} \tensor 1)$, for $a \in A$ and $b \in B$.
\end{proposition}

The canonical one-to-one correspondence between projection operators $p$ on $H$ and relations $P$ of arity $(\X)$ is defined by $P(H, \CC) = \{v \in L(H, \CC) \suchthat vp = v\} = L(H, \CC) \cdot p$ . This is an isomorphism of ortholattices \cite{Kornell}*{Appendix B}.

\begin{proof}[Proof of Proposition \ref{examples.G.2.5}]
For all $a \in A$ and $b \in B$ we calculate that
\begin{align*}
F(H \tensor \CC_a, \CC_b)
& =
L(\CC, \CC_b) \cdot L(\CC_b, \CC) \cdot F(H \tensor \CC_a, \CC_b) \cdot (1_H \tensor L(\CC, \CC_a)) \cdot (1_H \tensor L(\CC_a, \CC))
\\ &=
L(\CC, \CC_b) \cdot `b^\dagger(\CC_b, \CC) \cdot F(H \tensor \CC_a, \CC_b) \cdot (1_H \tensor `a(\CC, \CC_a)) \cdot (1_H \tensor L(\CC_a, \CC))
\\ &=
L(\CC, \CC_b) \cdot (`b^\dag \circ F \circ (I_\X \times `a))(H \tensor \CC, \CC) \cdot (1_H \tensor L(\CC_a, \CC))
\\ &=
L(\CC, \CC_b) \cdot P_{ab}(H \tensor \CC, \CC) \cdot (1_H \tensor L(\CC_a, \CC))
\\ &=
L(\CC, \CC_b) \cdot L(H \tensor \CC, \CC) \cdot (p_{ab} \tensor 1) \cdot (1_H \tensor L(\CC_a, \CC))
\\ &=
L(H \tensor \CC_a, \CC_b) \cdot (p_{ab} \tensor 1).
\end{align*}
The unitor $H \tensor \CC \to H$ has been suppressed.
\end{proof}

\begin{theorem}\label{examples.G.3}
Let $A$ and $B$ be sets equipped with binary relations $r$ and $s$, respectively, and let $R = `r$ and $S = `s$. Let $H$ be a nonzero finite-dimensional Hilbert space, and let $\X$ be the quantum set defined by $\At(\X) = \{H\}$. Then, the one-to-one correspondence of Proposition \ref{examples.G.2.5} restricts to a one-to-one correspondence between families of projections $(p_{ab} \in L(H) \suchthat a \in A,\, b\in B)$ such that
\begin{enumerate}
\item $\sum_{b \in B} p_{ab} = 1_H$ for all $a \in A$ and
\item $p_{a_1b_1} \perp p_{a_2 b_2}$ for all $(a_1, a_2) \in r$ and $(b_1, b_2) \in \NOT s$, 
\end{enumerate}
and functions $F\: \X \times `A \to `B$ such that
\begin{equation*}\label{eq.hom}
\begin{split}
&\[
(\forall (x \feq x_*) \fin \X \ftimes \X^*)\, (\forall (a_1 \feq a_{1*}) \fin `A \ftimes `A^*)\,(\forall (a_2 \feq a_{2*}) \fin `A \ftimes `A^*)\, \\ & \ghost \hspace{40ex}
(\bend R (a_{1},a_{2*}) \IMPLIES \bend S_* (F_*(x_*,a_{1*}), F(x,a_{2})))
\] = \top.
\end{split}
\end{equation*}
\end{theorem}

\begin{proof}
This follows immediately from Proposition \ref{examples.G.2} and Proposition \ref{examples.G.2.5}.
\end{proof}

\begin{definition}\label{examples.G.4}
Let $(A, \sim_A)$ and $(B, \sim_B)$ be finite simple graphs, and let $H$ be a nonzero finite-dimensional Hilbert space. We say that a family of projections $(p_{ab} \in L(H) \suchthat a \in A, b \in B)$ \emph{witnesses} $A \xrightarrow{q} B$ if
\begin{enumerate}
\item $\sum_{b \in B} p_{ab} = 1_H$ for all $a \in A$;
\item $p_{a_1b_1} \cdot p_{a_2b_2} = 0$ for all $a_1,a_2 \in A$ and $b_1,b_2 \in B$ satisfying either $(a_1 =_A a_2) \AND (b_1 \neq_B b_2)$ or $(a_1 \sim_A a_2) \AND (b_1 \not \sim_B b_2)$.
\end{enumerate}
\end{definition}

Alice and Bob have a winning strategy for the $(A,B)$-homomorphism game, possibly using finite entangled quantum systems, if and only if there exists a family of projections on some nonzero finite-dimensional Hilbert space that witnesses $A \xrightarrow{q} B$ \cite{MancinskaRoberson}*{Cor.~2.2}. 

\begin{corollary}\label{examples.G.5}
Let $(A, \sim_A)$ and $(B, \sim_B)$ be finite simple graphs, and let $R = `(\sim_A)$ and let $S = `(\sim_B)$. Let $H$ be a nonzero finite-dimensional Hilbert space, and let $\X$ be the quantum set defined by $\At(\X) = \{H\}$. Then, the one-to-one correspondence of Proposition \ref{examples.G.2.5} restricts to a one-to-one correspondence between families of projections  $(p_{ab} \in L(H) \suchthat a \in A, \, b \in B)$ witnessing $A \xrightarrow{q} B$ and functions $F\: \X \times `A \to `B$ satisfying 
\begin{equation*}
\begin{split}
&\[
(\forall (x \feq x_*) \fin \X \ftimes \X^*)\, (\forall (a_1 \feq a_{1*}) \fin `A \ftimes `A^*)\,(\forall (a_2 \feq a_{2*}) \fin `A \ftimes `A^*)\, \\ & \ghost \hspace{40ex}
(\bend R (a_{1},a_{2*}) \IMPLIES \bend S_* (F_*(x_*,a_{1*}), F(x,a_{2})))
\] = \top.
\end{split}
\end{equation*}
\end{corollary}

\begin{proof}
This corollary follows immediately from Theorem \ref{examples.G.3}.
\end{proof}

\begin{definition}\label{examples.G.6}
Let $A$ and $B$ be finite sets, and let $H$ be a nonzero finite-dimensional Hilbert space. We say that a family of projections $(p_{ab} \in L(H) \suchthat a \in A, b \in B)$ is a \emph{magic unitary} if
\begin{enumerate}
\item $\sum_{b \in B} p_{ab} = 1_H$ for all $a \in A$;
\item $\sum_{a \in A} p_{ab} = 1_H$ for all $b \in B$.
\end{enumerate}
\end{definition}

A magic unitary in this sense is essentially a quantum family of bijections between $A$ and $B$ indexed by the quantum set whose only atom is $H$. In particular, if $A = B$, then it is a quantum family of bijections. Indeed, the universal C*-algebra generated by projections $p_{ab}$, for  $a, b \in A$, satisfying conditions (1) and (2) is a well established quantum generalization of the permutation group $\mathrm{Aut}(A)$, introduced in \cite{Wang}*{sec.~3}.

\begin{corollary}\label{examples.G.7}
Let $A$ and $B$ be sets. Let $H$ be a nonzero finite-dimensional Hilbert space, and let $\X$ be the quantum set defined by $\At(\X) = \{H\}$. Then, the one-to-one correspondence of Proposition \ref{examples.G.2.5} restricts to a one-to-one correspondence between families of projections $(p_{ab} \in L(H) \suchthat a \in A, \, b \in B)$ that are magic unitaries and functions $F\: \X \times `A \to `B$ such that
\begin{enumerate}
\item $\[(\forall x \in \X)\,(\forall b_* \in `B_*)\,(\exists a \in `A)\, E_{`B}(F(x,a),b_*)\] = \top$;
\item $\[(\forall (x \feq x_*) \fin \X \ftimes \X^*)\, (\forall (a_1 \feq a_{1*}) \fin `A \ftimes `A^*)\,(\forall (a_2 \feq a_{2*}) \fin `A \ftimes `A^*)$ \\ $\ghost \hfill
(E_{`A} (a_{1},a_{2*}) \IFF  E_{`B*} (F_*(x_*,a_{1*}), F(x,a_{2})))
\] = \top.
$
\end{enumerate}
\end{corollary}

\begin{proof}
Applying Theorem \ref{examples.G.3} with $r$ equal to the identity binary relation on $A$ and $s$ equal to the identity binary relation on $B$, we find that
\begin{equation*}
\begin{split}
&\[
(\forall (x \feq x_*) \fin \X \ftimes \X^*)\, (\forall (a_1 \feq a_{1*}) \fin `A \ftimes `A^*)\,(\forall (a_2 \feq a_{2*}) \fin `A \ftimes `A^*)\, \\ & \ghost \hspace{35ex}
(E_{`A} (a_{1},a_{2*}) \IMPLIES E_{`B*} (F_*(x_*,a_{1*}), F(x,a_{2})))
\] = \top
\end{split}
\end{equation*}
for all functions $F\: \X \times `A \to `B$.
Applying Theorem \ref{examples.G.3} with $r$ equal to the negation of the identity binary relation $A$ and $s$ equal to the the negation of the identity binary relation on $B$, we find that the one-to-one correspondence of Proposition \ref{examples.G.2.5} restricts to a one-to-one correspondence between families of projections $(p_{ab})$ such that $\sum_{b \in B} p_{ab} = 1_H$ and $p_{a_1b}\perp p_{a_2b}$ for all distinct $a_1, a_2 \in A$ and all $b \in B$ and functions $F\: \X \times `A \to `B$ such that 
\begin{equation*}
\begin{split}
&\[
(\forall (x \feq x_*) \fin \X \ftimes \X^*)\, (\forall (a_1 \feq a_{1*}) \fin `A \ftimes `A^*)\,(\forall (a_2 \feq a_{2*}) \fin `A \ftimes `A^*)\, \\ & \ghost \hspace{35ex}
(\NOT E_{`A} (a_{1},a_{2*}) \IMPLIES \NOT E_{`B*} (F_*(x_*,a_{1*}), F(x,a_{2})))
\] = \top.
\end{split}
\end{equation*}
Applying Proposition \ref{examples.G.1.5}, we also find that the one-to-one correspondence of Proposition \ref{examples.G.2.5} restricts to a one-to-one correspondence between families of projections $(p_{ab})$ such that $\sum_{b \in B} p_{ab} = 1_H$ for all $a \in A$ and $\bigvee_{a \in A} p_{ab} = 1_H$ for all $b \in B$ and functions $F\: \X \times `A \to `B$ such that 
$\[(\forall x \in \X)\,(\forall b_* \in `B_*)\,(\exists a \in `A)\, E_{`B}(F(x,a),b_*)\] = \top.$

Combining these three observations and applying Proposition \ref{appendix.G.2}, we conclude the statement of the corollary, because the atomic formulas $E_{`A} (a_{1},a_{2*})$ and $E_{`B*} (F_*(x_*,a_{1*}), F(x,a_{2}))$ have no variables in common, and thus, $$\[\NOT E_{`A} (a_{1},a_{2*}) \IMPLIES \NOT E_{`B*} (F_*(x_*,a_{1*}), F(x,a_{2}))\] =
\[E_{`B*} (F_*(x_*,a_{1*}), F(x,a_{2})) \IMPLIES E_{`A} (a_{1},a_{2*})\].$$ 
\end{proof}

\begin{definition}\label{examples.G.8}
Let $(A, \sim_A)$ and $(B, \sim_B)$ be finite simple graphs, and let $H$ be a nonzero finite-dimensional Hilbert space. We say that a family of projections $(p_{ab} \in L(H) \suchthat a \in A, b \in B)$ \emph{witnesses} $A \iso_q B$ if it witnesses both $A \xrightarrow{q} B$ and $B \xrightarrow{q} A$.
\end{definition}

Alice and Bob have a winning strategy for the $(A,B)$-isomorphism game, possibly using finite entangled quantum systems, if and only if there exists a family of projections on some nonzero finite-dimensional Hilbert space that witnesses $A \iso_q B$ \cite{AtseriasMancinskaRobersonSamalSeveriniVarvitsiotis}*{Res.~2}. 

\begin{corollary}\label{examples.G.9}
Let $(A, \sim_A)$ and $(B, \sim_B)$ be finite simple graphs, and let $R = `(\sim_A)$ and $S = `(\sim_B)$. Let $H$ be a nonzero finite-dimensional Hilbert space, and let $\X$ be the quantum set defined by $\At(\X) = \{H\}$. Then, the one-to-one correspondence of Proposition \ref{examples.G.2.5} restricts to a one-to-one correspondence between families of projections  $(p_{ab} \in L(H) \suchthat a \in A, \, b \in B)$ witnessing $A \iso_q B$ and functions $F\: \X \times `A \to `B$ such that
\begin{enumerate}
\item $\[(\forall x \in \X)\,(\forall b_* \in `B_*)\,(\exists a \in `A)\, E_{`B}(F(x,a),b_*)\] = \top$;
\item $\[(\forall (x \feq x_*) \fin \X \ftimes \X^*)\, (\forall (a_1 \feq a_{1*}) \fin `A \ftimes `A^*)\,(\forall (a_2 \feq a_{2*}) \fin `A \ftimes `A^*)$ \\ $\ghost \hfill
(E_{`A} (a_{1},a_{2*}) \IFF  E_{`B*} (F_*(x_*,a_{1*}), F(x,a_{2})))
\] = \top$;
\item $\[
(\forall (x \feq x_*) \fin \X \ftimes \X^*)\, (\forall (a_1 \feq a_{1*}) \fin `A \ftimes `A^*)\,(\forall (a_2 \feq a_{2*}) \fin `A \ftimes `A^*)$ \\  $\ghost \hfill
(\bend R (a_{1},a_{2*}) \IFF \bend S_* (F_*(x_*,a_{1*}), F(x,a_{2})))
\] = \top.$
\end{enumerate}
\end{corollary}

\begin{proof}
We extend the proof of Corollary \ref{examples.G.7}, by applying Theorem \ref{examples.G.3} first with $r = (\sim_A)$ and $s = (\sim_B)$ and then with $r = \NOT ( \sim_A)$ and $s = \NOT (\sim_B)$, reasoning similarly.
\end{proof}

\subsection{Quantum groups}\label{examples.H} In their essence, discrete quantum groups are the Pontryagin duals of compact quantum groups, and this is how they first arose \cite{PodlesWoronowicz}*{sec.~3}\cite{Woronowicz3}. There are many equivalent definitions \cite{VanDaele}\cite{EffrosRuan}\cite{KustermansVaes2}. We could define a discrete quantum group structure on a quantum set $\X$ to be a suitable comultiplication on $c_0(\X)$, as in \cite{PodlesWoronowicz}, a suitable comultiplication on $c_c(\X)$, as in \cite{VanDaele}, or a suitable comultiplication on $\ell^\infty(\X)$, as implicitly in \cite{KustermansVaes2}. We work with Van Daele's definition \cite{VanDaele}*{Def.~2.3}, defining $$c_c(\X) = \{a \in \ell^\infty(\X) \suchthat a(X) = 0 \text{ for cofinitely many } X \in \At(\X)\}$$ for each quantum set $\X$. 

Discrete quantum groups are undoubtedly the most compelling example illustrating the naturality of quantum logic to noncommutative mathematics. First, discrete quantum groups are a firmly established quantum generalization, as firmly established as any class of discrete quantum structures. They were first defined more than thirty years ago, and their place in the noncommutative dictionary appears to have never been in doubt. Second, the considerations that motivated their definition are far removed from the considerations that motivated the definition of the semantics considered here. Only the duality between operator algebras and quantum spaces is shared in common. Third, the example of discrete quantum groups carries empirical weight. While the semantics considered here can readily be motivated from first principles, the interpretation of equality that is given in this paper was in fact motivated by many of the examples that we have considered, so the incorporation of these examples is not so surprising. However, discrete quantum groups were not among the examples first considered by the author. The correct axiomatization of discrete quantum groups was naively hypothesized by the author and then established by Stefaan Vaes \cite{Vaes}. It is essentially his proof that is recorded here.

Let $\X$ be a quantum set. For each atom $X$ of $\X$, we write $J_X$ for the inclusion function of the atomic quantum set $\Q\{X\}$ into $\X$ \cite{Kornell}*{Def.~2.3}, i.e., as an abbreviation for $J_{\Q\{X\}}^\X$ \cite{Kornell}*{Def.~8.2}. Similarly, we write $\top_X$ and $\bot_X$ as abbreviations for $\top_{\Q\{X\}}$ and $\bot_{\Q\{X\}}$, respectively. Let $\Rep(\ell^\infty(\X))$ be the category of representations of $\ell^\infty(\X)$, implicitly of finite-dimensional nondegenerate normal $*$-representations. A morphism in $\Rep(\ell^\infty(\X))$ from a representation $\rho\: \ell^\infty(\X) \to L(H)$ to a representation $\sigma\:\ell^\infty(\X) \to L(K)$ is an intertwiner between the two representations, i.e., an operator $v\in L( H,K)$ such that $v\cdot\rho(a) = \sigma(a)\cdot v$ for all $a \in \ell^\infty(\X)$. Up to isomorphism, the simple objects, i.e., the irreducible representations, are just the canonical projections $J_X^\star\: \ell^\infty(\X) \to L(X)$ for $X \in \At(\X)$ \cite{Kornell}*{Thm.~7.4}, and every representation is a finite direct sum of these. The category $\Rep(\ell^\infty(\X))$ is a C*-category in the obvious way \cite{NeshveyevTuset}*{Def.~2.1.1}.

Let $\X$ be a quantum set, and let $F\:\X \times \X \to \X$ and $C\: \mathbf 1 \to \X$ be functions such that $F \circ (F \times I_\X) = F \circ (I_\X \times F)$ and $F \circ (C \times I_\X) = I_\X = F \circ (I_\X \times C)$. Then, $F^\star$ is a comultiplication on $\ell^\infty(\X)$, and $C^\star$ is a counit for this comultiplication \cite{Kornell}*{Thm.~7.4}. In the usual way, we obtain a monoidal structure on the category $\Rep(\ell^\infty(\X))$: For representations $\rho_1\: \ell^\infty(\X) \to L(H_1)$ and $\rho_2\: \ell^\infty(\X) \to L(H_2)$, we define $\rho_1 \boxtimes \rho_2 \: \ell^\infty(\X) \to L(H_1 \tensor H_2)$ by $\rho_1 \boxtimes \rho_2 = (\rho_1 \stensor \rho_2) \circ F^\star$. For intertwiners $v_1$, from $\rho_1\: \ell^\infty(\X) \to L(H_1)$ to $\sigma_1\: \ell^\infty(\X) \to L(K_1)$, and $v_2$, from $\rho_2\: \ell^\infty(\X) \to L(H_2)$ to $\sigma_2\: \ell^\infty(\X) \to L(K_2)$, we define $v_1 \boxtimes v_2 = v_1 \tensor v_2$. The standard computations show that $\Rel(\ell^\infty(\X))$ is a monoidal C*-category with product $\boxtimes$ and unit $C^\star\: \ell^\infty(\X) \to L(\CC)$ \cite{NeshveyevTuset}*{Def.~2.1.1}. We will show that it has conjugates \cite{NeshveyevTuset}*{Def.~2.2.1}.

We make a number of simplifying assumptions without loss of generality. First, we represent $\ell^\infty(\X)$ in a small strict monoidal category of finite-dimensional Hilbert spaces \cite{NeshveyevTuset}*{sec.~2.1}. Second, we assume that each atom of $\X$ is a Hilbert space in this small category. Third, we assume that $C(\CC, \CC) \neq 0$. In other words, we assume that $\CC$ is an atom of $\X$ and that furthermore it is the unique atom in the range of $C$ \cite{KornellLindenhoviusMislove2}*{Def.~3.2}. Overall, we have that $\Rep(\ell^\infty(\X))$ is a small strict monoidal C*-category that contains $J_X^\star$ for all $X \in \At(\X)$ and that has unit $C^\star = J^\star_\CC$.

\begin{proposition}\label{examples.H.1}
Let $\X$, $\Y$ and $\Z$ be quantum sets, let $P$ be a binary relation of arity $(\X,\Y)$, and let $Q$ be a binary relation of arity $(\Y,\Z)$. Assume that 
\begin{enumerate}\item $\[(\forall y \fin \Y)\, (\exists x \fin \X)\, P(x,y)\] = \top$; \item $\[(\forall y \fin \Y)\, (\exists z \fin \Z)\, Q(y,z)\] = \top$.\end{enumerate} Then, for all $Y \in \At(\Y)$, there exist $X \in \At(\X)$ and $Z \in \At(\Z)$ such that $(P \times \top_\Z) \circ (J_X \times J_Y \times J_Z)$ is not orthogonal to $(\top_\X \times Q) \circ (J_X \times J_Y \times J_Z)$.
\end{proposition}

\begin{proof}
We are essentially given that
\begin{align*}
\begin{aligned}
\begin{tikzpicture}
\begin{pgfonlayer}{nodelayer}
    \node [style=box] (P) at (0, 0) {$\ghost \;P\;\ghost$};
    \node [style=none] (Yanch) at (0.2,0) {};
    \node [style=none] (aanch) at (-0.2,0) {};
    \node (Y) at (0.2,-1.0) {$\scriptstyle \Y$};
    \node [style=none] (bullet) at (-0.2, -0.7) {$\bullet$};
\end{pgfonlayer}
\begin{pgfonlayer}{edgelayer}
    \draw [arrow, markat =0.3] (Y) to (Yanch);
    \draw [arrow, markat=0.4] (bullet.center) to (aanch);
\end{pgfonlayer}
\end{tikzpicture}
\end{aligned}
\;=\;
\begin{aligned}
\begin{tikzpicture}
\begin{pgfonlayer}{nodelayer}
    \node [style=none] (bullet) at (0, 0) {$\bullet$};
    \node [style=none] (Y) at (0,-1.0) {$\scriptstyle \Y$};
\end{pgfonlayer}
\begin{pgfonlayer}{edgelayer}
    \draw [arrow, markat =0.5] (Y) to (bullet.center);
\end{pgfonlayer}
\end{tikzpicture}
\end{aligned}
\qquad
\text{and}
\qquad
\begin{aligned}
\begin{tikzpicture}
\begin{pgfonlayer}{nodelayer}
    \node [style=box] (Q) at (0, 0) {$\ghost \;Q\;\ghost$};
    \node [style=none] (Yanch) at (-0.2,0) {};
    \node [style=none] (aanch) at (0.2,0) {};
    \node (Y) at (-0.2,-1.0) {$\scriptstyle \Y$};
    \node [style=none] (bullet) at (0.2, -0.7) {$\bullet$};
\end{pgfonlayer}
\begin{pgfonlayer}{edgelayer}
    \draw [arrow, markat =0.3] (Y) to (Yanch);
    \draw [arrow, markat=0.4] (bullet.center) to (aanch);
\end{pgfonlayer}
\end{tikzpicture}
\end{aligned}
\;=\;
\begin{aligned}
\begin{tikzpicture}
\begin{pgfonlayer}{nodelayer}
    \node [style=none] (bullet) at (0, 0) {$\bullet$};
    \node [style=none] (Y) at (0,-1.0) {$\scriptstyle \Y$};
\end{pgfonlayer}
\begin{pgfonlayer}{edgelayer}
    \draw [arrow, markat =0.5] (Y) to (bullet.center);
\end{pgfonlayer}
\end{tikzpicture}
\end{aligned}.
\end{align*}
We now reason that
\begin{align*}
\bigvee_{\begin{smallmatrix} X \in \At(\X) \\ Z \in \At(\Z) \end{smallmatrix}}
\;
\begin{aligned}
\begin{tikzpicture}
\begin{pgfonlayer}{nodelayer}
    \node [style=box] (P) at (0.5,3) {$\ghost \quad P^{\phantom \dagger}\quad \ghost$};
    \node [style=box] (Q) at (1.5,0) {$\ghost \quad Q^{\dagger}\quad \ghost$};
    \node [style=none] (X3) at (0,3) {};
    \node [style=none] (Y3) at (1,3) {};
    \node [style=none] (Z3) at (2,3) {$\bullet$};
    \node [style=box] (X2) at (0,2) {$J_X^{\phantom \dagger}$};
    \node [style=box] (X1) at (0,1) {$J_X^{\dagger}$};
    \node [style=box] (Y2) at (1,2) {$J_Y^{\phantom \dagger}$};
    \node [style=box] (Y1) at (1,1) {$J_Y^{\dagger}$};
    \node [style=box] (Z2) at (2,2) {$J_Z^{\phantom \dagger}$};
    \node [style=box] (Z1) at (2,1) {$J_Z^{\dagger}$};
    \node [style=none] (X0) at (0,0) {$\bullet$};
    \node [style=none] (Y0) at (1,0) {};
    \node [style=none] (Z0) at (2,0) {};
\end{pgfonlayer}
\begin{pgfonlayer}{edgelayer}
    \draw[arrow, markat=0.55] (X1.center) to (X2.center);
    \draw[arrow, markat=0.55] (Y1.center) to (Y2.center);
    \draw[arrow, markat=0.55] (Z1.center) to (Z2.center);
    \draw[arrow, markat=0.5] (X0.center) to (X1.center);
    \draw[arrow, markat=0.5] (Y0.center) to (Y1.center);
    \draw[arrow, markat=0.5] (Z0.center) to (Z1.center);
    \draw[arrow, markat=0.6] (X2.center) to (X3.center);
    \draw[arrow, markat=0.6] (Y2.center) to (Y3.center);
    \draw[arrow, markat=0.7] (Z2.center) to (Z3.center);
\end{pgfonlayer}
\end{tikzpicture}
\end{aligned}
\quad
=
\quad
\begin{aligned}
\begin{tikzpicture}
\begin{pgfonlayer}{nodelayer}
    \node [style=box] (P) at (0.5,3) {$\ghost \quad P^{\phantom \dagger}\quad \ghost$};
    \node [style=box] (Q) at (1.5,0) {$\ghost \quad Q^{\dagger}\quad \ghost$};
    \node [style=none] (X3) at (0,3) {};
    \node [style=none] (Y3) at (1,3) {};
    \node [style=none] (Z3) at (2,3) {$\bullet$};
    \node [style=box] (Y2) at (1,2) {$J_Y^{\phantom \dagger}$};
    \node [style=box] (Y1) at (1,1) {$J_Y^{\dagger}$};
    \node [style=none] (X0) at (0,0) {$\bullet$};
    \node [style=none] (Y0) at (1,0) {};
    \node [style=none] (Z0) at (2,0) {};
\end{pgfonlayer}
\begin{pgfonlayer}{edgelayer}
    \draw[arrow, markat=0.55] (Y1.center) to (Y2.center);
    \draw[arrow, markat=0.5] (X0.center) to (X3.center);
    \draw[arrow, markat=0.5] (Y0.center) to (Y1.center);
    \draw[arrow, markat=0.5] (Z0.center) to (Z3.center);
    \draw[arrow, markat=0.6] (Y2.center) to (Y3.center);
\end{pgfonlayer}
\end{tikzpicture}
\end{aligned}
\quad
=
\quad
\begin{aligned}
\begin{tikzpicture}
\begin{pgfonlayer}{nodelayer}
    \node [style=none] (Y3) at (1,3) {$\bullet$};
    \node [style=box] (Y2) at (1,2) {$J_Y^{\phantom \dagger}$};
    \node [style=box] (Y1) at (1,1) {$J_Y^{\dagger}$};
    \node [style=none] (Y0) at (1,0) {$\bullet$};
\end{pgfonlayer}
\begin{pgfonlayer}{edgelayer}
    \draw[arrow, markat=0.55] (Y1.center) to (Y2.center);
    \draw[arrow, markat=0.5] (Y0.center) to (Y1.center);
    \draw[arrow, markat=0.6] (Y2.center) to (Y3.center);
\end{pgfonlayer}
\end{tikzpicture}
\end{aligned}
\quad
=
\quad
\top.
\end{align*}
We have applied the equalities $\bigvee_{X \in \At(\X)} J_X \circ J_X^\dagger = I_\X$ and $\bigvee_{Z \in \At(\Z)} J_Z \circ J_Z^\dagger = I_\Z$ \cite{KornellLindenhoviusMislove2}*{Lem.~A.4}. The last equality holds because $\top_\Y \circ J_Y = \top_Y \neq \bot_Y$ \cite{Kornell}*{Thm.~B.8}. Hence, at least one term in the join on the left is equal to $\top$. We conclude that there exist $X \in \At(\X)$ and $Z \in \At(\Z)$ such that $((P \times \top_\Z) \circ (J_X \times J_Y \times J_Z)) \circ ((\top_\X \times Q) \circ (J_X \times J_Y \times J_Z))^\dagger = \top$, and therefore, the two relations are not orthogonal.
\end{proof}

For each atom $X$ of $\X$, the inclusion function $J_X\: \Q\{X\} \to \X$ is injective \cite{Kornell}*{Prop.~8.4}, and therefore, the binary relation $J_X^\dagger\: \X \to \Q\{X\}$ is a partial function. Because $J_X(X,X)$ is spanned by the identity operator on $X$, with all other components of $J_X$ vanishing, it is easy to see that $(J_X^\dagger)^\star$ is equal to $[X]$, the minimal central projection in $\ell^\infty(\X)$ corresponding to $X$. This is also the support projection of $J_X^\star$ \cite{Kornell}*{Lem.~8.3}. Thus, the support projection of $C^\star = J_\CC^\star\: \ell^\infty(\X) \to \CC$ is $[\CC] = (C^\dagger)^\star(1)$.

\begin{lemma}\label{examples.H.2}
Let $\X$ be a quantum set, and let $F\:\X \times \X \to \X$ and $C\: \mathbf 1 \to \X$ be functions such that $F \circ (F \times I_\X) = F \circ (I_\X \times F)$ and $F \circ (C \times I_\X) = I_\X = F \circ (I_\X \times C)$. Assume that \begin{enumerate} \item $\[(\forall x_2 \fin \X)\, (\exists x_1 \fin \X)\, E_\X(F(x_1, x_2), C_*)\] = \top$; \item $\[(\forall x_2 \fin \X)\, (\exists x_3 \fin \X)\, E_\X(F(x_2, x_3), C_*)\] = \top$. \end{enumerate} Then, every simple object of the strict monoidal C*-category $\Rep(\ell^\infty(X))$ has a conjugate. In other words, for every atom $X \in \At(\X)$, there exist an atom $\bar X \in \At(X)$ and intertwiners $v_X\: C^\star \to J_{X}^\star \boxtimes J_{\bar X}^\star$ and $w_X\: C^\star \to J_{\bar X}^\star \boxtimes J_{X}^\star$ such that $(v_X^\dagger \boxtimes 1_{X})(1_{ X} \boxtimes w_X) = 1_{X}$ and $(w_X^\dagger \boxtimes 1_{\bar X})(1_{\bar X} \boxtimes v_X) = 1_{\bar X}$. Therefore, $\Rep(\ell^\infty(X))$ is rigid.
\end{lemma}

\begin{proof}
Let $X_2 \in \At(\X)$. Let $P = \[(x_1, x_2) \in \X \times  \X \suchthat E_\X(F(x_1, x_2), C_*)\] = C^\dagger \circ F$, and similarly, let $Q = \[(x_2, x_3) \in \X \times  \X \suchthat E_\X(F(x_2, x_3), C_*)\] = C^\dagger \circ F$. We apply Proposition \ref{examples.H.1} to obtain atoms $X_1$ and $X_3$ such that 
$$ ((C^\dagger \circ F) \times \top_\X) \circ (J_{X_1} \times J_{X_2} \times J_{X_3}) \not \perp  ( \top_\X \times (C^\dagger \circ F) ) \circ (J_{X_1} \times J_{X_2} \times J_{X_3})$$ as binary relations from $\Q\{X_1 \tensor X_2 \tensor X_3\}$ to $\mathbf 1$. Writing $G_{12} = C^\dagger \circ F \circ (J_{X_1} \times J_{X_2})$ and $G_{23} =C^\dagger \circ F \circ (J_{X_2} \times J_{X_3})$, we have that $G_{12} \times \top_{X_3}  \not \perp  \top_{X_1} \times G_{23}$.
In other words, $$G_{12}(X_1 \tensor X_2, \CC) \tensor L( X_3, \CC) \not \perp L(X_1, \CC) \tensor G_{23} (X_2 \tensor X_3, \CC)$$
as subspaces of $L(X_1 \tensor X_2 \tensor X_3, \CC)$. Hence, let $v_{12} \in G_{12}(X_1 \tensor X_2, \CC)$, $\xi_3 \in L(X_3, \CC)$, $\xi_1 \in L(X_1, \CC)$ and $v_{23} \in G_{23}(X_2 \tensor X_3, \CC)$ be such that $v_{12} \tensor \xi_3$ is not orthogonal to $\xi_1 \tensor v_{23}$ as elements of $L(X_1 \tensor X_2 \tensor X_3, \CC)$. It certainly follows that $(v_{12} \tensor 1) \cdot (1 \tensor v_{23}^\dagger) \neq 0$.

We now show that $v_{12}$ is an intertwiner. To do so, we recall our previous observation that $C^\dagger\: \X \to \mathbf 1$ is a partial function and that $[\CC]=(C^\dagger)^\star(1)$ is the support projection of $C^\star$. Hence, $G_{12}\: \Q\{X_1 \tensor X_2\} \to \mathbf 1$ is a partial function, and $v_{12}$ satisfies $v_{12} = 1\cdot v_{12} = v_{12}\cdot  G_{12}^\star(1) = v_{12} \cdot (J_{X_1}^\star \stensor J_{X_2}^\star)(F^\star([\CC]))$ \cite{Kornell}*{Thm.~6.3}. We now calculate that for all $a \in \ell^\infty(\X)$,
\begin{align*}&
v_{12} \cdot (J_{X_1}^\star \boxtimes J_{X_2}^\star)(a)
=
v_{12} \cdot (J_{X_1}^\star \stensor J_{X_2}^\star) (F^\star(a))
\\ &= 
v_{12} \cdot (J_{X_1}^\star \stensor J_{X_2}^\star)(F^\star([\CC])) \cdot (J_{X_1}^\star \stensor J_{X_2}^\star) (F^\star(a))
\\ & = 
v_{12} \cdot (J_{X_1}^\star \stensor J_{X_2}^\star)(F^\star([\CC]\cdot a))
 =
v_{12} \cdot (J_{X_1}^\star \stensor J_{X_2}^\star)(F^\star ([\CC] \cdot C^\star(a)))
\\ & = 
v_{12} \cdot (J_{X_1}^\star \stensor J_{X_2}^\star)(F^\star([\CC]) )\cdot C^\star(a)
=
v_{12} \cdot C^\star(a) =  C^\star(a) \cdot v_{12}.
\end{align*}
\noindent 
Therefore, $v_{12}$ is an intertwiner from $J_{X_1}^\star \boxtimes J_{X_2}^\star$ to $C^\star$. Similarly, $v_{23}$ is an intertwiner from $J_{X_2}^\star \boxtimes J_{X_3}^\star$ to $C^\star$.
Furthermore, by our choice of $v_{12}$ and $v_{23}$, we have that $(v_{12} \tensor 1) (1 \tensor v_{23}^\dagger) \neq 0$.

Altogether, we find that $s := (v_{12} \tensor 1)\cdot (1 \tensor v_{23}^\dagger)$ is a nonzero intertwiner from $J_{X_1}^\star$ to $J_{X_3}^\star$. These are irreducible representations of $\ell^\infty(\X)$, and thus, by Schur's Lemma,  $s$ is an isomorphism in $\Rep(\ell^\infty(\X))$. Defining $v:= (1 \tensor s^{-1})v_{23}^\dagger$ and $w:= v_{12}^\dagger$, we obtain intertwiners $v\:C^\star\to J_{X_2}^\star \boxtimes J_{X_1}^\star$ and $w\:C^\star\to J_{X_1}^\star \boxtimes J_{X_2}^\star$ such that $(w^\dagger \tensor 1)\cdot  (1 \tensor v) = 1$ and therefore also $(1 \tensor v^\dagger)\cdot(w \tensor 1)  = 1$.
We now calculate that
\begin{align*}&
(1 \tensor ((v^\dagger \tensor 1)\cdot (1\tensor w)))\cdot w
=
(1 \tensor v^\dagger \tensor 1)\cdot(1 \tensor 1 \tensor w)\cdot w
=
(1 \tensor v^\dagger \tensor 1)\cdot (w \tensor w)
\\ & =
(1 \tensor v^\dagger \tensor 1)\cdot (w \tensor 1 \tensor 1)\cdot w
=
(((1 \tensor v^\dagger)\cdot(w \tensor 1)) \tensor 1)\cdot w
=
(1 \tensor 1)\cdot w
=
w.
\end{align*}
Since $w$ is nonzero, we find that $(v^\dagger \tensor 1)\cdot(1\tensor w)$ is a nonzero intertwiner on $J_{X_2}^\star$. It is therefore a scalar, and it can be no scalar other than $1$. We conclude that $J_{X_1}^\star$ and $J_{X_2}^\star$ are conjugate objects in $\Rep(\ell^\infty(\X))$, and more generally, that every simple object in $\Rep(\ell^\infty(\X))$ has a conjugate. Because every object in $\Rep(\ell^\infty(\X))$ is a direct sum of simple objects, it follows that every object has a conjugate. In other words, $\Rep(\ell^\infty(\X))$ is rigid.
\end{proof}

\begin{lemma}\label{examples.H.3}
Let $\X$ be a quantum set, and let $F\:\X \times \X \to \X$ and $C\: \mathbf 1 \to \X$ be functions such that $F \circ (F \times I_\X) = F \circ (I_\X \times F)$ and $F \circ (C \times I_\X) = I_\X = F \circ (I_\X \times C)$. Assume that $\Rep(\ell^\infty(X))$ is rigid. Then, $F^\star(a)\cdot(1 \tensor b)$ and $(a \tensor 1) \cdot F^\star(b)$ are both in the algebraic tensor product $c_c(\X) \atensor c_c(\X)$ for all $a, b \in c_c(\X)$.
\end{lemma}

\begin{proof}
Let $a,b \in c_c(\X)$. Let $X_0$, $X_1$ and $X_2$ be atoms of $\X$, and assume that $J_{X_0}^\star(b) \neq 0$ and that $\Mor(J_{X_0}^\star, J_{X_1}^\star \boxtimes J_{X_2}^\star) \neq 0$. Then, there exists an operator $v \in L(X_0, X_1 \tensor X_2)$ such that
$(J_{X_1}^\star \boxtimes J_{X_2}^\star)(b) \cdot  v = v \cdot J_{X_0}^\star(b) \neq 0,$
and hence $(J_{X_1}^\star \boxtimes J_{X_2}^\star)(b) \neq 0$. Therefore, for all atoms $X_0$, $X_1$ and $X_2$, the conditions $J_{\X_0}^\star(b) \neq 0$ and $\Mor(J_{X_0}^\star, J_{X_1}^\star \boxtimes J_{X_2}^\star) \neq 0$ together imply that $(J_{X_1}^\star \boxtimes J_{X_2}^\star)(b) \neq 0$.

Let $X_0$ and $X_1$ be atoms of $\X$. For each atom $X \in \At(\X)$, let $\bar X \in \At(\X)$ be the unique atom such that the representations $J_{X}^\star$ and $J_{\bar X}^\star$ are conjugate. Applying Frobenius reciprocity \cite{NeshveyevTuset}*{Thm.~2.2.6}, we compute that for each atom $X_2$,
$\Mor(J_{X_0}^\star, J_{X_1}^\star \boxtimes J_{X_2}^\star) = \Mor( J_{X_1}^\star \boxtimes J_{X_2}^\star, J_{X_0}^\star)^\dagger = \Mor( J_{X_2}^\star, J_{ \bar X_1}^\star\boxtimes J_{X_0}^\star)^\dagger$. Because $\Rep(\ell^\infty(\X))$ is a rigid C*-category, there are only finitely many atoms $X_2$ such that $\Mor( J_{X_2}^\star, J_{\bar X_1}^\star\boxtimes J_{X_0}^\star) \neq 0$ \cite{NeshveyevTuset}*{Cor.~2.2.9}. Therefore, for all atoms $X_0$ and $X_1$, there are only finitely many atoms $X_2$ such that $\Mor(J_{X_0}^\star, J_{X_1}^\star \boxtimes J_{X_2}^\star) \neq 0$.

We now compute that for all atoms $X_1$ and $X_2$,
\begin{align*}
(J_{X_1}^\star \stensor J_{X_2}^\star)((a \tensor 1)\cdot F^\star(b))
& =
(J_{X_1}^\star \stensor J_{X_2}^\star)(a \tensor 1) \cdot (J_{X_1}^\star \stensor J_{X_2}^\star)(F^\star(b))
\\ & =
(J_{X_1}^\star(a) \tensor 1) \cdot(J_{X_1}^\star \boxtimes J_{X_2}^\star)(b).
\end{align*}
Because $a, b \in c_c(\X)$, we have that $J_{X}^\star(a) \neq 0$ and $J_{X}^\star(b) \neq 0$ for only finitely many atoms $X$. Thus, $J_{X_1}^\star(a) \tensor 1 \neq 0$ for only finite many atoms $X_1$, and for each of those atoms, $(J_{X_1}^\star \boxtimes J_{X_2}^\star)(b) \neq 0$ for only finitely many atoms $X_2$. We conclude that there are only finitely many pairs $(X_1, X_2) \in \At(\X) \times \At(\X)$ such that $(J_{X_1}^\star \stensor J_{X_2}^\star)((a \tensor 1)\cdot F^\star(b))$ is nonzero. Therefore, $(a \tensor 1)\cdot F^\star(b) \in c_c(\X) \atensor c_c(\X)$ as claimed. Similarly, $F^\star(a) \cdot (1 \tensor b) \in c_c(\X) \atensor c_c(\X)$.
\end{proof}

\begin{lemma}\label{examples.H.4}
Let $\X$ be a quantum set, and let $F\:\X \times \X \to \X$ and $C\: \mathbf 1 \to \X$ be functions such that $F \circ (F \times I_\X) = F \circ (I_\X \times F)$ and $F \circ (C \times I_\X) = I_\X = F \circ (I_\X \times C)$. Assume that $\Rep(\ell^\infty(\X))$ is rigid. Let $T_1$ and $T_2$ be the functions on the algebraic tensor product $c_c(\X) \atensor c_c(\X)$ that are defined by the equations $T_1(a \tensor b) = F^\star(a)\cdot (1 \tensor b)$ and $T_2(a \tensor b) = (a \tensor 1) \cdot F^\star(b)$, respectively. Then, $T_1$ and $T_2$ are both surjective.
\end{lemma}

\begin{proof}
We use the notation of Lemma \ref{examples.H.2}.
Let $c \in c_c(\X)$, and let $X_0 \in \At(\X)$. Then, $F^\star(c) \cdot (1 \tensor [\bar X_0])$ is in $c_c(\X) \atensor c_c(\X)$, so it can be written as $F^\star(c) \cdot (1 \tensor [\bar X_0]) = \sum_{i=1}^n a_i \tensor c_i$ for some operators $a_1, \ldots, a_n$ and $c_1, \ldots, c_n$ in $c_c(\X)$. For each index $i$, let $b_i = \iota^{X_0}((1 \tensor w_{X_0}^\dagger)\cdot (1 \tensor J_{\bar X_0}^\star(c_i) \tensor 1) \cdot (v_{X_0} \tensor 1))$, where $\iota^{X_0}$ is the canonical inclusion of $L(X_0)$ into $\ell^\infty(\X)$. We now compute that for all atoms $X_1$ and $X_2$,
\begin{align*} &
(J_{X_1}^\star \stensor J_{X_2}^\star)\left(T_1\left( \sum_{i=1}^n a_i \tensor b_i\right)\right)
 = 
\sum_{i=1}^n (J_{X_1}^\star \stensor J_{X_2}^\star)(F^\star(a_i)\cdot(1 \tensor b_i))
\\ & =
\sum_{i=1}^n (J_{X_1}^\star \stensor J_{X_2}^\star)(F^\star(a_i)) \cdot (1 \tensor J_{X_2}^\star(b_i))
\\ &=
\delta_{X_0 X_2} \cdot \sum_{i=1}^n (J_{X_1}^\star \stensor J_{X_2}^\star)(F^\star(a_i)) \cdot (1 \tensor 1 \tensor w_{X_2}^\dagger)\cdot(1 \tensor 1 \tensor J_{\bar X_2}^\star(c_i) \tensor 1) \cdot (1 \tensor v_{X_2} \tensor 1)
\\ & =
\delta_{X_0 X_2} \cdot (1 \tensor 1 \tensor w_{X_2}^\dagger)\cdot\left(\sum_{i=1}^n (J_{X_1}^\star \stensor J_{X_2}^\star)(F^\star(a_i)) \tensor J_{\bar X_2}^\star(c_i) \tensor 1\right) \cdot (1 \tensor v_{X_2} \tensor 1)
\\ & =
\delta_{X_0 X_2} \cdot (1 \tensor 1 \tensor w_{X_2}^\dagger)\cdot\left(\sum_{i=1}^n (J_{X_1}^\star \stensor J_{X_2}^\star \stensor J_{\bar X_2}^\star)((F^\star \stensor \idhom)(a_i \tensor c_i)) \tensor 1\right) \cdot (1 \tensor v_{X_2} \tensor 1)
\\ & =
\delta_{X_0 X_2} \cdot (1 \tensor 1 \tensor w_{X_2}^\dagger)\cdot\left((J_{X_1}^\star \stensor J_{X_2}^\star \stensor J_{\bar X_2}^\star)((F^\star \stensor \idhom)(F^\star(c)\cdot (1 \tensor [\bar X_2]))) \tensor 1\right) \\ & \ghost \hspace{70ex} \cdot (1 \tensor v_{X_2} \tensor 1)
\\ & =
\delta_{X_0 X_2} \cdot (1 \tensor 1 \tensor w_{X_2}^\dagger)\cdot\left((J_{X_1}^\star \stensor J_{X_2}^\star \stensor J_{\bar X_2}^\star)((F^\star \stensor \idhom)(F^\star(c))) \tensor 1 \right) 
\\ & \ghost \hspace{30ex}\cdot \left((J_{X_1}^\star \stensor J_{X_2}^\star \stensor J_{\bar X_2}^\star)( (F^\star \stensor \idhom)(1 \tensor [\bar X_2])) \tensor 1\right) \cdot (1 \tensor v_{X_2} \tensor 1)
\\ & =
\delta_{X_0 X_2} \cdot (1 \tensor 1 \tensor w_{X_2}^\dagger)\cdot\left((J_{X_1}^\star \stensor J_{X_2}^\star \stensor J_{\bar X_2}^\star)((F^\star \stensor \idhom)(F^\star(c))) \tensor 1 \right) 
\cdot (1 \tensor v_{X_2} \tensor 1)
\\ & =
\delta_{X_0 X_2} \cdot (1 \tensor 1 \tensor w_{X_2}^\dagger)\cdot\left((J_{X_1}^\star \boxtimes J_{X_2}^\star \boxtimes J_{\bar X_2}^\star)(c) \tensor 1 \right) 
\cdot ((1 \boxtimes v_{X_2}) \tensor 1)
\\ & =
\delta_{X_0 X_2} \cdot (1 \tensor 1 \tensor w_{X_2}^\dagger)\cdot ((1 \boxtimes v_{X_2}) \tensor 1)
\cdot 
\left((J_{X_1}^\star \boxtimes C^\star)(c) \tensor 1 \right)
\\ & =
\delta_{X_0 X_2} \cdot (1 \tensor 1 \tensor w_{X_2}^\dagger)\cdot (1 \tensor v_{X_2} \tensor 1)
\cdot 
\left(J_{X_1}^\star (c) \tensor 1 \right) 
\\ & =
\delta_{X_0 X_2} \cdot
\left(J_{X_1}^\star (c) \tensor 1 \right)
=
J^\star_{X_1}(c) \tensor J^\star_{X_2}([X_0])
=
(J^\star_{X_1} \stensor J^\star_{X_2})(c \tensor [X_0]).
\end{align*}
Therefore, $T_1\left( \sum_{i=1}^n a_i \tensor b_i\right) = c\tensor [X_0]$. More generally, for all $c, d \in c_c(\X)$ and all atoms $X_0 \in \At(\X)$, we have that $T_1\left( (\sum_{i=1}^n a_i \tensor b_i)\cdot(1 \tensor d)\right) = (c \tensor [X_0]) \cdot (1 \tensor d) = c \tensor (d \cdot[X_0])$, so $c \tensor (d \cdot[X_0]) \in \mathrm{ran}(T_1)$. The expression $d \cdot [X_0]$ is nonzero for only finitely many atoms $X_0$, so we conclude that $c \tensor d \in \mathrm{ran}(T_1)$ for all $c, d \in c_c(\X)$. Therefore, $T_1$ is surjective. Similarly, $T_2$ is surjective.
\end{proof}

\begin{theorem}[Vaes, cf.~\cite{Vaes}]\label{examples.H.5}
Let $\X$ be a quantum set, and let $F\:\X \times \X \to \X$ and $C\: \mathbf 1 \to \X$ be functions such that $F \circ (F \times I_\X) = F \circ (I_\X \times F)$ and $F \circ (C \times I_\X) = I_\X = F \circ (I_\X \times C)$. Assume that \begin{enumerate} \item $\[(\forall x_1 \fin \X)\, (\exists x_2 \fin \X)\, E_\X(F(x_1, x_2), C_*)\] = \top$; \item $\[(\forall x_2 \fin \X)\, (\exists x_1 \fin \X)\, E_\X(F(x_1, x_2), C_*)\] = \top$. \end{enumerate} Let $\Delta\: c_c(\X) \to \Mult(c_c(\X) \atensor c_c(\X))$  be defined by $\Delta = F^\star|_{c_c(\X)}$. Then $(c_c(\X), \Delta)$ is a discrete quantum group in the sense of \cite{VanDaele}*{Def.~2.3}.
\end{theorem}

\begin{proof}
Via the duality between quantum sets and hereditarily atomic von Neumann algebras \cite{Kornell}*{Thm.~7.4}, we obtain unital normal $*$-homomorphisms $F^\star\: \ell^\infty(\X) \stensor \ell^\infty(\X) \to \ell^\infty(\X)$ and $C^\star\: \ell^\infty(\X) \to \CC$ that satisfy $(F^\star \stensor \idhom) \circ F^\star =( \idhom \stensor F^\star) \circ F^\star$, $(C^\star \stensor \idhom ) \circ F^\star = \idhom$ and $(\idhom \stensor C^\star) \circ F^\star = \idhom$.
The hereditarily atomic von Neumann algebra $\ell^\infty(\X) \stensor \ell^\infty(\X)$ may be naturally regarded as an algebra of multipliers on $c_c(\X) \atensor c_c(\X)$, because $c_c(\X) \atensor c_c(\X)$ is a two-sided ideal in $\ell^\infty(\X) \stensor \ell^\infty(\X)$. Thus, we may define unital $*$-homomorphisms $\Delta\: c_c(\X) \to \Mult(c_c(\X) \atensor c_c(\X))$ by $\Delta = F^\star|_{c_c(\X)}$ and $\varepsilon\: c_c(\X) \to \CC$ by $\varepsilon = C^\star|_{c_c(\X)}$.

We now observe that $\Delta$ is a comultiplication on $c_c(\X)$. By Lemmas \ref{examples.H.2} and \ref{examples.H.3}, $\Delta(a)\cdot(1 \tensor b)$ and $(a \tensor 1)\cdot \Delta(b)$ are both in $c_c(\X) \atensor c_c(\X)$ for all $a, b \in c_c(\X)$. Furthermore, for all $a, b, c \in c_c(\X)$,
\begin{align*}
(a & \tensor 1 \tensor 1) \cdot (\Delta \tensor \idhom)(\Delta(b) \cdot (1 \tensor c))
 =
(a \tensor 1 \tensor 1) \cdot (F^\star \stensor \idhom)(F^\star(b) \cdot (1 \tensor c))
\\ & =
(a \tensor 1 \tensor 1) \cdot (F^\star \stensor \idhom)(F^\star(b)) \cdot (F^\star \stensor \idhom)(1 \tensor c)
\\ & =
(\idhom \stensor F^\star)( a \tensor 1)\cdot (\idhom \stensor F^\star)(F^\star(b)) \cdot (1 \tensor 1 \tensor c)
\\ & =
(\idhom \stensor F^\star)((a \tensor 1) \cdot F^\star(b)) \cdot (1 \tensor 1 \tensor c)
=
(\idhom \tensor \Delta)((a \tensor 1) \cdot \Delta(b)) \cdot (1 \tensor 1 \tensor c).
\end{align*}
Hence, $\Delta$ is indeed a comultiplication on $c_c(\X)$.

Let $a \in c_c(\X)$, and assume that $\Delta([\CC])\cdot (1 \tensor a) = 0$. It follows that $\Delta([\CC])\cdot (1 \tensor aa^\dagger) = 0$ and thus that $\Delta([\CC])\cdot( 1 \tensor [aa^\dagger]) = 0$, where $[aa^\dagger]$ is of course the support projection of the self-adjoint operator $aa^\dagger$. We now observe that the projection $\Delta([\CC]) \in \ell^\infty(\X) \stensor \ell^\infty(\X)$ corresponds to the relation $\[ E_\X(F(x_1, x_2), C_*)\] = C^\dagger \circ F \in \Rel(\X,\X)$ under the canonical correspondence \cite{Kornell}*{Thm.~B.8}. Indeed, $(C^\dagger \circ F)^\star(1) = F^\star((C^\dagger)^\star(1)) = F^\star([\CC]) = \Delta([\CC])$. Similarly, the projection $[aa^\dagger]$ corresponds to some relation $P \in \Rel(\X)$.
The projections $\Delta([\CC])$ and $( 1 \tensor [aa^\dagger])$ are orthogonal, and thus, the relations $C^\dagger \circ F$ and $\top_\X \times P$ are orthogonal. Condition (2) may be rendered as
\begin{align*}
\begin{aligned}
\begin{tikzpicture}
\begin{pgfonlayer}{nodelayer}
    \node [style=box] (CF) at (0, 0) {$\ghost C^\dagger \circ F\ghost$};
    \node [style=none] (Yanch) at (0.3,0) {};
    \node [style=none] (aanch) at (-0.3,0) {};
    \node [style=none] (X) at (0.3,-1.0) {$\scriptstyle \X$};
    \node [style=none] (bullet) at (-0.3, -0.7) {$\bullet$};
\end{pgfonlayer}
\begin{pgfonlayer}{edgelayer}
    \draw [arrow, markat =0.3] (X) to (Yanch);
    \draw [arrow, markat=0.4] (bullet.center) to (aanch);
\end{pgfonlayer}
\end{tikzpicture}
\end{aligned}
\;=\;
\begin{aligned}
\begin{tikzpicture}
\begin{pgfonlayer}{nodelayer}
    \node [style=none] (bullet) at (0, 0) {$\bullet$};
    \node [style=none] (X) at (0,-1.0) {$\scriptstyle \X$};
\end{pgfonlayer}
\begin{pgfonlayer}{edgelayer}
    \draw [arrow, markat =0.5] (X) to (bullet.center);
\end{pgfonlayer}
\end{tikzpicture}
\end{aligned}\;,
\end{align*}
so we calculate that
\begin{align*}
\begin{aligned}
\begin{tikzpicture}
\begin{pgfonlayer}{nodelayer}
    \node [style=none] (bullet) at (0, 0) {$\bullet$};
    \node [style=box] (P) at (0,-1.0) {$P^\dagger$};
\end{pgfonlayer}
\begin{pgfonlayer}{edgelayer}
    \draw [arrow, markat =0.5] (P) to (bullet.center);
\end{pgfonlayer}
\end{tikzpicture}
\end{aligned}
\; = \;
\begin{aligned}
\begin{tikzpicture}
\begin{pgfonlayer}{nodelayer}
    \node [style=box] (CF) at (0, 0) {$\ghost C^\dagger \circ F\ghost$};
    \node [style=none] (Yanch) at (0.3,0) {};
    \node [style=none] (aanch) at (-0.3,0) {};
    \node [style=box] (P) at (0.3,-1.0) {$P^\dagger$};
    \node [style=none] (bullet) at (-0.3, -1.0) {$\bullet$};
\end{pgfonlayer}
\begin{pgfonlayer}{edgelayer}
    \draw [arrow, markat =0.3] (P) to (Yanch);
    \draw [arrow, markat=0.4] (bullet.center) to (aanch);
\end{pgfonlayer}
\end{tikzpicture}
\end{aligned}
\; = \; \bot \;.
\end{align*}
We conclude that $P = \bot_\X$. Thus, $[aa^\dagger] = 0$, and therefore, $a=0$. We have shown that for all $a \in c_c(\X)$, the equation $\Delta([\CC])\cdot (1 \tensor a) = 0$ implies that $a = 0$.

Let $T_1$ and $T_2$ be the linear maps on $c_c(\X)  \atensor c_c(\X)$ defined by $T_1(a \tensor b) = \Delta(a)\cdot(1 \tensor b)$ and $T_2(a \tensor b) = (a \tensor 1)\cdot\Delta(b)$, respectively. By Lemmas \ref{examples.H.2} and \ref{examples.H.4}, both $T_1$ and $T_2$ are surjective. For all $a \in c_c(\X)$, we compute that $\Delta(a)\cdot (1 \tensor [\CC]) = F^\star(a) \cdot (1 \tensor[\CC]) = (\idhom \stensor C^\star)(F^\star(a)) \tensor [\CC] = a \tensor [\CC]$, because $[\CC]$ is the support projection of the unital normal 
$*$-homomorphism $C^*\: c_c(\X) \to \CC$. Furthermore, we have already shown that for all $a \in c_c(\X)$, we have that $\Delta([\CC])\cdot (1 \tensor a) = 0$ only if $a = 0$. We conclude by \cite{VanDaele}*{Thm.~3.4} that $(c_c(\X), \Delta)$ is a discrete quantum group.
\end{proof}

\begin{corollary}[Vaes]\label{examples.H.6}
Let $\X$ be a quantum set. Then, there is a one-to-one correspondence between $*$-homomorphisms $\Delta\: c_c(\X) \to \mathrm{Mult}(c_c(\X) \atensor c_c(\X))$ such that $(c_c(\X), \Delta)$ is a discrete quantum group in the sense of \cite{VanDaele}*{Def.~2.3} and pairs of functions, $F\: \X \times \X \to \X$ and $C: \mathbf 1 \to \X$, such that
\begin{enumerate}
\item $\[(\forall( x_1 \feq x_{1*}) \fin \X \ftimes \X^*)\,(\forall (x_2 \feq x_{2*}) \fin \X \ftimes \X^*)\,(\forall (x_3 \feq x_{3*}) \fin \X \ftimes \X^*), \\ \ghost \hfill  E_\X(F(F(x_1, x_2), x_3), F_*(x_{1*},F_*(x_{2*}, x_{3*}))\] = \top$;
\item $\[(\forall (x \feq x_*) \fin \X \ftimes \X^*)\, E_\X(F(x, C),x_*)\] = \top$;
\item $\[(\forall (x \feq x_*) \fin X \ftimes \X^*)\, E_\X(F(C,x), x_*)\] = \top$;
\item $\[(\forall x_1 \fin \X)\, (\exists x_2 \fin \X)\, E_\X(F(x_1, x_2), C_*)\]= \top$;
\item $\[(\forall x_2 \fin \X)\, (\exists x_1 \fin \X)\, E_\X(F(x_1, x_2), C_*)\] = \top$.
\end{enumerate}
This correspondence is given by $\Delta = F^\star|_{c_c(\X)}$ and $\varepsilon = C^\star|_{c_c(\X)}$, where $\varepsilon\: c_c(\X) \to \CC$ is the counit of $(c_c(\X), \Delta)$ \cite{VanDaele}*{sec.~3}.
\end{corollary}

\begin{proof}[Proof (cf.~\cite{Daws})]
Let $F$ be a function $\X \times \X \to \X$, and let $C$ be a function $\mathbf 1 \to \X$. By Lemma \ref{computation.F.3} and Proposition \ref{computation.F.4}, conditions (1), (2) and (3) are equivalent to $F\circ(F \times I_\X) = F \circ (I_\X \times F)$, $F\circ(I_\X \times C) = I_\X$ and $F\circ (C \times I_\X) = I_\X$, respectively. Therefore, by Theorem \ref{examples.H.5}, conditions (1)\ndash(5) imply that $(c_c(\X), F^\star|_{c_c(\X)})$ is a discrete quantum group.

Conversely, let $\Delta\: c_c(\X) \to \Mult(c_c(\X) \atensor c_c(\X))$ be a $*$-homomorphism, and assume that $(c_c(\X), \Delta)$ is a discrete quantum group. Let $a_0 \in c_c(\X)$. Let $A \subsetof c_c(\X)$ be the $*$-algebra generated by $a_0$. Because $a_0(X)=0$ for all but finitely many atoms $X$, we know that $A$ is a finite-dimensional C*-algebra. As observed in \cite{VanDaele}*{sec.~2}, the multiplier algebra $\Mult(c_c(\X) \atensor c_c(\X))$ is canonically isomorphic to $\ell(\X \times \X)$ \cite{Kornell}*{Def.~5.1}. For all atoms $X_1, X_2 \in \At(\X)$, the function $A \to L(X_1 \tensor X_2)$ that is defined by $a \mapsto \Delta(a)(X_1 \tensor X_2)$ is a $*$-homomorphism between finite-dimensional C*-algebras, and therefore $\|\Delta(a_0)(X_1 \tensor X_2)\| \leq \|a_0\|$. We conclude that $\|\Delta(a_0)\| \leq \|a_0\|$. Therefore, $\Delta(c_c(\X)) \subsetof \ell^\infty(\X \times \X)$.

The $*$-homomorphism $\Delta\: c_c(\X) \to \ell^\infty(\X \times \X)$ is bounded, and thus, it extends uniquely to a $*$-homomorphism $\Delta_0\: c_0(\X) \to \ell^\infty(\X \times \X)$. This is a nondegenerate representation of the C*-algebra $c_0(\X)$ because $(c_c(\X) \atensor 1) \cdot \Delta(c_c(\X)) = c_c(\X) \atensor c_c(\X)$ by the definition of a discrete quantum group. The enveloping von Neumann algebra of $c_0(\X)$ is of course $\ell^\infty(\X)$, and hence, $\Delta_0$ extends to a unital normal $*$-homomorphism $\Delta_1 \: \ell^\infty(\X) \to \ell^\infty(\X \times \X)$ \cite{Pedersen}*{Thm.~3.7.7}. Thus, we obtain a function $F\: \X \times \X \to \X$ such that $F^\star$ extends $\Delta$.

Let $b \in c_c(\X)$. Appealing to the definition of a comultiplication \cite{VanDaele}*{Def.~2.1}, we calculate that for all atoms $X_1$ and $X_2$,
\begin{align*}
([X_1] \tensor 1 \tensor 1)\cdot (F^\star \stensor \idhom)(F^\star(b)) \cdot (&1 \tensor 1 \tensor [X_2])
 \\ & =
([X_1]\tensor 1 \tensor 1) \cdot (F^\star \stensor \idhom)(F^\star(b)\cdot (1 \tensor [X_2]))
\\ & =
(\idhom \stensor F^\star)(([X_1] \tensor 1)\cdot F^\star(b))\cdot(1 \tensor 1 \tensor [X_2])
\\ &=
([X_1] \tensor 1 \tensor 1) \cdot (\idhom \stensor F^\star)(F^\star(b)) \cdot (1 \tensor 1 \tensor [X_2]).
\end{align*}
Therefore, for all $b \in c_c(\X)$, $(F^\star \stensor \idhom)(F^\star(b)) = (\idhom \stensor F^\star)(F^\star(b))$. Because $c_c(\X)$ is ultraweakly dense in $\ell^\infty(\X)$, we conclude that $(F \times I_\X) \circ F = (I_\X \times F) \circ F$, establishing condition (1).

The discrete quantum group has a counit $\varepsilon\: c_c(\X) \to \CC$ \cite{VanDaele}*{sec.~3}. By elementary algebra, $\varepsilon = J_{X}^\star|_{c_c(\X)}$ for some one-dimensional atom $X$, and without loss of generality, we may assume that $\varepsilon = J_\CC^\star|_{c_c(\X)}$. Let $C = J_\CC$, so that $\varepsilon = C^\star|_{c_c(\X)}$. Appealing to the definition of a counit, we calculate that for all $a \in c_c(\X)$ and all $X \in \At(\X)$,
\begin{align*}
(C^\star\stensor \idhom)(F^\star(a)) \cdot [X]
=
(C^\star \stensor \idhom)(F^\star(a) \cdot (1 \tensor [X]))
=
a \cdot [X].
\end{align*}
Therefore, for all $a \in c_c(\X)$, $(C^\star\stensor \idhom)(F^\star(a)) = a$. Because $c_c(\X)$ is ultraweakly dense in $\ell^\infty(\X)$, we conclude that $F\circ (C \times I_\X) = I_\X$. Similarly, $F \circ (I_\X \times C) = I_\X$. We have established conditions (2) and (3).

Let $p \in \ell^\infty(\X)$ be a nonzero projection. Hence, we have an atom $X$ such that $p \cdot [X]  \neq 0$. It certainly follows that $[\CC]\tensor (p \cdot [X])  \neq 0$. Appealing directly to the definition of a discrete quantum group \cite{VanDaele}*{Def.~2.3}, we infer that $\Delta([\CC])(1 \tensor (p \cdot [X]))  \neq 0$ and thus that $ \Delta([\CC]) \cdot (1 \tensor p) \neq 0$. Therefore, $\Delta([\CC])$ is not orthogonal to $1 \tensor p$ for any projection $p \neq 0$. Equivalently, $\Delta([\CC])$ is not below $1 \tensor r$ for any projection $r \neq 1$.

We have already observed in the proof of Theorem \ref{examples.H.5} that the projection $\Delta([\CC])$ corresponds to the relation $\[E_\X(F(x_1, x_2), C_*)\]$ in the sense of \cite{Kornell}*{Thm.~B.8}, so $\[ E_\X(F(x_1, x_2), C_*)\]$ is not below $\top_\X \times R$ for any $R \neq \top_\X$. Therefore, by Proposition \ref{definition.D.2}(2),
\begin{align*}
\[ x_2 \in \X \suchthat (\exists x_1 \in \X_1)\,& E_\X(F(x_1, x_2), C_*)\]
\\ & =
\inf \{R \in \Rel(\X) \suchthat \top_\X \times R \geq \[ E_\X(F(x_1, x_2), C_*)\] \}
=
\top_\X.
\end{align*}
We conclude by Lemma \ref{computation.C.1} that $\[(\forall x_2 \in \X)\, (\exists x_1 \in \X)\, E_\X(F(x_1, x_2), C_*)\] = \top$. Similarly, $\[(\forall x_1 \in \X)\, (\exists x_2 \in \X)\, E_\X(F(x_1, x_2), C_*)\] = \top$. We have established conditions (4) and (5).

The $*$-algebra $c_c(\X)$ is ultraweakly dense in $\ell^\infty(\X)$, and thus, the equation $\Delta = F^\star|_{c_c(\X)}$ defines a bijection between 
\begin{itemize}
\item $*$-homomorphisms $\Delta\: c_c(\X) \to \Mult(c_c(\X) \atensor c_c(\X))$ such that $(c_c(\X), \Delta)$ is a discrete quantum group and
\item functions $F\: \X \times \X \to \X$ for which there exists a function $C\: \mathbf 1 \to \X$ such that together $F$ and $C$ satisfy conditions (1)\ndash(5).
\end{itemize}
The function $C\: \mathbf 1 \to \X$ is easily seen to be unique because conditions (2) and (3) imply that $F \circ ( I_\X \times C) = I_\X$ and $ F \circ (C \times I_\X) = I_\X$. We constructed $C$ to satisfy $\varepsilon = C^\star|_{c_c(\X)}$, where $\varepsilon$ is the counit of the discrete quantum group $(c_c(\X), \Delta)$. Hence the theorem is proved.
\end{proof}

\appendix

\section*{appendix}

\setcounter{section}{1}

\subsection{Nondegenerate equality.}\label{appendix.D} We show that the equality relation on a von Neumann algebra is nondegenerate if and only if that von Neumann algebra is hereditarily atomic.

\begin{lemma}\label{appendix.D.1}
Let $M$ be a commutative von Neumann algebra that contains no minimal projections. There exists no normal state $\varphi$ on the spatial tensor product $M \stensor M$ such that $\varphi(p \tensor (1-p)) = 0$ for every projection $p \in M$.
\end{lemma}

\begin{proof}
Suppose that we have such a normal state $\varphi$ on $M \stensor M$. Let $\varphi_1$ and $\varphi_2$ be the normal states on $M$ defined by $\varphi_1(a) = \varphi(a \tensor 1)$ and $\varphi_2(a) = \varphi (1 \tensor a)$ for all $a \in M$. For both $i \in \{1,2\}$, let $p_i$ be the support projection of $\varphi_i$, in other words, the smallest projection in $M$ such that $\varphi_i(p_i) = 1$. It is easy to see that $\varphi_i$ is faithful on $p_i M$. Indeed, for any projection $q \leq p_i$, if $\varphi_i(q) = 0$, then $\varphi_i(p_i-q) = 1$, which implies that $q =0$, by the minimality of $p_i$. We conclude that $\varphi_1 \stensor \varphi_2$ is a faithful normal state on $p_1 M \stensor p_2 M$ \cite{Blackadar}*{III.2.2.31}.

Our given normal state $\varphi$ factors through $p_1 M \stensor p_2 M$, as we now show. Indeed, by our choice of $p_1$ and $p_2$, we have that $\varphi(p_1 \tensor 1) = 1 = \varphi (1 \tensor p_2)$. Writing $\varphi$ as a countable linear combination of vector states, we find that $\varphi(p_1 \tensor p_2) = 1$ and furthermore that $\varphi((p_1 \tensor p_2) b) = \varphi(b)$ for all $b \in M \stensor M$. Thus, $\varphi$ does factor through $p_1 M \stensor p_2 M$, as claimed.

Finite partitions of the identity $1 \in M$ into pairwise orthogonal projections form a directed set $\Lambda$, with finer partitions appearing higher in the order. For each such partition $\lambda \in \Lambda$, we define a projection $q_\lambda = \sum_{p \in \lambda} p \tensor p$. The net $(q_\lambda\suchthat \lambda \in \Lambda)$ is evidently decreasing, and it therefore has an ultraweak limit $q_\infty$, also a projection in $M \stensor M$. By our assumption on $\varphi$, we have that $\varphi(q_\lambda) = 1$ for each partition $\lambda$, and therefore $\varphi(q_\infty) = 1$. We conclude that $\varphi((p_1 \tensor p_2) q_\infty) = 1$.

We now obtain a contradiction by showing that $(p_1 \tensor p_2)  q_\infty = 0$.
\begin{align*}
(\varphi_1 \stensor \varphi_2)((p_1 \tensor p_2)  q_\infty) & = (\varphi_1 \stensor \varphi_2)(\lim_\lambda (p_1 \tensor p_2)  q_\lambda )
 =
\lim_\lambda  (\varphi_1 \stensor \varphi_2)((p_1 \tensor p_2)  q_\lambda )
\\ & =
\lim_\lambda  \sum_{p \in \lambda} \varphi_1 ( p_1  p)  \varphi_2(p_2 p )
=0
\end{align*}
The final equality is a consequence of the fact that for both $i \in \{1,2\}$, we can partition $p_i$ into projections $p$ that are arbitrarily small in the sense that each satisfies $\varphi_i(p) \leq \epsilon$ for arbitrarily small $\epsilon>0$. Indeed, $p_i$ is the identity of the von Neumann algebra $p_i M$, which has no atoms and on which $\varphi_i$ is a faithful normal state. Since $\varphi_1 \stensor \varphi_2$ is faithful, we conclude that $(p_1 \tensor p_2)  q_\infty = 0$, as claimed. Having obtained a contradiction, we infer that our opening supposition is false.
\end{proof}

\begin{theorem}\label{appendix.D.2}
Let $M$ be any von Neumann algebra. Let $\delta$ be the largest projection in the spatial tensor product $M \stensor M^{op}$ such that $(p \tensor (1-p))  \delta = 0$ for all projections $p$ in $M$. Then, $(p_1 \tensor p_1)\delta \neq 0$ for all nonzero projections $p_1$ in $M$ if and only if $M$ is hereditarily atomic.
\end{theorem}

\begin{proof}
Let $H$ be the Hilbert space on which $M$ is canonically represented. The von Neumann algebra $M^{op}$ is canonically represented on the conjugate Hilbert space $\overline H$, whose vectors are the same as those of $H$ but written with a conjugation symbol so that $\overline{\alpha h} = \overline \alpha \overline h$ for all $\alpha \in \CC$ and all $h \in H$. The inner product on $\overline H$ is defined by $\langle \overline h_1 | \overline h_2 \rangle = \overline{\langle h_1 | h_2 \rangle}$. For each $a \in M^{op}$ and each $\overline h \in \overline H$, we define $a \overline h = \overline{a^\dagger h}$, where $a^\dagger$ is the Hermitian adjoint of $a$. It is routine to verify that this defines a faithful representation of $M^{op}$ on $\overline H$.

Like any von Neumann algebra, $M$ is the direct sum of a hereditarily atomic von Neumann algebra $M_0$ and a von Neumann algebra $M_1$ that has no finite type I factors as a direct summand. Let $p_0$ and $p_1$ be the central projections in $M$ corresponding to $M_0$ and $M_1$, respectively. Hence, $M_0 = p_0 M$ and $M_1 = p_1M$.

Assume that $M$ is not hereditarily atomic or, in other words, that $p_1$ is nonzero. By our assumption on $\delta$, we have that $(p_0 \tensor p_1) \delta = 0$ and $(p_1 \tensor p_0) \delta = 0$, and therefore $\delta = (p_0 \tensor p_0)\delta + (p_1 \tensor p_1) \delta$. The projection $\delta_1 := (p_1 \tensor p_1) \delta$ is in $M_1 \stensor M_1^{op}$, and it satisfies $(p \tensor (p_1 -p))  \delta_1 = 0$ for all projections $p$ in $M_1$.

Assume for contradiction that $\delta_1$ is nonzero. It follows that there is a vector $w$ in the Hilbert space $(p_1 H) \tensor (p_1 \overline H)$ such that $(p \tensor (p_1 -p)) w = 0$ for all projections $p$ in $M_1$. Thus, we have a state $\varphi$ on $M_1 \stensor M_1^{op}$ such that $\varphi(p \tensor (p_1 - p)) =0$ for all projections $p$ in $M_1$.

The algebra $M_1$ need not be commutative, but it contains a unital ultraweakly closed $*$-subalgebra that is both commutative and \emph{diffuse} in the sense that it contains no minimal projections. Indeed, the center of $M_1$ is the direct sum of an atomic von Neumann algebra and a diffuse von Neumann algebra. It follows that $M_1$ is a direct sum of von Neumann algebras, each of which is either a factor that is not finite type I or a von Neumann algebra with diffuse center. Each such direct summand has a diffuse commutative unital ultraweakly closed $*$-subalgebra, and hence, so does $M_1$.

Let $N$ be any diffuse commutative unital ultraweakly closed $*$-subalgebra of $M_1$. The normal state $\varphi$ restricts to a normal state on $N \stensor N^{op}$ such that $\varphi(p \tensor (p_1 - p)) =0$ for all projections $p$ in $N$, and $p_1$ is the multiplicative unit of $N$ because $N$ is a unital $*$-subalgebra of $M_1$. Furthermore, since $N$ is commutative, $N^{op} = N$. Therefore, we may apply Lemma \ref{appendix.D.1} to obtain a contradiction. We conclude that $\delta_1 =0$, that is, $(p_1 \tensor p_1)\delta =0$. Therefore, if $M$ is not hereditarily atomic, then there does exists a nonzero projection $p_1$ such that $(p_1 \tensor p_1) \delta =0$.

Assume now that $M$ is a hereditarily atomic von Neumann algebra, and let $p_1$ be a nonzero projection in $M$. By \cite{Kornell}*{Prop.~5.4}, there exists a set $A$ of finite-dimensional Hilbert spaces such that $M$ is isomorphic to the $\ell^\infty$-direct sum of the operator algebras $L(X)$, for $X \in A$. Without loss of generality, we may assume that $M$ is equal to such an $\ell^\infty$-direct sum. For each $X \in A$, let $[X]$ be the corresponding minimal central projection in $M$.

Let $X_1$ be such that $p_1[X_1] \neq 0$. Choose an orthonormal basis $x_1, \ldots, x_n$ for $X_1$, and let $w = \sum_{i=1}^n x_i \tensor \overline x_i \in X_1 \tensor \overline X_1 \leq H \tensor \overline H$. Let $[\CC w]$ be the corresponding projection in $M \stensor M^{op}$. Using standard linear algebra, we may directly compute $\< w | (p_1 \tensor p_1) w\>$ to show that $(p_1 \tensor p_1)w \neq 0$; thus, $(p_1 \tensor p_1) [\CC w] \neq 0$. Similarly, for each projection $p \in M$, we may directly compute $\< w | (p \tensor (1-p)) w\>$ to show that $(p \tensor (1-p))w = 0$; thus, $[\CC w] \leq \delta$. Altogether, we find that $(p_1 \tensor p_1) \delta =0$ for any nonzero projection $p_1$ in $M$.
\end{proof}

\subsection{Weaver's quantum relations}\label{appendix.E} We substantiate the observation \cite{Kornell} that binary relations between quantum sets are essentially just Weaver's quantum relations \cite{Weaver2}. A quantum relation from a von Neumann algebra $M \leq L(H)$ to a von Neumann algebra $N \leq L(K)$ is defined to be an ultraweakly closed subspace $V \leq L(H,K)$ such that $N'\cdot V \cdot M' \leq V$. 

For each atom $X$ of a quantum set $\X$, we write $\incnag_X \in L(X, \bigoplus \At(\X))$ for the corresponding inclusion isometry.

\begin{proposition}\label{appendix.E.1}
Let $\X$ and $\Y$ be quantum sets. Let the von Neumann algebras $\ell^\infty(\X)$ and $\ell^\infty(\Y)$ be canonically represented on the Hilbert spaces $\bigoplus\At(\X)$ and $\bigoplus \At (\Y)$ respectively. Quantum relations $V$ from $\ell^\infty(\X)$ to $\ell^\infty(\Y)$ are in one-to-one correspondence with binary relations $R$ from $\X$ to $\Y$. The correspondence is given by $R(X, Y) = \inc_Y^\dagger \cdot V \cdot \inc_X^{\phantom \dagger}$, for $X \in \At(\X)$ and $Y \in \At(\Y)$.
\end{proposition}

\begin{proof}
For each quantum relation $V$ from $\ell^\infty(\X)$ to $\ell^\infty(\Y)$, let $R_V$ be the binary relation from $\X$ to $\Y$ defined by $R_V(X,Y) = \incdag_Y \cdot V \cdot \incnag_X$, for $X \in \At(\X)$ and $Y \in \At(\Y)$. For each binary relation $R$ from $\X$ to $\Y$, let $V_R \leq L(\bigoplus \At(\X), \bigoplus \At(\Y))$ be defined by $V_R = \overline{\sum\{ \incnag_Y \cdot R(X,Y) \cdot \incdag_X \suchthat X \in \At(\X),\, Y \in \At(\Y)\}}$, where the symbol $\sum$ denotes the algebraic span of the union and the line indicates closure with respect to the ultraweak topology. This is a quantum relation from $\ell^\infty(\X)$ to $\ell^\infty(\Y)$ because $\ell^\infty(\X)'$ is the closed span of the minimal central projections $[X] := \incnag_X \cdot \incdag_X$ for $X \in \At(\X)$, and it is likewise for $\ell^\infty(\Y)'$. Indeed, for all atoms $X_0 \in \At(\X)$ and $Y_0 \in \At(\Y)$, we calculate that $[Y_0] \cdot V_R \cdot [X_0] \leq \overline{\incnag_{Y_0} \cdot R(X_0, Y_0) \cdot \incdag_{X_0}} = \incnag_{Y_0} \cdot R(X_0, Y_0) \cdot \incdag_{X_0} \leq V_R$, and therefore $\ell^\infty(\Y)' \cdot V_R \cdot \ell^\infty(\X)' \leq V_R$.

We show that the two constructions invert each other by direct calculation. For each binary relation $R$ from $\X$ to $\Y$, and all atoms $X_0 \in \At(\X)$ and $Y_0 \in \At(Y)$, we calculate that
\begin{align*}
R_{V_R}(X_0, Y_0) & = \incdag_{Y_0} \cdot V_R \cdot \incnag_{X_0} \leq \overline{ \incdag_{Y_0} \cdot \incnag_{Y_0} \cdot R(X_0, Y_0) \cdot \incdag_{X_0} \cdot \incnag_{X_0}} = R(X_0, Y_0) 
\\ &
= \incdag_{Y_0} \cdot \incnag_{Y_0} \cdot R(X_0, Y_0) \cdot \incdag_{X_0} \cdot \incnag_{X_0}
\leq \incdag_{Y_0} \cdot V_R \cdot \incnag_{X_0} = R_{V_R}(X_0, Y_0).
\end{align*}
Similarly, for each quantum relation $V$ from $\ell^\infty(\X)$ to $\ell^\infty(\Y)$, we calculate that
\begin{align*}
V_{R_V} & = \overline{\sum\{\incnag_Y \cdot R_V(X,Y) \cdot \incdag_X \suchthat X \in \At(\X),\, Y \in \At(\Y)\}}
\\ & =
\overline{\sum\{\incnag_Y \cdot \incdag_Y \cdot V \cdot \incnag_X \cdot \incdag_X \suchthat X \in \At(\X),\, Y \in \At(\Y)\}}
\\ & =
\overline{\sum\{[Y] \cdot V \cdot [X] \suchthat X \in \At(\X),\, Y \in \At(\Y)\}}
=
V.
\end{align*}
The last equality can be proved by establishing both inclusions. The inclusion of the left side into the right side holds because $V$ is a quantum relation. The inclusion of the right side into the left side holds because the projections $[X]$ for $X \in \At(\X)$ sum to the identity, as do the projections $[Y]$ for $Y \in \At(\Y)$. Therefore, the constructions $R \mapsto V_R$ and $V \mapsto R_V$ invert each other.
\end{proof}

\begin{proposition}\label{appendix.E.2}
The one-to-one correspondence of Proposition \ref{appendix.E.1} is functorial.
\end{proposition}

\begin{proof}
Let $\X$, $\Y$ and $\Z$ be quantum sets, let $V$ be a quantum relation from $\ell^\infty(\X)$ to $\ell^\infty(\Y)$, and let $W$ be a quantum relation from $\ell^\infty(\Y)$ to $\ell^\infty(\Z)$. In the notation of the proof of Proposition \ref{appendix.E.1}, we are to show that $R_{\overline{W \cdot V}} = R_W \circ R_V$. For all atoms $X \in \At(\X)$ and $Z \in \At(\Z)$, we calculate that 
\begin{align*}
(R_W&\circ R_V)(X,Z)
=
\sum_{Y \in \At(\Y)} R_W(Y,Z) \cdot R_V(X,Y)
=
\sum_{Y \in \At(\Y)} \incdag_Z \cdot W \cdot \incnag_Y \cdot \incdag_Y \cdot V \cdot \incnag_X
\\ & \leq
\incdag_Z \cdot W \cdot V \cdot \incnag_X
 \leq
\incdag_Z \cdot \overline{W \cdot V} \cdot \incnag_X
 =
R_{\overline{W \cdot V}}(X, Z)
 \leq 
\overline{\incdag_Z \cdot W \cdot V \cdot \incnag_X}
\\ & =
\overline{\incdag_Z \cdot W \cdot 1_{\oplus \At(Y)} \cdot V \cdot \incnag_X}
\leq
\overline{\sum_{Y \in \At(\Y)} \incdag_Z \cdot W \cdot \incnag_Y \cdot \incdag_Y  \cdot V \cdot \incnag_X}
\\ & =
\overline{(R_W \circ R_V)(X,Z)}
=
(R_W \circ R_V)(X,Z).\qedhere
\end{align*}
\end{proof}

It is immediate from the definition of this one-to-one correspondence in Proposition \ref{appendix.E.1} that it preserves the partial order relation and the adjoint operation. Thus, we obtain an enriched equivalence of dagger categories from the category of quantum sets and binary relations to the category of hereditarily atomic von Neumann algebras and quantum relations.

\subsection{Permutation equivariance}\label{appendix.B} We prove Proposition \ref{definition.C.3}.

\begin{lemma}\label{appendix.B.1}
Let $\X$, $\Y$ and $\Z$ be quantum sets.
\begin{enumerate}
\item For every predicate $P$ on $\X$, we have $\neg(P \times \top_\Y) = (\neg P) \times \top_\Y$.
\item For all predicates $P_1$ and $P_2$ on $\X$, and all predicates $Q_1$ and $Q_2$ on $\Y$, we have $(P_1 \times Q_1) \AND (P_2 \times Q_2) = (P_1 \AND P_2) \times (Q_1 \AND Q_2)$.
\item Let $R$ be a predicate on $\X \times \Y$, let $Q$ be the largest predicate on $\Y$ such that $\top_\X \times Q \leq R$, and let $S$ be the largest predicate on $\Y \times \Z$ such that $\top_\X \times S \leq R \times \top_\Z$. Then, $S = Q \times \top_\Z$.
\end{enumerate}
\end{lemma}

\begin{proof}
All three claims are established most easily using the bijective correspondence between the predicates on any quantum set $\W$ and the projections in the corresponding von Neumann algebra $\ell^\infty(\W)$. Expressed in terms of projections, claims (1) and (2) are elementary. To prove claim (3), let $Q$, $R$, and $S$ correspond to projections $q \in \ell^\infty(\Y)$, $r \in \ell^\infty(\X) \stensor \ell^\infty(\Y)$, and $s \in \ell^\infty(\Y) \stensor \ell^\infty(\Z)$, respectively. The von Neumann algebra $\ell^\infty(\Z)$ is canonically represented on a Hilbert space $H$, the $\ell^2$-direct sum of the atoms of $\Z$. If $\Z$ has no atoms, then claim (3) holds trivially, so we may assume that $H$ is nonzero. We are essentially given that $q = \sup\{p \in \mathrm{Proj}(\ell^\infty(\Y)) \suchthat 1 \tensor p \leq r \}$, and $s= \sup\{p \in \mathrm{Proj}(\ell^\infty(\Y)\stensor \ell^\infty(\Z)) \suchthat 1 \tensor p \leq r \tensor 1\}$. In particular $1 \tensor q \leq r$, so $ 1 \tensor q \tensor 1 \leq r \tensor 1$, giving $q \tensor 1  \leq s$.

For the opposite inequality, we consider the projection $\tilde s$, defined to be the supremum of all projections $p$ in $\ell^\infty(\Y) \stensor L(H)$ satisfying the inequality $1 \tensor p \leq r \tensor 1$, where $L(H)$ is the von Neumann algebra of all bounded operators on $H$. If a projection $p$ satisfies the inequality $1 \tensor p \leq r \tensor 1$, then so does the projection $(1 \tensor u^\dagger)\cdot p\cdot(1 \tensor u)$ for every unitary operator $u \in L(H)$. It follows that $(1 \tensor u^\dagger) \cdot  \tilde s \cdot (1 \tensor u) = \tilde s$ for all unitaries $u \in L(H)$, so $\tilde s$ is in the commutant $(\CC \stensor L(H))'$. Since $\tilde s$ is also in $\ell^\infty(\Y) \stensor L(H)$, we conclude that $\tilde s$ is in $\ell^\infty(\Y) \stensor \CC$, a von Neumann subalgebra of $\ell^\infty(\Y) \stensor \ell^\infty(\Z)$. Therefore, $s = \tilde s = p_1 \tensor 1$, for some projection $p_1$ in $\ell^\infty(\Y)$. We now calculate that $1 \tensor p_1 \tensor 1 = 1 \tensor s \leq r \tensor 1$, which implies that $1 \tensor p_1 \leq r$, giving us $p_1 \leq q$, by the definition of $q$ as a supremum of projections satisfying this inequality. Finally, we obtain $s = p_1 \tensor 1 \leq q \tensor 1$.
\end{proof}

\begin{proposition}\label{appendix.B.2}
Let $\X_1, \ldots, \X_p$ be quantum sets, and let $x_1, \ldots, x_p$ be distinct variables of sorts $\X_1, \ldots, \X_p$, respectively. For each permutation $\sigma$ of $\{1, \ldots, p\}$ and each primitive formula $\phi(x_1, \ldots, x_n)$ with $n \leq p$, we have
\begin{align*}
\[ (x_{\sigma(1)}, & \ldots, x_{\sigma(p)})  \in \X_{\sigma(1)} \times \cdots \times \X_{\sigma(p)}
\suchthat
\phi (x_1, \ldots, x_n) \]
\\ & =
(\sigma \inv)_\# (
\[ (x_1, \ldots, x_n) \in \X_1 \times \cdots \times \X_n
\suchthat
\phi(x_1, \ldots, x_n) \]
\times \top_{\X_{n+1}} \times \cdots \times \top_{\X_{p}}).
\end{align*}
\end{proposition}

\begin{proof}
The proof proceeds by structural induction. To clarify the calculations, we introduce the notation $\Y_i = \X_{\sigma(i)}$ and $y_i = x_{\sigma(i)}$ for $i \in \{1, \ldots, p\}$.

Suppose that $\phi(x_1, \ldots, x_n)$ is atomic. In that case, $\phi(x_1, \ldots, x_n)$ is necessarily of the form $R(x_{\pi(1)}, \ldots, x_{\pi(m)})$, for some natural $m \leq n$, for some permutation $\pi$ of the set $\{1, \ldots, n\}$, and for some relation $R$ of arity $(\X_{\pi(1)}, \ldots, \X_{\pi(m)})$. We may extend $\pi$ to a permutation $\tilde \pi$ of the set $\{1, \ldots, p\}$ by defining $\tilde \pi(k) = k$ for all $k$ in $\{n+1, \ldots, p\}$.
\begin{align*}
\[ (x_{\sigma(1)}, & \ldots, x_{\sigma(p)}) \in \X_{\sigma(1)} \times \cdots \times \X_{\sigma(p)}
\suchthat R(x_{\pi(1)}, \ldots, x_{\pi(m)} ) \]
\\ & =
\[ (x_{\sigma(1)}, \ldots, x_{\sigma(p)}) \in \X_{\sigma(1)} \times \cdots \times \X_{\sigma(p)}
\suchthat R(x_{\tilde \pi(1)}, \ldots, x_{\tilde \pi(m)}) \]
\\ & =
\[ (y_1, \ldots, y_p) \in \Y_1 \times \cdots \times \Y_p \suchthat
R(y_{(\sigma \inv \circ \tilde \pi)(1)}, \ldots, y_{(\sigma \inv \circ \tilde \pi)(m)}) \]
\\ & = 
(\sigma \inv \circ \tilde \pi)_\# (R \times 
\top_{\Y_{(\sigma \inv \circ \tilde \pi)(m+1)}} \times \cdots \times \top_{\Y_{(\sigma \inv \circ \tilde \pi)(p)}} )
\\ & =
(\sigma \inv)_\# ( \tilde \pi_\#( R \times 
\top_{\X_{\tilde \pi(m+1)}} \times \cdots \times \top_{\X_{\tilde \pi (p)}} ))
\\ &=
(\sigma \inv)_\# ( \tilde \pi_\#( R \times 
\top_{\X_{\pi(m+1)}} \times \cdots \times \top_{\X_{\pi(n)}}
\times \top_{\X_{n+1}} \times \cdots \times \top_{\X_p} ))
\\ & =
(\sigma \inv)_\# ( \pi_\#( R \times
\top_{\X_{\pi(m+1)}} \times \cdots \times \top_{\X_{\pi(n)}})
\times \top_{\X_{n+1}} \times \cdots \times \top_{\X_p})
\\ & =
(\sigma \inv)_\# (
\[ (x_1, \ldots, x_n) \in \X_1 \times \cdots \times \X_n
\suchthat
R(x_{\pi(1)}, \ldots, x_{\pi(m)} ) \]
\times \top_{\X_{n+1}} \times \cdots \times \top_{\X_{p}}).
\end{align*}

Suppose that $\phi(x_1, \ldots, x_n)$ is of the form $\NOT \psi(x_1, \ldots, x_n)$ for some nonduplicating formula $\psi(x_1, \ldots, x_n)$.
\begin{align*}
\[ (x_{\sigma(1)}, & \ldots, x_{\sigma(p)}) \in \X_{\sigma(1)} \times \cdots \times \X_{\sigma(p)}
\suchthat \NOT \psi(x_1, \ldots, x_n) \]
\\ & =
\[ (y_1, \ldots, y_p) \in \Y_1 \times \cdots \times \Y_p \suchthat
\NOT \psi(y_{\sigma\inv(1)}, \ldots, y_{\sigma\inv(n)}) \]
\\ & =
\NOT \[ (y_1, \ldots, y_p) \in \Y_1 \times \cdots \times \Y_p \suchthat
 \psi(y_{\sigma\inv(1)}, \ldots, y_{\sigma\inv(n)}) \]
\\ & = 
\NOT \[ (x_{\sigma(1)}, \ldots, x_{\sigma(p)}) \in \X_{\sigma(1)} \times \cdots \times \X_{\sigma(p)}
\suchthat \psi(x_1, \ldots, x_n) \]
\\ & = 
\NOT (\sigma \inv)_\# (
\[ (x_1, \ldots, x_n) \in \X_1 \times \cdots \times \X_n
\suchthat
\psi(x_1, \ldots, x_n) \]
\times \top_{\X_{n+1}} \times \cdots \times \top_{\X_{p}})
\\ & = 
(\sigma \inv)_\# (\NOT(
\[ (x_1, \ldots, x_n) \in \X_1 \times \cdots \times \X_n
\suchthat
\psi(x_1, \ldots, x_n) \]
\times \top_{\X_{n+1}} \times \cdots \times \top_{\X_{p}}))
\\ & = 
(\sigma \inv)_\# ( \NOT
\[ (x_1, \ldots, x_n) \in \X_1 \times \cdots \times \X_n
\suchthat
\psi(x_1, \ldots, x_n) \]
\times \top_{\X_{n+1}} \times \cdots \times \top_{\X_{p}})
\\ & = 
(\sigma \inv)_\# ( 
\[ (x_1, \ldots, x_n) \in \X_1 \times \cdots \times \X_n
\suchthat
\NOT \psi(x_1, \ldots, x_n) \]
\times \top_{\X_{n+1}} \times \cdots \times \top_{\X_{p}}).
\end{align*}
\noindent We apply Lemma \ref{appendix.B.1}(1) in the second-to-last equality. The case in which $\phi(x_1, \ldots, x_n)$ is of the form $\psi_1(x_1, \ldots, x_n) \AND \psi_2(x_1, \ldots, x_n)$ is entirely similar; there, we apply Lemma \ref{appendix.B.1}(2).

Suppose that $\phi(x_1, \ldots, x_n)$ is of the form $(\forall x_0 \fin \X_0)\, \psi(x_0, \ldots, x_n)$ for some quantum set $\X_0$ and some variable $x_0$ that is distinct from the variables $x_1, \ldots, x_p$, and that has sort $\X_0$. We extend the permutation $\sigma$ to a permutation $\tilde \sigma$ of $\{0, \ldots, p\}$ by setting $\tilde \sigma (0) = 0$, and we write $\Y_0 = \X_0$ and $y_0 = x_0$.
\begin{align*}
& \[ (x_{\sigma(1)}, \ldots, x_{\sigma(p)}) \in \X_{\sigma(1)} \times \cdots \times \X_{\sigma(p)}
\suchthat (\forall x_0 \fin \X_0)\, \psi(x_0, \ldots, x_n) \]
\\ & =
\[ (y_1, \ldots, y_p) \in \Y_1 \times \cdots \times \Y_p \suchthat
(\forall y_0 \fin \Y_0)\, \psi(y_0, y_{\sigma\inv(1)}, \ldots, y_{\sigma\inv(n)}) \]
\\ & =
\sup
\{ R \in \Rel\{ \Y_1, \ldots, \Y_p \} \suchthat 
\\ & \quad
\top_{\Y_0} \times R \leq
\[(y_0, \ldots, y_p) \in \Y_0 \times \cdots \times \Y_p \suchthat \psi(y_0, y_{\sigma\inv(1)}, \ldots, y_{\sigma\inv(n)})
\]
\}
\\ & =
\sup
\{ R \in \Rel\{ \X_{\sigma(1)}, \ldots, \X_{\sigma(p)} \} \suchthat 
\\ & \quad \top_{\X_0} \times R \leq
\[(x_{\tilde \sigma (0)}, x_{\tilde \sigma(1)}, \ldots, x_{\tilde \sigma(p)}) \in \X_{\tilde \sigma(0)} \times \X_{\tilde \sigma(1)} \cdots \times \X_{\tilde \sigma (p)} \suchthat \psi(x_0, x_1, \ldots, x_n)
\]
\}
\\ & =
\sup
\{ R \in \Rel\{ \X_{\sigma(1)}, \ldots, \X_{\sigma(p)} \}\suchthat
\\ & \quad \top_{\X_0} \times R \leq
(\tilde \sigma \inv)_\#(\[(x_0, \ldots, x_n) \in \X_0  \times \cdots \times \X_n \suchthat \psi(x_0,  \ldots, x_n)\]
\times \top_{\X_{n+1}} \times \cdots \times \top_{\X_{p}}
)
\}
\\ & =
\sup
\{ R \in \Rel\{ \X_{\sigma(1)}, \ldots, \X_{\sigma(p)} \} \suchthat
\\ & \quad (\tilde \sigma )_\#(\top_{\X_0} \times R )\leq
\[(x_0, \ldots, x_n) \in \X_0  \times \cdots \times \X_n \suchthat \psi(x_0,  \ldots, x_n) \]
\times \top_{\X_{n+1}} \times \cdots \times \top_{\X_{p}}
\}
\\ & =
\sup
\{ R \in \Rel\{ \X_{\sigma(1)}, \ldots, \X_{\sigma(p)} \} \suchthat
\\ & \quad \top_{\X_0} \times \sigma_\#(R )\leq
\[(x_0, \ldots, x_n) \in \X_0  \times \cdots \times \X_n \suchthat \psi(x_0,  \ldots, x_n)\] \times \top_{\X_{n+1}} \times \cdots \times \top_{\X_{p}}
\}
\\ & =
( \sigma \inv)_\# (\sup
\{ R' \in \Rel\{ \X_1, \ldots, \X_{p} \} \suchthat
\\ & \quad \top_{\X_0} \times R' \leq
\[(x_0, \ldots, x_n) \in \X_0  \times \cdots \times \X_n \suchthat \psi(x_0,  \ldots, x_n)\] \times \top_{\X_{n+1}} \times \cdots \times \top_{\X_{p}}
\})
\\ & =
(\sigma \inv)_\# (\sup
\{ R'' \in \Rel\{ \X_1, \ldots, \X_n \} \suchthat
\\ & \quad \top_{\X_0} \times R'' \leq
\[(x_0, \ldots, x_n) \in \X_0  \times \cdots \times \X_n \suchthat \psi(x_0,  \ldots, x_n)\] 
\} \times \top_{\X_{n+1}} \times \cdots \times \top_{\X_{p}} )
\\ & =
(\sigma \inv)_\# (
\[ (x_1, \ldots, x_n) \in \X_1 \times \cdots \times \X_n
\suchthat
(\forall x_0 \fin \X_0 )\,\psi(x_0, \ldots, x_n) \]
\times \top_{\X_{n+1}} \times \cdots \times \top_{\X_{p}}).
\end{align*}
We apply Lemma \ref{appendix.B.1}(3) in the second-to-last equality.
\end{proof}

\subsection{Relating $E_\X$ and $\delta_M$}\label{appendix.H}
Let $\X$ be a quantum set. The von Neumann algebra $\ell^\infty(\X) \stensor \ell^\infty(\X)^{op}$ is canonically isomorphic to $\ell^\infty(\X \times \X^*)$ via the ultraweakly continuous map $\phi$ that is defined by $\phi(a_1 \tensor a_2)(X_1 \tensor X_2^*) = a_1(X_1) \tensor a_2(X_2)^*$, for $a_1, a_2 \in \ell^\infty(\X)$ and $X_1, X_2 \in \At(\X)$. The map $\phi$ exists and is a unital normal $*$-homomorphism because the spatial tensor product of hereditarily atomic von Neumann algebras is also their categorical tensor product \cite{Guichardet}. Clearly $\phi$ maps no minimal central projections to $0$, so it is injective on its entire domain. It is also clearly surjective onto each factor of $\ell^\infty(\X \times \X^*)$, so it is surjective onto its entire codomain. Thus, $\phi$ is an isomorphism.

Hence we regard the projection $\delta_{\ell^\infty(\X)}$ that was defined in subsection \ref{introduction.A} as an element of $\ell^\infty(\X \times \X^*)$. We show that this projection corresponds to the predicate $E_\X$ in the sense that $E_\X(X_1 \tensor X_2^*) = L(X_1 \tensor X_2^*, \CC) \cdot \delta_{\ell^\infty(\X)}(X_1 \tensor X_2^*)$
for all $X_1, X_2 \in \At(\X)$.

\begin{lemma}\label{appendix.H.1}
Let $X$ be a finite-dimensional Hilbert space, and let $\zeta\:X \tensor X^* \to \CC$ be a bounded functional. Then, $\zeta\in \CC \counit_X$ if and only if $\zeta \cdot (p \tensor (1-p)) = 0$ for all projections $p \in L(X)$.
\end{lemma}

\begin{proof}
Assume that $\zeta = \alpha \counit_X$ for some $\alpha \in \CC$, and let $p \in L(X)$ be a projection. Let $x_1 \in pX$ and let $x_2^* \in (1-p)^*X^*$. Then, $\zeta(x_1 \tensor x_2^*) = \alpha \langle x_2 | x_1 \rangle = 0$. We conclude that $\zeta$ vanishes on $pX \tensor (1-p)^*X^* = (p \tensor (1-p)^*) (X \tensor X^*)$, so $\zeta \cdot (p \tensor (1-p)^*)) = 0$.

Assume that $\zeta \cdot (p \tensor (1-p)^*) = 0$ for all projections $p \in L(X)$. Define $a \in L(X)$ by $\langle x_2 | a x_1\rangle = \zeta  (x_1 \tensor x_2^*)$ for all $x_1, x_2 \in X$. If $x_1$ and $x_2$ are orthogonal, then there is a projection $p \in L(X)$ such that $px_1 = x_1$ and $px_2 = 0$, so $\langle x_2 | a x_1\rangle = \zeta (x_1 \tensor x_2^*) = (\zeta \cdot (p \tensor (1-p)^*)) (x_1 \tensor x_2^*) = 0$. It follows that $a \in \CC 1_X$ and therefore that $\zeta \in \CC \counit_X$.
\end{proof}

\begin{lemma}\label{appendix.H.2}
Let $\X$ be a quantum set, let $r$ be a projection in $\ell^\infty(\X \times \X^*)$, and let $R$ be the corresponding relation of arity $(\X, \X^*)$, i.e., the relation defined by $R(X_1 \tensor X_2^*) = L(X_1 \tensor X_2^*, \CC) \cdot r(X_1 \tensor X_2^*)$, for $X_1, X_2 \in \At(\X)$. Then, $r$ is orthogonal to $p \tensor (1-p)^*$ for every projection $p \in \ell^\infty(\X)$ if and only if $R(X \tensor X^*) \leq \CC \counit_X$ for all $X \in \At(\X)$ and $R(X_1 \tensor X^*_2) = 0$ for all distinct $X_1, X_2 \in \At(\X)$.
\end{lemma}

\begin{proof}
Fix $X_1, X_2 \in \At(\X)$. For all projections $p$ in $\ell^\infty(\X)$, the condition $r(X_1 \tensor X_2^*) \cdot (p \tensor (1-p)^*)(X_1 \tensor X^*_2) = 0$ is clearly equivalent to $R(X_1 \tensor X_2^*) \cdot (p(X_1) \tensor (1-p(X_2)^*)) = 0$. If $X_1$ and $X_2$ are distinct, then the latter condition holds for all projections $p$ if and only if $R(X_1 \tensor X_2^*) = 0$ because there is a projection $p$ such that $p(X_1) = 1 $ and $p(X_2) = 0$. If $X_1$ and $X_2$ are identical, then by Lemma \ref{appendix.H.1}, the latter condition holds for all projections $p$ if and only if the elements of $R(X_1 \tensor X_2^*)$ are all scalar multiples of $\counit_{X_1}$. We vary $X_1, X_2 \in \At(\X)$ to conclude the statement of the proposition.
\end{proof}

\begin{proposition}\label{appendix.H.3}
Let $\X$ be a quantum set. Then, $E_\X(X_1 \tensor X_2) = L(X_1 \tensor X_2, \CC) \cdot \delta_{\ell^\infty(\X)}(X_1 \tensor X_2)$ for all $X_1, X_2 \in \At(\X)$.
\end{proposition}

\begin{proof}
The canonical one-to-one correspondence between projections $r$ in $\ell^\infty(\X \times \X^*)$ and relations $R$ of arity $(\X, \X^*)$ is an isomorphism of orthomodular lattices. Thus, by Lemma \ref{appendix.H.2}, the largest projection $r$ that is orthogonal to $p \tensor (1-p)^*$ for all projections $p \in \ell^\infty(\X)$, namely $r= \delta_{\ell^\infty(\X)}$, corresponds to the largest relation $R$ that is less than or equal to $E_\X$, namely $R = E_\X$. In other words, $E_\X(X_1 \tensor X_2) = L(X_1 \tensor X_2, \CC) \cdot \delta_{\ell^\infty(\X)}(X_1 \tensor X_2)$ for all $X_1, X_2 \in \At(\X)$, as claimed.
\end{proof}

\subsection{Canonical isomorphisms}\label{appendix.J}

We prove Proposition \ref{computation.A.1}, which concerns Definition \ref{definition.B.4}.

\begin{proposition}\label{appendix.J.1}
Let $\X_1, \ldots, \X_n$ be quantum sets, and let $\pi$ be a permutation of $\{1, \ldots, n\}$. Let $U_\pi$ be the canonical isomorphism \cite{MacLane}*{Thm.~XI.1.1} from $\X_1\times \cdots \times \X_n$ to $\X_{\pi(1)} \times \cdots \times \X_{\pi(n)}$ in the symmetric monoidal category of quantum sets and binary relations \cite{Kornell}*{sec.~3}. Then, $\pi_\#(R) = R \circ U_\pi$ for all relations $R$ of arity $(\X_{\pi(1)}, \ldots, \X_{\pi(n)})$.
\end{proposition}

\begin{proof}
We first consider the special case that $\pi$ exchanges $m, m+1 \in \{1, \ldots, n\}$, leaving the other elements fixed. Then, $U_\pi = I_{\X_1} \times \cdots \times I_{\X_{m-1}} \times B_{\X_m, \X_{m+1}} \times I_{\X_{m+2}} \times \cdots \times I_{\X_n}$, where $B_{\X_m, \X_{m+1}}$ is the braiding from $\X_m \times \X_{m+1}$ to $\X_{m+1} \times \X_m$. By the definition of this braiding, $B_{\X_m, \X_{m+1}}(X_m \tensor X_{m+1},  X_{m+1} \tensor X_m) = \CC b_{X_m, X_{m+1}}$ for all $X_m \in \At(\X_m)$ and $X_{m+1} \in \At(\X_{m+1})$, with the other components of $B_{\X_m, \X_{m+1}}$ vanishing. Here, $b_{X_m, X_{m+1}}$ is the braiding from $X_m \tensor X_{m+1}$ to $X_{m+1} \tensor X_m$ in the symmetric monoidal category of finite-dimensional Hilbert spaces and linear operators. In other words, $b_{X_m, X_{m+1}} = u_\pi$ in the notation of Definition \ref{definition.B.4}.

Let $\X_\pi = \X_{\pi(1)} \times \cdots \times \X_{\pi(n)}$. We compute that for all $X_1 \in \At(\X_1)$, $\X_2 \in \At(\X_2)$, etc., 
\begin{align*}
(R \circ U_\pi)(X_1 \tensor \cdots \tensor X_n, \CC)
& =
\sum_{X_\pi \in \At(\X_\pi)} R(X_\pi, \CC) \cdot U_\pi(X_1 \tensor \cdots \tensor X_n, X_\pi)
\\ & =
R(X_{\pi(1)} \tensor \cdots \tensor X_{\pi(n)}, \CC) \cdot \CC u_\pi
=
\pi_\#(R)(X_1 \tensor \cdots \tensor X_n).
\end{align*}
Therefore, $\pi_\#(R) = R \circ U_\pi$ whenever $\pi$ exchanges two consecutive elements of $\{1, \ldots, n\}$, leaving the other elements fixed.

The general case follows from the fact that any permutation of $\{0,\ldots,n\}$ is a product of such $2$-cycles. Let $\cat{qRel}$ be the category of quantum sets and binary relations, and let $\cat{FdHilb}$ be the category of finite-dimensional Hilbert spaces and linear operators. For each permutation $\pi$ of $\{1, \ldots, n\}$, we may regard $U_\pi$ as a natural transformation of functors $\cat{qRel}^n \to \cat{qRel}$, and similarly, we may regard $u_\pi$ as a natural transformation of functors $\cat{FdHilb}^n \to \cat{FdHilb}$. By category theory \cite{MacLane}*{Thm.~XI.1.1}, we have that $U_{\pi_2 \circ \pi_1} = U_{\pi_1} \circ U_{\pi_2}$ and $u_{\pi_2 \circ \pi_1} = u_{\pi_1} \circ u_{\pi_2}$ for all permutations $\pi_1$ and $\pi_2$ of $\{1, \ldots, n\}$, where the composition symbol denotes the ``horizontal'' composition of natural transformations \cite{MacLane}*{sec.~II.5}. It follows that for all relations $R$ of arity $(\X_{(\pi_2 \circ \pi_1)(1)}, \ldots, \X_{(\pi_2 \circ \pi_1)(n)})$, we have that $R \circ U_{\pi_2 \circ \pi_1} = R \circ U_{\pi_1} \circ U_{\pi_2}$ and similarly that $(\pi_2 \circ \pi_1)_\#(R) = \pi_{2\#}(\pi_{1\#}(R))$.

Overall, we find that the set of all permutations $\pi$ of $\{1, \ldots,n\}$ such that $\pi_\#(R) = R \circ U_\pi$ for all $n$-ary relations $R$ contains all the permutations that exchange two consecutive elements of $\{1,\ldots, n\}$ and is closed under composition. We conclude that this set consists of all the permutations of $\{1, \ldots,n\}$, and the proposition is proved.
\end{proof}

\subsection{Quantifying over ordinary sets.}\label{appendix.C} Let $A$ be a set. We show that existential quantification over $`A$ reduces to disjunction in the expected way. As usual, we view each element $a \in A$ also as a function $\{*\} \to A$, and we identify $`\{\ast\}$ with the monoidal unit $\mathbf 1$ of the dagger compact category of quantum sets and binary relations.

\begin{lemma}\label{appendix.C.1}
Let $A$ be a set, and let $\Y_1, \ldots, \Y_n$ be quantum sets.
If $\phi(x, y_1, \ldots, y_n)$ is a nonduplicating formula with $x$ being a variable of sort $`A$ and $y_1, \ldots, y_n$ being variables of sorts $\Y_1, \ldots, \Y_n$, respectively, then
\begin{align*}
\[(y_1, \ldots, y_n) \in \Y_1\times \cdots & \times \Y_n \suchthat (\exists x \fin `A)\,\phi(x,y_1, \ldots, y_n)\]
\\ &=
\bigvee_{a \in A}
\[(y_1, \ldots, y_n) \in \Y_1\times \cdots \times \Y_n \suchthat \phi(`a,y_1, \ldots, y_n)\].
\end{align*}
Similarly, if $\psi(x_*, y_1, \ldots, y_n)$ is a nonduplicating formula with $x_*$ being a variable of sort $`A^*$ and $y_1, \ldots, y_n$ being variables of sorts $\Y_1, \ldots, \Y_n$, respectively, then
\begin{align*}
\[(y_1, \ldots, y_n) \in \Y_1\times \cdots & \times \Y_n \suchthat (\exists x_* \fin `A^*)\,\phi(x_*,y_1, \ldots, y_n)\]
\\ &=
\bigvee_{a \in A}
\[(y_1, \ldots, y_n) \in \Y_1\times \cdots \times \Y_n \suchthat \phi(`a_*,y_1, \ldots, y_n)\].
\end{align*}
\end{lemma}

\begin{proof}
This lemma follows from Proposition \ref{computation.C.3}. The binary relation $\top_{`A}^\dagger$ can be written as a disjunction
$$
\begin{aligned}
      \begin{tikzpicture}[scale=1]
	\begin{pgfonlayer}{nodelayer}
		\node (1) at (0,0) {$\bullet$};
		\node [style = none] (2) at (0,1.2) {};
	\end{pgfonlayer}
	\begin{pgfonlayer}{edgelayer}
	    \draw [arrow] (1.center) to (2);
	\end{pgfonlayer}\end{tikzpicture}
\end{aligned}
\quad = 
\quad
\bigvee_{a \in A}
\left(
\begin{aligned}
      \begin{tikzpicture}[scale=1]
	\begin{pgfonlayer}{nodelayer}
		\node [style = none] (2) at (0,1.2) {};
		\node [style = box] (4) at (0,0){$`a$};
	\end{pgfonlayer}
	\begin{pgfonlayer}{edgelayer}
	    \draw [arrow] (4) to (2);
	\end{pgfonlayer}\end{tikzpicture}
\end{aligned}
\right).
$$
Therefore,
\begin{align*}
\[ (\exists x \fin `A)\,\phi(x,y_1, \ldots, y_n)\]
\quad  & = \quad
\bigvee_{a \in A}
\left( \;
\begin{aligned}
\begin{tikzpicture}[scale=1]
	\begin{pgfonlayer}{nodelayer}
        \node [style=box] (0) at (0,0) {\,$\[\phi(x, y_1, \ldots, y_n)\]$\,};
        \node [style=none] (A) at (-1.2,-0.2) {};
        \node [style=box]  (Abox) at (-1.2,-1) {$`a_{\phantom{*}}$};
        \node [style=none] (C) at (-0.4,-0.2) {};
        \node [style=none] (E) at (1.4,-0.2) {};
        \node [style=none] (C1) at (-0.4,-1.4) {};
        \node [style=none] (E1) at (1.4,-1.4) {};
        \node [style=none] (D) at (0.5,-0.85){$\cdots\cdots$};
	\end{pgfonlayer}
	\begin{pgfonlayer}{edgelayer}
	    \draw [arrow] (Abox.north) to (A);
	    \draw [arrow] (C1) to (C);
	    \draw [arrow] (E1) to (E);
	\end{pgfonlayer}
   \end{tikzpicture}
\end{aligned}
\; \right)
\\ & = \quad
\bigvee_{a \in A}\;
\[\phi(`a, y_1, \ldots, y_n)\].
\end{align*}
In the second case, the proof is entirely similar because $\top^\dagger_{`A^*} = \bigvee_{a \in A} `a_*$.
\end{proof}

\begin{lemma}\label{appendix.C.2}
Let $A$ be a set, and let $\Y_1, \ldots, \Y_n$ be quantum sets. Let $\phi(x, x_*, y_1, \ldots, y_n)$ be a nonduplicating formula with $x$ being a variable of sort $`A$, $x_*$ being a variable of sort $`A^*$ and $y_1, \ldots, y_n$ being variables of sorts $\Y_1, \ldots \Y_n$, respectively. Then,
\begin{align*}
\[(y_1, \ldots, y_n) \in \Y_1\times \cdots & \times \Y_n \suchthat (\exists (x \feq x_*) \fin `A \ftimes `A^*)\,\phi(x,x_*,y_1, \ldots, y_n)\]
\\ &=
\bigvee_{s \in A}
\[(y_1, \ldots, y_n) \in \Y_1\times \cdots \times \Y_n \suchthat \phi(`a,`a_*,y_1, \ldots, y_n)\].
\end{align*}
\end{lemma}

\begin{proof}
This lemma follows from Theorem \ref{computation.D.2}. The identity on $`A$ can be written as a disjunction
$$
\begin{aligned}
      \begin{tikzpicture}[scale=1]
	\begin{pgfonlayer}{nodelayer}
		\node (1) at (0,0) {$\scriptstyle `A$};
		\node [style = none] (2) at (0,1) {};
		\node [style = none] (3) at (0,-1) {};
	\end{pgfonlayer}
	\begin{pgfonlayer}{edgelayer}
	    \draw [arrow] (1) to (2);
	    \draw [arrow] (3) to (1);
	\end{pgfonlayer}\end{tikzpicture}
\end{aligned}
\quad = 
\quad
\bigvee_{s \in A}
\left(
\begin{aligned}
      \begin{tikzpicture}[scale=1]
	\begin{pgfonlayer}{nodelayer}
		\node [style = none] (2) at (0,1) {};
		\node [style = none] (3) at (0,-1) {};
		\node [style = box] (4) at (0,0.3){$`a^{\phantom \dagger}$};
		\node [style = box] (5) at (0,-0.3){$`a^\dagger$};
	\end{pgfonlayer}
	\begin{pgfonlayer}{edgelayer}
	    \draw [arrow] (4) to (2);
	    \draw [arrow] (3) to (5);
	\end{pgfonlayer}\end{tikzpicture}
\end{aligned}
\right).
$$
Therefore,
\begin{align*}
\[ (\exists (x \feq x_*) \fin `A \ftimes `A^*)\,\phi(x,x_*,y_1, \ldots, y_n)\]
\quad  & = \quad
\bigvee_{a \in A}
\left( \;
\begin{aligned}
\begin{tikzpicture}[scale=1]
	\begin{pgfonlayer}{nodelayer}
        \node [style=box] (0) at (0,0) {\,$\[\phi(x,x_*, y_1, \ldots, y_n)\]$\,};
        \node [style=none] (A) at (-1.6,-0.2) {};
        \node [style=box]  (Abox) at (-1.6,-1) {$`a_{\phantom{*}}$};
        \node [style=none] (B) at (-0.7,-0.2) {};
        \node [style=box]  (Bbox) at (-0.7,-1) {$`a_*$};
        \node [style=none] (C) at (0,-0.2) {};
        \node [style=none] (E) at (1.8,-0.2) {};
        \node [style=none] (C1) at (0,-1.4) {};
        \node [style=none] (E1) at (1.8,-1.4) {};
        \node [style=none] (D) at (0.9,-0.85){$\cdots\cdots$};
	\end{pgfonlayer}
	\begin{pgfonlayer}{edgelayer}
	    \draw [arrow] (Abox.north) to (A);
	    \draw [arrow,markat=0.7] (B) to (Bbox.north);
	    \draw [arrow] (C1) to (C);
	    \draw [arrow] (E1) to (E);
	\end{pgfonlayer}
   \end{tikzpicture}
\end{aligned}
\; \right)
\\ & = \quad
\bigvee_{a \in A}\;
\[\phi(`a,`a_*,y_1, \ldots, y_n)\]. \qedhere
\end{align*}
\end{proof}

\subsection{Quantifier laws}\label{appendix.G} We show that the existential quantifier distributes over disjunction, just as it does classically, and consequently, the universal quantifier distributes over conjunction, just as it does classically. We also show that the existential quantifier commutes with conjunction by a second formula, subject to the restriction that the two formulas have no free variables in common at all, and consequently, the universal quantifier commutes with disjunction by a second formulas, subject to the same restriction.

\begin{proposition}\label{appendix.G.1}
Let $\X_1, \ldots, \X_n$ and $\Y_1, \ldots, \Y_m$ be quantum sets. Let $x_1, \ldots, x_n$ and $y_1, \ldots, y_m$ be distinct variables of sorts $\X_1, \ldots, \X_n$ and $\Y_1, \ldots, \Y_m$, respectively. Let $\phi(x_1, \ldots, x_n)$ and $\psi(y_1, \ldots, y_m)$ be nonduplicating formulas. If $\X_2 = \X_1^*$, then 
\begin{align*}&
\[(x_3, \ldots, x_n, y_1, \ldots y_m) \in \X_3 \times \cdots \times \X_n \times \Y_1 \times \cdots \times \Y_m 
\\ & \hspace{35ex} \suchthat 
(\exists (x_1 \feq x_2) \fin \X_1 \ftimes \X_1^*)\, (\phi(x_1\ldots, x_n) \AND \psi(y_1, \ldots, y_m))\] 
\\ &=
\[(x_3, \ldots, x_n) \in \X_3 \times \cdots \times \X_n \suchthat (\exists (x_1 \feq x_2) \fin \X_1 \ftimes \X_1^*)\, \phi(x_1\ldots, x_n)\] \\ & \hspace{10ex} \AND \[(y_1, \ldots y_m) \in \Y_1 \times \cdots \times \Y_m \suchthat  \psi(y_1, \ldots, y_m)\].
\end{align*}
\end{proposition}

\begin{proof}
We reason graphically, applying Theorem \ref{computation.D.2}:
\begin{align*}
&\[(\exists (x_1 \feq x_2) \fin \X_1 \ftimes \X_1^*)\, (\phi(x_1\ldots, x_n) \AND \psi(y_1, \ldots, y_m))\]
\; = \quad
\begin{aligned}
\begin{tikzpicture}[scale=1.1]
\begin{pgfonlayer}{nodelayer}
    \node [style=box] (Phi) at (0,0) {$\[\phi(x_1, \ldots, x_n)\]$};
    \node [style=box] (Psi) at (3,0) {$\[\psi(y_1, \ldots, y_m)\]$};
    \node [style=none] (x1anch) at (-1,-0.2) {};
    \node [style=none] (x2anch) at (-0.5,-0.2) {};
    \node [style=none] (x3anch) at (0,-0.2) {};
    \node [style=none] (xnanch) at (1,-0.2) {};
    \node [style=none] (x3) at (0,-1) {$\scriptstyle x_3$};
    \node [style=none] (xdots) at (0.5,-0.6) {$\cdots$};
    \node [style=none] (xn) at (1,-1) {$\scriptstyle x_n$};
    \node [style=none] (y1anch) at (2,-0.2) {};
    \node [style=none] (ymanch) at (4,-0.2) {};
    \node [style=none] (y1) at (2,-1) {$\scriptstyle y_1$};
    \node [style=none] (ydots0) at (3,-0.6) {$\cdots \cdots$};
    \node [style=none] (ym) at (4,-1) {$\scriptstyle y_m$};
\end{pgfonlayer}
\begin{pgfonlayer}{edgelayer}
    \draw [arrow] (x3) to (x3anch);
    \draw [arrow] (xn) to (xnanch);
    \draw [arrow, bend left = 90, looseness=3] (x2anch) to (x1anch);
    \draw [arrow] (y1) to (y1anch);
    \draw [arrow] (ym) to (ymanch);
\end{pgfonlayer}
\end{tikzpicture}
\end{aligned}
\\ & = \; \[(\exists (x_1 \feq x_2) \fin \X_1 \ftimes \X_1^*)\, \phi(x_1\ldots, x_n)\] \AND \[\psi(y_1, \ldots, y_m)\].
\end{align*}
\end{proof}

\begin{proposition}\label{appendix.G.2}
Let $\X_1, \ldots, \X_n$ be quantum sets, and let $x_1, \ldots, x_n$ be distinct variables of sorts $\X_1, \ldots, \X_n$, respectively. Let $\phi(x_1, \ldots, x_n)$ and $\psi(x_1, \ldots, x_n)$ be nonduplicating formulas. If $\X_2 = \X_1^*$, then
\begin{align*}
&\[(x_3, \ldots, x_n) \in \X_3 \times \cdots \times \X_n \suchthat (\exists (x_1 \feq x_2) \fin \X_1 \ftimes \X_1^*)\, (\phi(x_1, \ldots, x_n) \OR \psi(x_1, \ldots, x_n))\]
\\ & =
\[(x_3, \ldots, x_n) \in \X_3 \times \cdots \times \X_n \suchthat (\exists (x_1 \feq x_2) \fin \X_1 \ftimes \X_1^*)\, \phi(x_1, \ldots, x_n)\] 
\\ & \hspace{10ex} \OR
\[(x_3, \ldots, x_n) \in \X_3 \times \cdots \times \X_n \suchthat (\exists (x_1 \feq x_2) \fin \X_1 \ftimes \X_1^*)\, \psi(x_1, \ldots, x_n)\].
\end{align*}
\end{proposition}

\begin{proof} We appeal to Theorem \ref{computation.D.2} and to the fact that composition distributes over the join of binary relations between quantum sets:
\begin{align*}
&\[(x_3, \ldots, x_n) \in \X_3 \times \cdots \times \X_n \suchthat (\exists (x_1 \feq x_2) \fin \X_1 \ftimes \X_1^*)\, (\phi(x_1, \ldots, x_n) \OR \psi(x_1, \ldots, x_n))\]
\\ &=
(
\[(x_1, \ldots, x_n) \in \X_1 \times \cdots \times \X_n \suchthat \phi(x_1, \ldots, x_n)\]
\OR
\[(x_1, \ldots, x_n) \in \X_1 \times \cdots \times \X_n \suchthat \psi(x_1, \ldots, x_n)\]
)
\\ & \hspace{10ex} \circ
(E_{\X_1}^\dagger \times I_{\X_3} \times \cdots \times I_{\X_n})
\\ & =
(\[(x_1, \ldots, x_n) \in \X_1 \times \cdots \times \X_n \suchthat \phi(x_1, \ldots, x_n)\] \circ 
(E_{\X_1}^\dagger \times I_{\X_3} \times \cdots \times I_{\X_n}))
\\ & \hspace{10ex} \OR
(\[(x_1, \ldots, x_n) \in \X_1 \times \cdots \times \X_n \suchthat \psi(x_1, \ldots, x_n)\] \circ 
(E_{\X_1}^\dagger \times I_{\X_3} \times \cdots \times I_{\X_n}))
\\ & =
\[(x_3, \ldots, x_n) \in \X_3 \times \cdots \times \X_n \suchthat (\exists (x_1 \feq x_2) \fin \X_1 \ftimes \X_1^*)\, \phi(x_1, \ldots, x_n)\]
\\ & \hspace{10ex}\OR
\[(x_3, \ldots, x_n) \in \X_3 \times \cdots \times \X_n \suchthat (\exists (x_1 \feq x_2) \fin \X_1 \ftimes \X_1^*)\, \psi(x_1, \ldots, x_n)\].
\end{align*}
\end{proof}

\begin{lemma}\label{appendix.G.3}
Let $\X_1, \ldots, \X_n$ be quantum sets, and let $x_1, \ldots, x_n$ be distinct variables of sorts $\X_1, \ldots, \X_n$, respectively. Let $\psi(x_3, \ldots, x_n)$ be a nonduplicating formula. If $\X_2 = \X_1^*$ and $\X_1 \neq `\emptyset$, then
\begin{align*}
&\[(x_3, \ldots, x_n) \in \X_3 \times \cdots \times \X_n \suchthat (\exists (x_1 \feq x_2) \fin \X_1 \ftimes \X_1^*)\, \psi(x_3, \ldots, x_n)\]
\\ & =
\[(x_3, \ldots, x_n) \in \X_3 \times \cdots \times \X_n \suchthat \psi(x_3, \ldots, x_n)\].
\end{align*}
\end{lemma}

\begin{proof} We reason graphically:
\begin{align*}
\[(x_1, \ldots, x_n) \in \X_1 \times \cdots \times \X_n \suchthat \psi(x_3, \ldots, x_n)\]
\; = \quad
\begin{aligned}
\begin{tikzpicture}
\begin{pgfonlayer}{nodelayer}[scale=1.1]
    \node [style=box] (Psi) at (3,0) {$\[\psi(x_3, \ldots, x_n)\]$};
    \node [style=none] (x3anch) at (2,-0.2) {};
    \node [style=none] (xnanch) at (4,-0.2) {};
    \node [style=none] (x3) at (2,-1) {$\scriptstyle \X_3$};
    \node [style=none] (xdots0) at (3,-0.6) {$\cdots \cdots$};
    \node [style=none] (xn) at (4,-1) {$\scriptstyle \X_n$};
    \node [style=none] (x2) at (1.25,-1) {$\scriptstyle \X_1$};
    \node [style=none] (x1) at (0.5,-1) {$\scriptstyle \X_1$};
    \node [style=none] (x2anch) at (1.25,0) {$\bullet$};
    \node [style=none] (x1anch) at (0.5,0) {$\bullet$};
\end{pgfonlayer}
\begin{pgfonlayer}{edgelayer}
    \draw [arrow] (x3) to (x3anch);
    \draw [arrow] (xn) to (xnanch);
    \draw [arrow, markat = 0.35] (x1) to (x1anch.center);
    \draw [arrow, markat =0.8] (x2anch.center) to (x2);
\end{pgfonlayer}
\end{tikzpicture}
\end{aligned}\quad ;
\end{align*} 

\begin{align*} &
\[(x_3, \ldots, x_n) \in \X_3 \times \cdots \times \X_n \suchthat (\exists (x_1 \feq x_2) \fin \X_1 \ftimes \X_1^*)\, \psi(x_3, \ldots, x_n)\]
\\ & = \;
\[(x_1, \ldots, x_n) \in \X_1 \times \cdots \times \X_n \suchthat \psi(x_3, \ldots, x_n)\] \circ 
(E_{\X_1}^\dagger \times I_{\X_3} \times \cdots \times I_{\X_n})
 \\ & = \quad
\begin{aligned}
\begin{tikzpicture}[scale=1.1]
\begin{pgfonlayer}{nodelayer}
    \node [style=box] (Psi) at (3,0) {$\[\psi(x_3, \ldots, x_n)\]$};
    \node [style=none] (x3anch) at (2,-0.2) {};
    \node [style=none] (xnanch) at (4,-0.2) {};
    \node [style=none] (x3) at (2,-1) {$\scriptstyle \X_3$};
    \node [style=none] (xdots0) at (3,-0.6) {$\cdots \cdots$};
    \node [style=none] (xn) at (4,-1) {$\scriptstyle \X_n$};
    \node [style=none] (x2anch) at (1.25,0) {$\bullet$};
    \node [style=none] (x1anch) at (0.5,0) {$\bullet$};
\end{pgfonlayer}
\begin{pgfonlayer}{edgelayer}
    \draw [arrow] (x3) to (x3anch);
    \draw [arrow] (xn) to (xnanch);
    \draw [arrow, bend left = 90, looseness=3] (x2anch.center) to (x1anch.center);
\end{pgfonlayer}
\end{tikzpicture}
\end{aligned}
\quad = \quad
\begin{aligned}
\begin{tikzpicture}[scale=1.1]
\begin{pgfonlayer}{nodelayer}
    \node [style=box] (Psi) at (3,0) {$\[\psi(x_3, \ldots, x_n)\]$};
    \node [style=none] (x3anch) at (2,-0.2) {};
    \node [style=none] (xnanch) at (4,-0.2) {};
    \node [style=none] (x3) at (2,-1) {$\scriptstyle \X_3$};
    \node [style=none] (xdots0) at (3,-0.6) {$\cdots \cdots$};
    \node [style=none] (xn) at (4,-1) {$\scriptstyle \X_n$};
\end{pgfonlayer}
\begin{pgfonlayer}{edgelayer}
    \draw [arrow] (x3) to (x3anch);
    \draw [arrow] (xn) to (xnanch);
\end{pgfonlayer}
\end{tikzpicture}
\end{aligned}
\\ & = \;
\[(x_3, \ldots, x_n) \in \X_3 \times \cdots \times \X_n \suchthat \psi(x_3, \ldots, x_n)\].
\end{align*}
The assumption that $\X_1 \neq `\emptyset$ is used in the second-to-last equality.
\end{proof}

\begin{proposition}\label{appendix.G.4}
Let $\X_1, \ldots, \X_n$ be quantum sets, and let $x_1, \ldots, x_n$ be distinct variables of sorts $\X_1, \ldots, \X_n$, respectively. Let $\phi(x_1, \ldots, x_n)$ and $\psi(x_3, \ldots, x_n)$ be nonduplicating formulas. If $\X_2 = \X_1^*$ and $\X_1 \neq `\emptyset$, then
\begin{align*}
&\[(x_3, \ldots, x_n) \in \X_3 \times \cdots \times \X_n \suchthat (\exists (x_1 \feq x_2) \fin \X_1 \ftimes \X_1^*)\, (\phi(x_1, \ldots, x_n) \OR \psi(x_3, \ldots, x_n))\]
\\ & =
\[(x_3, \ldots, x_n) \in \X_3 \times \cdots \times \X_n \suchthat (\exists (x_1 \feq x_2) \fin \X_1 \ftimes \X_1^*)\, \phi(x_1, \ldots, x_n)\] 
\\ & \hspace{10ex} \OR
\[(x_3, \ldots, x_n) \in \X_3 \times \cdots \times \X_n \suchthat \psi(x_3, \ldots, x_n)\].
\end{align*}
\end{proposition}

\begin{proof} 
We combine Proposition \ref{appendix.G.2} with Lemma \ref{appendix.G.3}.
\end{proof}

\section*{Acknowledgments}

I thank Matthew Daws for his reply to my question about discrete quantum groups \cite{Daws}. I thank Chris Heunen for suggesting a number of references and for his guidance toward producing the wire diagrams. I thank Bert Lindenhovius for valuable discussion. I thank Piotr So\l tan for his help in understanding quantum groups. I thank Stefaan Vaes for generously contributing the example of discrete quantum groups. Finally, I thank Dominic Verdon and Andreas Winter for the opportunity to share an early version of these results with the Focused Research Group on Noncommutative Mathematics and Quantum Information.

\end{document}